\documentclass[10pt]{amsart}
\usepackage[english]{babel}
\usepackage{csquotes}
\usepackage{bbm}
\usepackage{dsfont}									
\usepackage[inline,shortlabels]{enumitem}

\usepackage{scrextend}						

\usepackage{appendix}

\usepackage[normalem]{ulem}

\usepackage[hyperfootnotes=false]{hyperref} 
\usepackage{color}
\hypersetup{
	colorlinks = true,
	linkcolor = {blue},
	urlcolor = {red},
	citecolor = {blue}
}

\usepackage{verbatim}

\usepackage{amsmath}
\usepackage{amsthm}
\usepackage{amssymb}
\usepackage{mathrsfs} 
\usepackage{mathtools}
\usepackage{bm}

\usepackage{graphicx}
\usepackage{tikz}
\usepackage{tikz-cd}
\usetikzlibrary{calc, positioning, arrows, patterns}
\tikzcdset{arrow style=tikz}

\newcounter{results}[section]

\theoremstyle{plain}
\newtheorem{theorem}[results]{Theorem}

\newtheorem{lemma}[results]{Lemma}
\newtheorem{proposition}[results]{Proposition}
\newtheorem{corollary}[results]{Corollary}

\theoremstyle{remark}
\newtheorem{remark}[results]{Remark}

\theoremstyle{definition}
\newtheorem{definition}[results]{Definition}

\numberwithin{equation}{section}

\allowdisplaybreaks

\usepackage{xparse}

\ExplSyntaxOn
\NewDocumentCommand{\makeabbrev}{mmm}
 {
  \yoruk_makeabbrev:nnn { #1 } { #2 } { #3 }
 }

\cs_new_protected:Npn \yoruk_makeabbrev:nnn #1 #2 #3
 {
  \clist_map_inline:nn { #3 }
   {
    \cs_new_protected:cpn { #2 } { #1 { ##1 } }
   }
 }
 \ExplSyntaxOff

\makeabbrev{\textbf}{tbf#1}{a,b,c,d,e,f,g,h,i,j,k,l,m,n,o,p,q,r,s,t,u,v,w,x,y,z,A,B,C,D,E,F,G,H,I,J,K,L,M,N,O,P,Q,R,S,T,U,V,W,X,Y,Z}

\makeabbrev{\textbf}{bf#1}{a,b,c,d,e,f,g,h,i,j,k,l,m,n,o,p,q,r,s,t,u,v,w,x,y,z,A,B,C,D,E,F,G,H,I,J,K,L,M,N,O,P,Q,R,S,T,U,V,W,X,Y,Z}

\makeabbrev{\textsf}{tsf#1}{a,b,c,d,e,f,g,h,i,j,k,l,m,n,o,p,q,r,s,t,u,v,w,x,y,z,A,B,C,D,E,F,G,H,I,J,K,L,M,N,O,P,Q,R,S,T,U,V,W,X,Y,Z}

\makeabbrev{\mathsf}{mss#1}{a,b,c,d,e,f,g,h,i,j,k,l,m,n,o,p,q,r,s,t,u,v,w,x,y,z,A,B,C,D,E,F,G,H,I,J,K,L,M,N,O,P,Q,R,S,T,U,V,W,X,Y,Z}

\makeabbrev{\mathfrak}{mf#1}{a,b,c,d,e,f,g,h,i,j,k,l,m,n,o,p,q,r,s,t,u,v,w,x,y,z,A,B,C,D,E,F,G,H,I,J,K,L,M,N,O,P,Q,R,S,T,U,V,W,X,Y,Z}

\makeabbrev{\mathrm}{mrm#1}{a,b,c,d,e,f,g,h,i,j,k,l,m,n,o,p,q,r,s,t,u,v,w,x,y,z,A,B,C,D,E,F,G,H,I,J,K,L,M,N,O,P,Q,R,S,T,U,V,W,X,Y,Z}

\makeabbrev{\mathbf}{mbf#1}{a,b,c,d,e,f,g,h,i,j,k,l,m,n,o,p,q,r,s,t,u,v,w,x,y,z,A,B,C,D,E,F,G,H,I,J,K,L,M,N,O,P,Q,R,S,T,U,V,W,X,Y,Z}

\makeabbrev{\mathcal}{mc#1}{A,B,C,D,E,F,G,H,I,J,K,L,M,N,O,P,Q,R,S,T,U,V,W,X,Y,Z}

\makeabbrev{\mathbb}{mbb#1}{A,B,C,D,E,F,G,H,I,J,K,L,M,N,O,P,Q,R,S,T,U,V,W,X,Y,Z}

\makeabbrev{\mathscr}{ms#1}{A,B,C,D,E,F,G,H,I,J,K,L,M,N,O,P,Q,R,S,T,U,V,W,X,Y,Z}

\makeabbrev{\mathrm}{#1}{
id,ran,rk,diag,stab,ann,pr,ev,End,Hom,sgn,im,op,can,fin,ext,red,tot,Aut,Ad,ad,hor,ver,rad,
rot,usc,lsc,LocLip,bSymLip,osc,loc,uloc,spec,coz,z,ul,
Opt,Adm,Cpl,Geo,GeoOpt,GeoAdm,GeoCpl,reg,aut,
bd,co,Ric,Exp,dExp,seg,Seg,cut,fcut,Cut,SDiff,Iso,Isom,cl,Homeo,Diff,Der,vol,dvol,inj,relint,Graph,sub,eucl,Tube,
var,law,Gam,pa,so,iso,fs,inv,pqi,mix,inn,
TestF,
}

\makeabbrev{\mathsf}{#1}{CD,BE,MCP,Ent,wMTW,MTW,RCD,QCD,EVI,Irr,SC,wFe,VA,UP,Curv,Alex,CAT}

\makeabbrev{\textsc}{#1}{df,fv,tv}

\renewcommand{\paragraph}[1]{\medskip\emph{#1}.\quad}
\newcommand{\paragraphn}[1]{\medskip\emph{#1}}

\newcommand{\eqdef}{\coloneqq}
\newcommand{\defeq}{\eqqcolon}
\newcommand{\pfwd}{\sharp}
\newcommand{\emp}{\varnothing}
\newcommand{\emparg}{{\,\cdot\,}}

\newcommand{\trid}{\star}
\newcommand{\as}[1]{\quad #1\text{-a.e. }}
\newcommand{\forallae}[1]{\quad \text{for } #1\text{-a.e. } }

\newcommand{\abs}[1]{\left\lvert#1\right\rvert}						

\newcommand{\norm}[1]{\left\lVert#1\right\rVert}					
\newcommand{\set}[1]{\left\{#1\right\}}							
\newcommand{\tset}[1]{\big\{#1\big\}}							
\newcommand{\seq}[1]{\left(#1\right)}							
\newcommand{\tseq}[1]{\big(#1\big)}							
\newcommand{\paren}[1]{\left(#1\right)}							
\newcommand{\tparen}[1]{\big({#1}\big)}
\newcommand{\braket}[1]{\left[#1\right]}							
\newcommand{\tbraket}[1]{\big[#1\big]}							

\newcommand{\tscalar}[2]{\big\langle #1 \, \big |\, #2\big\rangle}			
\newcommand{\scalar}[2]{\left\langle #1 \,\middle |\, #2\right\rangle}		

\newcommand{\FC}[4]{\mcF^{#1}_{#2}\mcC^{#3}_{#4}}
\newcommand{\hFC}[4]{\mcF^{#1}_{#2}\widehat\mcC^{#3}_{#4}}

\newcommand{\toplus}{{\scriptscriptstyle\oplus}}

\DeclareMathOperator{\car}{\mathds 1}
\newcommand{\Li}[2][]{\mathrm{L}_{#1}(#2)}						

\newcommand{\comma}{\,\,\mathrm{,}\;\,}
\newcommand{\comm}{\,\,\mathrm{,}\;\,}
\newcommand{\semicolon}{\,\,\mathrm{;}\;\,}
\newcommand{\fstop}{\,\,\mathrm{.}}

\newcommand{\DF}[1]{\mcD_{#1}}
\newcommand{\LP}[1]{\mcL_{#1}}
\newcommand{\GP}[1]{\mcG_{#1}}

\newcommand{\dom}[1]{\msD(#1)}								

\newcommand{\grad}{\boldnabla}
\newcommand{\gradW}{\boldnabla^{\hor}}
\newcommand{\gradH}{\boldnabla^{\ver}}
\renewcommand{\div}{\mathrm{div}}

\newcommand{\acts}{\circlearrowleft} 

\newcommand{\multi}{\rmB_b(X)}
\newcommand{\Multi}{\mfM}
\newcommand{\Shift}{\mfS}
\newcommand{\normaliz}{N}
\newcommand{\isomor}{J}
\newcommand{\pb}{{}^*}

\newcommand{\ball}[3][]{\mrmB_{{#1}}(#2,#3)}

\newcommand{\boldnabla}{\boldsymbol\nabla}
\newcommand{\boldGamma}{\boldsymbol\Gamma}

\newcommand{\C}{\mathbb{C}}

\newcommand{\N}{\mathbb{N}}
\newcommand{\Q}{\mathbb{Q}}
\newcommand{\R}{\mathbb{R}}
\newcommand{\X}{\mathbb{X}}
\newcommand{\Y}{\mathbb{Y}}
\newcommand{\Z}{\mathbb{Z}}
\renewcommand{\AA}{\mathscr{A}}

\newcommand{\LL}{\mathscr{L}}

\newcommand{\TT}{\boldsymbol{T}}

\newcommand{\mm}{{\mbox{\boldmath$m$}}}

\newcommand{\vv}{{\mbox{\boldmath$v$}}}

\newcommand{\Beta}{\mrmB}
\newcommand{\aalpha}{{\mbox{\boldmath$\alpha$}}}

\newcommand{\ggamma}{{\mbox{\boldmath$\gamma$}}}

\newcommand{\eeta}{{\mbox{\boldmath$\eta$}}}
\newcommand{\iiota}{{\boldsymbol \iota}}

\newcommand{\jjmath}{{\boldsymbol \jmath}}

\newcommand{\sfd}{{\sf d}}
\newcommand{\sfe}{{\sf e}}

\newcommand{\sfr}{{\sf r}}

\newcommand{\sfx}{{\sf x}}

\newcommand{\sfD}{{\sf D}}

\newcommand{\sfQ}{{\mathsf Q}}

\newcommand{\rmd}{{\mathrm d}}
\newcommand{\rme}{{\mathrm e}}

\newcommand{\rmA}{{\mathrm A}}

\newcommand{\rmC}{{\mathrm C}}

\newcommand{\rmB}{{\mathrm B}}
\newcommand{\rmD}{{\mathrm D}}

\newcommand{\Kliminf}{K\kern-3pt-\kern-2pt\mathop{\rm lim\,inf}\limits}  
\newcommand{\supp}{\mathop{\rm supp}\nolimits}   
\newcommand{\LSC}{\operatorname{LSC}}          
\newcommand{\Lip}{\mathop{\rm Lip}\nolimits}          
\newcommand{\Lipb}{\mathop{\rm Lip}_b\nolimits}          

\newcommand{\lip}{\mathop{\rm lip}\nolimits}          

\renewcommand{\d}{{ \mathrm d}}

\newcommand{\restr}[1]{\big\lvert_{#1}}

\newcommand{\Leb}[1]{{\mathscr L}^{#1}}      
\newcommand{\la}{{\langle}}                  
\newcommand{\ra}{{\rangle}}

\newcommand{\eps}{\varepsilon}  
\newcommand{\nchi}{{\raise.3ex\hbox{$\chi$}}}
\newcommand{\weakto}{\rightharpoonup}
\newcommand{\Cb}{\rmC_b}

\newcommand{\Cc}{\rmC_c}

\newcommand{\de}{{\,\rm d}}

\renewcommand{\mm}{\mathfrak m}

\newcommand{\J}I

\newcommand{\CE}{\mathsf{C\kern-1pt E}}
\newcommand{\NE}{\mathsf{N\kern-2.5pt E}}
\newcommand{\wCE}{\mathsf{wC\kern-1pt E}}

\newcommand{\pCE}{\mathsf{pC\kern-1pt E}}

\newcommand{\kkappa}{{\boldsymbol \kappa}}

\newcommand{\mres}{\mathbin{\vrule height 1.6ex depth 0pt width
0.13ex\vrule height 0.13ex depth 0pt width 1.3ex}}

\newcommand{\dY}{{\delta}}
\newcommand{\W}{\mathbb W}

\newcommand{\diff}{\mrmd}

\newcommand{\pc}{\f{C}[U, U]} 
\newcommand{\HK}{{\mathsf H\!\!\mathsf K}}
\newcommand{\gd}{\mathsf{g}^{\varrho}}
\newcommand{\GHK}{\mathsf {G\!H}\!\!\mathsf K}
\newcommand{\f}[1]{\mathfrak{#1}} 
\newcommand{\ac}[1]{\mathsf{AC}^2\left ( #1 \right )}

\renewcommand{\W}{\mathsf{W}}
\newcommand{\He}{\mathsf{He}}

\newcommand{\anon}[1]{#1} 

\title[The Hellinger--Kantorovich metric measure geometry]{The Hellinger--Kantorovich\\metric measure geometry\\on spaces of measures}

\anon{
\author{Lorenzo Dello Schiavo}
\address{Lorenzo Dello Schiavo: Dipartimento di Matematica -- Universit\`a degli studi di Roma ``Tor Vergata'', Via della Ricerca Scientifica~1, 00133 Roma, Italy}
\email{delloschiavo@mat.uniroma2.it}

\author{Giacomo Enrico Sodini}
\address{Giacomo Enrico Sodini: Institut für Mathematik - Fakultät für Mathematik - Universität Wien, Oskar-Morgenstern-Platz 1, 1090 Wien, Austria}
\email{giacomo.sodini@univie.ac.at}
}


\begin{document}

\begin{abstract}
Let~$(M,g)$ be a  Riemannian manifold  with Riemannian distance~$\mssd_g$, and~$\mcM(M)$ be the space of all non-negative Borel measures on~$M$, endowed with the Hellinger--Kantorovich distance~$\HK_{\mssd_g}$ induced by~$\mssd_g$.

Firstly, we prove that~$\tparen{\mcM(M),\HK_{\mssd_g}}$ is a universally infinitesimally Hilbertian metric space, and that a natural class of cylinder functions is dense in energy in the Sobolev space of every finite Borel measure on~$\mcM(M)$.

Secondly, we endow~$\mcM(M)$ with its canonical reference measure, namely A.M.~Vershik's multiplicative infinite-dimensional Lebesgue measure~$\LP{\theta}$,~$\theta>0$, and we consider
\begin{enumerate*}[$(a)$]
\item\label{i:abstract:A} the \emph{geometric structure} on~$\mcM(M)$ induced by the natural action on~$\mcM(M)$ of the semi-direct product of diffeomorphisms and densities on~$M$, under which~$\LP{\theta}$ is the unique invariant measure; and
\item\label{i:abstract:B} the \emph{metric measure structure} of~$\tparen{\mcM(M),\HK_{\mssd_g},\LP{\theta,\nu}}$, inherited from that of~$(M,\mssd_g,\vol_g)$.
\end{enumerate*}
We identify the canonical Dirichlet form~$\tparen{\mcE,\dom{\mcE}}$ of~\ref{i:abstract:A} with the Cheeger energy of~\ref{i:abstract:B}, thus proving that these two structures coincide.
We further prove that~$\tparen{\mcE,\dom{\mcE}}$ is a conservative quasi-regular strongly local Dirichlet form on~$\mcM(M)$, recurrent if and only if~$\theta\in (0,1]$, and properly associated with the Brownian motion of the Hellinger--Kantorovich geometry on~$\mcM(M)$.
\end{abstract}

\maketitle

\newpage

\setcounter{tocdepth}{3}
\tableofcontents
\thispagestyle{empty}

\section{Introduction}

Let~$(M,g)$ be a  smooth, connected, orientable, complete  Riemannian manifold with Riemannian distance~$\mssd_g$. 
We denote by~$\mcM(M)$ the cone of all non-negative and finite Borel measures on~$M$, endowed with the \emph{Hellinger--Kantorovich distance}~$\HK_{\mssd_g}$ induced by~$\mssd_g$, see~\S\ref{sss:IntroHK} below.

Firstly, for a suitable class~$\FC{}{}{}{}$ of cylinder functions on~$\mcM(M)$, we prove the following Myers--Serrin-type theorem.

\begin{theorem}
Let~$\mcQ$ be any non-negative Borel measure on~$\tparen{\mcM(M),\HK_{\mssd_g}}$ with all exponential moments, i.e.\ such that
\[
\int \rme^{-t \mu M} \, \diff \mcQ(\mu) < + \infty \quad \text{ for every } t>0 \fstop
\]
Then, the space $\FC{}{}{}{}$ is dense in $2$-energy in the metric Sobolev space
\[
H^{1,2}\tparen{\mcM(M), \HK_{\mssd_g}, \mcQ}\comma
\]
and the latter is a Hilbert space.
\end{theorem}

Secondly, we focus on a specific choice for~$\mcQ$.
For a parameter~$\theta>0$ and a probability measure~$\nu \in \mcP(M)$, we consider Vershik's \emph{infinite-dimensional multiplicative Lebesgue measure}~$\LP{\theta,\nu}$ on~$\mcM(M)$, see~\S\ref{sss:InfDimLeb} below, and we show that it is the unique natural measure for the Hellinger--Kantorovich geometry on~$\mcM(M)$.
Further denote by~$\grad$ the gradient for real-valued functions on~$\mcM(M)$ associated to the Hellinger--Kantorovich geometry, see~\S\ref{sss:DifferentiationIntro}, and by $\scalar{\emparg}{\emparg}_\mu$ a suitably weighted scalar product, see~\eqref{eq:DirichletForm0}.

\begin{theorem}
For every~$\theta>0$, the canonical energy form
\[
\mcE(u,v)\eqdef \int \scalar{(\grad u)_\mu}{(\grad v)_\mu}_\mu \diff\LP{\theta,\nu}(\mu)\comma \qquad u,v\in\FC{}{}{}{}\comma
\]
is closable on~$L^2(\LP{\theta,\nu})$. Its closure~$\tparen{\mcE,\dom{\mcE}}$
\begin{itemize}
\item is a conservative quasi-regular strongly local Dirichlet form on~$L^2(\LP{\theta,\nu})$;
\item coincides with the Cheeger energy of the metric measure space
\[
\tparen{\mcM(M), \HK_{\mssd_g}, \LP{\theta,\nu}}\semicolon
\]
\item is properly associated with a Hunt diffusion with state space~$\mcM(M)$, the `Brownian motion' of the Hellinger--Kantorovich geometry on~$\mcM(M)$.
\end{itemize}
Finally, the latter process is recurrent if~$\theta\in (0,1]$, and transient if~$\theta\in (1,\infty)$.
\end{theorem}

We proceed to explain the motivations and significance of our results.

\subsection{Large Lie groups and stochastic partial differential equations}
A \emph{large} Lie group is a Lie group modelled on an infinite-dimensional Hilbert, Banach, or Fréchet space.
Prototypical examples of such groups are 
\begin{itemize}[wide]
\item \emph{transformation groups}, as diffeomorphism groups of differential manifolds;
\item $G$-\emph{current groups}, i.e.\ groups of $G$-valued functions, for some Lie group~$G$;
\item \emph{multiplier groups}, i.e.\ (Abelian) $(\R^+,\cdot)$-current groups, which are maximal toral subgroups in the corresponding $SL_2$-current groups.
\end{itemize}

In spite of their Lie-group structure, these groups are typically wild, displaying a number of strange phenomena that do not occur in the standard finite-dimensional theory.
For instance, concerning diffeomorphism groups: the exponential map from the corresponding Lie algebra into the group fails to be open, hence it is not surjective even on the identity component of the group~\cite{Mil84}, its image is meager~\cite{Pal74}, and in fact extremely small~\cite{Gra88}.

A fruitful approach to the study of large Lie groups via their \emph{representations} has been the subject of a longstanding programme initiated for diffeomorphism groups by A.M.~Vershik,~I.M.~Gel'fand, and M.I.~Graev in~\cite{VerGelGra75}, and continued by Vershik and Graev for current groups~\cite{VerGra08,VerGra09}, and more recently by Yu.G.~Kondratiev, E.W.~Lytvynov, and Vershik for certain semidirect products of diffeomorphisms and multipliers in~\cite{KonLytVer15}.

In order for the representations of these groups to be \emph{faithful} ---i.e.\ for them to retain sufficient information on the group--- the representations need to be constructed on some `large' Hilbert space.
Especially in the case of diffeomorphisms and of multipliers, a concrete realization of such a Hilbert space is the space $L^2(\mcQ)$ of some measure~$\mcQ$ on a space of measures. 
Indeed, diffeomorphisms naturally act on measures by push-forward, and multipliers simply act on measures by multiplication by densities (hence the name).
When~$\mcQ$ is a probability measure, it is usually regarded as (the law of) a \emph{random measure}, typically, a random point process. 
This is the case in:~\cite{VerGelGra75}, concerned with \emph{Poisson} point processes; \cite{KonLytVer15}, concerned with \emph{Gamma} compound Poisson point processes; and~\cite{LzDS19a,LzDS17+}, concerned with \emph{Dirichlet--Ferguson} point processes.

\paragraph{Geometric Brownian motions}
As already noted in~\cite{LzDS17+,KonLytVer15}, this action of diffeomorphisms, multipliers, or a combination thereof, on a space of measures induces an energy functional on~$L^2(\mcQ)$. 
As it turns out, the functional is, in many of these settings, a \emph{Dirichlet form}, and it is therefore uniquely associated with a measure-valued Markov process.
We call this process the \emph{geometric measure-valued Brownian motion} induced by the group action.

On the one hand, it is one goal of the aforementioned programme to study properties of the representation on $L^2(\mcQ)$ of a given large Lie group via the corresponding geometric measure-valued Brownian motion.
For instance, it is usually expected that invariant sets of this Brownian motion are in one-to-one correspondence with irreducible sub-representations of the group action.

On the other hand, geometric measure-valued Brownian motions are very interesting stochastic processes in their own right.
This is readily seen from two important examples in the case when~$\mcQ$ is the Dirichlet--Ferguson measure. (See below.)
In this case, one process induced by the action of multipliers is the \emph{Fleming--Viot process with parent independent mutation}~\cite{FleVio79,OveRoeSch95}, while one process induced by the action of diffeomorphisms is the \emph{Dirichlet--Ferguson diffusion}~\cite{LzDS17+}, the unique solution to the \emph{Dean--Kawasaki stochastic partial differential equation with singular drift}~\cite{KonvRe18,KonvRe17,LzDS24}.

\paragraph{Metric-measure Brownian motions}
Another fundamental approach to the construction of energy functionals on spaces of measures is as follows.
When the space of measures in question is endowed with some natural distance (e.g., Hellinger, Bhattacharyya, Kantorovich--Rubinstein a.k.a.\ Wasserstein, Hellinger--Kantorovich a.k.a.\ Wasser\-stein--Fisher--Rao, etc.) and with a reference random measure~$\mcQ$, we consider the \emph{Cheeger energy} of the resulting \emph{metric measure space}.

For $L^p$-Kantorovich--Rubinstein distances on spaces of probability measures, the study of these energy functionals has been undertaken in~\cite{FSS22,S22}.
Here, we rather consider the space of all non-negative finite measures with the \emph{Hellinger--Kantorovich distance}. (See below.)
Also in this case, the Cheeger $L^2$-energy is a quadratic functional, and thus a Dirichlet form.
We call the unique Markov process associated to it the \emph{metric measure measure-valued Brownian motion} induced by the distance.

\paragraph{Main results: Identification}
It is one main result of this work that, for a specific choice of~$\mcQ$, the geometric point of view (group actions) and the metric-measure point of view (distances) are one and the same, i.e.\ that the geometric measure-valued Brownian motion coincides with the metric measure Brownian motion just described.

Not only this provides an identification of the stochastic process in question; it will also grant us the possibility to import tools from metric measure geometry in the study of geometric Brownian motions and of the corresponding representations for a given group action, and vice versa to use the Lie-group construction for the study of the metric measure space arising from the Hellinger--Kantorovich distance and the reference measure~$\mcQ$.

In future work, we will address the complete identification of this Brownian motion as the unique solution to some singular stochastic partial differential equation with measure-valued solutions.

\bigskip

Let us now present in greater detail the constructions we touched on above.

\subsection{The geometric point of view}\label{ss:IntroGeometric}
Let~$X$ be a Polish topological space.
Denote by~$\rmC_b(X)$, resp.~$\rmC_0(X)$, the space of all continuous bounded, resp.\ continuous vanishing at infinity, real-valued functions on~$X$.
For a non-negative and finite Borel measure~$\mu$ on~$X$ and a Borel function~$f\colon X\to\R$, write
\[
f^\trid\mu \eqdef \int_X f \diff\mu
\]
whenever the integral makes sense, and~$\mu_x\eqdef \mu\set{x}$. Analogously, for a vector $\mbff = (f_1, \dotsc, f_k)$, $k \in \N$, we set $\mbff^\trid \eqdef (f_1^\trid, \dotsc, f_k^\trid)$.
The push-forward of a measure~$\mu$ by a measurable map~$T$ is the measure
\[
T_\pfwd\mu\eqdef \mu\circ T^{-1}\fstop
\]

\paragraph{Spaces of measures}
We denote by~$\mcM(X)$, resp.~$\mcP(X)$, the space of all non-negative and finite Borel, resp.\ Borel probability, measures on~$X$, always endowed with the weak topology, i.e.\ the coarsest topology under which the maps~$\mu\mapsto f^\trid \mu$ are continuous for every~$f\in\rmC_b(X)$, and with the corresponding Borel $\sigma$-algebra.
As it is well-known, both~$\mcM(X)$ and~$\mcP(X)$ are Polish, and thus standard Borel spaces.
Let~$\normaliz\colon \mcM(X) \to \mcP(X)\cup\set{0}$ be the \emph{normalization map}
\[
\normaliz\colon \mu\longmapsto \tfrac{1}{\mu X}\mu\comma \qquad \mu\in\mcM(X)\comma
\]
where, conventionally,~$\normaliz(0)=0$ is the zero measure.
Note that there is a Borel bi-measurable isomorphism~$\isomor\colon \mcM(X) \to \mcP(X)\times \R^+$
\begin{equation}\label{eq:IsomorphismMbp}
\isomor\colon \mu \longmapsto \tparen{\normaliz(\mu), \mu X}\comma \qquad \mu\in\mcM(X)\fstop
\end{equation}
For~$\msS$ either~$\mcM(X)$ or~$\mcP(X)$, we write~$\msS^{\pa}$ for the subspace of~$\msS$ consisting of all purely atomic measures in~$\msS$.

\subsubsection{Group actions}
For a topological group~$G$ acting measurably on a measurable space~$(\Omega,\msF)$, we write~$\acts\colon G\times \Omega\to \Omega$, $(g,\omega)\mapsto g.\omega$ for its action.
We consider the following groups acting on~$\mcM(X)$.

\paragraph{Multipliers}
Denote by~$\multi$ the space of real-valued bounded Borel functions on~$X$, regarded as an Abelian Lie algebra with the pointwise product of functions. For~$\nu\in\mcM(X)$, further define the \emph{$\nu$-traceless} subalgebra~$\multi_\nu\eqdef \set{f\in \multi: \int f \de \nu=0}$ of~$\multi$.
The corresponding Abelian Lie groups are the groups of \emph{multipliers}~$\Multi(X)\eqdef \set{e^a: a\in\multi}$ and its subgroup~$\Multi_\nu(X)\eqdef \set{e^a: a\in\multi_\nu}$, both endowed with pointwise product of functions.

The group~$\Multi(X)$ (hence all its subgroups) acts naturally on~$\mcM(X)$ by setting
\begin{equation}\tag{$\acts_\cdot$}\label{eq:ActionMultipliers}
e^a. \colon \mu\longmapsto e^a \cdot \mu\comma \qquad a\in \multi\comma \qquad \mu\in\mcM(X)\fstop
\end{equation}

\paragraph{Shifts}
Denote by~$\Shift(X)$ the group of \emph{shifts}, i.e.\ Borel bi-measurable bijections of~$X$ with the composition of functions.
The group~$\Shift(X)$ (hence all its subgroups) acts naturally on~$\mcM(X)$ (hence on~$\mcP(X)$), by setting
\begin{equation}\tag{$\acts_{\!\pfwd}$}\label{eq:ActionShifts}
\psi. \colon \mu\longmapsto \psi_\pfwd\mu \comma \qquad \psi\in \Shift(X) \comma \qquad \mu\in\mcM(X)\fstop
\end{equation}
The action~\ref{eq:ActionShifts} commutes on~$\mcM(X)$ with the normalization of measures and thus factors over the map~$J$ in~\eqref{eq:IsomorphismMbp} in the sense that~$J(\psi_\pfwd\mu)=(\psi_\pfwd\normaliz(\mu),\mu (X))$.

For~$\nu\in\mcM(X)$, further denote by~$\Shift_\nu(X)$ the subgroup of~$\Shift(X)$ consisting of all elements fixing~$\nu$, or, equivalently,~$\normaliz(\nu)$.

\paragraph{Products}
We denote by~$\psi^*k\eqdef k\circ\psi$ the \emph{pull-back} of a function~$k$ by a map~$\psi$. 
The pull-back operator~$\pb\colon \psi\longmapsto \psi^*$ is a group homomorphism~$\pb\colon \Shift(X)\to \Aut(\Multi(X))$ on~$\Shift(X)$ into the automorphism group~$\Aut(\Multi(X))$ of $\Multi(X)$.
Thus, $\Shift(X)$~acts on~$\Multi(X)$ by automorphisms via~$\pb$, viz.
\[
\acts_\pb\colon (k,\psi) \longmapsto (\pb(\psi)).k = k\circ \psi \fstop
\]
This action\footnote{In order to avoid confusion, we shall always consider \emph{left} actions. However, this choice forces the somewhat unusual definition of \emph{right} semidirect product in~\ref{eq:ActionSemidirect}. Indeed, note the indices in the composition~$(\psi_2^*a_1)a_2$. This also motivates the difference of our action from the one in~\cite{KonLytVer15}.} induces a \emph{right} semidirect product~$\Shift(X)\rtimes_\pb \Multi(X)$ with group operation and inverse
\begin{equation*}
\begin{aligned}
h_1h_2=&\ \tparen{\psi_1\circ\psi_2, (\psi_2^*k_1) k_2} \comma 
\\
h^{-1}=&\ \tparen{ \psi^{-1},(\psi^{-1})^*\tfrac{1}{k}}\comma
\end{aligned}
\qquad h_i=(\psi_i,k_i)\in \Shift(X)\rtimes_\pb \Multi(X)\comma \quad i=1,2,\emp \fstop
\end{equation*}
The product~$\Shift(X)\rtimes_\pb \Multi(X)$ acts naturally on~$\mcM(X)$ by setting
\begin{equation}\tag{$\acts_{\!\rtimes_\pb}$}\label{eq:ActionSemidirect}
h. \colon \mu\longmapsto \psi_\pfwd(k\cdot\mu) \comma \qquad h=(\psi,k)\in \Shift(X)\rtimes_\pb \Multi(X)\comma \qquad \mu\in\mcM(X)\fstop
\end{equation}

\subsubsection{Differentiation in the smooth category}\label{sss:DifferentiationIntro}
The groups $\Multi(X)$ and~$\Shift(X)$ and the group actions~\ref{eq:ActionMultipliers} and~\ref{eq:ActionShifts} above are given in the measurable category.
The same actions restrict to subgroups of~$\Multi(X)$ and~$\Shift(X)$ in other categories, e.g.\ the continuous category, where~$\Multi(X)$ is replaced by the subgroup of all its continuous functions and~$\Shift(X)$ by the group of self-homeomorphisms of~$X$.
The same applies to the action~\ref{eq:ActionSemidirect} provided that both~$\Multi(X)$ and~$\Shift(X)$ are restricted to the same category, in order for the semidirect product to be defined in that category.

The smooth category will be of particular interest. Such restriction is possible when~$X$ is endowed with a structure of \emph{smooth} (i.e.~$\rmC^\infty$-smooth), connected,  orientable  manifold, henceforth denoted by~$M$.
In this case the restrictions of both $\Multi(M)$ and~$\Shift(M)$ are Lie groups, and we may discuss the corresponding Lie algebras.
Indeed, we may replace~$\rmB_b(M)$ with the (Abelian Lie) subalgebra~$\Cc^\infty(M)$ of smooth compactly supported functions on~$M$, the corresponding Lie group~$\exp[\Cc^\infty(M)]$ being defined in the obvious way as a subgroup of~$\Multi(M)$.
The suitable restriction of~$\Shift(M)$ is the group~$\Diff^+_0(M)$ of orientation-preserving compactly non-identical (smooth) (self-)diffeomorphisms of~$M$, corresponding to the Lie algebra~$\mfX^\infty_c(M)$ of (smooth) compactly supported vector fields on~$M$ with the standard Lie bracket of vector fields.

In the following, let us replace~$X$ by a manifold~$M$ as above.
For each~$w\in\mfX^\infty_c(M)$ we denote by~$\psi^w_t$ the \emph{flow} of~$w$ at time~$t\in\R$, satisfying
\[
\begin{aligned}
\diff_t \psi^w_t(x) &= w\tparen{\psi^w_t(x)}
\\
\psi^w_0(x)&= x
\end{aligned}\comma
\qquad x\in M\comma t\in\R \fstop
\]
Since~$w$ is compactly supported,~$\psi^w_t$ is well-defined everywhere on~$M$ and an element of~$\Diff^+_0(M)$ for every~$t\in\R$, with inverse~$(\psi^w_t)^{-1}=\psi^w_{-t}$.
The (\emph{Lie}) \emph{exponential} of~$\mfX^\infty_c(M)$ is then the map
\[
\exp\colon w \longmapsto \psi^w_1 \comma \qquad w\in\mfX^\infty_c(M)\fstop
\]

Let us write~$\mfG(M)\eqdef \Diff^+_0(M)\rtimes_\pb \exp[\Cc^\infty(M)]$ for the semidirect product of~$\Diff^+_0(M)$ and~$\exp[\Cc^\infty(M)]$ induced by~\ref{eq:ActionSemidirect}, and
\[
\mfG_{\!\scriptscriptstyle\times}(M)\eqdef \Diff^+_0(M)\times \exp[\Cc^\infty(M)]
\]
for their direct product.
The same (right) action~\ref{eq:ActionSemidirect} and the semidirect product~$\mfG(M)$ have been previously considered by T.~Gallou\"et and F.-X.~Vialard in~\cite[Eqn.~(2.29)]{GalVia17}, where~$\mfG(M)$ is interpreted as the automorphism group of the principal fiber bundle of half-densities\footnote{In the terminology and notation of~\cite{GalVia17}, the \emph{space of half-densities} is~$\Lambda_{1/2}\eqdef \rmC^\infty(M;\R^+)=\exp[\rmC^\infty(M)]$. Note that~$\Lambda_{1/2}$ is a group under pointwise multiplication, and that our group of (compactly non-identical smooth) multipliers~$\exp[\rmC^\infty_c(M)]$ is a subgroup thereof.} on~$M$, see~\cite[\S2.4]{GalVia17}, also cf.~\cite[\S3.2.1]{GalGheVia21}.

Both~$\mfG(M)$ and~$\mfG_{\!\scriptscriptstyle\times}(M)$ are Lie groups, see e.g.~\cite[\S5.16, p.~48, Eqn.~(3)]{KolSloMic93}, and we denote by~$\mfg(M)$, resp.~$\mfg_{\scriptscriptstyle\oplus}(M)$, the corresponding Lie algebras.
As vector spaces, both~$\mfg(M)$ and~$\mfg_{\scriptscriptstyle\oplus}(M)$ are linearly isomorphic to the direct sum~$\mfX^\infty_c(M)\oplus \Cc^\infty(M)$.
However,~$\mfg(M)$ is different from~$\mfg_{\scriptscriptstyle\oplus}(M)$, i.e.\ their brackets do not coincide, and the same holds for their exponentials~$\exp^{\mfg(M)}\colon\mfg(M)\to\mfG(M)$ and~$\exp^{\mfg_{\scriptscriptstyle\oplus}(M)}\colon \mfg_{\scriptscriptstyle\oplus}(M)\to \mfG_{\!\scriptscriptstyle\times}(M)$.

\paragraph{Directional derivatives}
For every sufficiently smooth ---to be clarified later on--- function~$u\colon \mcM(M)\to\R$, for every~$w\in\mfX^\infty_c(M)$ and every~$a\in\Cc^\infty(M)$, we may define the directional derivatives in the directions~$w$ and~$a$ by setting
\begin{equation}\label{eq:SingleDirDer}
(\partial_{w} u)_\mu \eqdef \diff_t\restr{t=0} u (\psi^w_t.\mu) \quad \text{and} \quad (\partial_a u)_\mu \eqdef \diff_t\restr{t=0} u (e^{ta}.\mu) \fstop
\end{equation}

Analogously, for every pair~$(w,a)\in\mfg(M)$, we may define a directional derivative in the direction~$(w,a)$, viz.
\[
(\partial_{w,a}u)_\mu=(\partial_{w,a}u) (\mu)\eqdef \diff_t\restr{t=0} u \paren{\exp^{\mfg(M)}\tparen{t(w,a)}.\mu} \fstop
\]
 A simple heuristic argument ---which can be made rigorous for finite-dimensional groups--- shows that~$\exp^{\mfg(M)}$ and~$\exp^{\mfg_{\scriptscriptstyle\oplus}(M)}$ are tangent at first order, that is
\[
\exp^{\mfg(M)}\tparen{t(w,f)}=(\psi^w_t,e^{ta})+o(t) \comma \qquad \abs{t}\ll 1 \fstop
\]
Thus, for every sufficiently smooth function~$u\colon \mcM(M)\to\R$ and every~$(w,a)\in\mfg(M)$,
\begin{equation}\label{eq:TotalDirDer}
(\partial_{w,a}u)_\mu= \diff_t\restr{t=0} u \tparen{(\psi^w_t, e^{ta}).\mu} = (\partial_w u)_\mu + (\partial_a u)_\mu \fstop
\end{equation}

The directional derivatives in~\eqref{eq:SingleDirDer} have been widely considered, see e.g.~\cite{AlbKonRoe98,Han02,LzDS19b,LzDS17+,RoeSch99,Sch97,vReStu09,RehRoe23}.
Notably, the directional derivative in~\eqref{eq:TotalDirDer} has been considered by Yu.G.~Kondratiev, E.W.~Lytvynov, and A.M.~Vershik in~\cite{KonLytVer15} as arising from the group action on~$\mcM(M)$ of the \emph{left} semidirect product of~$\Diff^+_0(M)$ and~$\exp[\rmC^\infty_c(M)]$.

\subsubsection{The Dirichlet form}\label{sss:TangentSpacesIntro}
Assume further that~$M$ is endowed with a (smooth) Riemannian metric~$g$ and set~$\abs{w}_g\eqdef g(w,w)^{1/2}$ for~$w\in\mfX^\infty_c(M)$.
For each~$\mu\in\mcM(M)$, this allows us to define pre-Hilbert norms on~$\mfX^\infty_c(M)$,~$\Cc^\infty(M)$, and~$\mfg(M)$, respectively by setting, for every $w\in\mfX^\infty_c(M)$ and~$a\in\Cc^\infty(M)$,
\begin{align}\label{eq:normu}
\norm{w}_{T^\hor_\mu}\eqdef \braket{\int_X \abs{w}_g^2 \diff\mu}^{1/2}\comma \quad \norm{a}_{T^\ver_\mu}\eqdef \braket{\int_X a^2 \diff \mu}^{1/2} \comma
\end{align}
so that
\begin{align}\label{eq:NormTot}
 \norm{(w,a)}_{T_\mu} \eqdef \sqrt{\norm{w}_{T^\hor_\mu}^2+\norm{a}_{T^\ver_\mu}^2} \fstop
\end{align}

\paragraph{Tangent spaces}
We respectively define the
\begin{itemize}[leftmargin=1.1em]
\item \emph{horizontal tangent space $T^\hor_\mu\mcM(M)$ to~$\mcM(M)$ at~$\mu$} as the completion of~$\mfX^\infty_c(M)$;
\item \emph{vertical tangent space $T^\ver_\mu\mcM(M)$ to~$\mcM(M)$ at~$\mu$} as the completion of~$\Cc^\infty(M)$;
\item (\emph{total}) \emph{tangent space $T_\mu\mcM(M)$ to~$\mcM(M)$ at~$\mu$} as the completion of~$\mfg(M)$.
\end{itemize}
The three spaces above are Hilbert spaces when endowed with the (non-relabeled) extensions of the respective pre-Hilbert norms.
As detailed below, these spaces have been widely considered in the literature.
Firstly, let us note that horizontal objects are occasionally called \emph{intrinsic} and vertical objects are occasionally called \emph{extrinsic}.
This terminology is motivated from the perspective of the geometry of~$\mcP(M)$, since the action of~\ref{eq:ActionShifts} leaves~$\mcP(M)$ invariant, while the action~\ref{eq:ActionMultipliers} does not.
We prefer the terminology of horizontal/vertical tangent space, since we mostly consider actions on~$\mcM(M)$, rather than on~$\mcP(M)$.

When~$\mu$ is a configuration, i.e.\ $\N$-valued on all compact sets, the space~$T^\hor_\mu\mcM(M)$ was first considered in~\cite{AlbKonRoe98,AlbKonRoe98b}.
When~$\mu$ is a probability measure in the $L^2$-Wasserstein space~$\mcP_2(M)$, it was considered in~\cite{GanKimPac10}.
When~$\mu\in\mcP(M)$ and in the context of group actions, it has been widely considered, e.g., in~\cite{LzDS17+,LzDS19b,RehRoe23,RenWan22}.
It is an extension to vector fields of non-gradient type of the classical tangent space to~$\mcP_2(M)$ in e.g.~\cite{Ott01, AmbGigSav08}, also cf.~\cite[\S{3.D.5}, pp.~155ff.]{Ebe99}.
For further comments on the terminology as well as for other notions of tangent spaces to~$\mcP_2(M)$ (hence to~$\mcM(M)$) see the Appendix to~\cite{LzDS19b} and references therein.

The space~$T^\ver_\mu\mcM(M)$ has been widely considered, again usually for measures in~$\mcP(M)$, occasionally with an equivalent norm, e.g.~\cite{OveRoeSch95,Sch97, Han02}, in relation with the Dirichlet form of the Fleming--Viot and related processes,~\cite{FleVio79}.

Finally, the space~$T_\mu\mcM(M)$ is an extension to vector fields of non-gradient type of the \emph{Hellinger--Kantorovich tangent space} in~\cite{LMS22,GM17}.
We have a natural orthogonal decomposition
\begin{equation}\label{eq:OrthoTangent}
T_\mu\mcM(M)\cong T^\hor_\mu\mcM(M) \oplus^{\perp_\mu} T^\ver_\mu\mcM(M)\comma
\end{equation}
where~$\perp_\mu$ denotes orthogonality w.r.t.\ the~$T_\mu\mcM(M)$-scalar product.

\paragraph{Gradients}Now, fix~$u\colon \mcM(M)\to\R$, sufficiently smooth. Whenever the linear operator
\[
(w,a)\longmapsto(\partial_{w,a} u)_\mu
\]
is $\norm{\,\cdot\,}_{T_\mu}$-bounded on~$\mfg(M)$, it extends uniquely to~$T_\mu\mcM(M)$.
Since the latter is a Hilbert space, by the standard Riesz Representation Theorem for Hilbert spaces this extension may be represented as
\begin{equation}\label{eq:GradRepresentation}
(w,a)\longmapsto \scalar{(\boldnabla u)_\mu}{(w,a)}_{T_\mu} \comma \qquad (w,a)\in T_\mu\mcM(M)\comma
\end{equation}
for some unique element~$(\boldnabla u)_\mu$ of~$T_\mu\mcM(M)$, satisfying
\[
\scalar{(\boldnabla u)_\mu}{(w,a)}_{T_\mu} = (\partial_{w,a} u)_\mu \comma \qquad (w,a)\in\mfg(M)\fstop
\]
We stress that the norm~$\norm{\emparg}_{T_\mu}$ and therefore the gradient~$\boldnabla u$ both depend on the choice of the Riemannian metric~$g$.
We assume~$g$ to be fixed and thus omit this dependence from the notation.

Finally, we denote by~$\gradW$, resp.~$\gradH$, the component of~$\boldnabla$ in~$T^\hor_\mu\mcM(M)$, resp.~$T^\ver_\mu\mcM(M)$, w.r.t.\ the orthogonal decomposition in~\eqref{eq:OrthoTangent}, satisfying
\begin{align} \label{eq:GradRepresentationhor}
\tscalar{(\gradW u)_\mu}{w}_{T^\hor_\mu}=&\ (\partial_w u)_\mu\comma & w&\in \mfX^\infty_c(M) \comma
\\ \label{eq:GradRepresentationver}
\tscalar{(\gradH u)_\mu}{a}_{T^\ver_\mu}=&\ (\partial_a u)_\mu\comma & a&\in \Cc^\infty(M) \comma
\end{align}
and
\begin{equation}\label{eq:OrthoGrad}
(\boldnabla u)_\mu = \tparen{(\gradH u)_\mu, (\gradW u)_\mu} \fstop
\end{equation}
The operators~$\boldnabla^\sharp$, with~$\sharp=\emp,\hor,\ver$, enjoy some of the properties that are expected of a natural gradient operator.
For example, it is readily verified from the standard Leibniz rule for~$\diff_t$ in~\eqref{eq:SingleDirDer} that they satisfy the \emph{Leibniz rule}
\begin{equation}\label{eq:Leibniz}
\boldnabla^\sharp (uv) = u\boldnabla^\sharp v + v\boldnabla^\sharp u
\end{equation}
thus acting as derivations on the space of smooth functions and taking values into the space of sections to the corresponding tangent bundles, viz.
\begin{equation}\label{eq:TangentBundles}
\boldnabla^\sharp u\colon \mcM(M) \longrightarrow T^\sharp\mcM(M)\eqdef \bigcup_{\mu\in\mcM(M)} T^\sharp_\mu\mcM(M)\comma \qquad \sharp= \emp,\hor,\ver\fstop
\end{equation}

\paragraph{Dirichlet form and cylinder functions}
Let~$\mcQ$ be a non-negative Radon ($\sigma$-finite, not necessarily finite) Borel measure on~$\mcM(M)$. 
Then,
\begin{equation}\label{eq:DirichletForm0}
\mcE(u,v)\eqdef \int \braket{ \scalar{(\gradW u)_\mu}{(\gradW v)_\mu}_{T^\hor_\mu}+ 4\scalar{(\gradH u)_\mu}{(\gradH v)_\mu }_{T^\ver_\mu}} \diff \mcQ(\mu)
\end{equation}
is a symmetric bilinear form on~$L^2(\mcM(M), \mcQ)$ defined on the class of sufficiently smooth functions. The presence of a factor~$4$ in the vertical-direction will be clarified in~\S\ref{sss:IntroHK} below.
\emph{If} this form is densely defined and closable in~$L^2(\mcM(M), \mcQ)$, it is readily seen that its closure is a local Dirichlet form.

Let us now turn to the definition of some class of smooth functions sufficiently large for our purposes and admitting a gradient as in \eqref{eq:GradRepresentation} (while in this introduction we focus only on one class of such functions, in the rest of the work we will use also other classes of functions).
For every linear functional of the form~$f^\trid$,~$f\in\Cc^\infty(M)$, we may easily compute
\[
(\partial_{w,a} f^\trid)_\mu= \int_M \tparen{(\diff f) w+f a}\diff\mu \comma
\]
where~$\diff f$ denotes the exterior differential of~$f$ on~$M$.
For~$w\in \mfX^\infty_c(M)$ and~$a\in\Cc^\infty(M)$, the functional~$(\partial_{w,a} f^\trid)_\mu$ is thus, again, the potential energy induced by~$(\diff f)w+f a \in \Cc^\infty(M)$.
Thus, the potential energy~$f^\trid$ is the prototypical smooth function, and we consider the algebra of smooth \emph{cylinder functions} induced by potential energies, viz.
\begin{align}\label{eq:cylman}
\FC{\infty,\infty}{c,c}{\infty}{c}\eqdef \set{\hat u\colon \mcM(M)\to\R : \begin{gathered} u=F\circ \mbff^\trid\comma F \in\rmC^{\infty}_{c}(\R^{k+1};\R) \comma
\\
k\in \N\comma \mbff\eqdef\tseq{ f_i}_{0\leq i\leq k} \comma f_0\equiv 1 \comma 
\\
 f_i\in \rmC^{\infty}_{c}\tparen{ M} \text{ for } 1\leq i\leq k\end{gathered}} \fstop
\end{align}

Note that every~$u\in \FC{\infty,\infty}{c,c}{\infty}{c} $ vanishes on measures with total mass~$\mu M> \kappa\eqdef \max\sup_{\boldsymbol t\in \R^k}\supp F(\emparg, t_1, \dotsc, t_k)$, and that~$u(\mu)=u(\mu\restr{K})$ where~$K\eqdef \cup_{i\leq N}\supp f_i$ is a compact subset of~$M$.
Since the subset~$\set{\mu\in\mcM(M): \mu M=\mu K\leq \kappa}$ is compact in~$\mcM(M)$ by Prokhorov's Theorem, we have~$\FC{\infty,\infty}{c,c}{\infty}{c}\subset L^2(\mcM(M),\mcQ)$ for every~$\mcQ$, and we may thus consider the form~$\tparen{\mcE,\FC{\infty,\infty}{c,c}{\infty}{c}}$.

\subsection{A candidate measure}
In this section we discuss a natural choice of the measure~$\mcQ$ in~\eqref{eq:DirichletForm0}, namely the \emph{multiplicative infinite-dimensional Lebesgue measure}.
We start by recalling the notion of (quasi-)invariance of a measure w.r.t.\ a group action.

\subsubsection{Quasi-invariance under group actions}\label{sss:PQI}
Let~$G$ be a topological group acting measurably on a $\sigma$-finite measure space~$(\Omega,\msF,\mcQ)$ and write~$\acts\colon G\times \Omega\to \Omega$, $(g,\omega)\mapsto g.\omega$ for the action.

\begin{definition}\label{d:PQI}
We say that~$\mcQ$ is:
\begin{enumerate}[$(a)$]
\item\label{i:d:PQI:1} ($\acts$-)\emph{quasi-invariant} if
\[
\mcQ_g\eqdef (g.)_\pfwd\mcQ= R_g\cdot \mcQ
\]
for some $\msF$-measurable Radon--Nikod\'ym derivative~$R_g\colon \Omega\to [0,\infty]$;

\item\label{i:d:PQI:2} \emph{projectively \emph{($\acts$-)}invariant} if, additionally,~$R_g$ is a constant function on~$\Omega$ (possibly depending on~$g$);

\item\label{i:d:PQI:3} ($\acts$-)\emph{invariant} if, additionally,~$\mcQ_g=\mcQ$ for every~$g\in G$;

\item\label{i:d:PQI:4} \emph{partially} ($\acts$-)\emph{quasi-invariant}~\cite[Dfn.~9]{KonLytVer15} if there exists a filtration~$\seq{\msF_t}_{t\in T}$ of~$\msF$, indexed by a (possibly uncountable) totally ordered set~$T$, so that
\begin{itemize}
\item $\msF$ is the minimal $\sigma$-algebra generated by~$\seq{\msF_t}_t$;
\item for each~$g\in G$ and~$s\in T$ there exists~$t\in T$ such that~$g.\msF_s\subset \msF_t$ (in which case it must be~$s\leq t$);
\item for each~$g\in G$ and~$t\in T$ the restriction~$\mcQ^t$ of~$\mcQ$ to~$\msF_t$ is $\acts$-quasi-invariant with $\msF_t$-measurable Radon--Nikod\'ym derivative~$R_{g,t}\colon \Omega\to [0,\infty]$, viz.~$(g.)_\pfwd \mcQ^t=R_{g,t} \cdot \mcQ^t$.
\end{itemize}
\end{enumerate}
Note that~\ref{i:d:PQI:3}$\implies$\ref{i:d:PQI:2}$\implies$\ref{i:d:PQI:1}$\implies$\ref{i:d:PQI:4}.
\end{definition}

These properties are related to the theory of representations of~$G$.
Indeed, each \mbox{($\acts$-quasi-)}inv\-ariant measure~$\mcQ$ induces a so-called (\emph{quasi-})\emph{regular representation} of~$G$ on the Hilbert space~$L^2(\mcM(X), \mcQ)$ by defining a unitary operator
\[
U^\mcQ_g\colon f\longmapsto R_g^{1/2} \cdot f \circ (g^{-1}.) \fstop
\]
If~$\mcQ$ is merely partially quasi-invariant, no representation of~$G$ is induced on $L^2(\mcM(X),\mcQ)$.
However, a representation is induced on~$L^2(\mcM(X),\mcQ^t)$ for every~$t$ and the family of all such representations may still be used to infer properties of the action on~$L^2(\mcM(X), \mcQ)$.
We refer the reader to the Introduction in~\cite{KonLytVer15} for further heuristics about partial quasi-invariance.

In the setting of~\S\ref{sss:DifferentiationIntro}, the (partial quasi-)invariance of~$\mcQ$ is also instrumental in establishing the closability of the form~\eqref{eq:DirichletForm0}.
Notable examples of this fact are the forms induced by: the Dirichlet--Ferguson measure (see below), the $\Diff^+_0(M)$\ref{eq:ActionShifts}-quasi-invariant \emph{entropic measure} in~\cite{vReStu09}, general $\Diff^+_0(M)$\ref{eq:ActionShifts}-quasi-invariant measures on configuration spaces in~\cite{RoeSch99}, and on~$\mcP(M)$ in~\cite{LzDS19b}.

We shall therefore seek for measures on~$\mcM(M)$ that are (partially) $\mfG(M)$\ref{eq:ActionSemidirect}-quasi-invariant.
Nonetheless, let us first address the case when~$X$ has no smooth structure.

\subsubsection{The Dirichlet--Ferguson measure}\label{sec:dfm}
Let~$I\eqdef [0,1]$ and denote by~$\Beta_\beta$ the Beta distribution of parameters~$1$ and~$\beta>0$, viz.
\[
\diff\Beta_\beta(t)\eqdef \beta(1-t)^{\beta-1}\diff t \comma \qquad t\in I\fstop
\]

Let~$\nu\in\mcP(X)$ be diffuse (i.e.\ atomless) and, for ease of notation, set
\begin{equation}\label{eq:Shorthand}
\mu^x_t\eqdef (1-t)\mu+t \delta_x\comma \qquad \mu\in\mcM(X)\comma x\in X\comma  t\in I\fstop
\end{equation}

The \emph{Dirichlet--Ferguson measure~$\DF{\beta\nu}$ with intensity~$\beta\nu$}~\cite{Fer73} is the unique Borel probability measure on~$\mcP(X)$ satisfying, for every bounded Borel~$F\colon \mcP(X)\times X\times I\to\R$, the Mecke-type identity~\cite{LzDSLyt17} (also cf.~\cite{Set94})
\begin{equation}\label{eq:MeckeDF}
\int\int_X F(\eta,x,\eta_x) \diff\eta(x) \diff\DF{\beta\nu}(\eta) = \int \int_I \int_X F(\eta^x_t,x,t)\diff\nu(x)\, \diff\Beta_\beta(t)\, \diff\DF{\beta\nu}(\eta)\fstop
\end{equation}

We refer the reader to~\cite{Fer73} for the original construction of~$\DF{\beta \nu}$ via Kolmogorov consistency as a limit of Dirichlet distributions on standard simplices, to~\cite{LzDS19a} for a construction and characterization of~$\DF{\beta \nu}$ via Fourier transform, and to~\cite{Las18,Set94,JiaDicKuo04,LzDSQua23} for other characterizations.

 Dirichlet--Ferguson measures have appeared throughout mathematics, with many important applications to ---only to name a few--- the theory of random permutations (see e.g.~\cite{Ber09} and references therein), Bayesian non-parametrics (see e.g.~\cite{Fer73}), population genetics (see e.g.~\cite{Fen10} and references therein), infinite-dimensional stochastic analysis~\cite{OveRoeSch95,LzDS17+}, and the representation theory of groups of diffeomorphisms and multipliers~\cite{LzDS17+,LzDS19a}.

\subsubsection{The multiplicative infinite-dimensional Lebesgue measure}\label{sss:InfDimLeb}
For each~$\theta>0$, we define a $\sigma$-finite Borel measure on~$\R^+$ by
\begin{equation}\label{eq:IntroLambda}
\diff\lambda_\theta(t)\eqdef \frac{t^{\theta-1}\,\diff t}{\Gamma(\theta)} \comma \qquad t>0\fstop
\end{equation}
The family~$\seq{\lambda_\theta}_{\theta>0}$ is a convolution semigroup, in the sense that~$\lambda_\theta*\lambda_\tau=\lambda_{\theta+\tau}$ for all~$\theta,\tau>0$.

In~\cite{TsiVerYor01}, N.V.~Tsilevich, A.M.~Vershik, and M.~Yor introduced the \emph{multiplicative infinite-dimensional Lebesgue measure~$\LP{\theta,\nu}$ with shape parameter~$\theta>0$ and intensity measure~$\nu \in \mcP(X)$} as
\begin{equation}\label{eq:InfLebesgueIsomor}
\LP{\theta,\nu} \eqdef \isomor^{-1}_\pfwd\tparen{\DF{\nu} \otimes \lambda_\theta} \comm
\end{equation}
with~$J$ as in~\eqref{eq:IsomorphismMbp}.
The $\sigma$-finite measure~$\LP{\theta, \nu}$ ---which ought to be regarded as an infinite-dimensional analogue of~$\lambda_\theta$--- displays a number of remarkable properties.

Here and everywhere in the following, let
\begin{equation}\label{eq:IntroBall}
\mcB_r\eqdef \set{\mu\in\mcM(X) : \mu X \leq r}\comma \qquad r\geq 0 \fstop
\end{equation}

\begin{proposition}[Tsilevich--Vershik--Yor, see~{\cite[\S4]{TsiVerYor01}}]\label{p:TVY}
The measures~$\LP{\theta, \nu}$, with~$\theta>0$, enjoy the following properties:
\begin{enumerate}[$(i)$]
\item\label{i:p:TVY:1} $\seq{\LP{\theta, \nu}}_{\theta>0}$ is a convolution semigroup, viz.~$\LP{\theta, \nu}*\LP{\tau, \nu}= \LP{\theta+\tau, \nu}$ for every~$\theta,\tau>0$;

\item\label{i:p:TVY:2} $\LP{\theta, \nu}$ is projectively invariant for the \ref{eq:ActionMultipliers}-action of~$\Multi(X)$, with Radon--Nikod\'ym derivative
\begin{equation}\label{eq:ProjQInvLTheta}
\frac{\diff (k \cdot)_\pfwd \LP{\theta, \nu}}{\diff \LP{\theta, \nu}}\equiv e^{-\theta\, \int \log k \de \nu} \comma \qquad k \in\Multi(X)\fstop
\end{equation}
In particular, $\LP{\theta, \nu}$ is invariant for the \ref{eq:ActionMultipliers}-action of~$\Multi_\nu(X)$ and $\theta$-homo\-gen\-eous, i.e.\ $\LP{\theta, \nu}(c\emparg) = c^\theta \LP{\theta, \nu}$ for every constant~$c>0$.

\item\label{i:p:TVY:3} $\LP{\theta, \nu}$ is invariant for the \ref{eq:ActionShifts}-action of~$\mfS_\nu(X)$.
\end{enumerate}
\end{proposition}

In combination with several other properties, the invariance for the \ref{eq:ActionMultipliers}-action of~$\Multi_\nu(X)$ or the projective invariance for the \ref{eq:ActionMultipliers}-action of~$\Multi(X)$ have been used to characterize the measures~$\LP{\theta, \nu}$; see~\cite[Thm.~4.2]{TsiVerYor01} and~\cite[Thm.~5]{Ver08b}.
We extend these uniqueness results by providing the following characterization under minimal assumptions.

\begin{proposition}[Prop.~\ref{p:UniquenessLP}]\label{p:UniquenessIntro}
The following are equivalent:
\begin{enumerate}[$(i)$]
\item $\mcQ$ is a non-negative Borel measure on~$\mcM(X)$ satisfying:
\begin{itemize}[wide]
\item \emph{non-triviality}: $\mcQ\mcB_0=0$, cf.~\eqref{eq:IntroBall};
\item \emph{normalization}: $\mcQ\mcB_1=1$, cf.~\eqref{eq:IntroBall};
\item \emph{\ref{eq:ActionMultipliers}-invariance}: $\mcQ$ is projectively invariant for the \ref{eq:ActionMultipliers}-action of~$\Multi(X)$.
\end{itemize}

\item $\mcQ=\frac{1}{\Gamma(\theta+1)}\LP{\theta,\nu}$ with
\[
\theta\eqdef \int \mu X \, \rme^{-\mu X}\diff\mcQ(\mu)\qquad \text{and} \qquad \nu\eqdef \theta^{-1} \int \mu(\emparg)\, \rme^{-\mu X}\diff\mcQ(\mu)\fstop
\]
\end{enumerate}
\end{proposition}

Let us now turn to the smooth category.
In the case when~$X=M$ is a manifold as in~\S\ref{sss:DifferentiationIntro}, Proposition~\ref{p:UniquenessIntro} holds as well if we replace~$\Multi(M)$ by its smaller analogue~$\exp[\rmC^\infty_c(M)]$ in the smooth category.
Now, let us further assume, as in~\S\ref{sss:TangentSpacesIntro}, that~$M$ is endowed with a Riemannian metric~$g$.
As shown below, in this case we may further completely identify the `pseudo-intensity' measure~$\nu$.

Indeed, we look for measures~$\mcQ$ on~$\mcM(M)$ which are both \emph{natural} (in the smooth category) and \emph{universal}.
By `natural' we mean that given a measure~$\mcQ$ on a  smooth, connected, orientable  Riemannian manifold~$(M,g)$ we can identify~$\mcQ$ by~$g$.
Rigorously,~$\mcQ$ is invariant for the \eqref{eq:ActionShifts}-action on~$\mcM(M)$ of the isometry group~$\Iso(g)$ of~$(M,g)$.
(Of course, depending on~$(M,g)$, this constraint might be void as~$\Iso(g)$ might be the trivial group.)
By `universal' we mean that for \emph{every} (smooth, connected, orientable) manifold~$M$ and every Riemannian metric~$g$, we can find a natural~$\mcQ$.
Rigorously, for every manifold~$M$ there exists a multi-valued map~$g\rightrightarrows\mcQ$ whose image consists of natural measures.

Now, assume further that~$\mcQ$ is as in Proposition~\ref{p:UniquenessIntro}.
It is not difficult to show that (the infinite-dimensional multiplicative Lebesgue measure)~$\mcQ$ is natural and universal if and only if~$\nu$ is invariant for the  natural action on~$M$ of~$\Iso(g)$.
Since~$\mcQ$ is completely determined by its pseudo-intensity~$\nu$ and by the scale (homogeneity) parameter~$\theta>0$, it is reasonable to strengthen the above notion of naturalness by requiring that~$\mcQ$ be determined only by~$\vol_g$ rather than by~$g$.
Rigorously, we require the multi-valued map~$g\rightrightarrows\mcQ$ to factor over the map~$g\mapsto \vol_g$ assigning to~$g$ its volume measure~$\vol_g$.
That is, we have a multi-valued map~$\vol_g\rightrightarrows\mcQ$.

Finally, let~$\Diff^+_0(g)$ be the group of compactly non-identical, orientation-preserv\-ing diffeomorphisms on~$M$ preserving the Riemann volume form of~$g$ by pullback or, equivalently, preserving~$\vol_g$ by push-forward.
In other words, we require~$\mcQ$ to be invariant for the \eqref{eq:ActionShifts}-action on~$\mcM(M)$ of~$\Diff^+_0(g)$.
Contrary to the previous definition, this stronger version of naturalness is always non-void, since~$\Diff^+_0(g)$ is the infinite-dimensional Lie group modelled on the (infinite-dimensional) Lie algebra of compactly supported $\div_g$-free vector fields, and it is in fact sufficient to turn the multi-valued map $g\rightrightarrows\mcQ$ into a uniquely determined function.

\begin{corollary}[See Prop.~\ref{p:FullIdentification}]
Let~$(M,g)$ be a  smooth, connected, orientable  Riemannian manifold with finite total volume.
Then, the following are equivalent:
\begin{enumerate}[$(i)$, leftmargin=2em]
\item $\mcQ$ is a non-negative Borel measure on~$\mcM(M)$ satisfying:
\begin{itemize}[leftmargin=1em]
\item \emph{non-triviality}: $\mcQ\mcB_0=0$;
\item \emph{normalization}: $\mcQ\mcB_1=1$;
\item \emph{\ref{eq:ActionMultipliers}-invariance}: $\mcQ$ is projectively invariant for the \ref{eq:ActionMultipliers}-action of~$\exp[\rmC^\infty_c(M)]$;
\item \emph{\ref{eq:ActionShifts}-invariance}: $\mcQ$ is invariant for the \ref{eq:ActionShifts}-action of~$\Diff^+_0(g)$.
\end{itemize}

\item $\mcQ=\frac{1}{\Gamma(\theta+1)}\LP{\theta,\nu}$ with~$\theta\nu=\vol_g$ and~$\nu M=1$.
\end{enumerate}
\end{corollary}

\subsubsection{Closability of the canonical form}
It is shown in~\cite{LzDS17+} that~$\DF{\beta \nu}$ is \emph{not} quasi-invariant w.r.t.\ the action~\ref{eq:ActionShifts} of~$\Diff^+_0(M)$ on~$\mcP(M)$, but merely partially quasi-invariant with respect to some filtration~$\seq{\msF_t}_t$ of the Borel $\sigma$-algebra of~$\mcP(M)$.
Since~\ref{eq:ActionShifts} on~$\mcM(M)$ factorizes over~$\isomor$ in~\eqref{eq:IsomorphismMbp}, it follows that~$\LP{\theta, \nu}$ too is \emph{not} quasi-invariant for the same action.
Its partial quasi-invariance may not be immediately deduced from that of~$\DF{\beta \nu}$ since the normalization map is not Borel/$\msF_t$-measurable for any~$t$.
As in turns out however, $\LP{\theta, \nu}$~is indeed partially quasi-invariant under the \ref{eq:ActionShifts}-action of~$\Diff^+_0(M)$, and thus under the \ref{eq:ActionSemidirect}-action of~$\mfG(M)\eqdef \Diff^+_0(M)\rtimes_\pb \exp[\Cc^\infty(M)]$, Prop.~\ref{p:PQI}.

This is of particular importance, since partial \ref{eq:ActionSemidirect}-quasi-invariance can be used to prove the closability on~$L^2(\mcM(M), \LP{\theta,\nu})$ of the pre-Dirichlet energy in~\eqref{eq:DirichletForm0} with~$\mcQ=\LP{\theta, \nu}$.

\begin{theorem}\label{t:IntroForm}
The quadratic form~$\tparen{\mcE, \FC{\infty,\infty}{c,c}{\infty}{c}}$ is densely defined and closable on~$L^2(\mcM(M),\LP{\theta,\nu})$.
Its closure~$(\mcE,\dom{\mcE})$ is a quasi-regular conservative strongly local Dirichlet form on~$\mcM(M)$, recurrent if~$\theta\in (0,1]$ and transient if~$\theta\in(1,\infty)$.
\end{theorem}

\subsubsection{The Vershik diffusion on the cone of non-negative measures}
As a consequence of Theorem~\ref{t:IntroForm} and by the standard theory of Dirichlet forms, there exists a conservative Markov diffusion process~$\mu_\bullet$ with state space~$\mcM(M)$ and invariant measure~$\LP{\theta,\nu}$ properly associated with the form~$\tparen{\mcE,\dom{\mcE}}$.
We call~$\mu_\bullet$ the \emph{Vershik diffusion}.
It is the counterpart on~$\mcM(M)$ to the Dirichlet--Ferguson diffusion on~$\mcP(M)$ constructed in~\cite{LzDS17+} and the $\LP{\theta,\nu}$-reversible \emph{Brownian motion} for the Hellinger--Kantorovich geometry of~$\mcM(M)$.

A detailed study of the Vershik diffusion lies beyond the scope of this paper and will be addressed in future work.
However, instrumentally to a proof of conservativeness and recurrence/transience for~$\tparen{\mcE,\dom{\mcE}}$ in Theorem~\ref{t:IntroForm}, we prove the following statement.
Let~$x_t$ be a \emph{squared Bessel process of dimension~$\theta$}, i.e.\ the (pathwise-unique, strong) solution to the \textsc{sde}
\begin{equation}
\diff x_t= \sqrt{2 x_t}\, \diff W_t + \theta\,\diff t \comma \qquad t>0 \comma
\end{equation}
driven by a standard Wiener process~$W_t$ on~$\R$.

\begin{proposition}
For every~$\theta>0$, the radial-part process~$\mu_t M\geq 0$ of the $\LP{\theta,\nu}$-reversible Vershik diffusion is distributed as~$x_{t/4}$.
\end{proposition}

The time-rescaling factor~$\tfrac{1}{4}$ should be regarded as an artefact of our convention that the Wiener process is generated by the Laplacian~$\Delta$ (rather than by~$\tfrac{1}{2}\Delta$) together with the factor~$4$ in the vertical part of the Dirichlet form, cf.~\eqref{eq:DirichletForm0}.

The appearance of the squared Bessel process of dimension~$\theta$ in the description of the radial part process of~$\mu_\bullet$ is significant.
Indeed, the family of squared Bessel processes indexed by~$\theta> 0$ forms a convolution semigroup of diffusions on the real line ---which is reflected by the same property for the measures~$\lambda_\theta$ and~$\LP{\theta,\nu}$--- and is one instance in the larger class of \emph{continuous-state branching processes with immigration},~\cite{ShiWat73,KawWat71}. (See Appendix~\ref{app:Bessel}.)
This suggests that the `vertical random motion' associated to the Dirichlet form obtained by replacing~$\grad$ with~$\gradH$ and choosing~$\mcQ=\LP{\theta,\nu}$ in~\eqref{eq:DirichletForm0} is an unconstrained version of the celebrated Fleming--Viot process~\cite{FleVio79,OveRoeSch95}: it describes the distribution of alleles in a selectively neutral genetic population as influenced by mutation and random genetic drift, as opposed to the distribution of allelic \emph{frequencies} in the Fleming--Viot process.

\subsection{The metric point of view}
In this section, we will present another quite natural construction of an energy form for functions on~$\mcM(M)$, namely the \emph{Cheeger energy} induced by the Hellinger--Kantorovich distance and by~$\LP{\theta,\nu}$.

\subsubsection{Hellinger--Kantorovich distance on the cone of positive measures}\label{sss:IntroHK}

One of the most remarkable results in the theory of optimal transport is to provide notions of distances between probability measures: given a complete and separable metric space $(X, \sfd)$, the \emph{$L^2$-Kantorovich--Rubinstein} (also: \emph{Wasserstein}) \emph{distance}~$\W_{2, \sfd}$ between two probability measures $\mu_0, \mu_1 \in \mcP(X)$ is given by the optimal transport cost induced by $\sfd^2$ i.e.
\begin{align}\label{eq:WassersteinIntro}
\W_{2, \sfd}(\mu_0, \mu_1) \eqdef \braket{ \inf_{\ggamma \in \Cpl(\mu_0, \mu_1)} \int_{X^2} \sfd^2 \de \ggamma }^{1/2}\comma
\end{align}
where $\Cpl(\mu_0, \mu_1) \subset \mcP(X^2)$ is the space of \emph{couplings} between $\mu_0$ and $\mu_1$ i.e.~probabilities on the product space with marginals $\mu_0$ and $\mu_1$. When restricted to the set of probability measures with finite second moment $\mcP_2(X)$, the Wasserstein distance is complete and separable, length (resp.~geodesic) if $\sfd$ is a length (resp.~geodesic) distance.
The Wasserstein distance has also the remarkable property of making the map $x \mapsto \delta_x$ an isometric embedding of $X$ into $\mcP_2(X)$.
We refer to the classic monographs \cite{AGS08, Villani03, Villani09, santambrogio, Rachev-Ruschendorf98I, Rachev-Ruschendorf98II} and to the more recent \cite{ abs21, fg21} for a treatment of optimal transport and Wasserstein distances.

The generalization of the optimal transport problem to the \emph{unbalanced setting} (i.e.~to the case when $\mu_0$ and $\mu_1$ may have different total non-negative masses) has been considered in a variety of papers, see for example the generalization of the dynamical approach to balanced optimal transport~\cite{BenamouBrenier00} to the unbalanced setting in \cite{msg, 49eee, RosPic, Rospicprop,ChiPeySchVia16}; static generalizations connected to the Kantorovich formulation were explored in \cite{CaffarelliMcCann06TR}; see also the work \cite{11uu} where the so-called optimal partial transport has been introduced. A very general framework for unbalanced optimal-transport problems can be also found in \cite{SS24} where, using the methods developed in \cite{SS20}, the ideas already contained in \cite{LMS18, CPSV18}  are extended.

We focus here on the class of distances $\HK_\sfd$ introduced at the same time in the works \cite{LMS18, CPSV18}  and called \emph{Wasserstein--Fisher--Rao} and \emph{Hellinger--Kantorovich} respectively: these can be defined in many ways (see in particular Section \ref{sec:hk} for one of these possibilities) and are rightfully considered to be the correct generalization to $\mcM(X)$ of the Wasserstein distances.
Indeed, the space $(\mcM(X), \HK_\sfd)$ is a complete and separable metric space whose topology is the one of the weak convergence of measures, it preserves the length and geodesic properties of the underlying space, and makes the map $(X \times \R_0^+) \ni (x,r) \mapsto r \delta_x \in \mcM(X)$ an isometric immersion, provided $X \times \R_0^+$ is endowed with the natural cone distance, cf.~\eqref{ss22:eq:distcone} below.
Perhaps the simplest way to grasp the idea behind the Hellinger--Kantorovich distance is to look at its \emph{dynamical formulation} on the Euclidean space $\R^d$: if $\mu_0, \mu_1 \in \mcM(\R^d)$, we have
\begin{equation} \label{eq:hkintro}
\HK(\mu_0, \mu_1)^2 = \inf \set{ \int_0^1 \int_{\R^d} \left [ |\vv_t|^2 + \tfrac{1}{4}|w_t|^2\right ] \de \mu_t \de t }
\end{equation}
where the infimum is taken among all triplets $(\mu_\cdot, \vv_\cdot, w_\cdot)$ with $\mu_t |_{t=i}=\mu_i$ for~$i=0,1$, and solving the \emph{continuity equation with reaction}
\[
\partial_t \mu_t + \div (\vv_t \mu_t) = w_t \mu_t \quad \text{ in } \mathscr{D}'\tparen{(0,1) \times \R^d} \fstop
\]
This formula, a generalization of the classical Benamou--Brenier theorem~\cite{BenamouBrenier00}, shows that the $\HK$-geometry encompasses both a horizontal movement of mass (driven by~$\vv$) and a reaction/dilation behavior driven by the term~$w$.
We refer to the second part of~\cite{LMS18} for details and proofs regarding the Hellinger--Kantorovich distance.

Let us finally remark that deep relationships between the group~$\mfG(M)$ (here $M$ is a manifold  as in~\S\ref{sss:TangentSpacesIntro} which is additionally assumed to be complete ) and the Hellinger--Kantorovich distance are already apparent in the work~\cite{GalVia17}, where it is noted that a \emph{left} action of~$\mfG(M)$ ---thus opposite to~\ref{eq:ActionSemidirect}--- gives a Riemannian submersion between an $L^2$-type of metric on~$\mfG(M)$ and the Hellinger--Kantorovich metric on the space of densities, see~\cite[Prop.~10]{GalVia17}.

\subsubsection{Metric measure geometry and Cheeger energies}
One of the earliest and simplest construction of Sobolev space on a smooth domain $\Omega \subset \R^d$ is to consider the closure of smooth functions w.r.t.~a given Sobolev norm i.e.
\[
H^{1,2}(\Omega) \eqdef \overline{\rmC_\infty(\Omega)}^{H^{1,2}}\comma \qquad \|\varphi\|_{H^{1,2}}^2 \eqdef \int_\Omega \left [ |\varphi(x)|^2 + |\nabla \varphi(x)|^2 \right ]\de x\fstop
\]
The key idea is very simple: we have a class of regular functions for which the notion of derivative is given and a norm to measure it: we consider limits of smooth functions w.r.t.~this norm. The fundamental observation to generalize this approach to metric spaces is to notice that what is needed is not really the gradient of a smooth function, but rather its norm. This leads to the following construction: if $(X, \sfd)$ is a complete and separable metric space endowed with a non-negative, finite, Borel measure $\mssm$, for a Lipschitz and bounded function $f\in \Lip_b(X, \sfd)$ we define the so called \emph{asymptotic Lipschitz constant} as 
\begin{equation}
 \lip_\sfd f(x) \eqdef \limsup_{y,z \to x, y \ne z} \frac{ |f(z)-f(y)|}{\sfd(z,y)}, \quad x \in X \comma
 \end{equation}
 and the corresponding \emph{$2$-pre-Cheeger energy} as
 \begin{equation}\label{eq:pce}
\pCE_{2, \sfd, \mm}(f) \eqdef \int_{X} (\lip_\sfd f)^2 \de \mssm, \quad f \in \Lip_b(X, \sfd) \fstop
 \end{equation}
The quantity~$ \lip_\sfd f$ is a surrogate for the modulus of the gradient of a smooth function. 
The functional $\pCE_{2, \sfd, \mssm}$ is however defined only on Lipschitz functions.
In order to extend it to more general functions, we can consider its lower semicontinuous relaxation in $L^2(\mssm)$, namely the $2$-\emph{Cheeger} energy~$\CE_{2, \sfd, \mssm}$, \cite{Cheeger99, Gigli-Kuwada-Ohta13}.
The vector subspace of functions in $L^2(X, \mssm)$ for which $\CE_{2,\sfd, \mssm}$ is finite is the \emph{metric Sobolev space} $H^{1,2}(X, \sfd, \mssm)$ \cite{Bjorn-Bjorn11,HKST15, GP20, Savare22} which can be proven to be a Banach space with the norm
\[
\norm{f}_{H^{1,2}} \eqdef \braket{ \|f\|_{L^2(X,\mssm)}^2 + \CE_{2,\sfd, \mssm}(f) }^{1/2} \fstop
\]
Here, we are interested in the case when $(X, \sfd) = (\mcM(M), \HK_{\sfd_g})$ for a manifold $(M,g)$  as in~\S\ref{sss:TangentSpacesIntro}  with Riemannian distance $\sfd_g$.

\subsubsection{Density of cylinder functions}  Let $(M,g)$ be a smooth, connected, complete Riemannian manifold.  
Given $\mu \in \mcM(M)$ and $(w,a) \in T_\mu \mcM(M)$, we set
\[
\abs{(w,a)}_\toplus\eqdef \tbraket{g(w,w) + 4\abs{a}^2}^{1/2}\comma
\]
and
\begin{align}\label{eq:NormTot4}
\norm{(w,a)}_{T_\mu^{1,4}}\eqdef \sqrt{\norm{w}_\mu^2+4\norm{a}_\mu^2} = \braket{\int_M\abs{(w,a)}_\toplus^2\diff\mu}^{1/2} \fstop
\end{align}
We will show in Proposition~\ref{prop:equalityriem} that
\begin{equation}\label{eq:equalityi}
\norm{(\boldnabla u)_\mu}_{T_\mu^{1,4}} = \lip_{\HK_{\sfd_g}} u (\mu) \comma \qquad \mu\in \mcM(M)\comma
\end{equation}
whenever $u \in \FC{\infty,\infty}{c,c}{\infty}{c}$ is a smooth cylinder function as in \eqref{eq:cylman}; that is, the metric slope of a cylinder function pointwise coincides with the modulus of its geometric gradient.

In order to compare the geometric and the metric approach, pursuing the same strategy as in \cite{FSS22}, we aim at showing that cylinder functions are dense in $2$-energy in the Sobolev space~$H^{1,2}(\mcM(M), \HK_{\sfd_g}, \mcQ)$ for every finite, non-negative, Borel measure~$\mcQ$ on~$\mcM(M)$. 

The idea is to exploit the equality in \eqref{eq:equalityi} and to approximate the $\HK$ distance function from a fixed reference measure with suitable cylinder functions: this was done for the Wasserstein distance in \cite{FSS22} using regular \emph{Kantorovich potentials}, also cf.~\cite{LzDS19b}.
Even if suitably-adapted versions of Kantorovich potentials are available for the $\HK$-distance, their regularity properties are scarce \cite{LMS22}, so that it seems difficult to be able to reproduce the same argument as in \cite{FSS22} in this case.
However, we show in Corollary \ref{cor:ineq} that we can approximate with cylinder functions a smoothed version of the $\HK$-distance (namely, the so called Gaussian Hellinger--Kantorovich distance, $\GHK$, which induces $\HK$ as length-distance) instead of~$\HK$ directly, employing the regularity of its adapted version of Kantorovich potentials (cf. Theorem~\ref{teo:final}). 

Together with Theorem \ref{theo:startingpoint}, which shows that it is enough to approximate a distance function~$\delta$ from a reference point with an algebra of functions~$\AA$ in order to obtain the density of~$\AA$ in the Sobolev space induced by the length distance generated by~$\delta$, we thus obtain the sought density of a suitable class of cylinder functions in $H^{1,2}(\mcM(\R^d), \HK_{\sfd_e}, \mcQ)$, where $\sfd_e$ is the Euclidean distance on $\R^d$ (see Theorem~\ref{teo:density}).
The result is then extended first to the Riemannian setting in Theorem~\ref{thm:hilbm} and then further refined to the smaller class of smooth cylinder functions.
We then have the following general result (Corollary~\ref{c:DensityCylC}).

\begin{theorem}\label{thm:introdens}
 Let   $\mcQ$ be a non-negative Borel measure on $\mcM(M)$. If 
\begin{equation}\label{eq:lpex}
\int \rme^{-t \mu M} \, \diff \mcQ(\mu) < + \infty \quad \text{ for every } t>0 \comma
\end{equation}
then, $\tparen{H^{1,2}(\mcM(M), \HK_{\sfd_g}, \mcQ}$ is a Hilbert space, the subalgebra $\FC{\infty,\infty}{c,c}{\infty}{c}$ is strongly dense in $H^{1,2}(\mcM(M), \HK_{\mssd_g}, \mcQ)$ and for every $u \in H^{1,2}(\mcM(M), \HK_{\mssd_g}, \mcQ)$ there exists a sequence $(u_n)_n \subset \FC{\infty,\infty}{c,c}{\infty}{c}$ such that 
\[
u_n \to u  \quad \text{ in } L^2(\mcM(M), \mcQ), \quad \pCE_{2, \HK_{\sfd_g}, \mcQ} (u_n) \to \CE_{2, \HK_{\sfd_g}, \mcQ} (u) \fstop
\]
\end{theorem}

\subsubsection{Identification of the metric and the geometric points of view}
Let us now come to a comparison of the Dirichlet form~$(\mcE,\dom{\mcE})$ in Theorem~\ref{t:IntroForm} with the Cheeger energy~$\CE_{2,\HK_{\sfd_g},\LP{\theta,\nu}}$ in Theorem \ref{thm:introdens}, when we chose as reference measure $\mcQ$ the multiplicative infinite-dimensional Lebesgue measure $\LP{\theta,\nu}$. Firstly let us stress that these two objects (and the relative constructions) are \emph{a priori} related only by the choice of the same reference measure~$\LP{\theta,\nu}$.
Indeed, on the one hand, the construction of~$(\mcE,\dom{\mcE})$ relies on the infinite-dimensional Lie group~$\mfG(M)$ (motivating also the choice of~$\LP{\theta,\nu}$) and on the class of cylinder functions.
On the other hand, the definition of~$\CE_{2,\HK,\LP{\theta,\nu}}$ relies solely on the Hellinger--Kantorovich distance~$\HK_{\sfd_g}$ and on the class of $\HK_{\sfd_g}$-Lipschitz functions. 
As a consequence of Theorem \ref{thm:introdens} and, the equality~\eqref{eq:equalityi}, and the definition in~\eqref{eq:DirichletForm0} we obtain the following result.

\begin{theorem}\label{t:CoincidenceIntro}
As quadratic forms on~$L^2(\mcM(M), \LP{\theta,\nu})$,
\[
(\mcE,\dom{\mcE}) = \tparen{\CE_{2,\HK_{\sfd_g},\LP{\theta,\nu}},H^{1,2}(\mcM(M), \HK_{\sfd_g}, \LP{\theta,\nu})} \fstop
\]
In other words, the geometric construction and the metric construction coincide, and the Vershik diffusion is the metric measure Brownian motion of the metric measure space~$\tparen{\mcM(M), \HK_{\mssd_g}, \LP{\theta,\nu}}$.
\end{theorem}

\subsection{Applications and outlook}

\subsubsection{Applications to PDEs}
As described in the beginning of~\S\ref{sss:TangentSpacesIntro}, on the space $\mcP_2(\R^d)$ of Borel probability measures on~$\R^d$ with finite second moment the horizontal gradient corresponds to the geometry of Otto calculus, inducing the $L^2$-Kantorovich--Rubinstein distance $W_2$, in~\eqref{eq:WassersteinIntro}.
This was first formally introduced in \cite{JordanKinderlehrerOtto98} and later developed in e.g.~\cite{AGS08} to study PDEs of the form 
\begin{equation}\label{eq:ottopde}
\partial_t \rho - \nabla \cdot \left ( \rho\,  \nabla \frac{\delta \mathcal{F}}{\delta \rho} \right ) = 0 \quad \text{ in } \R^d \times (0,+\infty),
\end{equation}
where $ \frac{\delta \mathcal{F}}{\delta \rho}$ is the first variation of a typical integral functional
\[
\mathcal{F}(\rho) = \int_{\R^d} F(x, \rho(x), \nabla \rho(x)) \d x.
\]
The use of Otto calculus allows to identify~\eqref{eq:ottopde} as a gradient flow dynamics for the functional $\mathcal{F}$ on the space $\mcP_2(M)$ w.r.t.~$W_2$. There are many advantages in recogninzing this kind of dynamics on the space of probability measures: the gradient-flow formulation suggests the use of the minimizing-movement scheme to prove existence of solutions or to numerically approximate them; it allows for contraction and energy estimates, and provides an important tool for the study of the dependence of solutions from perturbation of the functional; working in the space of probability measures allows for weaker assumptions on the initial data (which can be general probability measures), and forces the solutions to be non-negative a priori. See e.g.~\cite[Chapter 11]{AGS08} for an overview.

The addition of a reaction term in \eqref{eq:ottopde} destroys the possibility to resort to Otto calculus, since then mass is no longer preserved during the evolution, which is then set in the space $\mcM(\R^d)$, rather than in~$\mcP_2(\R^d)$.
To recover a similar framework, one can consider the full gradient introduced in the end of~\S\ref{sss:TangentSpacesIntro} inducing the Hellinger--Kantorovich distance \eqref{eq:hkintro}, which thus serves as a generalization of the Otto calculus to functions defined on $\mcM(M)$.
The present paper provides more insight in the structure of $\mcM(M)$ and the differentiation of functionals defined therein, which will serve as a tool to study the gradient flow structure of natural generalizations of the PDE in \eqref{eq:ottopde}.

\subsubsection{Applications to SPDEs}
Let~$H\colon \mcP_2(M)\to\R$ be a measurable function, and~$\xi$ be a space-time white noise on~$M$ with values in the tangent bundle to~$M$.
The Dean--Kawasaki equation~\cite{Dea96,Kaw94} is the stochastic partial differential equation (\textsc{spde})
\begin{equation}\label{eq:DK}
\partial \rho= \Delta\rho + \nabla\cdot (\sqrt{\rho\,}\, \xi) + \nabla\cdot \paren{\rho\, \nabla\frac{\delta F}{\delta\rho}} \fstop
\end{equation}

Equation~\eqref{eq:DK} has been proposed, independently, by D.S.~Dean and by K.~Kawa\-saki,~\cite{Dea96,Kaw94}, to describe the density function of a large particle system subject to a diffusive Langevin dynamics, combining a deterministic interaction~$H$ with a noise~$\xi$ accounting for the particles' thermal fluctuations.
Together with its variants, it has been used to describe super-cooled liquids, colloidal suspensions, the glass-liquid transition, some bacterial patterns, and other systems; see, e.g., the recent review~\cite{VruLoeWitt20} and the introduction to~\cite{LzDS24}.

Recently, the noise term in~\eqref{eq:DK} has been identified as the natural noise for the geometry of~$\mcP_2(M)$. Indeed, in the `free case' $H\equiv 0$, a solution~$t\mapsto \rho=\rho_t$ to~\eqref{eq:DK} is an intrinsic random perturbation of the gradient flow of the Boltzmann--Shannon entropy on~$\mcP_2(M)$ by a noise distributed according to the energy dissipated by the system, i.e.\ by the natural isotropic noise arising from the Riemannian structure of~$\mcP_2(M)$, see~\cite{JacZim14,KonvRe17,KonvRe18}.
Generalizations to~\eqref{eq:DK} have been further completely characterized in~\cite{LzDS24}, where it shown that Markov solutions to the free Dean--Kawasaki equation correspond to the Dirichlet form induced by the horizontal form on~$\mcP_2(M)$.

It is therefore natural to ask whether the Vershik diffusion~$\mu_t$ can be identified as the (unique) solution to some \textsc{spde} with $\mcM(M)$-valued solutions, and ---if so--- what is dynamics of the corresponding infinite marked particle system on~$M$.
By analogy with the \emph{massive particle systems} in~\cite{LzDS24}, we expect the Vershik diffusion to solve a combination of the Dean--Kawasaki and Dawson--Watanabe \textsc{spde}s, cf.~\cite[\S1.5.3]{LzDS24}.
Since the chosen invariant measure~$\LP{\theta}$ is concentrated on purely atomic measures, we also expect the solution~$\mu_t$ to this \textsc{spde} to be the empirical measure of a massive particle system in the sense of~\cite{LzDS24}, viz.
\[
\mu_t = \sum_i {S^i_t}\, \delta_{X^i_t} \fstop
\]
The atoms' locations~$X^i_t$ will solve the infinite system of \textsc{sde}s:
\[
\diff X^i_t = \diff \mssW^i_{t/S^i_t} 
\]
where the~$\mssW^i$'s are mutually independent instances of the Brownian motion on~$M$, while the atoms' masses~$S^i_t$ will solve an infinite system of \textsc{sde}s of Bessel type driven by mutually independent standard Brownian motions on the real line also independent of the~$\mssW^i$'s.

\subsection*{Plan of the work}
In~\S\ref{sec:msob} we collect preliminary results about Sobolev spaces on arbitrary metric measure spaces.
In~\S\ref{s:DistMeas} we recall the definition and some properties of the Hellinger--Kantorovich distance, and we prove new regularity estimates for optimal potentials in its dual formulation.
In~\S\ref{s:ALC-Cylinder} and~\ref{s:DensityCylinder} we introduce cylinder functions on~$\mcM(M)$ and prove their density in energy in the metric Sobolev space of the metric measure space~$\tparen{\mcM(\R^d),\HK_{\sfd_e},\mcQ}$ for any finite Borel measure~$\mcQ$ on $\mcM(\R^d)$. 
In~\S\ref{s:Extensions} we extend the density (and consequently the Hilbertianity) result to the case of a Riemannian manifold and to the class of smooth cylinder functions in $\FC{\infty,\infty}{c,c}{\infty}{c}$, and we also draw some consequences of the latter in terms of the density of cylinder functions in (and the Hilbertianity of) the Sobolev spaces $H^{1,2}(\mcM(\R^d),\sfd,\mcQ)$, where~$\sfd$ is either the $L^2$-Wasserstein (extended) distance or the $2$-Hellinger distance.
In~\S\ref{sec:lebm} we study the uniqueness and invariance properties of the measure $\LP{\theta,\nu}$.
In~\S\ref{sec:extcyl} and~\S\ref{sec:ident} we prove the closability of the form $(\mcE,\FC{\infty,\infty}{c,c}{\infty}{c})$ and we study its properties.
In Appendix~\ref{app:Bessel} we construct the Dirichlet form associated to the squared Bessel process whose properties are instrumental to many results related to the closure of $(\mcE,\FC{\infty,\infty}{c,c}{\infty}{c})$.
 Finally, in Appendix~\ref{app:diffeo} we collect some properties of measure-preserving diffeomorphisms on manifolds.

\anon{
\subsection*{Acknowledgments} 
The authors are grateful to Professor Giuseppe Savar\'e for his encouragement in pursuing the results in this work.

This research was funded in part by the Austrian Science Fund (FWF) project \href{https://doi.org/10.55776/ESP208}{{10.55776/ESP208}}.
This research was funded in part by the Austrian Science Fund (FWF) project \href{https://doi.org/10.55776/F65}{{10.55776/F65}}.
}

\section{Metric Sobolev spaces and the density of subalgebras}\label{sec:msob} 

\subsection{Preliminaries}
We recall some standard facts in the theory of metric measure spaces and establish the necessary notation.

\paragraph{Metric-topological objects}
Let~$(X, \tau)$ be a Hausdorff topological space and $\sfd\colon X \times X \to [0,+\infty]$ be an extended distance on $X$, i.e.~$\sfd$ satisfies all the axioms of a distance but may attain the value $+\infty$.
Given a function~$f\colon X\to\R$ and a set~$A\subset X$ we define
\begin{itemize}
\item the \emph{local $\mssd$-Lipschitz constant}
\[
\Li[\mssd,A]{f}\eqdef \sup_{x,y\in A, \, x \neq y}\frac{|f(x)-f(y)|}{\sfd(x,y)} \semicolon
\]
\item the (\emph{global}) \emph{$\sfd$-Lipschitz constant}~$\Li[\mssd]{f}\eqdef \Li[\mssd,X]{f}$;
\item the \emph{asymptotic $(\tau,\sfd)$-Lipschitz constant}
\begin{equation}
  \label{eq:1}
  \lip_\sfd^\tau f(x)\eqdef \inf_{U \in \mathcal{U}_x}\Li[\mssd,U]{f} \comma \qquad x \in X,
\end{equation}
where $\mathcal{U}_x$ is the directed set of all the $\tau$-neighbourhoods of $x$.
\end{itemize}
We denote by $\Lip_b(X, \tau, \sfd)$ the set of bounded $\tau$-continuous and $\sfd$-Lipschitz functions in $X$ i.e.
\[ \Lip_b(X, \tau, \sfd) \eqdef  \{ f \in \rmC_b(X, \tau) : \Li[\mssd]{f} < + \infty \}.\]

\begin{definition}[Extended metric-topological measure space (e.m.t.m. space)] Let $(X, \tau)$ be a Hausdorff topological space and let $\sfd\colon X \times X \to [0,+\infty]$ be an extended distance on $X$. We say that $(X, \tau, \sfd)$ is an extended metric-topological space if
\begin{enumerate}
\item $\tau$ coincides with the initial topology induced by $\Lip_b(X, \tau, \sfd)$ on $X$, i.e.\ the coarsest topology on $X$ such that all functions in $\Lip_b(X, \tau, \sfd)$ are continuous;
\item the distance $\sfd$ can be recovered starting from non-expansive $\sfd$-Lipschitz $\tau$-continuous functions:
\[
\sfd(x,y)\eqdef  \sup \{ |f(x)-f(y)| : f \in \Lip_b(X, \tau, \sfd), \,  \Li[\mssd]{f} \le 1 \}\comma \qquad x,y \in X\fstop
\]
\end{enumerate}
Given a non-negative, Radon measure $\mssm$ on $(X, \mathscr{B}(X, \tau))$ finite on $\sfd$-balls, we call $\X\eqdef (X, \tau, \sfd, \mssm)$ an extended metric-topological measure space. In the particular case when $(X, \sfd)$ is a complete and separable metric space and $\tau=\tau_\sfd$ is the topology induced by $\sfd$, we say that $(X, \sfd, \mssm)$ is a Polish metric measure space.
\end{definition}

\begin{remark}
When~$(X,\mssd)$ is a complete and separable metric space, every non-negative Borel measure~$\mssm$ finite on $\mssd$-bounded sets is automatically Radon; see, e.g.,~\cite[Thm.~II.11, p.~125]{Sch73}.
\end{remark}

A subalgebra~$\AA$ of~$\Lip_b(X, \tau, \sfd)$ is
\begin{itemize}
\item \emph{unital} if the constant function~$\car$ is an element of~$\msA$;
\item \emph{point-separating} if for every $x_0,x_1\in X$ there exists $f\in \AA$ with~$f(x_0)\neq f(x_1)$.
\end{itemize}

\paragraph{Curves}
Given an interval $I \subset \R$ and a metric space $(X, \sfd)$, we call a curve a continuous function $\gamma\colon I \to X$ and we denote its metric speed by 
\begin{equation}
|\dot\gamma|_\sfd(t)\eqdef \limsup_{h\to0}\frac{\sfd\tparen{\gamma(t+h),\gamma(t)}}{|h|}, \quad t \in I \fstop
\end{equation}
The distance~$\mssd$ induces a \emph{length functional}~$\ell_\mssd$ on the space of curves in $\Lip([0,1];(X,\sfd))$:
\begin{align*}
\ell_{\sfd}(\gamma)&\eqdef 
\sup\set{\sum_{n=1}^N\sfd(\gamma(t_{n-1}),\gamma(t_n)) :
      0=t_0<t_1<\cdots<t_{N-1}<t_N=1}
      \\
      &=\int_0^1|\dot\gamma|_\sfd(t)\,\d t \fstop
\end{align*}

The \emph{length} (extended) \emph{distance $\hat\mssd$} induced by $\sfd$ on a subset $Y \subset X$ is then, for every~$y_0, y_1 \in Y$,
\begin{align}
\hat\mssd_Y(y_0,y_1)&\eqdef 
    \inf\set{\ell_\sfd(\gamma): \begin{gathered}\gamma\in \Lip([0,1];(Y,\sfd))\comma \\
                                \gamma(0)=y_0,\ \gamma(1)=y_1
                                \end{gathered}
                                } \\
                               &{}=                         
                                 \inf\set{\ell>0: \begin{gathered}\gamma\in \Lip([0,\ell];(Y,\sfd))\comma \\
                                \gamma(0)=y_0,\ \gamma(\ell)=y_1,\
      |\dot\gamma|_\sfd\le 1\text{ a.e.}
      \end{gathered}
      } \fstop
\end{align}
In case $Y=X$, we simply write $\hat \mssd$.

\paragraph{Measure objects}
Unless explicitly stated otherwise, by a \emph{measure} we mean a non-negative non-zero measure.
Let~$(X,\Sigma,\mssm)$ be a measure space.
We denote by~$L^0(X, \mssm)$ the space of real-valued measurable functions on~$X$, identified up to equality $\mssm$-a.e., endowed with the topology of the local convergence in $\mssm$-measure.
For~$r\in [1,+\infty]$, we denote by~$L^r(X,\mssm)$ the usual Lebesgue spaces of real-valued, measurable $r$-summable functions, identified up to equality $\mssm$-a.e., with its usual norm.
For~$Y$ (any subset of) a Banach space, we write~$L^r(X,\mssm; Y)$ for the corresponding space of~$Y$-valued functions.

\subsection{Relaxed gradients, Cheeger energies, and Sobolev spaces}
We fix for this subsection a e.m.t.m. space $\X=(X, \tau, \sfd, \mssm)$, a unital point-separating subalgebra $\AA \subset \Lip_b(X, \tau, \sfd)$, and $q \in (1,+\infty)$.

\begin{definition}[Cheeger energy and Sobolev space]
The \emph{$(q,\AA)$-Cheeger energy} is the functional
\begin{equation}\label{eq:relpreq}
  \CE_{q,\AA}(f) = \inf \left \{ \liminf_{n \to + \infty}
    \pCE_{q}(f_n) : \begin{gathered} (f_n)_n \subset \AA\comma \\ f_n \to f \text{ in } L^0(X,\mssm) \end{gathered}
  \right \}\comma \qquad f\in L^0(X,\mssm)\comma
\end{equation}
where the pre-Cheeger energy $\pCE_q: \Lip_b(X, \tau, \sfd) \to [0, +\infty]$ is defined as
\begin{equation}\label{eq:prec}
\pCE_{q}(f) \eqdef  \int_X (\lip^\tau_\sfd f)^q \,\d \mssm\comma \qquad f \in \Lip_b(X,\tau, \sfd) \fstop
\end{equation}
We denote by $D^{1,q}(\X;\AA)$ the subspace of functions in $L^0(X, \mssm)$ with finite $(q, \AA)$-Cheeger energy. The Sobolev space $H^{1,q}(\X;\AA)$ is defined as
$L^q(X,\mssm)\cap D^{1,q}(\X;\AA)$. The Sobolev norm of $f \in H^{1,q}(\X;\AA)$ is defined by
\[
\|f\|_{H^{1,q}(\X;\AA)}^q\eqdef \|f\|_{L^q(X, \mssm)}^q+\CE_{q, \AA}(f)\fstop
 \]
\end{definition}

We adopt the following definition of relaxed gradient, see~\cite{AGS14I,AGS13,Savare22}, also cf.~\cite{Shanmugalingam00, Bjorn-Bjorn11} for a different approach.
\begin{definition}[$(q,\AA)$-relaxed gradient]
  \label{def:relgrad}
  We say that $G\in L^q(X,\mssm)$ is a $(q,\AA)$-relaxed gradient of $f\in L^0(X,\mssm)$ if
  there exist a sequence $(f_n)_{n}\subset \AA$ with $(\lip^\tau_\sfd f_n)_{n} \subset L^q(X,\mssm)$ and $\tilde{G} \in L^q(X, \mssm)$ such that:
  \begin{itemize}
  \item $f_n\to f$ in $L^0(X, \mssm)$ and
    $\lip_\sfd^\tau f_n\weakto \tilde G$ in $L^q(X,\mssm)$;
  \item $\tilde G\le G$ $\mssm$-a.e.~in $X$.  
  \end{itemize}
\end{definition}
\begin{remark}
In general, it is not the case that $\lip^\tau_\sfd f \in L^q(X,\mssm)$ for~$f \in \Lip_b(X, \tau,\sfd)$.
\end{remark}

For~$f\in L^0(X,\mssm)$ set
\[
S_{\X, q, \AA}(f)\eqdef\set{G \in L^q(X,\mssm):
\text{$G$ is a $(q,\AA)$-relaxed gradient of $f$}} \fstop
\]
It is a simple but important property of relaxed gradients that, if~$S_{\X, q, \AA}(f)$ is non-empty, then it has a unique element of minimal $L^q(X,\mssm)$-norm, see~\cite{AGS14I,AGS13,Savare22}, denoted by $|\rmD f|_{\star,q,\AA}$ and called the \emph{minimal $(q,\AA)$-relaxed gradient} of $f$.

\begin{remark}[Notation] Whenever $\tau=\tau_\sfd$ for a complete and separable distance $\sfd$ on $X$, and $\AA=\Lip_b(X, \tau, \sfd)$ we will omit the dependence on either~$\tau$,~$\AA$, or both from the notation.
Furthermore, the notion of relaxed gradient, the minimal relaxed gradient (if it exists) and the Cheeger (resp.~pre-Cheeger) energy of a function in $L^0(X,\mssm)$ depend of course on $\X$, $\AA$ and $q$ (resp.~on $\X$ and $q$) but, while we will always keep the dependence on $\AA$ (except for the case specified above) and $q$ explicit, we will usually not state explicitly the one w.r.t.~$\X$. In some circumstances however, it will be useful to do so, and we will talk about $(\X, q, \AA)$-relaxed gradient and write $|\rmD f|_{\star, \X,q, \AA}$, $\CE_{\X, q, \AA}(f)$ and $\pCE_{\X, q}(f)$.
\end{remark}

Let us collect a few properties of relaxed gradients and Sobolev spaces that will be useful in the following. For a more comprehensive list and references to the proofs, see \cite{FSS22}.

\begin{theorem} Let $\X=(X, \tau, \sfd, \mssm)$ be a e.m.t.m.~space, let $\AA \subset \Lip_b(X, \tau, \sfd)$ be a unital point-separating subalgebra and let $q \in (1,+\infty)$. The following properties hold true:
  \label{thm:omnibus}\ 
  \begin{enumerate}[$(i)$, wide]
  \item\label{i:t:Omnibus:1} \emph{Completeness:} $(H^{1,q}(\X;\AA),\|\cdot\|_{H^{1,q}(\X;\AA)}) $ is a Banach space.
  \item  \label{i:t:Omnibus:2} \emph{Closure:} The set
    \begin{displaymath}
      S_{\X, q, \AA}\eqdef \Big\{(f,G)\in L^0(X,\mssm)\times L^q(X,\mssm):
      \text{$G$ is a $(q,\AA)$-relaxed gradient of $f$}\Big\}
    \end{displaymath}
    is convex and closed with respect to the product topology of the local convergence in $\mssm$-measure and
    the weak convergence in $L^q(X,\mssm)$.
    In particular, for every $r \in (1,+\infty)$, the restriction $S^r_{\X, q, \AA} \eqdef S_{\X, q, \AA} \cap L^r(X,\mssm)\times
    L^q(X,\mssm)$ is weakly closed in
    $L^r(X,\mssm)\times
    L^q(X,\mssm)$.
    
    \item\label{i:t:Omnibus:3} \emph{Strong approximation:} If
   $f \in D^{1,q}(X, \sfd,\mssm; \AA)$ then
    there exists a sequence $f_n\in \AA$
    such that
    \begin{equation}
      \label{eq:4}
      f_n \to f\text{ $\mssm$-a.e.~in $X$},\quad
      \lip^\tau_\sfd f_n\to |\rmD f|_{\star,q,\AA}\text{ strongly in }L^q(X,\mssm).
    \end{equation} 
    \item\label{i:t:Omnibus:boh} \emph{Local representation:} It holds
      \begin{equation}\label{eq:3}
    \CE_{q,\AA}(f)\eqdef \int_X |\rmD f|_{\star,q,\AA}^q\,\d\mssm\quad
    \text{for every $f\in
  D^{1,q}(\X;\AA)$}.
  \end{equation}
\end{enumerate} 

Further assume that $\mssm$ is finite and let~$f,g\in D^{1,q}(\X; \AA)$. Then,
\begin{enumerate}[$(i)$, resume, wide]

\item\label{i:t:Omnibus:3bis} \emph{Refined approximation:}  If in addition $f \in L^r(X, \mssm)$ for some $r \in \{0\} \cup [1,+\infty)$ and takes values in a closed (possibly unbounded) interval $I\subset \R$, the sequence in \ref{i:t:Omnibus:3} can be found taking also values in $I$ and converging to $f$ in $L^r(X,\mssm)$.
  \item\label{i:t:Omnibus:4} \emph{Pointwise minimality:}
    If $G$ is a $(q,\AA)$-relaxed gradient of~$f$,  then $|\rmD f|_{\star,q,\AA}\le G$ $\mssm$-a.e.
    
  \item \emph{Leibniz rule:}
    If $f,g\in D^{1,q}(\X; \AA)\cap L^\infty(X,\mssm)$, then
    $fg \in D^{1,q}(\X; \AA)$ and
    \begin{equation}
      \label{eq:5}
      |\rmD (fg)|_{\star,q,\AA}\le |f|\,|\rmD g|_{q,\star,\AA}+
      |g|\,|\rmD f|_{\star,q,\AA} \quad\text{$\mssm$-a.e.}
    \end{equation}
  \item \emph{Sub-linearity:}
    If~$\alpha, \beta \in \R$, then~$\alpha f+\beta g \in D^{1,q}(\X; \AA)$ and
    \begin{equation}
      \label{eq:6}
      |\rmD(\alpha f+\beta g)|_{\star,q,\AA}\le |\alpha|\,
      |\rmD f|_{q,\star,\AA}+
      |\beta|\,|\rmD g|_{\star,q,\AA} \quad\text{$\mssm$-a.e.}
    \end{equation}
  \item \emph{Locality:}
     for any $\LL^1$-negligible Borel subset $N\subset \R$ we have
    \begin{equation}
      \label{eq:7}
      |\rmD f|_{\star,q,\AA}=0\quad\text{$\mssm$-a.e.~on $f^{-1}(N)$} \fstop
    \end{equation}
    Furthermore, 
    \begin{equation}\label{eq:locsuper1}
        |\rmD f|_{\star, q, \AA} = |\rmD g|_{\star, q, \AA} \text{ $\mssm$-a.e.~ on $\{f=g\}$}.
    \end{equation}   

  \item\label{i:t:Omnibus:8} \emph{Chain rule:}
    If $\phi\in \Lip(\R)$ then $\phi\circ f\in D^{1,q}(\X; \AA)$ and
    \begin{equation}
      \label{eq:8}
      |\rmD (\phi\circ f)|_{\star,q,\AA}\le |\phi'(f)|\,|\rmD f|_{\star,q,\AA}  \quad\text{$\mssm$-a.e.},
    \end{equation}
    and equality holds in \eqref{eq:8} if $\phi$ is monotone or
    $\rmC^1$.
    
    \item\label{i:t:Omnibus:9} \emph{Truncations:} If $f_j \in D^{1,q}(\X; \AA)$, $1\leq j\leq k$,
      then~$f_+\eqdef \max(f_1,\dotsc,f_k)$ and
      $f_-\eqdef \min(f_1,\dotsc,f_k)$ satisfy~$f_\pm\in D^{1,q}(\X; \AA)$
      and
      \begin{align}
        \label{eq:9}
        |\rmD f_\pm|_{\star,q,\AA}=|\rmD f_j|_{\star,q,\AA}&\quad\text{$\mssm$-a.e.~on }
                                                       \{f_\pm=f_j\}\comma \qquad 1\leq j \leq k\fstop
      \end{align}
  \end{enumerate}
\end{theorem}

\begin{remark} In the notation of Theorem \ref{thm:omnibus}, we have that if $\mssm$ is finite, the restriction of $\CE_{q, \AA}$ to $L^r(X, \mssm)$, $r \in [1,+\infty)$, can be equivalently obtained as
\begin{equation}\label{eq:relaxr}
 \CE_{q,\AA}(f) = \inf \left \{ \liminf_{n \to + \infty}
    \pCE_{q}(f_n) : \begin{gathered} (f_n)_n \subset \AA\comma \\ f_n \to f \text{ in } L^r(X,\mssm) \end{gathered}
  \right \}\comma
    \qquad f\in L^r(X,\mssm).
\end{equation}
\end{remark}

\subsection{Density of sub-algebras of Lipschitz functions}
The main property we are interested in is the density of the subalgebra $\AA$ in the metric Sobolev space
\[
D^{1,q}(\X)= D^{1,q}(\X; \Lip_b(X, \tau, \sfd))\fstop
\]

\begin{definition}[Density in energy of a subalgebra of Lipschitz functions]
  \label{def:density} Let $\X=(X, \tau, \sfd, \mssm)$ be an e.m.t.m.\ space, let $\AA \subset \Lip_b(X, \tau, \sfd)$ be a subalgebra and let $q \in (1,+\infty)$.
  We say that $\AA$ is \emph{dense in $q$-energy} in $D^{1,q}(\X)$ if for every $f\in D^{1,q}(\X)$ there exists a sequence $(f_n)_{n}$ satisfying
  \begin{equation}
    \label{eq:4bis}
    f_n\in \AA,\quad
    f_n \to f\text{ $\mssm$-a.e.~in $X$},\quad
    \lip^\tau_\sfd f_n\to |\rmD f|_{\star,q} \text{~~strongly in~~} L^q(X,\mssm).
  \end{equation}
(In particular~$f_n\to f$ in $L^0(\mfm)$ as well.)
\end{definition}

\begin{remark}\label{rem:dense2} When~$\mssm$ is finite, Definition~\ref{def:density} is also equivalent to the following strong approximation property: for every $f \in H^{1,q}(\X)$ there exists a sequence $(f_n)_{n}$ satisfying
  \begin{equation}
    \label{eq:4tris}
    f_n\in \AA,\quad
    f_n \to f\text{ in $L^q(X,\mssm)$},\quad
    \lip^\tau_\sfd f_n\to |\rmD f|_{\star,q}\text{ strongly in }L^q(X,\mssm).
  \end{equation}
  When $\AA$ is unital and point-separating, Definition \ref{def:density} is equivalent to either the equality $D^{1,q}(\X;\AA)=D^{1,q}(\X)$ with equal minimal relaxed gradients, or the equality of the Sobolev spaces $H^{1,q}(\X;\AA)=H^{1,q}(\X)$ with equal norms.
\end{remark}

\begin{remark}[Strong density]\label{rem:strong} If a subalgebra $\AA \subset \Lip_b(X, \sfd)$ satisfies \eqref{eq:4tris} and the space $H^{1,q}(\X)$ is reflexive, then $\AA$ is also strongly dense i.e. for every $f \in H^{1,q}(\X)$ there exists a sequence $(f_n)_n \subset \AA$ such that
\[ f_n \to f \quad \text{ in } H^{1,q}(\X).\]
This follows noticing that, for a sequence $(f_n)_n \subset \AA$ as in \eqref{eq:4tris}, the sequence $(|\rmD f_n|_{\star, q})_n$ is uniformly bounded in $L^q(X, \mssm)$ so that $(f_n)_n$ is uniformly bounded in $H^{1,q}(\X)$. By reflexivity we have that $f_n$ converges weakly in $H^{1,q}(\X)$ to some $g \in H^{1,q}(\X)$. Applying Mazur's theorem, we find a sequence $(g_n)_n \subset \AA$ converging strongly in $H^{1,q}(\X)$ (and thus in particular in $L^q(X, \mssm)$) to $g$. Since $f_n$ was converging in $L^q(X, \mssm)$ to $f$, it must be that $f=g$.
\end{remark}

In the following theorem we are going to consider a situation that may appear a bit artificial but it is actually what it is going to happen in the concrete case we are going to analyze.

\begin{theorem}\label{theo:startingpoint} Let $(X,\sfd)$ be a complete and separable metric space, let $\mssm$ be a finite, non negative Borel measure on $(X, \sfd)$ and let $\AA \subset \Lip_b(X, \sfd)$ be a unital and point-separating subalgebra.  Let $\delta$ be a distance on $X$ such that: 
\begin{enumerate}[$(a)$]
    \item $(X,\delta)$ is separable;
    \item the topology $\tau$ induced by $\delta$ coincides with the one induced by $\sfd$;
    \item it holds $\sfd \wedge 1 \le \delta \le \hat\mssd$ in $X \times X$.
\end{enumerate}
Let $Y\subset X$ be a countable $\mssd$-dense set. Setting $\X\eqdef (X, \tau, \sfd, \mssm)$ and $\X'\eqdef (X, \tau, \delta, \mssm)$, if
\begin{equation}\label{eq:214-15}
  \sfd_y\in D^{1,q}(\X';\AA)\quad \text{and}\quad
  \big|\rmD \sfd_y\big|_{\star, \X',q,\AA}\leq 1\comma \qquad y\in Y\comma
\end{equation}
then
$\AA$ is dense in $q$-energy in $D^{1,q}(\X')$ in the sense of Definition \ref{def:density}. 
  \end{theorem}
  \begin{proof}
  Note that, under the given assumptions, $(X,\sfd)$ and~$(X,\delta)$ are both complete and separable metric spaces. %
By \cite[Cor.~5.3.6]{Savare22} we have~$D^{1,q}(\X)=D^{1,q}(\X')$ with same minimal relaxed gradient, viz.
\[ |\rmD f|_{\star,\X, q}= |\rmD f|_{\star,\X',q}\comma \qquad f \in D^{1,q}(\X) = D^{1,q}(\X')\fstop\]
  The proof is precisely the same of \cite[Theorem 2.12]{FSS22} but for one (simple, yet crucial) detail: the distances $\sfd$ and $\delta$ are different in general: we are considering here the relaxed gradient of $\sfd_y$ induced by $\delta$. Note that the analogous condition with $\sfd$ in place of $\delta$ is stronger. Since the use of the two distances~$\delta$ and~$\sfd$ may cause confusion, we prove the assertion in full.
We split the proof in various steps.
    It suffices to prove that
      \begin{equation}
        \label{eq:15}
        |\rmD f|_{\star, \X',q, \AA}
        \le |\rmD f|_{\star, \X',q} = |\rmD f|_{\star, \X,q} \quad\text{$\mssm$-a.e.}\comma\qquad f \in D^{1,q}(\X)\ \tparen{= D^{1,q}(\X')} \comma
      \end{equation}
      since the opposite inequality holds by definition.

\begin{enumerate}[$(1)$,wide]
\item\label{i:t:startingpoint:1} \emph{Claim: It is not restrictive
        to assume $\sfd$ bounded above by $1$.}\quad Indeed: metric completeness, the induced length distance, the class of Lipschitz and bounded functions, and the definition of asymptotic Lipschitz constant are all invariant under truncation.

\item\label{i:t:startingpoint:2} \emph{Claim: It is sufficient to prove that}
      \begin{equation}
        \label{eq:92}
         \CE_{\X', q, \AA}(f)\le \int_X (\lip_\sfd f)^q\,\d\mssm \comma \qquad
f\in \Lip_b(X, \sfd) \fstop
      \end{equation}
      Indeed, if $f \in D^{1,q}(\X)=D^{1,q}(\X')$, we can find a sequence $f_n\in \Lip_b(X,\sfd)$ such that
    $f_n\to f$ $\mssm$-a.e.~and $\lip_\sfd f_n\to |\rmD f|_{\star, \X,q}= |\rmD f|_{\star, \X',q}$ strongly in
    $L^q(X,\mssm)$
    as $n\to\infty$.
     By the $L^0$-lower semicontinuity of the
    $\CE_{\X', q, \AA}$-energy, letting~$f=f_n$ in~\eqref{eq:92} and letting $n\to\infty$ we get
    \begin{displaymath}
      \CE_{\X', q, \AA}(f)\le \int_X |\rmD f|_{\star, \X',q}^q\,\d\mssm=\CE_{\X', q}(f)= \CE_{\X, q}(f) <\infty \fstop
    \end{displaymath}
    We deduce that $f$ has a $(\X', q, \AA)$-relaxed gradient (equivalently a $(\X, q, \AA)$-relaxed gradient) and that \eqref{eq:15} holds, 
    since $|\rmD f|_{\star, \X',q} = |\rmD|_{\star, q, \sfd} \le |\rmD f|_{\star, \X',q, \AA}$ $\mssm$-a.e.

\item \emph{Hopf--Lax regularizations.\quad} Let~$r \in(1,+\infty)$ be the H\"older conjugate exponent of $q$, satisfying i.e.~$1/r+1/q=1$. 
    For every $f\in \Lip_b(X, \sfd)$ and $t>0$ we introduce the Hopf--Lax
    regularization $\mssQ_t f:X\to \R$,
    \begin{align}
        \label{eq:16}
      \sfQ_t f(x)\eqdef {}&\inf_{y\in X}\frac{1}{ rt^{r-1}}\sfd^{ r }(x,y)+f(y)\comma
                       \quad x\in X\fstop
    \end{align} 
    It is clear that $\mssQ_t f$ is bounded (with values in the
    interval
    $[\inf_X f,\sup_X f]$) and $\sfd$-Lipschitz, being the infimum of a family
    of  uniformly  $\sfd$-Lipschitz functions.
    We further consider the upper semicontinuous function
    \cite[(3.4) and Prop.~3.2]{AGS14I}
      \begin{align}
       \label{eq:18}
        \sfD_t^+ f(x)\eqdef &\sup\limsup_{n\to\infty} \sfd(x,y_n),
      \end{align}
      where the supremum ranges over all the minimizing sequences~$\seq{y_n}_n$
      of \eqref{eq:16}.
      The function~$\sfD_t^+ f$ too is bounded uniformly in~$t$ and satisfies
      (see e.g.~\cite[Lem.~3.2.1]{Savare22})
      \begin{equation}
        \label{eq:20}
         \left (  \frac{\sfD_t^+ f(x)}t  \right )^{r}  \leq \tparen{ r\Li[\mssd]{f}}^q \fstop
      \end{equation}
      In fact, if $\seq{y_n}_n$ is a minimizing sequence for the right-hand side of~\eqref{eq:16}, for
      every $\eps>0$ we have, eventually in~$n$,
      \[
      \frac{1}{ rt^{r-1}}\sfd^{ r }(x,y_n)+f(y_n)\le
      \sfQ_tf(x)+\eps\le f(x)+\eps\comma
      \]
      i.e., setting $L\eqdef \Li[\mssd]{f}$,
      \begin{align*}
         \frac{1}{t^{ r}}\sfd^{ r }(x,y_n) &\leq \frac{\eps r}{t} + \frac{r}{t}\tparen{f(x)-f(y_n)} \le \frac{\eps r}{t} + rL \frac{\sfd(x,y_n)}{t} 
         \\
         &\leq \frac{\eps r}{t} + (rL)^q + \frac{\sfd^{r}(x,y_n)}{rt^{r}q^{1/(q-1)}} \fstop
      \end{align*}
      We thus get
      \[
        \limsup_{n\to\infty}\frac{1}{ t^{r}}\sfd^{ r}(x,y_n)\le  \frac{\eps r}{t} + (rL)^q 
      \]
      which yields \eqref{eq:20} since $\eps>0$ is arbitrary.

\item\label{i:t:startingpoint:5} \emph{Claim: For every $f\in \Lip_b(X,\sfd)$ and for every $t>0$,
        \begin{equation}
          \label{eq:19}
          |\rmD \sfQ_t f|_{\star, \X',q, \AA}\le  \tparen{ t^{-1} \sfD_t^+ f }^{r-1}  \quad
          \text{$\mssm$-a.e.} \fstop
        \end{equation}
      }
Indeed, let~$\set{y_n}_n$ be an enumeration of the countable $\mssd$-dense set~$Y$, fix~$f\in \Lip_b(X,\sfd)$, and set
\[
\sfQ^n_t f(x)\eqdef \min_{1\le k\le n}\frac{1}{ r t^{ r-1}}\sfd^{ r}(x,y_k)+f(y_k) \fstop
\]
It is readily checked that
\begin{equation}
        \label{eq:21}
        \sfQ_t f(x)=\inf_{y\in Y}\frac{1}{ r t^{ r-1}}\sfd^{ r}(x,y)+f(y)=
        \lim_{n\to\infty}\sfQ^n_tf(x)\fstop
      \end{equation}
       We consider now the upper semicontinuous   function
      \begin{equation}
        \label{eq:23}
        \sfD^n_t(x)\eqdef  \max\set{\sfd(x,y_k):1\le k\le n,\ 
        \sfQ_t^n(x)=\frac{1}{ r t^{ r-1}}\sfd^{ r}(x,y_k)+f(y_k)} \fstop
      \end{equation}
       By \eqref{eq:214-15} and Theorem \ref{thm:omnibus}\ref{i:t:Omnibus:9}, we have that $ (  t^{-1}\sfD^n_t  )^{r-1} $ is  a $(\X', q, \AA)$-relaxed gradient of~$\sfQ^n_t f$.
      It is then clear that for every $x$ there exists a a minimizing sequence~$\seq{z_{n,x}}_n$ of \eqref{eq:16} with~$z_{n,x}\in \set{y_1,\dotsc,y_n}$, i.e.\ such that 
      \[
      \begin{gathered}
      \sfD^n_t(x)= \sfd(x, z_{n,x}) \qquad \text{and}
      \\
      \lim_n \sfQ^n_tf(x)=\frac{1}{ r t^{ r-1}}\sfd^{ r}(x, z_{n,x} )+f( z_{n,x} )\to \sfQ_t f(x)\comma
      \end{gathered}
      \qquad x\in X\fstop
      \]
      We deduce that
      \begin{equation}\label{eq:24}
        \limsup_{n\to\infty}\sfD^n_t(x)=
         \limsup_{n\to\infty}\sfd(x, z_{n,x} )\le \sfD_t^+f(x)\comma \qquad x \in X \fstop
      \end{equation}

      Since~$\mssd\leq 1$, we have~$\sfD^n_t\leq 1$.
      In particular,~$\sfD^n_t$ is bounded uniformly in~$n$ for every~$t$.
      Therefore, we can assume with no loss of generality ---up to extracting a suitable non-relabeled subsequence--- that
      $(t^{-1}  \sfD^n_t )^{ r-1}$ converges weakly* in~$L^\infty(X,\mssm)$ to some~$G$.
      Thus,~$G$ is a $(\X', q, \AA)$-relaxed gradient of~$\sfQ_t f$ by Theorem \ref{thm:omnibus}\ref{i:t:Omnibus:2}, hence $|\rmD \sfQ_t f|_{\star, \X',q, \AA}\le G$ $\mssm$-a.e.~by Theorem~\ref{thm:omnibus}\ref{i:t:Omnibus:4}.
      Now, since~$\mssm$ is finite,~$\car_B\in L^1(X,\mssm)$ for every Borel~$B\subset X$. 
      Furthermore, since~$\mssm$ is finite, $\mssD^n_t\leq \car \in L^1(\mssm)$ uniformly in~$n$ and~$t$.
      We can thus apply the reverse Fatou lemma and conclude from the $L^\infty(X,\mssm)$-weak*-convergence of~$(t^{-1}\sfD^n_t)^{r-1}$ to~$G$ that, for every Borel~$B \subset X$,
      \begin{align*}
      \int_B G \, \d \mssm =&\ \lim_n \int_B \tparen{t^{-1}\sfD^n_t}^{r-1} \, \d \mssm \le \int_B \limsup_n \tparen{t^{-1}\sfD^n_t}^{r-1} \, \d \mssm
      \\
      \leq& \int_B \tparen{t^{-1}\rmD_t^+ f}^{r-1} \, \d \mssm \comma
      \end{align*}
      where the last inequality follows from~\eqref{eq:24}.
      We conclude that $|\rmD \sfQ_t f|_{\star, \X',q, \AA}\le (t^{-1}\rmD_t^+ f)^{r-1}$ $\mssm$-a.e..

     \item \emph{Claim: For every $x\in X,\ t>0,$ and $f\in
        \Lip_b(X,\sfd)$
        we have
        \begin{align}
          \label{eq:25}
          \frac{f(x)-\sfQ_tf(x)}t&=\frac 1{ q}\int_0^1
                                   \paren{\frac{\rmD_{ut}^+f(x)}{ut}}^{ r}
                                   \,\d u\comma
                                   \\
          \label{eq:27}
          \limsup_{t\downarrow0}\frac{f(x)-\sfQ_tf(x)}t&\le \frac1{ q}\tparen{\lip_\sfd f(x)}^{ q} \fstop
        \end{align}
      }
      This follows by \cite[Thm. 3.2.4]{Savare22};
      also cf.~\cite[Thm.~3.1.4, Lemma 3.1.5]{AGS08}.

    \item \emph{Conclusion.}
      We argue as  in  \cite[Theorem 3.2.7]{Savare22}:
      \eqref{eq:25} and \eqref{eq:20} yield the uniform bound
      \begin{equation}
        \label{eq:26}
        \frac{f(x)-\sfQ_tf(x)}t\le  \frac{1}{q}  \tparen{r \Li[\mssd]{f}}^{ q}\comma \qquad\quad
 x\in X\comma t>0 \fstop
      \end{equation}
      Integrating \eqref{eq:27} in $X$ 
      and applying the reverse Fatou Lemma we get
      \begin{equation}
        \label{eq:28}
        \limsup_{t\downarrow0}
        \int_{X} \frac{f-\sfQ_tf}{t}\,\d\mssm \le
        \frac1{ q}\int_{ X} (\lip_\sfd f)^{ q} \,\d\mssm \fstop
      \end{equation}
      On the other hand, \eqref{eq:25} and Fubini's Theorem yield
      \begin{equation}
        \label{eq:29}
        \int_{ X}  \frac{f-\sfQ_tf}{t}\,\d\mssm
        =\frac 1{ q} \int_0^1 \int_{ X}  \paren{\frac{\sfD_{ut}^+f}{ut}}^{ r}
        \,\d \mssm\,\d u \comma
      \end{equation}
and an application of Fatou's Lemma yields
      \begin{equation}
        \label{eq:29bis}
        \liminf_{t\downarrow0}\int_{ X}  \frac{f-\sfQ_tf}{t}\,\d\mssm
        \ge  \frac 1{ q}\liminf_{t\downarrow0} \int_{ X}  \paren{\frac{\sfD_{t}^+f}{t}}^{ r}
        \,\d \mssm \fstop
      \end{equation}
      Using the fact that
      $t^{-1}\sfD_t^+f$ is uniformly bounded by \eqref{eq:20}, we can
      find a decreasing and vanishing sequence
      $\seq{t_n}_n\subset \R^+$ and a limit function $G\in
      L^\infty(X,\mssm)$
      such that
      \begin{gather}
          \left (  t_n^{-1}\sfD_{t_n}^+f  \right )^{r-1}  \weakto^* G\quad \text{weakly$^*$ in
           $L^\infty(X,\mssm)$}\quad \text{as }n\to\infty,\notag
           \\
         \label{eq:117}
        \lim_{n\to\infty}
        \int_{ X}  \paren{\frac{\sfD_{t_n}^+f}{t_n}}^{ r}
        \,\d \mssm=\liminf_{t\downarrow0} \int_{ X}  \paren{\frac{\sfD_{t}^+f}{t}}^{ r}
        \,\d \mssm \fstop
      \end{gather}
      Since $ \left (  t^{-1}\sfD_t^+f  \right )^{r-1} $  is  a $(\X', q, \AA)$-relaxed gradient of
      $\sfQ_t f$ by Claim~\ref{i:t:startingpoint:5} and~$\sfQ_t f\to f$ pointwise
      everywhere on~$X$, we see that~$G$ is
      a $(\X',q,\AA)$-relaxed gradient of $f$ by Theorem \ref{thm:omnibus}\ref{i:t:Omnibus:2}.
      
      Using the lower semicontinuity of the
       $L^{ q}$-norm w.r.t.~the weak$^*$ $L^\infty(X,\mssm)$
      convergence,
      \begin{equation}\label{eq:30}
      \begin{aligned}
        \lim_{n\to\infty} \int_{ X} \paren{\frac{\sfD_{t_n}^+f}{t_n}}^{ r}
        \,\d \mssm &=\lim_{n\to\infty} \int_{ X} \paren{\frac{\sfD_{t_n}^+f}{t_n}}^{ q(r-1)}
        \,\d \mssm
        \\
        &\ge \int_{ X} G^{ q} \,\d\mssm
      \ge \int_{ X} |\rmD f|_{\star, \X',q, \AA}^{ q}\,\d\mssm \comma
      \end{aligned}
      \end{equation}
      where we also used the pointwise minimality of  $|\rmD f|_{\star, \X',q, \AA}$  given by Theorem \ref{thm:omnibus}\ref{i:t:Omnibus:4}. 
      Combining \eqref{eq:30}, \eqref{eq:117}, \eqref{eq:29bis} and
      \eqref{eq:28} 
      we deduce that
      \begin{displaymath}
        \int_{ X}  |\rmD f|_{\star, \X',q, \AA}^{ q}\,\d\mssm \le \int_{ X} (\lip_\sfd
        f)^{ q}\,\d\mssm
      \end{displaymath}
      so that \eqref{eq:92} holds. \qedhere
      \end{enumerate}
  \end{proof}

\subsection{Stability under changes of measures and distances}

We recall the notion of infinitesimal Hilbertianity.

\begin{definition}
A Polish metric measure space~$(X,\mssd,\mssm)$ is \emph{infinitesimally Hilbertian} if~$H^{1,2}(X, \mssd,\mssm)$ is a Hilbert space --- or, equivalently, if~$\CE_{2,\mssm}$ is a quadratic form.

A complete and separable metric space~$(X,\mssd)$ is \emph{universally infinitesimally Hilbertian} if~$(X,\mssd,\mssm)$ is infinitesimally Hilbertian for \emph{every} non-negative Borel measure~$\mssm$ on~$(X,\sfd)$ finite on $\mssd$-bounded sets.
\end{definition}

\subsubsection{Changes of measures}
In this section we study how to derive the infinitesimal Hilbertianity and the density of subalgebras of Lipschitz functions for a Polish metric measure space with possibly infinite measure $\mssm$ when the same property is satisfied by a class of finite measures that have a good compatibility with $\mssm$. 

In particular, if $(X, \sfd, \mssm)$ is a Polish metric measure space, we will consider measures $\mssm'$ satisfying 
\begin{equation}\label{eq:thecond}
\begin{split}
&\,\,\vartheta \eqdef  \frac{ \de \mssm'}{\de \mssm} \in L^1 (X,\mssm) \text{ and for every compact set $K \subset X$ there exist } \\ &\text{$r>0$ and~$c>0$ satisfying }
0<c\le \vartheta \le 1 \quad \text{on}\quad \set{ x \in X : \sfd(x,K) \le r} \fstop
\end{split}
\end{equation}
This condition guarantees that~$H^{1,2}(X, \sfd, \mssm) \subset H^{1,2}(X, \sfd,\mssm')$ with equal minimal relaxed gradients, see \cite[Lem.~4.11]{AGS14I} or~\cite[Lem.~4.7]{LzDSSuz23a}. (Note that $\mssm'$ and $\mssm$ share the same class of negligible sets.)

\begin{lemma}\label{l:TransferFunction} Let $(X, \sfd, \mssm)$ be a Polish metric measure space and fix~$x_0 \in X$. Then, there exist a continuous function~$\vartheta=\vartheta_{x_0} \in L^1(X,\mssm)^+$ and, for  every~$r>0$, a constant~$c(r)>0$, such that~$0<c(r) \leq \vartheta \le 1$ in $\rmB(x_0,r)$.
In particular, $\mssm'\eqdef \vartheta \mssm$ satisfies~\eqref{eq:thecond}.
\begin{proof}
Let $f: [0,+\infty) \to [1,+\infty)$ be a nondecreasing continuous function such that
\[
f(n)=\mssm  \ball{x_0}{n+1}+1\comma \qquad n \in \N_0 \comma
\]
and set~$\varrho\colon t\mapsto t+\log f(t))$ for~$t>0$.
Now, define~$V=V_{x_0}\eqdef \varrho\,\circ\,\mssd_{x_0}$, and~$\vartheta=\vartheta_{x_0} \eqdef e^{-V}$, and note that~$\vartheta$ is continuous and takes values in~$(0,1]$.
It is clear from the monotonicity of~$f$ (hence, in turn, of~$\varrho$) that, wherever~$\mssd_{x_0} \leq r$ for some $r >0$, we have~$\vartheta \geq e^{-r}/f(r)>0$, which shows the lower bound for~$\vartheta$ with~$c(r)\eqdef e^{-r}/f(r)>0$.
Analogously, wherever~$\mssd_{x_0} \geq r$ for some $r >0$, we have~$\vartheta \leq e^{-r}/f(r)>0$, and thus
\begin{align*}
    \int_X \abs{\vartheta} \,\d\mssm &= 
    \sum_{n=1}^{\infty} \int_{\ball{x_0}{n} \setminus \ball{x_0}{n-1}} \vartheta \,\d \mssm 
    \leq 
    \sum_{n=1}^{\infty} \mssm\, \rmB(x_0,n) \frac{e^{1-n}}{f(n-1)} 
    \\
    &
    \leq 
    \sum_{n=1}^{\infty} e^{1-n} <+\infty \fstop\qedhere
\end{align*}
\end{proof}
\end{lemma}

\begin{proposition}\label{p:IH} Let $(X, \sfd, \mssm)$ be a Polish metric measure space and let~$\seq{\mssm_k}_k$ be a sequence of measures~$\mssm_k\eqdef \vartheta_k\mssm$ satisfying~\eqref{eq:thecond}, and assume either of the following:
\begin{enumerate}[$(a)$]
\item\label{i:p:IH:1} $\vartheta_k \in L^2(X,\mssm)$ and~$\vartheta_k\rightharpoonup \car$ in $L^2(X,\mssm)$;
\item\label{i:p:IH:2} $\vartheta_k \uparrow \car$ $\mssm$-a.e..
\end{enumerate}
If~$H^{1,2}(X, \sfd,\mssm_k)$ is Hilbert for every~$k$, then~$H^{1,2}(X, \sfd, \mssm)$ too is Hilbert.

\begin{proof}
In light of~\eqref{eq:thecond},~$|\rmD (\emparg)|_{\star, \X, 2}$ coincides with~$|\rmD (\emparg)|_{\star, \X_k, 2}$ on $H^{1,2}(\X)\subset H^{1,2}(\X_k)$ for every~$k$ by~\cite[Lemma~4.11]{AGS14I}, where $\X\eqdef (X, \sfd, \mssm)$ and $\X_k\eqdef (X, \sfd, \mssm_k)$.
Thus, throughout the proof we may write~$|\rmD (\emparg)|_{\star, 2}$ in place of both~$|\rmD (\emparg)|_{\star, \X, 2}$ and~$|\rmD (\emparg)|_{\star, \X_k, 2}$.
Fix~$f,g\in H^{1,2}(\X)$.

Assume~\ref{i:p:IH:1}.
By the infinitesimal Hilbertianity of~$H^{1,2}(\X_k)$,
    \begin{align*}
    \int_X |\rmD (f+g)|_{\star, 2}^2 \vartheta_k \, \d \mssm &+ \int_X |\rmD (f-g)|_{\star, 2}^2 \vartheta_k \, \d \mssm 
    \\
    &=  \int_X |\rmD (f+g)|_{\star, 2}^2 \, \d \mssm_k + \int_X |\rmD (f-g)|_{\star, 2}^2 \, \d \mssm_k \\
    &= 2\int_X |\rmD f|_{\star, 2}^2 \, \d \mssm_k + 2 \int_X |\rmD g|_{\star, 2}^2 \, \d \mssm_k\\ 
  &= 2\int_X |\rmD f|_{\star, 2}^2 \vartheta_k\, \d \mssm + 2 \int_X |\rmD g|_{\star, 2}^2  \vartheta_k \, \d \mssm
    \end{align*}
    Passing to the limit on both sides as $k \to + \infty$ we get
    \[ \CE_{\X,2}(f+g)+ \CE_{\X,2}(f-g) = 2\,\CE_{\X,2}(f) + 2\,\CE_{\X,2}(g)\comma \]
    which concludes that $\CE_{\X,2}$ is a quadratic form and thus that $H^{1,2}(\X)$ is a Hilbert space.

Assume now~\ref{i:p:IH:2}. Since~$\vartheta_k\leq 1$ we have~$\mssm_k \le \mssm$.
By the infinitesimal Hilbertianity of~$H^{1,2}(\X_k)$,
    \begin{align*}
    \int_X |\rmD (f+g)|_{\star, 2}^2 \vartheta_k \, \d \mssm &+ \int_X |\rmD (f-g)|_{\star, 2}^2 \vartheta_k \, \d \mssm
    \\
    &=  \int_X |\rmD (f+g)|_{\star, 2}^2 \, \d \mssm_k + \int_X |\rmD (f-g)|_{\star, 2}^2 \, \d \mssm_k \\
    &= 2\int_X |\rmD f|_{\star, 2}^2 \, \d \mssm_k + 2 \int_X |\rmD g|_{\star, 2}^2 \, \d \mssm_k \\
    & \le 2\int_X |\rmD f|_{\star, 2}^2 \, \d \mssm + 2 \int_X |\rmD g|_{\star, 2}^2 \, \d \mssm \\
    &= 2\, \CE_{\X,2}(f)+ 2\, \CE_{\X,2}(g).
    \end{align*}
    Passing to the limit as $k \to + \infty$ we get that 
    \[ \CE_{\X,2}(f+g)+ \CE_{\X,2}(f-g) \le 2\, \CE_{\X,2}(f) + 2\, \CE_{\X,2}(g)\]
    which is enough (cf.\ e.g.~\cite[Prop.~11.9]{DalMaso93}) to conclude that $\CE_{\X,2}$ is a quadratic form and thus that $H^{1,2}(\X)$ is a Hilbert space.
\end{proof}
\end{proposition}

Proposition~\ref{p:IH} shows that infinitesimal Hilbertianity is stable under suitable limits.
As we now show, this entails that universal infinitesimal Hilbertianity can be checked on finite measures only.

\begin{proposition}\label{prop:extt}
Let~$(X,\mssd)$ be a complete and separable metric space and assume that~$(X,\mssd,\mssm)$ is infinitesimally Hilbertian for every non-negative, \emph{finite} Borel measure $\mssm$ on~$X$.
Then,~$(X,\mssd)$ is universally infinitesimally Hilbertian.
\end{proposition}
\begin{proof}
Let~$\mssm$ be a non-negative Borel measure on~$X$ finite on $\mssd$-bounded sets.
It suffices to construct a sequence of finite measure~$\mssm_k\eqdef \vartheta_k \mssm$ with~$\vartheta_k\in L^1(X,\mssm)$ as in \eqref{eq:thecond} and satisfying~\ref{i:p:IH:2} in Proposition~\ref{p:IH}.
Indeed, in this case~$H^{1,2}(X,\mssd,\mssm_k)$ is Hilbert by assumption, since~$\mssm_k$ is a finite measure, and the infinitesimal Hilbertianity of~$H^{1,2}(X,\mssd,\mssm)$ follows from Proposition~\ref{p:IH}.

Fix~$x_0 \in X$, let~$\varrho$,~$V$, and~$\vartheta$ be as in the proof of Lemma~\ref{l:TransferFunction}, and set~$V_k\eqdef \min(V,k)$, $\vartheta_k\eqdef e^{-(V-V_k)}$, and~$\mssm_k\eqdef \vartheta_k \mssm$. Observe that each measure $\mssm_k$ satisfies \eqref{eq:thecond}: if $\sfd(x,x_0) \le r$ for some $r>0$, then $V-V_k \le V \le \varrho(r)$ so that $\vartheta_k \ge \frac{e^{-r}}{f(r)}>0$. 
Since every compact set~$K$ is contained in some ball, this proves that~$\vartheta_k$ stays strictly positive in any enlargement of~$K$; by construction we also have that $0< \vartheta_k \le 1$ in~$X$ so that~$\mssm_k \le \mssm$. Furthermore,
\begin{align*}
\mssm_k (X) &= \int_X \vartheta_k \,\d \mssm \le \mssm \set{V \le k} + e^{k} \int_X e^{-V} \,\d \mssm 
\\
&\leq \mssm\, \ball{x_0}{k} + e^{k} \int_X e^{-V} \,\d \mssm <+\infty \fstop
\end{align*}
    It is also clear that $\vartheta_k \uparrow \car $ in $X$, which verifies the assumption in Proposition~\ref{p:IH}\ref{i:p:IH:2} and thus concludes the proof.
\end{proof}

In the next result we show how the density of subalgebras of Lipschitz functions can be transferred from finite measures to infinite measure, assuming that suitable truncation functions are available in the algebra.
For simplicity of notation, let~$\lip_\sfd\AA\eqdef \{ \lip_\sfd f : f \in \AA \}$.
\begin{lemma}\label{le:densinf} Let $(X, \sfd, \mssm)$ be a Polish metric measure space, let $q \in (1,+\infty)$, and let $\AA \subset \Lip_b(X, \sfd)$ be a subalgebra such that $\AA, \lip_\sfd\AA \subset L^q(X, \mssm)$.  If there exists functions $(\vartheta_k)_k \subset \Lip_b(X, \sfd)$ such that:
\begin{enumerate}[$(a)$]
\item for every $k \in \N$, $\mssm_k \eqdef \vartheta_k \mssm$ satisfies \eqref{eq:thecond}, $\AA$ is dense in $q$-energy in $D^{1,q}(X, \sfd, \mssm_k)$, $\frac{(\lip_\sfd \vartheta_k)^q}{\vartheta_k} \in L^\infty(X, \mssm)$, and $f \vartheta_k \in \AA$ for every $f \in \AA$; 
\item  $\sup_k \|\lip_\sfd \vartheta_k \|_\infty < + \infty$, $\vartheta_k \uparrow \car$ and $\lip_\sfd \vartheta_k \to 0$ $\mssm$-a.e.,~as $k \to + \infty$.
\end{enumerate}
Then~$\AA$ is dense in $H^{1,q}(X, \sfd, \mssm)$ in the following sense: for every $f \in H^{1,q}(X, \sfd, \mssm)$ there exist a sequence $(f_n)_n \subset \AA$ satisfying
\[ f_n \to f, \quad \lip_{\sfd} f_n \to |\rmD f|_{\star, q} \quad \text{ in } L^q(X, \mssm) \text{ as } n \to + \infty.\]

\begin{proof}
Let us set $\X\eqdef (X, \sfd, \mssm)$ and $\X_k\eqdef (X, \sfd, \mssm_k)$. Recall that~$H^{1,q}(\X) \subset H^{1,q}(\X_k)$ as a consequence of~\eqref{eq:thecond}.
Since $\AA$ is dense in $q$-energy in $D^{1,q}(\X_k)$, by Definition \ref{def:density} and Remark \ref{rem:dense2} ($\mssm_k$ is finite), for every $f \in H^{1,q}(\X)$ there exist sequences $\seq{f_n^k}_n \subset \AA$ satisfying
\[
f_n^k \to f, \quad \lip_\sfd f_n^k \to |\rmD f|_{\star, \X_k, q} \text{ in } L^q(X, \mssm_k) \text{ as } n\to + \infty\fstop
\]
By \cite[Lemma 4.11]{AGS14I} we may replace $|\rmD f|_{\star, \X_k, q'}$ with $|\rmD f|_{\star, \X, q}$.  Let us set 
\begin{gather*}
g_n^k \eqdef f_n^k \vartheta_k\comma \qquad g^k\eqdef f \vartheta_k\comma \qquad \lip_\sfd g_n^k \le G_n^k\eqdef  \vartheta_k \lip_\sfd f_n^k + |f_n^k| \lip_\sfd \vartheta_k
\\
G^k \eqdef \vartheta_k|\rmD f|_{\star, \X, q}  + |f|\lip_\sfd \vartheta_k \fstop
\end{gather*}
We note that $g_n^k \in \AA$, $g_n^k \to g^k$ and $G_n^k \to G^k$ in $L^q(X, \mssm)$ as $n \to + \infty$, and $g^k \to f$ and $G^k \to |\rmD f|_{\star, \X, q}$ in $L^q(X, \mssm)$ as $k \to + \infty$. The diagonal argument in $L^q(X, \mssm) \times L^q(X, \mssm)$ gives the existence of a subsequence $k \mapsto n_k$ such that setting $f_k\eqdef g_{n_k}^k \in \AA$ we have
\[
f_k \to f \qquad \text{and} \qquad \lip_\sfd f_k \le G_{n_k}^k \to |\rmD f|_{\star, \X, q} \qquad \text{in } L^q(X, \mssm) \fstop
\]
 Up to passing to a (non-relabeled) subsequence we get that
\[
f_k \to f  \qquad \text{and} \qquad \lip_{\sfd} f_k \weakto G \qquad \text{in } L^q(X, \mssm) 
\]
for some~$G\in L^q(X,\mssm)^+$ with~$G \le |\rmD f|_{\star, \X, q}$.
Since~$G$ is a $q$-relaxed gradient of $f$ by Definition \ref{def:relgrad}, we conclude that~$G=|\rmD f|_{\star, \X, q}$ by the $L^q(X, \mssm)$-minimality of~$|\rmD f|_{\star, \X, q}$.
The convergence is also strong in $L^q(X, \mssm)$ since
    \[ \limsup_n \int_X|\lip_\sfd f_k|^q \, \d \mssm \le \limsup_n \int_X |G_{n_k}^k|^q \, \d \mssm = \int_X |\rmD f|_{\star, \X, q}^q \, \d \mssm.\]
    This concludes the proof.
\end{proof}
\end{lemma}

\subsubsection{Changes of distances}
Here we consider the problem of transferring the density or Hilbertianity property from a (family of) distance(s) to another distance.

The following lemma is an immediate consequence of the fact that, setting $\sfd_\lambda\eqdef  \lambda \sfd$, $\lambda >0$, we have for any function $f:X \to \R$ and $A \subset X$
\[ \Li[\mssd_\lambda,A]{f} = \lambda^{-1} \Li[\mssd,A]{f}.\]

\begin{lemma}\label{le:easy} Let $\X=(X, \tau, \sfd, \mssm)$ be a e.m.t.m.~space, let $\AA \subset \Lip_b(X, \tau, \sfd)$ be a subalgebra and let $q \in (1,+\infty)$. Then, setting $\X_\lambda\eqdef  (X, \tau, \sfd_\lambda, \mssm)$, we have that $\AA$ is dense in $q$-energy in $D^{1,q}(\X)$ if and only if it is dense in $q$-energy in $D^{1,q}(\X_\lambda)$ and $H^{1,2}(\X)$ is Hilbert if and only if $H^{1,2}(\X_\lambda)$ is Hilbert.
\end{lemma}

The next result shows how to transfer the density of a subalgebra $\AA$ from a family of distances approximating $\sfd$ to $\sfd$ itself.

\begin{proposition}\label{prop:erbar} Let $\X\eqdef (X, \tau, \sfd, \mssm)$ be an e.m.t.m.~space with $\mssm$ finite, let $\seq{\sfd_i}_{i \in I}$ be a directed family of complete $\tau$-continuous distances in $X$ such that $\sfd= \sup_{i \in I} \sfd_i$ and let $\AA  \subset \cap_{i \in I} \Lip_b(X, \tau, \sfd_i) \subset \Lip_b(X, \tau, \sfd)$ be a unital and point-separating subalgebra. Set $\X_i \eqdef  (X, \tau_{\sfd_i}, \sfd_i, \mssm)$.
Then,
\begin{enumerate}[$(i)$]
\item\label{i:p:Erbar:1} if $\AA$ is dense in $2$-energy in $D^{1,2}(\X_i)$ for every $i \in I$, then $\AA$ is dense in $2$-energy in~$D^{1,2}(\X)$;
\item\label{i:p:Erbar:2} if $H^{1,2}(\X_i)$ is Hilbert for every $i \in I$, then $H^{1,2}(\X)$ is Hilbert too.
\end{enumerate}
\end{proposition}
\begin{proof} We want to apply \cite[Theorem 9.1]{aes} in our setting so that we need to make a few comments:
\begin{enumerate*}[$(a)$]
\item The Cheeger energy in \cite{aes} is defined as a functional on $L^2(X, \mssm)$ so that we can compare it with the restriction of our notion of Cheeger energy to $L^2(X, \mssm)$.
\item The Cheeger energy in \cite{aes} is defined starting from a different notion of asymptotic Lipschitz constant but it turns out that it coincides with our notion of Cheeger energy \cite[proof of Proposition 2.22]{LzDSSuz23a}.
\item In the notation of the discussion above \cite[Theorem 9.1]{aes} we have that, in our case, $\pi^i$ is the identity map in $X$ so that $\tilde{\sfd}_i=\sfd_i$, $X_i=X$, $\mm_i=\mssm$ and $\pi^i_\star$ is the identity in $L^2(X, \mssm)$.
\end{enumerate*}
Again in the notation of \cite[Theorem 9.1]{aes}, we deduce that $\mathsf{Ch}_i$ coincides with the restriction of $\CE_{\X_i, 2}$ to $L^2(X, \mssm)$. Applying \cite[Theorem 9.1]{aes} we deduce that 
\[ \CE_{\X, 2}(f) = \inf \left \{ \liminf_{n \to + \infty} \inf_{i \in I} \CE_{\X_i, 2}(f_n) : \begin{gathered} (f_n)_n \subset L^2(X, \mssm)\comma \\ f_n \to f \in L^2(X, \mssm)\end{gathered} \right \}, \quad f \in L^2(X, \mssm).\] 
We thus have that for every $f \in L^2(X, \mssm)$ there exists a sequence $(f_k)_k \subset L^2(X, \mssm)$ and a subsequence $(i_k)_k \subset I$ such that 
\begin{equation}\label{eq:altrochep}
\| f - f_k\|_{L^2(X, \mssm)} < 1/k, \quad \CE_{\X, 2}(f) + 1/k \ge \CE_{\X_{i_k}, 2} (f_k) \quad \text{ for every } k \in \N.
\end{equation}

\paragraph{Proof of~\ref{i:p:Erbar:1}}
If $\AA$ is dense in $2$-energy in $D^{1,2}(\X_i)$ for every $i \in I$, we deduce that for every $f \in L^2(X, \mssm)$ there exists a sequence $(f_n)_n \subset \AA$ and a subnet $(i_k)_k \subset I$ such that 
\[
\| f - f_k\|_{L^2(X, \mssm)} < 1/k\comma \quad \CE_{\X, 2}(f) + 1/k \ge \pCE_{\X_{i_k}, 2} (f_k) \ge \pCE_{\X, 2}(f_k)\comma \qquad k \in \N \fstop
\]
Passing the above inequality to the $\liminf_k$ we get that, for every $f \in L^2(X, \mssm)$ it holds
\[ \CE_{\X, 2}(f) \ge \liminf_k  \pCE_{\X, 2}(f_k)  \ge \CE_{\X,2, \AA}(f).\]
This proves the density of in $2$-energy of $\AA$ in $D^{1,2}(\X)$.

\paragraph{Proof of~\ref{i:p:Erbar:2}}
Let $f,g \in L^2(X, \mssm)$ and let $(f_k)_k, (g_k)_k$ be sequences as in \eqref{eq:altrochep} for $f$ and $g$ respectively (note that this can be done for the same subnet $(i_k)_k$ since the hypotheses of the present theorem also holds restricting $I$ to $(i_k)_k$). We have
\begin{align*}
\inf_{i \in I} \CE_{\X_i, 2} (f_k+g_k) + \inf_{i \in I} \CE_{\X_i, 2} (f_k-g_k)  & \le  \CE_{\X_{i_k}, 2} (f_k+g_k) + \CE_{\X_{i_k}, 2} (f_k-g_k)
\\
& = 2\, \CE_{\X_{i_k}}(f_k) + 2\CE_{\X_{i_k}}(g_k)
\\
& \le 2\, \CE_{\X}(f) + 2\CE_{\X}(g) + \frac{4}{k} \fstop
\end{align*}
Taking~$\liminf_k$ in the above inequality, we deduce that 
\begin{align*}
2\, \CE_{\X}(f) + 2\, \CE_{\X}(g) \ge&\ \liminf_k \inf_{i \in I} \CE_{\X_i, 2} (f_k+g_k) + \liminf_k \inf_{i \in I} \CE_{\X_i, 2} (f_k-g_k) 
\\
\ge&\ \CE_{\X}(f+g) + \CE_{\X}(f-g) \fstop
\end{align*}
This is enough (cf.\ e.g.~\cite[Prop.~11.9]{DalMaso93}) to conclude that $\CE_{\X,2}$ is a quadratic form in $L^2(X, \mssm)$ and thus that $H^{1,2}(\X)$ is a Hilbert space.
\end{proof}

\section{Distances on measures}\label{s:DistMeas}
In this section we introduce the distances on the space of measures we are going to work with. To this aim we fix a complete and separable metric space $(U,\varrho)$ and we consider the space $\mcM(U)$ of finite, non-negative Borel measures on $(U, \varrho)$.

The first distance we will work with is the Wasserstein distance induced by $\varrho$.
\begin{definition}[Extended Wasserstein distance]\label{def:wass}
We define the \emph{extended} (\emph{Kant\-orovich--Rubinstein}) \emph{Wasserstein $(2,\varrho)$-distance}
\begin{gather*}
\W_{2,\varrho}\colon \mcM(U) \to [0,+\infty]
\\ 
\W_{2,\varrho}(\mu_0, \mu_1)^2\eqdef  \inf \left \{ \int_{U \times U} \varrho^2 \de \ggamma : \ggamma \in \Cpl(\mu_0, \mu_1) \right \}, \quad \mu_0, \mu_1 \in \mcM(\R^d)\comma
\end{gather*}
where $\Cpl(\mu_0, \mu_1)$ denotes the set of couplings between $\mu_0$ and $\mu_1$ defined as
\[ \Cpl(\mu_0, \mu_1) \eqdef  \{ \ggamma \in \mcM(U \times U) : \pi^i_\sharp \ggamma = \mu_i, \, i=0,1 \},\]
being $\pi^i: U \times U \to U$ the projection $\pi^i(x_0, x_1)\eqdef  x_i$ for every $(x_0, x_1) \in U \times U$.
\end{definition}
Note that, whenever $\mu_0(U) \ne \mu_1(U)$, then $\Cpl(\mu_0, \mu_1) = \emptyset$ so that $\W_{2,\varrho}(\mu_0, \mu_1)=+\infty$. When we restrict $\W_{2,\varrho}$ to the subset of $\mcP(U)$ (the Borel probability measures in $U$) given by 
\[ \mcP_2(U)\eqdef  \{ \mu \in \mcP(U) : \int_U \varrho^2(x_0, x) \de \mu(x) < + \infty \text{ for some (hence for every) $x_0 \in U$} \} \]
then $\W_{2, \varrho}$ is finite and $(\mcP_2(U), \W_{2, \varrho})$ is a complete and separable metric space whose topology is stronger than (the restriction of ) the weak topology.
We refer to the monographs \cite{Villani09, Villani03, AGS08} for a throughout presentation of Wasserstein distances.
\medskip

The second distance we introduce is the well known Hellinger distance.

\begin{definition}[Hellinger distance]\label{def:he}  We define the \emph{Hellinger $2$-distance}
\begin{gather*}
\He_2\colon \mcM(U) \to [0,+\infty)\comma
\\
\He_2(\mu_0, \mu_1)^2\eqdef  \int_{U} \abs{ \sqrt{ \frac{\de \mu_0}{\de \eta}} - \sqrt{\frac{\de \mu_1}{\de \eta}} }^2 \de \eta, \quad \mu_0, \mu_1 \in \mcM(\R^d)\comma
\end{gather*}
where $\eta \in \mcM(U)$ is any measure such that $\mu_i \ll \eta$ for $i=0,1$. Notice that, since $(s,t) \mapsto | \sqrt{t}-\sqrt{s} |^2$ is positively $1$-homogeneous, the above definition does not depend on $\eta$.
\end{definition}

\subsection{The Hellinger--Kantorovich distances}\label{sec:hk}
We follow \cite{LMS18} to introduce the Hellinger--Kantorovich distance, see also \cite{CPSV18}.
We introduce on $U \times \R_+$ the equivalence relation 
\begin{equation*}
(x,r) \sim (y,s) \quad \overset{\text{def}}{\iff} \quad \left [ x=y, r=s \ne 0 \quad \vee \quad r=s=0 \right ] 
\end{equation*}
and the corresponding geometric cone $\f{C}[U]\eqdef (U \times \R_+)/\sim$, whose points are denoted by fraktur letters as $\f{y}$.
We denote by $\f{p}$ the quotient map $\f{p}\colon U \times \R_+ \to \f{C}[U]$ mapping a point~$(x,r)$ to its equivalence class~$[x,r]$.
Note that~$\f{p}$ is just the identity map except for those points with~$r=0$, all mapped to the same equivalence class, the so called \emph{vertex} of the cone, that we denote by~$\f{o}$.

On the cone~$\f{C}[U]$ we introduce the projections on $\R_+$ and $U$ simply defined as $\sfr([x,r]) = r$ and $\sfx([x,r]) = x$ if $r>0$ and $\sfx([x,r])=\bar{x}$ if $r=0$, where $\bar{x} \in U$ is some fixed point. We omit the dependence of $\sfx$ on $\bar{x}$ since in the constructions where $\sfx$ is involved this will be irrelevant.

On $\f{C}[U]$ we consider the following topology, weaker then the quotient one: a local system of neighbourhoods of a point $[x,r]$ is just the image trough $\f{p}$ of the local system of neighbourhoods given by the product topology at $(x,r) \in U \times \R_+$, if $r>0$. A local system of neighbourhoods at $\f{0}$ is given by
\begin{equation*}
\tset{ \set{ [x,r] \in \f{C}[U] : 0 \le r < \eps } : \eps >0 } \fstop
\end{equation*}
The topology of $\f{C}[U]$ is induced, for every $a \in (0,\pi]$, by the distance $\varrho_{a,\f{C}}: \f{C}[U]\times \f{C}[U] \to [0, +\infty)$ defined as
\begin{equation}\label{ss22:eq:distcone} \varrho_{a,\f{C}}([x,r],[y,s]) \eqdef  \tparen{ r^2+s^2-2rs\cos(\varrho(x,y) \wedge a) }^{\frac{1}{2}}, \quad [x,r], [y,s] \in \f{C}[U].
\end{equation}
With this distance, $(\f{C}[U], \varrho_{a,\f{C}})$ is a Polish metric space. We introduce the set
\begin{equation*}
\mathcal{\f{N}}_+^2(\f{C}[U]) \eqdef  \left \{ \alpha \in \mcM(\f{C}[U]) : \int_{\f{C}[U]} \sfr^2 \d \alpha < + \infty  \right \}, 
\end{equation*}
and the map
\begin{equation*}
\f{h} \colon \mathcal{\f{N}}_+^2(\f{C}[U]) \to \mcM(U), \quad \f{h}(\alpha) = \sfx_\sharp (\sfr^2 \alpha)\fstop
\end{equation*}
Note that the map $\f{h}$ does not depend on the point $\bar{x}$ in the definition of $\sfx$.

\bigskip

We introduce now the product cone: we define $\pc\eqdef  \f{C}[U]\times \f{C}[U]$ endowed with the product topology.  
On the product cone we can consider the projections on the two components $\pi^{\f{C}_i} : \pc \to \f{C}[U]$ sending $([x_0, r_0], [x_1, r_1])$ to $[x_i, r_i]$ and the projections on $\R_+$ and on the two copies of $U$ simply defined as $\sfr_i \eqdef  \sfr \circ \pi^{\f{C}_i}$ and $\sfx_i \eqdef  \sfx \circ \pi^{\f{C}_i}$ ($\sfx_i$ depends on the choice of points $\bar{x}_i \in U$, but this will be irrelevant). We introduce the set
\begin{equation*}
\mathcal{\f{N}}_+^2(\pc) \eqdef  \left \{ \aalpha \in \mcM(\pc) : \int_{\pc} (\sfr_0^2 + \sfr_1^2) \d \aalpha < + \infty  \right \}, 
\end{equation*}
and the maps
\begin{equation*}
\f{h}_i \colon \mathcal{\f{N}}_+^2(\pc) \to \mcM(U_i)\comma \qquad \f{h}_i(\aalpha) = (\sfx_i)_\sharp (\sfr_i^2 \aalpha)\comma \quad i=0,1 \fstop
\end{equation*}
Note that the map $\f{h}_i$ does not depend on the point $\bar{x}_i \in U_i$ in the definition of $\sfx_i$.
Finally we define, for every $(\mu_0, \mu_1) \in \mcM(U) \times \mcM(U)$, the set 
\begin{equation}\label{eq:hommarg}
\f{H}(\mu_0, \mu_1) \eqdef  \left \{ \aalpha \in  \mathcal{\f{N}}_+^2(\pc) : \f{h}_i(\aalpha) = \mu_i, \, i=0,1\right \}.
\end{equation}
If $\aalpha \in \f{H}(\mu_0, \mu_1)$, we say that $\mu_0$ and $\mu_1$ are the $2 $-homogeneous marginals of $\aalpha$. 
\quad \\
\noindent
We can now give the definition of Hellinger--Kantorovich distance induced by $\varrho$.

\begin{definition}[Hellinger--Kantorovich distance] We define the Hellinger--Kant\-orovich $\varrho$-distance $\HK_{\varrho}\colon \mcM(U) \to [0,+\infty)$ as
\begin{equation}\label{eq:defhk}
 \HK_\varrho(\mu_0,\mu_1)^2= \inf \left \{ \int_{\f{C}[U,U]} \varrho^2_{\pi,\f{C}} \d \aalpha : \aalpha \in \f{H}(\mu_0, \mu_1) \right \}.
\end{equation}
\end{definition}

The minimum in \eqref{eq:defhk} is always attained (see \cite[Theorem 7.6]{LMS18}), and we denote by~$\mathsf{OPT}_\varrho(\mu_0,\mu_1)$ the set of plans~$\aalpha$ attaining it.
The pair $(\mcM(U),\HK_\varrho)$ is a complete and separable metric space which is a length (resp.~geodesic) space if $(U,\varrho)$ is a length (resp.~geodesic) space (cf.~\cite[Theorems 7.15, 7.17, 8.5]{LMS18}). The topology of induced by $\HK_\varrho$ on $\mcM(U)$ coincides with the weak topology $\sigma(\mcM(\R^d), \rmC_b(U, \varrho))$.

We introduce the Gaussian Hellinger--Kantorovich distance: let us consider the function $g:[0,+\infty) \to [0,\frac{\pi}{2})$ defined as 
\[
g(z)\eqdef  \arccos(\rme^{-z^2/2})\comma \qquad z \in [0,+\infty)\fstop
\]
Since $g$ is a concave increasing function with $g(0)=0$, then $\gd\eqdef  g \circ \varrho$ is a distance on $U$ inducing the same topology of $\varrho$. 
Note that
\[
\gd_{\pi,\f{C}}([x_0,r_0],[x_1,r_1])= \braket{ r_0^2+r_1^2-2r_1 r_2 \rme^{-\varrho(x_0,x_1)^2/2} }^{1/2}\comma \quad [x_0,r_0], [x_1,r_1] \in \f{C}[U]\fstop
\] 
\begin{definition} We define the \emph{Gaussian Hellinger--Kantorovich $\varrho$-distance}
\[
\GHK_{\varrho}\colon \mcM(U) \to [0,+\infty)\comma \qquad \GHK_\varrho(\mu_0,\mu_1) \eqdef  \HK_{\gd}(\mu_0,\mu_1)\comma \quad \mu_0,\mu_1 \in \mcM(U) \fstop
\]
\end{definition}
Again, $(\mcM(U), \GHK_\varrho)$ is a complete and separable metric space and the topology of $(\mcM(U),\GHK_\varrho)$ coincides with the weak topology,~\cite[Theorem 7.25]{LMS18}. Moreover, if $(U, \varrho)$ is a length space, then $\HK_\varrho$ is the length distance induced by $\GHK_\varrho$,~\cite[Corollary 8.7]{LMS18}.
\medskip

The following result is the duality theorem for the $\GHK_\varrho$ distance. We will use the following notation: given two functions $\phi_i\colon U \to [-\infty, + \infty]$, $i=0,1$, the map $\phi_0 \oplus_o \phi_1 \colon U \times U \to [-\infty, + \infty]$ is defined as
\[
\phi_0 \oplus_o \phi_1 (x_0,x_1)\eqdef  \lim_{n \to + \infty} (-n \vee \phi_0(x_0) \wedge n) + (-n \vee \phi_1(x_1) \wedge n), \quad (x_0,x_1) \in U \times U\fstop
\]

\begin{theorem}\label{teo:ghk}
    Let $(U,\varrho)$ be a complete and separable metric space. Then, for every $\mu_0,\mu_1 \in \mcM(U)$ we have
    \begin{align*}
    \GHK_\varrho(\mu_0,\mu_1)^2 &= \inf \braket{ \sum_i \int_U (\sigma_i\log \sigma_i-\sigma_i+1) \de \mu_i + \int_{U \times U} \varrho^2 \de \ggamma : \ggamma \in \mcM(U\times U) }
    \\
    &=\sup \braket{ \sum_i\int_U (1-\rme^{-2\phi_i}) \de \mu_i : \phi_i \in \rmA_i(U), \, \phi_0 \oplus_o \phi_1 \le \varrho^2/2 } \comma    
    \end{align*} 
where $\sigma_i\eqdef  \frac{\de \gamma_i}{\de \mu_1}$, being $\gamma_i$, $i=0,1$, the marginals of $\ggamma$, and $\rmA_i(U)$ can be chosen between the spaces $\rmC_b(U)$, $\LSC_b(U)$, $\rmB_b(U)$ and 
\[
L^1_\rme (\mu_i)\eqdef  \set{ \phi\colon U \to [-\infty, + \infty] : \phi \text{ is Borel and } \int_U \rme^{-2\phi} \de \mu_i < + \infty } \fstop
\]
Furthermore:
\begin{enumerate}[$(i)$]
    \item\label{i:t:GHK:1} The set of plans $\ggamma$ which realize the infimum above is non-empty and it is denoted by $\mathsf{OPT}'_{\varrho}(\mu_0,\mu_1)$.
    
    \item\label{i:t:GHK:2} If $\ggamma \in \mathsf{OPT}'_{\varrho}(\mu_0,\mu_1) $ there exist two Borel sets $A_i \subset \supp(\gamma_i)$ on which $\gamma_i$ is concentrated and two Borel densities $\sigma_i \colon A_i \to (0,+\infty)$ of $\gamma_i$ w.r.t.~$\mu_i$ such that 
\begin{align} \label{eq:primacond}
\sigma_0 \sigma_1 &\ge \rme^{-\varrho^2} \quad \text{ in } A_0 \times A_1,\\ \label{eq:secondacond}
\sigma_0 \sigma_1 &= \rme^{-\varrho^2} \quad \text{ $\ggamma$-a.e.~in } A_0 \times A_1.
\end{align}

\item\label{i:t:GHK:3} If we define the Borel potentials $\phi_i \colon U \to [-\infty, +\infty]$ as
\[
\phi_i\eqdef  \begin{cases} -\frac{1}{2}\log \sigma_i, \quad &\text{ in } A_i, 
\\
- \infty, \quad &\text{ in } \supp(\mu_i) \setminus A_i,\\
+ \infty, \quad &\text{ otherwise,}
\end{cases}
\]
then the pair $(\phi_0, \phi_1)\in L^1_\rme(\mu_0) \times L^1_\rme(\mu_1)$ and realizes the supremum above and satisfies $\phi_0 \oplus_o \phi_1 \le \varrho^2/2$.

\item\label{i:t:GHK:4} Every $\ggamma \in \mathsf{OPT'}_{\varrho}(\mu_0,\mu_1)$ has the following properties:
\begin{enumerate}[$(a)$]
    \item $\mu_i \ll \gamma_i \ll \mu_i$, for $i=0,1$;
    \item $\ggamma$ is a solution for the Optimal Transport problem between its two marginals with cost $\varrho^2$;
    \item the plan 
    \[ \aalpha \eqdef  \tparen{[\id_U, \sigma_0^{-1/2}],[\id_U, \sigma_1^{-1/2}]}_\sharp \ggamma\]
    belongs to $\mathsf{OPT}_{\gd}(\mu_0,\mu_1)$.
\end{enumerate}

\item\label{i:t:GHK:5} If $(\tilde{\phi}_0, \tilde{\phi}_1) \in L^1_\rme(\mu_0) \times L^1_\rme(\mu_1)$ realizes the supremum above and satisfies $\tilde{\phi}_0 \oplus_o \tilde{\phi}_1 \le \varrho^2/2$, then $\phi_i= \tilde{\phi}_i$ $\mu_i$-a.e.~on $U$.
\end{enumerate}
\end{theorem}

\begin{proof}
The first equality is \cite[Theorem 7.20]{LMS18} while the equality between infimum and supremum is \cite[Theorem 6.3(a)]{LMS18}.
Assertion~\ref{i:t:GHK:1} is a consequence of \cite[Theorem 6.2(a)]{LMS18}.
Assertion~\ref{i:t:GHK:2} is \cite[Theorem 6.3(b)]{LMS18} while assertion~\ref{i:t:GHK:3} is \cite[Theorem 6.3(d)]{LMS18}.
Assertion~\ref{i:t:GHK:4} is \cite[Theorem 6.3(c)]{LMS18} together with comments (a) and (b) after \cite[Theorem 7.20]{LMS18}.
Finally, assertion~\ref{i:t:GHK:5} is \cite[Formula (4.21)]{LMS18}; see also \cite[Theorem 2.14]{LMS22}.
\end{proof}

\subsection{Estimates on optimal potentials in the Euclidean case}\label{sec:spec}
We now specialize Theorem \ref{teo:ghk} to the case of $(U,\varrho)=(\R^d, \sfd_e)$, where $\sfd_e$ is the distance induced by the Euclidean norm in $\R^d$.
The symbol $\mcM^{ac}(\R^d)$ denotes the set of measures in $\mcM(\R^d)$ which are absolutely continuous w.r.t.~the $d$-dimensional Lebesgue measure.  Analogous notation is adopted for $\mcP^{ac}(\R^d)$ and $\mcP_2^{ac}(\R^d)$.

\begin{proposition}\label{prop:thepot}
    Let $\nu, \mu \in \mcM^{ac}(\R^d)$ and let $R,\delta>0$ be such that $\supp(\nu)=\overline{\rmB(0,R)}$ and $\nu \ge \delta\Leb{d} \mres \rmB(0,R)$. Then there is a unique pair of convex and continuous functions $\varphi:\rmB(0,R) \to \R$ and $\psi:\R^d \to \R$ such that 
    \begin{enumerate}[$(i)$] 
        \item\label{i:p:THePot:1} $\psi$ and $\varphi$ are the Legendre conjugate of each other:
\begin{align*}
    \varphi(x)&\eqdef  \sup_{y \in \R^d} \tbraket{ \la x, y \ra - \psi(y) }\comma &&x \in \rmB(0,R) \comma
    \\
    \psi(y)&\eqdef  \sup_{x \in \rmB(0,R)} \tbraket{ \la x, y \ra- \varphi(x) }\comma &&y \in \R^d \semicolon
\end{align*}

\item\label{i:p:THePot:2} $\psi$ is $R$-Lipschitz;

\item\label{i:p:THePot:3} the pair $(\frac{1}{2}|\cdot|^2-\varphi, \frac{1}{2}|\cdot|^2-\psi)$ is optimal for the dual problem for the $\GHK_{\sfd_e}$ distance i.e.
\begin{align}
\label{eq:11}
\varphi(x) + \psi(y) &\le \la x, y \ra \quad \text{ for every } (x,y) \in \rmB(0,R) \times \R^d,
\\
\label{eq:12}
\GHK_{\sfd_e}(\nu,\mu)^2 &= \int_{\rmB(0,R)} \braket{1-\rme^{-|x|^2+2\varphi(x)} } \de \nu(x) + \int_{\R^d} \braket{ 1-\rme^{-|y|^2+2\psi(y)} } \de \mu(y) \fstop
\end{align}
\end{enumerate}
Such pair satisfies also
\begin{align}\label{eq:13}
\GHK_{\sfd_e}(\nu, \mu)^2 &=\int_{\R^d} \braket{ \braket{1-\rme^{-|y|^2+2\psi(y)}}^2+\rme^{-2|y|^2+4\psi(y)} \braket{\rme^{|y-\nabla \psi(y)|^2}-1} } \de \mu(y) \fstop
\end{align}

\end{proposition}
\begin{proof} By Theorem \ref{teo:ghk}\ref{i:t:GHK:3}, we can find Borel functions $\phi_i: \R^d \to [-\infty, +\infty]$ which are optimal for the dual problem for the $\GHK_{\sfd_e}$ distance and such that 
\begin{align}
\phi_0 \oplus_o \phi_1 &\le \tfrac{1}{2}\sfd_e^2 \quad \text{ in } \R^d \times \R^d,
\\
\label{eq:theass}
\phi_0+\phi_1 &= \tfrac{1}{2}\sfd_e^2 \quad \text{$\ggamma$-a.e.} \fstop
\end{align}
Let us consider $\ggamma \in \mathsf{OPT'}_{\sfd_e}(\nu, \mu)$ (whose existence is granted by Theorem \ref{teo:ghk}\ref{i:t:GHK:1}) and let us define
\begin{align*}
\tilde{\gamma}_0&\eqdef  \gamma_0 / \gamma(\R^d \times \R^d), \quad \tilde{\gamma}_1\eqdef  \gamma_1 / \gamma(\R^d \times \R^d)\comma
\\
\tilde{\varphi}(x)&\eqdef  \inf_{y \in \R^d} \braket{ \tfrac{1}{2}|x-y|^2- \phi_1(y) }, \quad x \in \R^d \comma
\\
\tilde{\psi}(y)&\eqdef  \inf_{x \in \R^d} \braket{ \tfrac{1}{2}|x-y|^2- \tilde{\varphi}(x) }, \quad y \in \R^d \comma
\\
\varphi(x)&\eqdef  \tfrac{1}{2} |x|^2 - \tilde{\varphi}(x), \quad \psi(y)\eqdef  \tfrac{1}{2}|y|^2 - \tilde{\psi}(y), \quad x,y \in \R^d \fstop
\end{align*}
Clearly $\tilde{\gamma}_i \in \mcP^{ac}(\R^d)$.
Since, by Theorem \ref{teo:ghk}(4a), $\nu \ll \gamma_0 \ll \nu$, we have that $\supp(\tilde{\gamma}_0)=\overline{\rmB(0,R)}$, so that $\tilde{\gamma}_0 \in \mcP_2^r(\R^d)$. Moreover, since 
\[ +\infty >\GHK_{\sfd_e}(\nu, \mu)^2 \ge \int_{\R^d \times \R^d} \sfd_e^2 \de \ggamma \ge \W_{2, \sfd_e}^2(\tilde{\gamma}_0, \tilde{\gamma}_1),\]
we also get that $\tilde{\gamma}_1 \in \mcP_2(\R^d)$.

Since $\tilde{\varphi}+\tilde{\psi}\le \sfd_e^2/2$ in $\R^d \times \R^d$ and $\tilde{\varphi} \ge \phi_0$, $\tilde{\psi} \ge \phi_1$ in $\R^d$, we have that $\tilde{\varphi}=\phi_0$ $\nu$-a.e.~and $\tilde{\psi}=\phi_1$ $\mu$-a.e..
This, together with \eqref{eq:theass} and the optimality of~$\ggamma$ (cf.~Theorem \ref{teo:ghk}(4b)), gives that 
\[
\tfrac{1}{2} \int_{\R^d} |x|^2 \de \tilde{\gamma}_0(x) + \tfrac{1}{2} \int_{\R^d} |x|^2 \de \tilde{\gamma}_1(x) - \W_{2, \sfd_e}^2(\tilde{\gamma}_0, \tilde{\gamma}_1) = \int_{\rmB(0,R)} \varphi \de \tilde{\gamma}_0 + \int_{\R^d} \psi \de \tilde{\gamma}_1 \fstop 
\]

It is also clear that $\varphi$ and $\psi$ are the Legendre transform of each other. The proof of \cite[Theorem~3.2]{FSS22} shows that, after restricting $\varphi$ to $\rmB(0,R)$, both~$\varphi$ and~$\psi$ are finite, convex, continuous and satisfy~\ref{i:p:THePot:1}. 
The representation in~\eqref{eq:12} follows since $\frac{1}{2}|\cdot|^2 - \varphi$ coincides with $\phi_0$ $\nu$-a.e.~and $\frac{1}{2}|\cdot|^2 - \psi$ coincides with $\phi_1$ $\mu$-a.e.

Assertion~\ref{i:p:THePot:2}, as well as \eqref{eq:11}, follows immediately from~\ref{i:p:THePot:1}.

We now prove~\eqref{eq:13}. It is a classical result in Optimal Transport theory that $\ggamma$ must be concentrated on the graph of $\nabla \psi$ in the sense that 
\[
\ggamma = (\nabla \psi, \id_{\R^d})_\sharp \gamma_1\fstop
\]
By Theorem \ref{teo:ghk}(4c), we have that 
\[
\aalpha \eqdef  \paren{[\id_{\R^d}, \sigma_0^{-1/2}], [\id_{\R^d}, \sigma_1^{-1/2}]}_\sharp \ggamma\comma
\]
belongs to $\mathsf{OPT}_{\mathsf{g}^{\mathsf{d}_e}}(\nu, \mu)$ so that 
\begin{align*}
\GHK_{\sfd_e}(\nu,\mu)^2 &= \int_{\f{C}[\R^d, \R^d]} \paren{\mathsf{g}^{\mathsf{d}_e}_{\pi, \f{C}}}^2 \de \aalpha
\\
&= \int_{\R^d \times \R^d} \braket{ \frac{1}{\sigma_0(x_0)}+\frac{1}{\sigma_1(x_1)}-\frac{2\rme^{-|x_0-x_1|^2/2}}{\sigma_0(x_0)^{1/2}\, \sigma_1(x_1)^{1/2}} } \de \ggamma(x_0,x_1)
\\
&= \int_{\R^d} \braket{ \frac{1}{\sigma_0(\nabla \psi(y))}+\frac{1}{\sigma_1(y)}-\frac{2\rme^{-|y-\nabla \psi(y)|^2/2}}{\sigma_0(\nabla \psi(y))^{1/2}\, \sigma_1(y)^{1/2}} } \de \ggamma_1(y)
\\
&= \int_{\R^d} \left (1+ \sigma_1(y)^2 \rme^{|y-\nabla \psi(y)|^2}-2\sigma_1(y) \right )\de \mu(y)
\\
&= \int_{\R^d} \left ( 1+ \rme^{-4\phi_1(y)} \rme^{|y-\nabla \psi(y)|^2}-2\rme^{-2\phi_1(y)} \right ) \de \mu(y)
\\
&= \int_{\R^d} \left ( (1-\rme^{-2\phi_1(y)})^2+\rme^{-4\phi_1(y)}(\rme^{|y-\nabla \psi(y)|^2}-1) \right )\de \mu(y)
\\
&= \int_{\R^d} \left ( (1-\rme^{-|y|^2+2\psi(y)})^2+\rme^{-2|y|^2+4\psi(y)}(\rme^{|y-\nabla \psi(y)|^2}-1) \right )\de \mu(y) \fstop
\end{align*}
We now prove uniqueness. If another pair $(\varphi_0, \psi_0)$ is as in the statement and satisfies~\ref{i:p:THePot:1}-\ref{i:p:THePot:3} above, we have by Theorem \ref{teo:ghk}\ref{i:t:GHK:5} that $\frac{1}{2}|\cdot|^2-\varphi_0$ must coincide $\nu$-a.e.~with $\phi_0$ which coincides $\nu$-a.e.~with $\frac{1}{2}|\cdot|^2-\varphi$.
Since $\nu$ is equivalent to the Lebesgue measure, we deduce that $\varphi_0=\varphi$ $\Leb{d}$-a.e.~in $\rmB(0,R)$. Being both functions continuous, we deduce that they coincide everywhere, so that by~\ref{i:p:THePot:1} also $\psi_0$ and $\psi$ coincide in $\R^d$.
\end{proof}

Our aim is to fix a suitably regular measure $\nu \in \mcM(\R^d)$ and consider a sufficiently regular class of measures $\mu$ for which the resulting potentials $\psi$ are \emph{uniformly} regular. To do so we are going to consider a family of regularizing operators $\mathsf{T}_\eps$. Let us start with the relevant definitions.

We consider two functions $\kappa \in \rmC^{\infty}_c(\R^d, \R)$ and $f \in \rmC^\infty(\R^d, \R^d)$ satisfying the following properties:
\begin{enumerate}
    \item $\kappa(x)\ge 0$ for every $x \in \R^d$, $\supp(\kappa)=\overline{\rmB(0,1)}$, $\kappa(x)=\kappa(-x)$ for every $x \in \R^d$ and $\int_{\R^d} \kappa(x) \de x =1$;
    \item $f(x)=x$ for every $x \in \rmB(0,1)$, $|f(x)|\le 2$ for every $x \in \R^d$ and $\|Df(x)\| \le 1$ for every $x \in \R^d$, where $\|\cdot\|$ denotes the operator norm of a matrix in $\R^{d \times d}$.
\end{enumerate}

We then consider $\kappa_\eps(x)\eqdef \frac{1}{\eps^d}\kappa(x/\eps)$ and $f_\eps(x)\eqdef \frac{1}{\eps} f(\eps x)$ for every $x \in \R^d$ and $\eps\in (0,1)$.
It is thus clear that $\kappa_\eps \in \rmC^{\infty}_c(\R^d, \R)$, $f_\eps \in \rmC^\infty(\R^d, \R^d)$ and they satisfy the following properties:
\begin{enumerate}
    \item $\kappa_\eps(x)\ge 0$ for every $x \in \R^d$, $\supp(\kappa_\eps)=\overline{\rmB(0,\eps)}$, $\kappa_\eps(x)=\kappa_\eps(-x)$ for every $x \in \R^d$ and $\int_{\R^d} \kappa_\eps(x) \de x =1$;
    \item $f_\eps(x)=x$ for every $x \in \rmB(0,1/\eps)$, $|f_\eps(x)|\le 2/\eps$ for every $x \in \R^d$ and $\|Df_\eps(x)\| \le 1$ for every $x \in \R^d$.
\end{enumerate}

We define the operator $\mathsf{T}_\eps: \mcM(\R^d) \to \mcM^{ac}(\R^d)$ as
\begin{equation}\label{eq:theop}
\mathsf{T}_\eps (\mu) = \left ( (f_\eps)_\sharp \mu + \eps \delta_0 \right ) \ast \kappa_\eps, \quad \mu \in \mcM(\R^d).  
\end{equation} 

\begin{proposition}\label{prop:thereg} Let $\mathsf{T}_\eps$ be as in \eqref{eq:theop}. Then $\mathsf{T}_\eps$ is well defined and $\mathsf{T}_\eps(\mu) \to \mu$ as $\eps \downarrow 0$ for every $\mu \in \mcM(\R^d)$. Let $\nu \in \mcM^{ac}(\R^d)$ and let $R,B, \delta>0$ be such that $\supp(\nu)=\overline{\rmB(0,R)}$, $\nu \R^d \le B$ and $\nu \ge \delta\Leb{d} \mres \rmB(0,R)$. Then, for every $\eps>0$, there exists a constant $K=K(\eps, B, R)>0$ such that, for every $\mu \in \mcM(\R^d)$, if $(\varphi, \psi)$ is as in Proposition \ref{prop:thepot} for the pair $(\nu, \mathsf{T}_\eps(\mu))$, then $\psi(0)\le K$ and the map $y \mapsto \frac{1-\rme^{-|y|^2+2\psi(y)}}{2}$ is $K$-Lipschitz and $K$-bounded.
\end{proposition}
\begin{proof} The proof is divided in several claims.

    \paragraph{Claim (a): the operator $\mathsf{T}_\eps$ is well defined and tends to the identity as $\eps \downarrow 0$}

    \paragraph{Proof of claim (a)} Since $f_\eps$ is a Borel function and the convolution preserves mass and non-negativity of the measure, we have that $\mathsf{T}_\eps(\mu) \in \mcM(\R^d)$ for every $\mu \in \mcM(\R^d)$. It is also clear that $\mathsf{T}_\eps(\mu)$ has a density w.r.t.~the Lebesgue measure in $\R^d$ so that $\mathsf{T}_\eps(\mu) \in \mcM^{ac}(\R^d)$ for every $\mu \in \mcM(\R^d)$. Let $\varphi \in \rmC_b(\R^d)$; then
\begin{align*}
    \int_{\R^d} \varphi \de \mathsf{T}_n(\mu) = \eps(\varphi \ast \kappa_\eps)(0)+ \int_{\R^d} ((\varphi \ast \kappa_\eps) \circ f_\eps) \de \mu.
\end{align*}
Since $\varphi \ast \kappa_\eps \to \varphi$ locally uniformly in $\R^d$ as $\eps \downarrow 0$, we have in particular that $(\varphi \ast \kappa_\eps)(0)$ is bounded, so that $\eps(\varphi \ast \kappa_\eps)(0) \to 0$ as $\eps \downarrow 0$. For the convergence of the integral we use the Dominated Convergence Theorem: note that $\|(\varphi \ast \kappa_\eps) \circ f_\eps)\|_\infty \le  \|\varphi \ast \kappa_\eps\|_\infty \le \|\varphi\|_\infty$ and, for a fixed $x \in \R^d$, if $\eps < 1/|x|$, then $f_\eps(x)=x$ so that $(\varphi \ast \kappa_\eps) \circ f_\eps)(x) \to \varphi(x)$ for every $x \in \R^d$. We can thus conclude that 
\[ \int_{\R^d} \varphi \de \mathsf{T}_n(\mu) \to \int_{\R^d} \varphi \de \mu \quad \text{ for every } \varphi \in \rmC_b(\R^d) \text{ as } \eps \downarrow 0.\]

    \paragraphn{Claim (b): there exist a constant $a=a(\eps, B, R)>0$ such that
    \[
    \sup_{\mu \in \mcM(\R^d)} \frac{B+\mathsf{T}_\eps(\mu) \R^d}{\displaystyle \int_{\R^d} \rme^{-|x|^2 - 2R|x|} \de\mathsf{T}_\eps(\mu)(x)}  \le a.
    \]}

    \paragraph{Proof of claim (b)} Denote by $f^R: \R^d \to \R$ and $C(\eps,R)>0$ the function and the constant
    \[
    f^R(x)\eqdef  \rme^{-|x|^2-2R|x|} \quad x \in \R^d, \quad C(\eps, R)\eqdef  \min_{\rmB(0,2/\eps)} f^R \ast \kappa_\eps >0\fstop
    \]
We have
\begin{align*}
 \frac{B+\mathsf{T}_\eps(\mu) \R^d}{\displaystyle \int_{\R^d} \rme^{-|x|^2 - 2R|x|} \de\mathsf{T}_\eps(\mu)(x)} &=  \frac{B+\left ( (f_\eps)_\sharp \mu + \eps \delta_0\right )(\R^d)}{\displaystyle \int_{\R^d} (f^R \ast \kappa_{\eps}) \de\left ( (f_\eps)_\sharp \mu + \eps \delta_0\right )(x)}
 \\
 &=  \frac{B+\eps + \mu \R^d}{ \eps(f^R \ast \kappa_{\eps})(0)+ \displaystyle\int_{\R^d} ((f^R \ast \kappa_{\eps})\circ f_\eps) \de \mu(x)}
 \\
 & \le \frac{B+\eps + \mu \R^d}{ \eps(f^R \ast \kappa_{\eps})(0)+ C(\eps,R) \mu \R^d}
 \\
 & \le \frac{B+\eps}{\eps(f^R \ast \kappa_{\eps})(0)} + \frac{1}{C(\eps,R)} 
 \\
 &\defeq a \fstop
 \end{align*}

    \paragraph{Claim (c): for every $\mu \in \mcM(\R^d)$, if $(\varphi, \psi)$ is as in Proposition \ref{prop:thepot} for the pair $(\nu, \mathsf{T}_\eps(\mu))$, then $\rme^{2\psi(0)} \le a$, where $a$ is the constant from claim (b)}

    \paragraph{Proof of claim (c)} By Proposition \ref{prop:thepot}, we know that $\psi$ is $R$-Lipschtiz and the pair $(\varphi,\psi)$ satisfies
    \begin{align*}
    \int_{\R^d} (1-\rme^{-|y|^2 +2\psi(y)}) \de \mathsf{T}_\eps(\mu)(y) + \nu \R^d &\ge \int_{\R^d} (1-\rme^{-|y|^2 +2\psi(y)}) \de \mathsf{T}_\eps(\mu)(y)
    \\
    &\qquad +  \int_{\rmB(0,R)} (1-\rme^{-|x|^2+2\varphi(x)}) \de \nu(x) \\
    &= \GHK_{\sfd_e}(\nu, \mathsf{T}_\eps(\mu))^2 \\
    &\ge 0
    \end{align*}
so that
\begin{align*}
\int_{\R^d} \rme^{-|y|^2-2R|y| +2\psi(0)} \de \mathsf{T}_\eps(\mu)(y) &\leq \int_{\R^d} \rme^{-|y|^2 +2\psi(y)} \de \mathsf{T}_\eps(\mu)(x) \le \nu \R^d+\mathsf{T}_\eps(\mu) \R^d
\\
&\leq B + \mathsf{T}_\eps(\mu) \R^d
\end{align*}
thus giving that $\rme^{2\psi(0)} \le a$.

    \paragraph{Claim (d): there exists a constant $K=K(\eps, B, R)>0$ such that, for every $\mu \in \mcM(\R^d)$, if $(\varphi, \psi)$ is as in Proposition \ref{prop:thepot} for the pair $(\nu, \mathsf{T}_\eps(\mu))$, then $\psi(0) \le K$ and the map $y \mapsto \frac{1-\rme^{-|y|^2+2\psi(y)}}{2}$ is $K$-Lipschitz and $K$-bounded.}
    \quad \smallskip\\
    \textit{Proof of claim (d):} it is clear that $\psi(0) \le \ln(a)/2$; observe that 
\[ -|y|^2 + 2\psi(y) \le -|y|^2 + 2R|y|+ 2\psi(0) \le R^2 + 2\psi(0) \text{ for every } y \in \R^d,\]
so that
\[ 0 \le \rme^{-|y|^2 + 2\psi(y)} \le \rme^{2\psi(0)}\rme^{R^2} \le a\rme^{R^2}\text{ for every } y \in \R^d, \]
where we have used claim (c). Moreover, we can write
\[ \rme^{-|y|^2+2\psi(y)}= \rme^{-|y|^2+2R|y|+2\psi(0)} \rme^{2\psi(y)-2R|x|-2\psi(0)} =: \rme^{-|y|^2+2R|y| + 2\psi(0)} \rme^{-f(y)}.\]
Notice that $y\mapsto \rme^{-|y|^2+2R|y| + 2\psi(0)}$ is Lipschtz and bounded and $f(y) \ge 0$ so that $\rme^{-f} \le 1$. Moreover, 
\[ \Li[\R^d]{\rme^{-f}} \le \Li[\R^-]{\exp} \Li[\R^d]{f}\le \Li[\R^d]{f}\comma \]
so that also $\rme^{-f}$ is Lipschitz and bounded. We deduce that $\rme^{-|\cdot|^2+2\psi}$ is Lipschitz, so that its Lipschtiz constant can be estimated by the essential supremum of the norm of its derivative:
\[ |\nabla \rme^{-|y|^2+2\psi(y)}| = \rme^{-|y|^2+2\psi(y)}|2y+2\nabla \psi(y)| \le 2\rme^{2\psi(0)}\rme^{-|y|^2+2R|y|}(|y|+R) \le 2a C'(R),\]
where 
\[ C'(R)= \left (\sqrt{R^2+1/2}+R \right ) \rme^{-R^2 +2R\sqrt{R^2+1/2}-1/2}. \]
This concludes the proof.
\end{proof}

We can summarize the results of this subsection in the following theorem.

\begin{theorem}\label{teo:final} Let $\nu \in \mcM^{ac}(\R^d)$ and let $R,B, \delta>0$ be such that $\supp(\nu)=\overline{\rmB(0,R)}$, $\nu \R^d \le B$ and $\nu \ge \delta\Leb{d} \mres \rmB(0,R)$; let $\eps \in (0,1)$ be fixed and let $\mathsf{T}_\eps$ be as in \eqref{eq:theop}. Then there exists a constant $K(\eps,B,R)$ that depends only on $\eps$, $R$ and $B$ such that, for every measure $\mu \in \mcM(\R^d)$, there exist a unique pair of convex and continuous functions $\varphi:\rmB(0,R) \to \R$ and $\psi:\R^d \to \R$ such that:
\begin{enumerate}[$(1)$]
\item $\psi$ is $R$-Lipschitz and $\psi(0) \le K$;
\item the map $y \mapsto \frac{1-\rme^{-|y|^2+2\psi(y)}}{2}$ is $K$-Lipschitz and $K$-bounded;
\item $\varphi$ and $\psi$ are the Legendre conjugate of each other:
\begin{align*}
    \varphi(x)&\eqdef  \sup_{y \in \R^d} \braket{ \la x, y \ra - \psi(y) } \comma &&x \in \rmB(0,R)\comma
    \\
    \psi(y)&\eqdef  \sup_{x \in \rmB(0,R)} \braket{ \la x, y \ra- \varphi(x) }\comma &&y \in \R^d \semicolon
\end{align*}
\item the pair $(\frac{1}{2}|\cdot|^2-\varphi, \frac{1}{2}|\cdot|^2-\psi)$ is optimal for the dual problem for the $\GHK$ distance between $\nu$ and $\mathsf{T}_\eps(\mu)$:
\begin{align*}
    \varphi(x) + \psi(y) &\le \la x, y \ra \quad \text{ for every } (x,y) \in \rmB(0,R) \times \R^d, \\
    \tfrac{1}{2}\GHK_{\sfd_e}(\nu,\mu)^2 &= \int_{\rmB(0,R)} \frac{1-\rme^{-|x|^2+2\varphi(x)}}{2} \de \nu(x) + \int_{\R^d} \frac{1-\rme^{-|y|^2+2\psi(y)}}{2} \de \mathsf{T}_\eps(\mu) \fstop
\end{align*}
\end{enumerate}
The same pair also satisfies
\begin{align*}
\GHK_{\sfd_e}&(\nu, \mathsf{T}_\eps(\mu))^2 =
\\
=&\int_{\R^d} \braket{\braket{1-\rme^{-|y|^2+2\psi(y)}}^2+\rme^{-2|y|^2+4\psi(y)}\braket{\rme^{|y-\nabla \psi(y)|^2}-1} }\de \mathsf{T}_\eps(\mu)(y) \fstop
\end{align*}
\end{theorem}

\section{Cylinder functions}

Let $(M,g)$ be a  smooth, connected, complete  Riemmanian manifold with Riemannian distance $\mssd_g$; to every $f \in \Lipb(M, \mssd_g)$ we can associate the functional $f^\trid$ on $\mcM(M)$
\begin{equation}\label{eq:monocyl}
  f^\trid:\mu \to \int_{M} f \,\d \mu.
\end{equation}
By \cite[Proposition 7.18]{LMS18}, we have that 
\begin{equation}\label{eq:lipest}
\begin{aligned}
\frac{\left | f^\trid(\mu)- f^\trid(\mu') \right |}{\HK_{\mssd_g}(\mu, \mu')} &\leq (\Li{f} \vee \|f\|_\infty) \cdot 
\\
&\quad \cdot \sqrt{2+\tfrac{\pi^2}{2}}\, \tparen{\mu M+\mu'M}^{1/2}\comma
\end{aligned}\qquad
\mu \neq \mu' \in \mcM(M)\comma
\end{equation}
so that, if~$f \in \Lipb(M,\mssd_g)$, then~$ f^\trid$ is $\HK_{\mssd_g}$-Lipschitz on every $\HK_{\mssd_g}$-bounded set in~$\mcM(M)$.
If $\mbff=(f_1,\dotsc,f_N)\in
\big(\Lipb(M,\mssd_g)\big)^N$, we denote by 
$\mbff^\trid \eqdef  (f_1^\trid,\dotsc, f_N^\trid)$ the corresponding map
from $\mcM(M)$ to $\R^N$.

We denote by $\rmC_b^m(M)$ the space of $m$-times continuously differentiable functions on~$M$ with bounded derivatives up to order~$m$, and by $\rmC_c^m(M)$ the subset of $\rmC_b^m(M)$ of functions with compact support.

\begin{definition}[Cylinder functions]\label{def:cyl}
We consider several algebras of cylinder functions on~$\mcM(M)$, viz.
\begin{equation}\label{eq:CylinderF}
\FC{m_1,m_2}{\sharp_1,\sharp_2}{m_3}{\sharp_3}(\mcM(M))\eqdef \set{u\colon \mcM(M)\to\R : \begin{gathered} u=(\nchi\circ \car^\trid) \cdot (F\circ \mbff^\trid)\comma
\\
N\in\N\comma \nchi\in \rmC^{m_1}_{\sharp_1}(\R^+_0)\comma
\\
F\in \rmC^{m_2}_{\sharp_2}(\R^N)\comma \mbff\in \tparen{\rmC^{m_3}_{\sharp_3}(M)}^N
\end{gathered}}\comma
\end{equation}
where~$m_1,m_2,m_3\in \N\cup\set{\infty}$ and~$\sharp$ stands for either~$b$ for \emph{bounded} or~$c$ for \emph{compact support}. Further define
\[
\FC{m_1,m_2}{u,\sharp_2}{m_3}{\sharp_3}(\mcM(M))\eqdef \FC{m_1,m_2}{c,\sharp_2}{m_3}{\sharp_3}(\mcM(M)) \oplus \R
\]
to include constant functions.
\end{definition}

\begin{remark}
Note that:
\begin{enumerate}[$(a)$, leftmargin=*]
\item the representation of a cylinder function~$u$ by~$\nchi$,~$F$ and~$\mbff$ is \emph{not} unique;
\item we have the following inclusions
\begin{center}
\begin{tikzcd}[row sep="0pt", column sep= {100pt,between origins}]
&& \FC{1,1}{b,b}{1}{b}(\mcM(M)) 
\\
\FC{\infty,\infty}{c,c}{\infty}{c}(\mcM(M) \ar[r, phantom,"\subsetneq"] & \FC{1,1}{u,b}{1}{b}(\mcM(M)) \ar[ru, phantom,sloped,"\subsetneq"] \ar[rd, phantom,sloped,"\subsetneq"]
\\
&& \Lip_b(\mcM(M), \HK_{\mssd_g})
\end{tikzcd}
\end{center}

\item cylinder functions in~$\FC{m_1,m_2}{b,\sharp_1}{m_3}{\sharp_3}(\mcM(M))$ may be expressed only in terms of~$F$ and~$\mbff$ by replacing~$F$ with~$\nchi\otimes F$ and $\mbff$ with $(\car, f_1, \dotsc, f_N)$;

\item $\FC{\infty,\infty}{c,c}{\infty}{c}(\mcM(M)$ is point-separating, hence so are all algebras in~\eqref{eq:CylinderF};

\item in particular,
\begin{equation}\label{eq:LipSubAlg}
\begin{gathered}
\FC{1,1}{u,b}{1}{b}(\mcM(M)) \text{ is a unital point-separating} 
\\
\text{subalgebra of } \Lip_b(\mcM(M), \HK_{\mssd_g})\fstop
\end{gathered}
\end{equation}
\end{enumerate}
\end{remark}

To shorten the notation we will usually remove the $\mcM(M)$ symbol in front of sets of cylinder functions and write $\FC{m_1,m_2}{\sharp_1,\sharp_2}{m_3}{\sharp_3}$ in place of $\FC{m_1,m_2}{\sharp_1,\sharp_2}{m_3}{\sharp_3}(\mcM(M))$, unless we are considering spaces of cylinder functions on specific manifolds such as~$\R^d$.

The following result is an immediate consequence of the definition of cylinder functions and of the construction in Section \ref{sss:DifferentiationIntro}.

\begin{lemma} For every $u \in \FC{1,1}{b,b}{1}{b}$ and every $\mu \in \mcM(M)$, the gradient $(\boldnabla u)_\mu \in T_\mu\mcM(M)$ as in \eqref{eq:OrthoGrad}  is well defined and, whenever $u= F \circ \mbff^\trid$ for some $N \in \N$, $F \in \rmC_b^1(\R^N)$ and $\mbff= (f_1, \dotsc, f_N) \in (\rmC_b^1(M))^N$, we have
\begin{equation}\label{eq:Fdiff}
(\boldnabla u)_\mu (x) = \sum_{n=1}^N \partial_n F ( \mbff^\trid\mu ) (\nabla f_n(x), f_n(x)) \quad \text{ for $\mu$-a.e.~$x \in \R^d$}.
\end{equation}
\end{lemma}

\begin{lemma}\label{le:trunc1} Let $\varsigma \in \rmC_c^\infty([0,+\infty))$ be a non-increasing function such that $\varsigma(r)=1$ if $0\le r \le 1$, $\varsigma(r)=0$ if $r \ge 2$ and $\|\varsigma'\|_\infty \le 2$. Define the continuous functions $u_k: \mcM(M) \to [0,1]$ as
\[
u_k(\mu) \eqdef \varsigma (\mu M/k), \quad \mu \in \mcM(M)\fstop
\]
Then $u_k \in \FC{\infty,\infty}{c,c}{\infty}{c}$, $\|u_k\|_\infty =1$, $u_k(\mu)=1$ if $ \mu M \le k$, $u_k(\mu)=0$ if $\mu M \ge 2k$ and 
\[  \lip_{\HK_{\mssd_g}} u_k(\mu) \le \frac{| \varsigma'(\mu M/k)|}{k} \sqrt{4+\pi^2}\, (\mu M)^{1/2} \le \frac{2 \sqrt{8+2\pi^2}}{\sqrt{k}} \nchi_{\{\mu M \le 2k\}} (\mu).\]
\end{lemma}
\begin{proof} This follows by the definition of $u_k$ and \eqref{eq:lipest} together with the fact that $ \lip_{\sfd} (\psi \circ f) = (|\psi'| \circ f) \lip_\sfd f$ whenever $\psi \in \rmC^1(\R)$.
\end{proof}

\begin{remark}[$2$-integrable variations] 
Let $(\mu_t)_{t \in [0,1]}\in \ac{[0,1];(\mcM(M), \HK_{\sfd_g})}$ and let $(\vv_t, w_t)$ be a vector field in $L^2((0,1)\times M, \mu_I; TM \times \R)$, where $\mu_I\eqdef  \int_0^1 (\mu_t \otimes \delta_t) \d t$, such that
\[
\partial_t \mu_t + \nabla \cdot (\vv_t \mu_t) = w_t  \mu_t  \text{ in }  \mathscr{D}'((0,1)\times M)\fstop
\]
If $u\in \FC{1,1}{b,b}{1}{b}$ then 
\begin{equation}\label{eq:derivative}
    \frac{\d}{\d t} u(\mu_t) = \scalar{(\boldnabla u)_{\mu_t}}{(\vv_t, w_t)}_{T_{\mu_t}\mcM(M)}\quad \text{ for a.e.~} t \in (0,1) \fstop
\end{equation}
This is indeed a simple consequence of the chain rule.
If $\mu \in \mcM(M)$, $(\TT_1, T_2) \in L^2(M, \mu; TM \times \R)$, and we define the curve $\mu_t\eqdef \exp( t\TT_1)_\sharp \tparen{(1+ t T_2)^2 \mu}$, $t \in [0,1]$, then 
\begin{equation}\label{eq:derivative1} 
\lim_{t \downarrow 0} \frac{u(\mu_t)-u(\mu)}{t} = \scalar{(\boldnabla u)_\mu}{(\TT_1, 2T_2 )}_{T_\mu} \fstop
\end{equation}
This follows again by the chain rule and the fact that for every $f \in \rmC_b^1(M)$ we have
\[
\int_M f\, \d \mu_t = \int_M f \d \mu + t \int_M \tparen{\scalar{\nabla f}{\TT_1}_g + 2f T_2} \d \mu + o(t)\fstop
\]
\end{remark}

The following Lemma is a simple extension of \cite[Lemma 4.6]{FSS22} and will be useful in the proof of Proposition \ref{prop:equalityriem}.
\begin{lemma}\label{lem:limit}
   Let $(U, \varrho)$ be a complete and separable metric space
  and let $G:\mcM(U) \times U \to [0, + \infty)$ be a bounded and
  continuous function w.r.t.~the product topology induced by the weak topology and~$\varrho$.
  If
  $(\mu_n)_{n}$ is a sequence in $\mcM(U)$ converging weakly
  to $\mu$ as $n \to + \infty$, then
\[
\lim_{n\to\infty} \int_U G(\mu_n, y)\, \d \mu_n(y) = \int_U G(\mu, y)\, \d \mu(y) \fstop
\]
\end{lemma}
    \begin{proof} Let $\eps>0$ be fixed; since $\mu_n$ converges weakly to $\mu$ the set $\{\mu_n\}_n \cup \{\mu\}$ is uniformly tight and uniformly bounded in mass by \cite[Theorem~8.6.2(ii)]{Bogachev07}.
    We can thus find a compact set $K_\eps \subset U$ and $M_\eps \in \N$ such that, for every $n > M_\eps$, it holds
\[ \sup_n \mu_n(U \setminus K_\eps) < \eps, \quad \mu(U \setminus K_\eps) < \eps, \quad \left | \int_U G(\mu, y) \d \mu_n(y) - \int_U G(\mu, y) \d \mu(y) \right |  < \eps,\]
where the last inequality comes from the weak convergence of $\mu_n$ to $\mu$ and the regularity of $y \mapsto G(\mu, y)$. Let $C_\eps \subset \mcM(U) \times U$ be the compact set defined as
\[
C_\eps\eqdef  \left ( \{\mu_n\}_n \cup \{\mu\} \right ) \times K_\eps \fstop
\]

Recall that $\HK_\varrho$ metrizes the weak topology. Since~$G$ is continuous, it is uniformly continuous on $C_\eps$ so that we can find $\eta_\eps>0$ such that  
\[ \left |G(\mu', y') - G(\mu'', y'') \right | < \eps \qquad \begin{aligned} & \text{for every } (\mu', y'), (\mu'', y'') \in C_\eps\\ & \text{with } \HK_\varrho(\mu', \mu'') + |y'-y''| < \eta_{\eps} \fstop\end{aligned}
\]
Since $\HK_\varrho(\mu_n, \mu)\to0$, we can find $L_\eps>0$ such that
\[
\HK_\varrho(\mu_n, \mu) < \eta_\eps \quad \text{ for every } n > L_\eps \fstop
\]
Let us define $N_\eps\eqdef  M_\eps \vee L_\eps$; then, if $n> N_\eps$,
\begin{align*}
    \Bigg | \int_U G(\mu_n, y) \d \mu_n(y) - \int_U G(\mu, y) &\d \mu(y) \Bigg| 
    \\
    \leq& \abs{ \int_{U} G(\mu_n, y) \d \mu_n(y) - \int_{U} G(\mu, y) \d \mu_n(y) }
    \\
    &+ \abs{ \int_{U} G(\mu, y) \d \mu_n(y)- \int_{U} G(\mu, y)\d \mu(y)  }
    \\
    \leq& \abs{ \int_{K_\eps} G(\mu_n, y) \d \mu_n(y) - \int_{K_\eps} G(\mu, y) \d \mu_n(y) }
    \\
    & + 2\eps \|G\|_\infty + \eps
    \\
    \leq&\ \eps \braket{ \sup_n \mu_nU + 1 + 2\|G\|_\infty}\comma
\end{align*}
where we have used that for every $y \in K_\eps$ we have $(\mu_n,y), (\mu, y) \in C_\eps$ with $ \HK_\varrho(\mu_n, \mu) + |y-y| =  \HK_\varrho(\mu_n, \mu) < \eta_\eps$. This concludes the proof.
\end{proof}

\subsection{Asymptotic Lipschitz constant of cylinder functions}\label{s:ALC-Cylinder}
The following proposition corresponds to \cite[Proposition 4.9]{FSS22} and the proof is quite similar but we still report it here because of a few differences.

\begin{proposition} \label{prop:equalityriem}
Fix~$u\in \FC{1,1}{b,b}{1}{b}$. Then,
\[ \|(\boldnabla u)_\mu\|_{T_\mu^{1,4}} = \lip_{\HK_{\sfd_g}} u (\mu)\comma \qquad \mu \in \mcM(M)\comma
\]
where $\|(\boldnabla u)_\mu\|_{ T_{\mu}^{1,4}}$ is as in \eqref{eq:NormTot4} and $\lip_{\HK_{\mssd_g}}$ is as in \eqref{eq:1}.
\end{proposition}
\begin{proof} We prove separately the two inequalities.
We start from $\lip_{\HK_{\mssd_g}} u (\mu) \le \|(\boldnabla u)_\mu\|_{T_\mu^{1,4}}$ and we first prove it on~$(M,\mssd_g)=(\R^d,\mssd_e)$.  Let $\mu \in \mcM(\R^d)$ and let $(\mu'_n, \mu''_n) \in \mcM(\R^d)\times \mcM(\R^d)$ with $\mu'_n \ne \mu''_n$ be such that $(\mu'_n, \mu''_n) \to (\mu, \mu)$ in $\HK_{\mssd_e}$ and
\[ \lim_n \frac{ \left |u(\mu'_n) - u(\mu''_n) \right |}{\HK_{\mssd_e}(\mu'_n, \mu''_n)} = \lip_{\HK_{\mssd_e}} u (\mu).\]
Let $(\mu^n_t)_{t \in [0,1]}$ be constant speed geodesics (cf.~\cite[Proposition 8.3]{LMS18}) connecting $\mu'_n$ to $\mu''_n$. By \cite[Theorems 8.16, 8.17]{LMS18} we can find Borel vector fields $(\vv^n_t, w^n_t) \in L^2((0,1)\times \R^d, \mu^n_I; \R^d \times \R)$, where $\mu^n_I\eqdef  \int_0^1 (\mu^n_t \otimes \delta_t) \d t$, such that 
\[ \partial_t \mu^n_t + \nabla \cdot (\vv^n_t \mu^n_t) = w_t^n \text{ in }  \mathscr{D}'((0,1)\times \R^d)\]
and 
\[ \int_{\R^d} \left ( |\vv^n_t|^2 +\tfrac{1}{4}|w_t^n|^2 \right ) \d \mu^n_t = |\dot \mu^n_t|^2 \quad \text{ for a.e.~} t \in (0,1)\comma
\]
where $|\dot \mu^n_t|$ denotes the $\HK_{\mssd_e}$-metric speed of the curve $(\mu^n_t)_{t \in [0,1]}$ at time $t$.
We have
\begin{align*}
\left |u(\mu'_n) - u(\mu''_n) \right |&= \left |\int_0^1 \scalar{ (\boldnabla u)_{\mu^n_t}}{(\vv^n_t, w^n_t)}_{T_{\mu^n_t}} \d t \right |
\\
 &\le \braket{ \int_0^1 \norm{(\boldnabla u)_{\mu^n_t}}_{T_{\mu^n_t}^{1,4}}^2 \diff t }^{\frac{1}{2}}  \braket{ \int_0^1 \int_{\R^d} \left (|\vv^n_t|^2 + \tfrac{1}{4}|w_t^n|^2 \right ) \,\d \mu^n_t \,\d t }^{\frac{1}{2}}
 \\
&= \HK_{\mssd_e}(\mu'_n, \mu''_n) \braket{ \int_0^1 \norm{(\boldnabla u)_{\mu^n_t}}_{T_{\mu^n_t}^{1,4}}^2 \diff t }^{\frac{1}{2}} \comma
\end{align*}
where the first equality holds integrating~\eqref{eq:derivative}.
Dividing both sides by $\HK_{\mssd_e}(\mu'_n, \mu''_n)$, we obtain
\[
\frac{\left |u(\mu'_n) - u(\mu''_n) \right |}{\HK_{\mssd_e}(\mu'_n, \mu''_n)} \le \braket{ \int_0^1 \norm{(\boldnabla u)_{\mu^n_t}}_{T_{\mu^n_t}^{1,4}}^2 \diff t }^{\frac{1}{2}} \fstop
\]
Observe that $\mu^n_t\to \mu$ in $\mcM(\R^d)$ for every $t\in [0,1]$.
We can then let~$n \to \infty$ in the above inequality using the Dominated Convergence Theorem and Lemma \ref{lem:limit} with
\[
G(\mu, x)\eqdef  \abs{ (\boldnabla u)_{\mu}(x) }_\toplus^2\comma \qquad (\mu,x) \in \mcM(\R^d)\times \R^d\comma
\]
where~$\abs{\emparg}_\toplus$ is defined immediately before~\eqref{eq:NormTot4}.
We hence get 
\begin{align*}
    \lip_{\HK_{\mssd_e}} u (\mu) &\le \braket{ \int_0^1 \norm{(\boldnabla u)_{\mu}}_{T_\mu^{1,4} }^2 \diff t }^{\frac{1}{2}}  =
                 \|(\boldnabla u)_{\mu}\|_{T_\mu^{1,4} } \fstop
\end{align*}

We now turn to the case of a general manifold~$M$.
Let $\iota: M \to \R^d$ be a smooth Nash embedding, let $\iiota\eqdef \iota_\sharp$ and let $\iiota^*$ be its pullback. By definition, $u= F \circ \mbff$ for some $F \in \rmC_b^1(\R^N)$, $\mbff=(f_1, \dots, f_N) \in (\rmC_b^1(M))^N$ and $N \in \N$.
For every $i \in \{1, \dots, N\}$ let us construct~$\tilde{f}_i \in \rmC_b^1(\R^d)$ such that 
\[
f_i = \tilde{f}_i \circ \iota\comma \qquad (\diff\iota)_p (\nabla^g f_i)_p = (\nabla\tilde{f}_i)_{\iota(p)}\comma \quad p\in M \fstop
\]
Indeed, by e.g.~\cite[Prop.~7.26, p.~200]{ONe83},~$\iota(M)$ has a smooth tubular neighborhood with (smooth) variable radius~$\eps\colon x\to \R^+$, locally isometric to its trivialization~$\bigcup_x \set{x}\times B^{T_x^\perp\iota(M)}_{\eps(x)}$.
Let~$\nchi\in \rmC^\infty(\R^+\times \R^+_0)$ with~$\nchi(\eps,\emparg)\equiv 1$ on~$[0,\eps/2)$ and~$\nchi(\eps,\emparg)\equiv 0$ on~$[\eps,\infty)$ for every~$\eps\in\R^+$.
Finally, set
\[
\tilde f_i(y)\eqdef  \begin{cases} f_i(x) \chi\tparen{\eps(x),\abs{x-y}} & \text{if } y\in B^{T_x^\perp\iota(M)}_{\eps(x)} \text{for some~$x\in \iota(M)$} 
\\
0 & \text{otherwise}
\end{cases}\comma \qquad y\in \R^d\fstop
\]
It is readily seen that~$\tilde f_i$ is a smooth function on~$\R^d$, and that, letting~$x=\iota(p)$,
\begin{align*}
(\nabla\tilde f_i)_x=&\ (\nabla^{\iota_*g}\tilde f_i)_x\oplus (\nabla^\perp \tilde f_i)_x
\\
=&\ \tparen{f_i(x)\, \tparen{\nabla^{\iota_*g} \nchi\tparen{\eps(\emparg),0}}_x + (\nabla^{\iota_*g}\tilde f_i)_x \nchi\tparen{\eps(x),0}} \oplus (\nabla^\perp \tilde f_i)_x 
\\
=&\ (\nabla^{\iota_*g}\tilde f_i)_x \oplus \tparen{\nabla^\perp\nchi\tparen{\eps(x,\emparg}}_x =(\nabla^{\iota_*g}\tilde f_i)_x= (\diff\iota)_p(\nabla^gf_i)_p \fstop
\end{align*}

Now, set~$\tilde{u} \eqdef F \circ \tilde{\mbff}^\trid \in \FC{1,1}{b,b}{1}{b}(\mcM(\R^d))$, and note that~$u= \iiota^*\tilde{u}$ on~$\mcM(M)$.

On the hand, since~$\iota\colon (M,\mssd_g)\to (\R^d,\mssd_e)$ is non-expansive, and since~$\mssd\mapsto\HK_{\mssd}$ is a monotone assignment,
\begin{align*}
\lip_{\HK_{\sfd_g}} u (\mu) 
&\leq \lip_{\HK_{\sfd_e}} \tilde{u} (\iiota \mu) =  \|(\boldnabla \tilde u)_{\iiota\mu}\|_{ T_{\iiota\mu}^{1,4}}.
\end{align*}
On the other hand, since~$\iota\colon (M,g)\to (\iota(M),\iota_*g)$ is a Riemannian isometry, 
\begin{align*}
\|(\boldnabla \tilde u)_{\iiota\mu}\|^2_{ T_{\iiota\mu}^{1,4}} &= 
\|(\boldnabla u)_{\mu}\|^2_{ T_{\iiota\mu}^{1,4}}  \fstop
\end{align*}
This proves the first inequality in the general case.

 \medskip

To prove the other inequality we fix $\mu \in \mcM(M)$ and we consider the curve
\[
\mu_t\eqdef (\exp(t\TT_1))_\sharp((1+ 2tT_2)^2 \mu) \comma \quad t \in [0,1]\comma
\]
where 
\[
(\TT_1(x), T_2(x))= (\boldnabla u)_{\mu}(x)\comma \qquad x \in M\fstop
\] 
It is not difficult to check that
\[
\aalpha_t\eqdef \tparen{[\text{id}_M, 1], [\exp(t\TT_1), 1+ 2tT_2]}_\sharp \mu  \in \mcM(\f{C}(M, M))
\]
belongs to $\f{H}(\mu, \mu_t)$ (cf.~\eqref{eq:defhk}) so that we can estimate
\begin{align*}
    \HK_{\mssd_g}&(\mu, \mu_t)^2 \leq
    \\
    &\leq\int_M (\mssd_g)_{\pi, \f{C}}\tparen{[\text{id}_M, 1], [\exp(t\TT_1), 1+ 2tT_2]}^2 \d \mu
    \\
    &= \int_M \braket{ 4 t^2 |T_2(x)|^2 + 4(1+t 2T_2(x)) \sin^2 \left ( \frac{ \mssd_g\tparen{x, \exp_x(t\TT_1(x)} \wedge \pi}{2} \right ) } \d \mu(x)
    \\
    &\leq \int_M \braket{4t^2 |T_2(x)|^2 + (1+t 2T_2(x))\, \mssd_g^2\tparen{x, \exp_x(t\TT_1(x)} } \d \mu(x)
    \\
    &\leq \int_M \braket{4t^2 |T_2(x)|^2 + (1+t 2T_2(x))\, \abs{t\TT_1(x)}^2 } \d \mu(x)
    \\
    & \le t^2  \| (\boldnabla u)_{\mu} \|^2_{T_\mu^{1,4}} + o(t^2) \quad \text{ as } t \downarrow 0,
\end{align*}
where we have used the uniform boundedness of $(\TT_1,T_2)$ and \cite[Formula (7.4)]{LMS18}.
On the other hand, by \eqref{eq:derivative1}, we get that
\[
\lim_{t \downarrow 0} \frac{u(\mu_t)-u(\mu)}{t} = \scalar{(\boldnabla u)_{\mu}}{(\TT_1, 4  T_2)}_{T_\mu} = \| (\boldnabla u)_{\mu}\|^2_{T_\mu^{1,4}}\fstop
\]
Thus
\begin{align*}
\lip_{\HK_{\mssd_g}} u (\mu) &\geq \limsup_{t \downarrow 0} \frac{u(\mu_t)-u(\mu) }{\HK_{\mssd_g}(\mu, \mu_t)}
\\
&\geq \limsup_{t \downarrow 0}\frac{u(\mu_t)-u(\mu) }{t} \frac{t}{\HK_{\mssd_g}(\mu, \mu_t)} \ge \|(\boldnabla u)_{\mu}\|_{T_\mu^{1,4}} \fstop \qedhere
\end{align*}
\end{proof}

As a final remark we compute the asymptotic Lipschitz constant of cylinder functions w.r.t.~the distances $\He_2$ and $\W_{2, \sfd_g}$.

\begin{proposition}\label{prop:equalityhew} Let $u \in \FC{1,1}{b,b}{1}{b}$ and $\mu \in \mcM(M)$. Then
\[
\lip_{\W_{2, \sfd_g}}^{\tau_w} u (\mu) = \| (\gradW u)_{\mu}\|_{T_\mu^{1,4}}\comma \qquad \lip_{\He_2}^{\tau_w} u (\mu) = \| (\gradH u)_{\mu}\|_{T_\mu^{1,4}}\comma
\]
where $\gradW$ and $\gradH$ are as in \eqref{eq:GradRepresentationhor}  and \eqref{eq:GradRepresentationver}, respectively and~$\tau_w$ is the weak topology on~$\mcM(M)$.
\end{proposition}
\begin{proof}
We treat separately the two distances.

\paragraph{The extended Wasserstein case} The proof is very similar to the one of \cite[Proposition 4.7]{S22}, we only have to take care that the distance can also attain also the value $+\infty$. To show the first inequality $\lip_{\W_{2, \sfd_g}}^{\tau_w} u (\mu) \le \| (\gradW u)_{\mu}\|_{T_\mu^{1,4}}$ we observe that for every $\mu \in \mcM(M)$, we can find a sequence $(\mu'_n, \mu''_n) \in \mcM(M) \times \mcM(M)$ such that $\mu'_n \ne \mu''_n$, $\W_{2, \sfd_g}(\mu'_n, \mu''_n) < + \infty$ for every $n \in \N$ and
\[
\mu'_n, \mu''_n \weakto \mu, \quad \lip_{\W_{2, \sfd_g}}^{\tau_w} u (\mu) = \lim_{n \to + \infty} \frac{ |u(\mu'_n)-u(\mu''_n)|}{\W_{2, \sfd_g}(\mu'_n, \mu''_n)} \quad \text{ as } n \to + \infty \fstop
\]
This is because $\tau_w$ is metrized by $\HK_{\mssd_g}$ and for every $r>0$ we can always find $\mu', \mu'' \in \rmB_{\HK_{\mssd_g}}(\mu, r)$ such that $\mu' \ne \mu''$ and $\W_{2, \sfd_g}(\mu', \mu'')< + \infty$. Indeed it is enough to consider three distinct points $x_0, x_1, x_2 \in M$ and note that 
\[ 
\mu'_\eps \eqdef \mu \mres \rmB_{\mssd_g}(x_0, 1/\eps) + \eps \delta_{x_1}\comma \quad  \mu''_\eps \eqdef \mu \mres \rmB_{\mssd_g}(x_0, 1/\eps) + \eps \delta_{x_2}
\]
weakly converge to $\mu$ as $\eps \downarrow 0$, they are different and compactly supported non-negative measures with the same positive total mass. Since $W_2^2(\mu'_n, \mu''_n)< +\infty$ we can find $\ggamma^n \in \Cpl(\mu'_n, \mu''_n)$ such that
\[
\int_{M \times M} \mssd^2_g \de \ggamma^n = W_2^2(\mu'_n, \mu''_n) \quad \text{ for every } n \in \N \fstop
\]
Let $G\colon M \times M \to \rmC([0,1]; M)$ be a (universally measurable) map associating to every pair of points $(x_0,x_1) \in M \times M$ the curve $G(x_0,x_1)=(\gamma_t^{x_0,x_1})_{t \in [0,1]}$ which is a constant speed geodesic connecting $x_0$ to $x_1$. Finally let $\mu^n_t\eqdef  (e_t \circ G)_\sharp \ggamma^n$, $t \in [0,1]$, where $e_t: \rmC([0,1];M) \to M$ is the evaluation map sending a curve $(\gamma_t)_{t \in [0,1]}$ to its value at time $t$, $\gamma_t \in M$. It is not difficult to check that
\[ \frac{\de}{\de t} u(\mu_t^n) = \int_{M \times M} g\tparen{(\gradW u)_{\mu_t^n}(\gamma_t^{x,y}), \dot{\gamma}^{x,y}_t } \de \ggamma^n(x,y) \quad \text{ for a.e.~$t \in (0,1)$}, \]
so that
\begin{align*}
\left |u(\mu'_n) - u(\mu''_n) \right |&= \left |\int_0^1 \int_{M \times M} g\tparen{(\gradW u)_{\mu_t^n}(\gamma_t^{x,y}), \dot{\gamma}^{x,y}_t } \de \ggamma^n(x,y) \d t \right |
\\
 &\le  \braket{ \int_0^1 \norm{(\gradW u)_{\mu^n_t}}_{T^\hor{\mu^n_t}}^2 \,\d t }^{\frac{1}{2}}
 \cdot \braket{ \int_{M \times M} \mssd_g^2 \de \ggamma^n }^{\frac{1}{2}}
 \\
&= \W_{2, \mssd_g}(\mu'_n, \mu''_n)  \braket{ \int_0^1 \norm{(\gradW u)_{\mu^n_t}}_{T^\hor{\mu^n_t}}^2 \,\d t }^{\frac{1}{2}} \fstop
\end{align*}
Dividing by $\W_{2, \mssd_g}(\mu'_n, \mu''_n)$ and passing to the limit as $n\to+\infty$ leads to the sought inequality, also using Lemma~\ref{lem:limit} with
\[
G(\mu, x)\eqdef  \abs{ (\gradW u)_{\mu}(x) }^2, \quad (\mu,x) \in \mcM(M)\times M \fstop
\]
The other inequality is obtained similarly to the one in Proposition \ref{prop:equalityriem} using the curve $\mu_t\eqdef  \exp(t\TT)_\sharp \mu$ and the plan $\ggamma_t\eqdef  (\id_M, \exp(t \TT)) \in \Cpl(\mu, \mu_t)$, where $\TT(x) = (\nabla^\hor u)_\mu(x)$.

\paragraph{The Hellinger case} The proof is again a modification of \cite[Proposition 4.7]{S22} but we reproduce the argument in full. Since $\tau_w$ is metrizable (e.g.~by $\HK_{\mssd_g}$), we can find $(\mu'_n, \mu''_n) \in \mcM(M)\times \mcM(M)$ with $\mu'_n \ne \mu''_n$  such that $\mu'_n, \mu''_n \weakto (\mu, \mu)$ and
\[
\lip_{\He_2}^{\tau_w} u (\mu)  = \lim_n \frac{ \left |u(\mu'_n) - u(\mu''_n) \right |}{\He_2(\mu'_n, \mu''_n)} \fstop
\]
For every $n \in \N$, let $\eta_n \in \mcM(M)$ be such that $\mu'_n, \mu''_n \ll \eta_n$ and let us define for every $t \in [0,1]$ the quantities
\begin{align*}
\rho'_n &\eqdef \frac{\de \mu'_n}{\de \eta_n}\comma & \rho''_n &\eqdef \frac{\de \mu''_n}{\de \eta_n}\comma & \rho^t_n &\eqdef \braket{(1-t) \sqrt{\rho'_n} + t \sqrt{\rho''_n}}^2\comma
\\
w_t &\eqdef 2 \frac{ \sqrt{\rho''_n}-\sqrt{\rho'_n}}{\sqrt{\rho^t_n}}\comma & \mu^t_n &\eqdef \rho^t_n \eta_n\fstop
\end{align*}
It is not difficult to check that $t \mapsto \mu^t_n$ is weakly continuous, $\mu^0_n = \mu'_n$, $\mu^1_n = \mu''_n$, $\partial_t \mu^t_n = w_n^t \mu^t_n$ in the sense of distributions in $(0,1)\times M$ and 
\begin{equation*}
\begin{aligned}
4\, \He_2(\mu'_n, \mu''_n)^2 &= \int_0^1 \int_{\R^d} |w_t^n|^2 \de \mu^n_t \de t\comma
\\
\frac{\de}{\de t} u(\mu^t_n) &= \int_{\R^d} (\gradH u)_{\mu^n_t} w^n_t \d \mu^n_t\comma 
\end{aligned}
\qquad \text{ for a.e.~} t \in (0,1) \fstop
\end{equation*}
We thus have
\begin{align*}
\left |u(\mu'_n) - u(\mu''_n) \right |&= \left |\int_0^1 \int_{\R^d} (\gradH u)_{\mu^n_t} w^n_t \d \mu^n_t \d t \right |
\\
 &\le  \braket{ \int_0^1 4  \norm{(\gradH u)_{\mu^n_t}}_{T^\ver_{\mu^n_t}}^2 \,\d t }^{\frac{1}{2}}
 \cdot \braket{ \int_0^1 \int_{M} \left ( \tfrac{1}{4}|w^t_n|^2 \right ) \,\d \mu^n_t \,\d t }^{\frac{1}{2}}
 \\
&= \He_2(\mu'_n, \mu''_n)  \braket{ \int_0^1  4  \norm{(\gradH u)_{\mu^n_t}}_{T^\ver_{\mu^n_t}}^2 \,\d t }^{\frac{1}{2}} \fstop
\end{align*}
Dividing both sides by $\He_2(\mu'_n, \mu''_n)$, we obtain
\[ \frac{\left |u(\mu'_n) - u(\mu''_n) \right |}{\He_2(\mu'_n, \mu''_n)} \le  \braket{  4 \int_0^1 \norm{(\gradH u)_{\mu^n_t}}_{T^\ver_{\mu^n_t}}^2 \,\d t }^{\frac{1}{2}}\fstop
\]
Observe that $\mu^t_n\weakto \mu$ for every $t\in [0,1]$, so that we can pass to the limit as $n \to + \infty$ the above inequality using the Dominated Convergence Theorem and Lemma~\ref{lem:limit} with
\[
G(\mu, x)\eqdef 4 \abs{ (\gradH u)_{\mu}(x) }^2, \quad (\mu,x) \in \mcM(M)\times M \fstop
\]
We hence get 
\begin{align*}
    \lip^{\tau_w}_{\He_2} u (\mu) &\le 
    2   \norm{(\gradH u)_\mu}_{T^\ver_\mu} =  \norm{(\gradH u)_\mu}_{T_\mu}  \fstop
\end{align*}
To prove the other inequality we fix $\mu \in \mcM(M)$ and we consider the curve
\[
\mu_t\eqdef (1+ 2tT)^2 \mu , \quad t \in [0,1]\comma
\]
where 
\[
T(x)= (\gradH u)_{\mu}(x), \quad x \in M \fstop
\] 
We can estimate, for $t$ sufficiently small, that
\[
    \He_2(\mu, \mu_t)^2 = \int_M \braket{\sqrt{(1+2tT)^2}-\sqrt{1} }^2 \de \mu=t^2 \int 4T^2 \de \mu = 4  t^2 \norm{(\gradH u)_\mu}_{T^\ver_\mu}^2 \fstop
\]
On the other hand it is not difficult to see that
\[
\lim_{t \downarrow 0} \frac{u(\mu_t)-u(\mu)}{t} = 4\int_M (\gradH u)_{\mu} T\d \mu = 4 \norm{(\gradH u)_\mu}_{T^\ver_\mu}^2 \fstop
\]
Thus
\[
\lip_{\He_2}^{\tau_w} u (\mu) \ge \limsup_{t \downarrow 0}\frac{u(\mu_t)-u(\mu) }{t} \frac{t}{\He_2(\mu, \mu_t)} \ge  \norm{(\gradH u)_\mu}_{ T_\mu} \fstop \qedhere
\]
\end{proof}

\subsection{Density of cylinder functions}\label{s:DensityCylinder}
In this section, we consider again the (Gaussian) Hellinger--Kantorovich distance only on the complete and separable metric space $(U, \varrho)=(\R^d, \sfd_e)$, where $\sfd_e$ denotes the distance induced by the Euclidean norm on $\R^d$. For this reason, we remove the dependence on $\sfd_e$ in the notation for the $\HK=\HK_{\sfd_e}$ and the $\GHK= \GHK_{\sfd_e}$ distances. 

\medskip

Recall that $(\mcM(\R^d), \HK)$ and $(\mcM(\R^d), \GHK)$ are complete and separable metric spaces and they induce the same topology on $\mcM(\R^d)$ which also coincides with the weak topology $\sigma(\mcM(\R^d), \rmC_b(\R^d))$. We fix a finite, non-negative, non-zero Borel measure $\mcQ$ on $(\mcM(\R^d), \HK)$ and we consider the Polish metric measure space $\X(\mcQ)\eqdef (\mcM(\R^d), \HK, \mcQ)$.

\medskip

Our aim is to prove that the unital and point-separating subalgebra $\AA\eqdef \FC{1,1}{u,b}{1}{b}(\mcM(\R^d)) \subset \Lip_b(\mcM(\R^d), \HK)$ (see Definition \ref{def:cyl}) is dense in $2$-energy in $D^{1,2}(\X(\mcQ))$ in the sense of Definition~\ref{def:density}. In order to do so, we are going to apply Theorem \ref{theo:startingpoint} to $X=\mcM(\R^d)$, $\delta=\HK$ and $\sfd=\GHK$ since~$\HK$ is the length distance generated by~$\GHK$, as noted in~\cite[Corollary  8.7]{LMS18}.
Note that the underlying metric-measure space is \emph{always} $\X(\mcQ)$ so that, in particular, $(2,\AA)$-relaxed gradients are computed w.r.t.~the distance $\HK$. We will thus devote this section to prove inequality \eqref{eq:214-15} i.e.
\begin{equation}\label{eq:aimsec}
 \text{for every } \nu \in \mcM(\R^d) \quad \left | \rmD \GHK(\nu, \emparg) \right|_{\star,2,\AA} \le 1 \text{ $\mcQ$-a.e.~in $\mcM(\R^d)$} \fstop
 \end{equation}
 
 The structure of this section and the proofs we present closely follow the ones of \cite[Section 4.2]{FSS22}. There are however several important differences so that we will report all the arguments in detail. We start with a technical result.

\begin{lemma}\label{le:trunc2}
   \label{le:limsup-approximation}
   Let $(v_k)_k$ be a sequence of functions in $D^{1,2}(\X(\mcQ); \AA) \cap L^\infty(\mcM(\R^d), \mcQ)$ such
   that $v_k$ and 
   $|\rmD v_k|_{\star,2, \AA}$ are uniformly bounded in every bounded subset
   of $\mcM(\R^d)$ and let $v,G$ be Borel functions in
   $L^2(\mcM(\R^d),\mcQ)$, $G$ nonnegative.
   If
   \begin{equation}
     \label{eq:170}
     \lim_{k\to\infty}v_k(\mu)= v(\mu),\quad
     \limsup_{k\to\infty}|\rmD v_k|_{\star,2, \AA}(\mu)\le G(\mu)
     \quad\text{$\mcQ$-a.e.~in $\mcM(\R^d)$},
   \end{equation}
   then $v\in H^{1,2}(\X(\mcQ); \AA)$ and $|\rmD v|_{\star,2, \AA}\le G$.
 \end{lemma}
 \begin{proof}
   Let $u_k$ be as in Lemma \ref{le:trunc1} for $M=\R^d$: since $u_k \in \AA$ we have $|\rmD u_k|_{\star,2, \AA} \le \lip_{\HK} u_k$ so that, setting $C\eqdef  2\sqrt{8+2\pi^2}$, we have
   \begin{equation}
     \label{eq:178}
     \begin{gathered}
     u_k\in H^{1,2}(\X(\mcQ); \AA)\comma \qquad
     |\rmD u_k|_{\star,2, \AA}\le C/\sqrt{k}\comma
     \\
     |\rmD u_k|_{\star,2, \AA}(\mu)=0 \quad \text{if} \quad \mu \R^d \le
     k \ \text{or}\ \mu \R^d \ge 2k \fstop
     \end{gathered}
   \end{equation}
   Notice also that $u_k$ vanishes if $\mu \R^d \ge 2k$. Thanks to the Leibniz rule, setting $
   v_{k,m}(\mu)\eqdef v_k(\mu)u^2_m(\mu)$ and $G_k\eqdef |\rmD
   v_k|_{\star,2, \AA}$, we have
   \begin{equation}
     \label{eq:179}
     \begin{gathered}
     v_{k,m}\in \rmD^{1,2}(\X(\mcQ); \AA)\comma
     \\
     |\rmD v_{k,m}|_{\star,2, \AA}(\mu)\le G_k(\mu)u^2_{m}(\mu)+
     2Cm^{-1/2} v_{k}(\mu)u_{m}(\mu) \fstop
     \end{gathered}
   \end{equation}
   Since for every $m\in \N$ the sequence
   $k\mapsto G_ku_m^2$ is uniformly bounded, we can find an
   increasing subsequence $j\mapsto k(j)$ such that
   $j\mapsto G_{k(j)}u_m^2$ is
   weakly$^*$ convergent in $L^\infty(\mcM(\R^d),\mcQ)$ and we
   denote by $\tilde G_m$ is
   weak$^*$ limit. By Fatou's lemma, for every Borel set
   $B\subset \mcM(\R^d)$ we get
   \begin{align*}
     \int_B \tilde G_m\,\d\mcQ&=
     \lim_{j\to\infty}\int_B G_{k(j)}(\mu)u^2_m(\mu)\,\d\mcQ(\mu)
                               \\&\le
     \int_B
     \limsup_{j\to\infty}\Big(G_{k(j)}(\mu)u^2_m(\mu)\Big)\,\d\mcQ(\mu)\\
     &\le \int_B G^2u_m^2\,\d\mcQ
   \end{align*}
   so that we deduce
   \begin{equation}
     \label{eq:181}
     \tilde G_m\le G^2u_m^2\quad\text{$\mcQ$-a.e.~in
       $\mcM(\R^d)$, for every $m\in \N$.}
   \end{equation}
   On the other hand, passing to the limit in \eqref{eq:179} along the
   subsequence $k(j)$ and recalling that
   $\lim_{j\to\infty}v_{k(j),m}=vu_m^2$ $\mcQ$-a.e.~we get
   \begin{equation}\label{eq:182}
   \begin{aligned}
     |\rmD (vu_m^2)|_{\star,2, \AA}(\mu)\leq&\ \tilde G_m(\mu)+\frac{2C}{m^{1/2}} v(\mu)u_m(\mu)
     \\
     \leq&\ G(\mu)u_m^2(\mu)+\frac{2C}{m^{1/2}} v(\mu)u_m(\mu)
     \end{aligned}
     \qquad\text{for $\mcQ$-a.e.~$\mu\in \mcM(\R^d)$.}
   \end{equation}
   We eventually pass to the limit as $m\to\infty$ concluding the proof
   of the Lemma.
 \end{proof}

The main result of this section is the following Proposition which is just one step away from \eqref{eq:aimsec}. Recall the notation introduced in Section \ref{sec:spec} and in Theorem \ref{teo:final}.

\begin{proposition} Let $\nu \in \mcM^{ac}(\R^d)$ and let $R, \delta>0$ be such that $\supp(\nu)=\overline{\rmB(0,R)}$, and $\nu \ge \delta\Leb{d} \mres \rmB(0,R)$; let $\eps \in (0,1)$ be fixed and let $\mathsf{T}_\eps$ be as in \eqref{eq:theop}. Let $\zeta\in \rmC^1(\R)$ be a non-decreasing function whose derivative has compact support. Then, for $\mcQ$-a.e.~$\mu \in \mcM(\R^d)$
\[
\left |\rmD \zeta \circ \left (\tfrac{1}{2} \GHK\tparen{\nu, \mathsf{T}_\eps(\emparg)}^2 \right)\right|_{\star,2, \AA} (\mu) \le \zeta' \left ( \tfrac{1}{2} \GHK\tparen{\nu, \mathsf{T}_\eps(\mu)}^2 \right )  \GHK\tparen{\nu, \mathsf{T}_\eps(\mu)} \fstop
\]
\end{proposition}
\begin{proof}
We set $B\eqdef \nu \R^d$; for the whole proof we are going to keep $\nu$, $B$, $R$, $\eps$ and~$\zeta$ fixed and then we will not stress the dependence of the objects we are going to define w.r.t.~them, even if many of them do depend on $\nu$, $R$, $B$, $\eps$ and $\zeta$.

Let $\mathcal{G} \eqdef (\mu^h)_h \subset \mcM(\R^d)$ be a countable and dense subset of $\mcM(\R^d)$. We define the pair $(\varphi_h, \psi_h)$ as the pair coming from Theorem \ref{teo:final} for $(\nu, \mu^h)$ and the functions $\eta_h$ and $\eta_h^*$ as
\[ q_h(y) \eqdef \frac{1-\rme^{-|y|^2+2\psi_h(y)}}{2}, \quad q_h^*(x)\eqdef \frac{1-\rme^{-|x|^2+2\varphi_h(x)}}{2}, \quad (x,y) \in \rmB(0,R) \times \R^d. \]
Recall that, by Theorem \ref{teo:final}(1)-(2), we have that $\psi^h$ is $C$-Lipschitz and $q_h$ is $C$-Lipschitz and $C$-bounded, for a constant $C$ that does not depend on $h$. In particular we have that 
\[
\tfrac{1}{2} \GHK(\nu, \mathsf{T}_\eps(\mu^h))^2 = \int_{\R^d} q_h \d \mathsf{T}_\eps(\mu^h) + \int_{\rmB(0,R)} q_h^* \d \nu \quad \text{ for every } h \in \N\fstop
\]
Let us define, for every $h,k \in \N$, the functions 
\[ w_h(\mu)\eqdef  \int_{\R^d} q_h \d \mathsf{T}_\eps(\mu) + \int_{\rmB(0,R)} q_h^* \d \nu, \quad z_k(\mu)\eqdef  \max_{1 \le h \le k} w_h(\mu), \quad  \mu \in \mcM(\R^d). \]
Notice that, while $q_h^* \in L^1(\R^d,\nu)$ by definition, $q_h$ is Borel and bounded and therefore integrable for every measure in $\mcM(\R^d)$. 
    
\paragraphn{Claim $(1)$. We have that
    \begin{equation}\label{eq:supeq}
  \lim_{k \to + \infty} z_k(\mu) = \sup_k z_k(\mu) = \tfrac{1}{2} \GHK(\nu, \mathsf{T}_\eps(\mu))^2 \quad \text{ for every } \mu \in \mcM(\R^d).  
\end{equation}}

\paragraph{Proof of claim $(1)$} Since $z_k(\mu)$ is increasing w.r.t.~$h$, the supremum
\[
\sup_k z_k(\mu)\comma \qquad \mu \in \mcM(\R^d)\comma
\]
is well defined and satisfies by Theorem \ref{teo:ghk} the inequality
\[
\sup_k z_k(\mu) \le \tfrac{1}{2} \GHK(\nu, \mathsf{T}_\eps(\mu))^2 \quad \text{ for every } \mu \in \mcM(\R^d)
\]
with equality if $\mu \in \mathcal{G}$. By density of $\mathcal{G}$, to prove the claim, it is enough to prove that $\sup_k z_k(\cdot)$ is a locally $\HK$-Lipschitz function. To this aim, let us fix $\mu',\mu'' \in \mcM(\R^d)$ and observe that, since $q_h$ is $C$-Lipschitz and $C$-bounded, we have (cf.~\cite[Proposition 7.18]{LMS18}) that
\begin{align*}
\int_{\R^d} &q_h \de\tbraket{\mathsf{T}_\eps(\mu')-\mathsf{T}_\eps(\mu'')}
\\
&\qquad \leq C \sqrt{2+\tfrac{\pi^2}{2}}\, \tbraket{\mathsf{T}_\eps(\mu')(\R^d)+ \mathsf{T}_\eps(\mu'')(\R^d)}^{1/2}\HK\tparen{\mathsf{T}_\eps(\mu'),\mathsf{T}_\eps(\mu'')}\comma
\end{align*}
hence
\[
|z_k(\mu')-z_k(\mu'')| \le C \sqrt{2+\tfrac{\pi^2}{2}} \tbraket{\mathsf{T}_\eps(\mu')(\R^d)+ \mathsf{T}_\eps(\mu'')(\R^d)}^{1/2}\HK\tparen{\mathsf{T}_\eps(\mu'),\mathsf{T}_\eps(\mu'')}\fstop
\]
Passing to the limit as $k \to + \infty$ we obtain the sought local $\HK$-Lipschitz property, also noting that $\mathsf{T}_\eps$ is $\HK$-Lipschitz (cf.~\cite[Section 8.7]{LMS18}). This proves the claim.

\paragraphn{Claim $(2)$. Let $u_k$ be as in Lemma \ref{le:trunc1}; If we set $v_k\eqdef  u_k \cdot (\zeta \circ z_k) $, then $v_k \in H^{1,2}(\X(\mcQ); \AA)$ and 
\begin{equation}\label{eq:limsup}
\begin{aligned}
    |\rmD v_k|_{\star, 2, \AA}(\mu) \leq&\ u_k(\mu) (\zeta'(w_h(\mu))) \sqrt{\int \left ( |\nabla q_h |^2  + 4|q_h |^2  \right ) \de \mathsf{T}_\eps(\mu)}
    \\
    &\qquad + \frac{\sqrt{4+\pi^2}}{k}|\theta'(\mu \R^d/k)|\sqrt{\mu \R^d} \|\zeta\|_\infty
    \end{aligned}
\end{equation}
for $\mcQ$-a.e.~$\mu$ such that $z_k(\mu)=w_h(\mu)$.
}

\paragraph{Proof of claim $(2)$} First of all note that $u_k \cdot (\zeta \circ w_h) \in \FC{1,1}{c,b}{1}{b}(\mcM(\R^d)) \subset \AA$ since it can be written as
    \[
    u_k \cdot (\zeta \circ w_h) = (\theta_k \circ \car^\trid) (\zeta_{h} \circ g_h^\trid)
    \]
where $\varsigma_k = \varsigma (\emparg/k) \in \rmC_c^1(\R)$ and 
\begin{gather*}
\zeta_{h}(s)\eqdef  \zeta \left (s+ \int_{\rmB(0,R)}q_h^* \d \nu + \eps(q_h \ast \kappa_{\eps})(0)  \right)\comma \qquad s\in \R\comma
\\
g_h\eqdef  (q_h \ast \kappa_{\eps})\circ f_\eps \in \rmC_b^1(\R^d)\fstop
\end{gather*}
By Lemma \ref{le:trunc1} we have
\[
\lip_{\HK} u_k (\mu) \le  \frac{| \varsigma'(\mu \R^d/k)|}{k} \sqrt{4+\pi^2} (\mu \R^d)^{1/2}\comma \qquad \mu \in \mcM(\R^d)\comma
\]
we can estimate
\begin{align}
\lip_{\HK} (u_k \cdot (\zeta \circ w_h))(\mu) &\le u_k(\mu)\lip_{\HK}(\zeta \circ w_h)(\mu) +\lip_{\HK} u_k (\mu) |(\zeta \circ w_h)| \\ \label{eq:esti2}
& \le u_k(\mu)\lip_{\HK}(\zeta \circ w_h)(\mu) + \frac{\sqrt{4+\pi^2}}{k} \abs{\frac{\varsigma'(\mu \R^d)}{k}} \sqrt{\mu \R^d}\, \|\zeta\|_\infty \fstop
\end{align}    

We can also estimate the asymptotic Lipschitz constant of $(\zeta \circ w_h)$: still by Proposition \ref{prop:equalityriem}, we have 
\begin{align*}
\lip_{\HK}&(\zeta_h \circ g_h^\trid)(\mu)^2
\\
&= (\zeta_h'(g_h^\trid(\mu)))^2 \int \braket{ |\nabla g_h|^2 + 4|g_h|^2 } \d \mu
\\
& \le \zeta_h'(g_h^\trid(\mu))^2 \int \braket{ |(\nabla (q_h \ast \kappa_{\eps}))\circ f_\eps|^2\|\rmD f_\eps\|^2  + 4|(q_h \ast \kappa_{\eps})\circ f_\eps|^2 } \d \mu
\\
& \le \zeta_h'(g_h^\trid(\mu))^2 \int \braket{ |(\nabla (q_h \ast \kappa_{\eps}))\circ f_\eps|^2  + 4|(q_h \ast \kappa_{\eps})\circ f_\eps|^2 } \d \mu
\\
&= \zeta_h'(g_h^\trid(\mu))^2 \int \braket{ |\nabla (q_h \ast \kappa_{\eps})|^2  + 4|q_h \ast \kappa_{\eps}|^2 } \d (f_\eps)_\sharp \mu
\\
& \le \zeta_h'(g_h^\trid(\mu))^2 \int \braket{ |\nabla (q_h \ast \kappa_{\eps})|^2  + 4|q_h \ast \kappa_{\eps}|^2 } \d \braket{ (f_\eps)_\sharp \mu +\eps \delta_0}
\\
& = \zeta_h'(g_h^\trid(\mu))^2 \int \braket{ |(\nabla q_h )\ast \kappa_{\eps}|^2  + 4|q_h \ast \kappa_{\eps}|^2 } \d \braket{ (f_\eps)_\sharp \mu +\eps \delta_0}
\\
& \leq 
\zeta_h'(g_h^\trid(\mu))^2 \int \braket{ |\nabla q_h |^2  + 4|q_h |^2  }\ast \kappa_{\eps} \d \braket{ (f_\eps)_\sharp \mu +\eps \delta_0}
\\
& = \zeta_h'(g_h^\trid(\mu))^2 \int \braket{ |\nabla q_h |^2  + 4|q_h |^2  } \d \braket{\braket{ (f_\eps)_\sharp \mu +\eps \delta_0}\ast \kappa_{\eps} }
\\
& = \zeta_h'(g_h^\trid(\mu))^2 \int \braket{ |\nabla q_h |^2  + 4|q_h |^2  } \de \mathsf{T}_\eps(\mu)
\end{align*}
where we have also used that $\|\rmD f_\eps(x)\| \le 1$ for every $x \in \R^d$.  We also deduce that 
\begin{equation}\label{eq:claim2}
\begin{aligned}
\lip_{\HK}(\zeta_h \circ g_h^\trid)(\mu) \leq&\ \zeta_h'(g_h^\trid(\mu)) \sqrt{\int \left ( |\nabla q_h |^2  + 4|q_h |^2  \right ) \de \mathsf{T}_\eps(\mu)}
\\
\leq&\ \|\zeta'\|_\infty \sqrt{5}C\sqrt{\eps + \mu \R^d}.
\end{aligned}
\end{equation}

We can now estimate the minimal relaxed gradient of the function $v_k= u_k \cdot (\zeta \circ z_k)$: since $\zeta$ is an increasing function we have that 
\[
u_k \cdot (\zeta \circ z_k) = u_k \cdot \max_{1 \le h \le k} (\zeta \circ w_h) \fstop
\]
By~\eqref{eq:locsuper1}, for $\mcQ$-a.e.~$\mu$ where $w_h(\mu)=z_k(\mu)$, we have 
\begin{align*}
|\rmD v_k|_{\star, 2, \AA}(\mu) &=|\rmD u_k\cdot(\zeta \circ z_k)|_{\star, 2, \AA}(\mu)
\\
&= |\rmD u_k\cdot(\zeta \circ w_h)|_{\star, 2, \AA}(\mu)\\
& \le \lip_{\HK} (u_k\cdot(\zeta \circ w_h))(\mu) \\
& \le u_k(\mu)\lip_{\HK}(\zeta \circ w_h)(\mu) + \frac{\sqrt{4+\pi^2}}{k}|\varsigma'(\mu \R^d/k)|\sqrt{\mu \R^d} \|\zeta\|_\infty
\\
& \le u_k(\mu) (\zeta_h'(g_h^\trid(\mu))) \sqrt{\int \left ( |\nabla q_h |^2  + 4|q_h |^2  \right ) \de \mathsf{T}_\eps(\mu)}
\\
&\qquad + \frac{\sqrt{4+\pi^2}}{k}|\varsigma'(\mu \R^d/k)|\sqrt{\mu \R^d} \|\zeta\|_\infty.
\end{align*}
This concludes the proof of the second claim.

\paragraphn{Claim $(3)$. Let $\mu \in \mcM(\R^d)$ and let $(h_n)_n \subset \N$ be such that
\[
\int_{\R^d} q_{h_n} \d \mathsf{T}_\eps(\mu) + \int_{\rmB(0,R)} q_{h_n}^* \d \nu = w_{h_n}(\mu) \to \tfrac{1}{2} \GHK^2(\nu, \mathsf{T}_\eps(\mu))^2 \quad \text{ as } n \to + \infty \fstop
\]
Then there exist two convex and continuous functions $\varphi\colon \rmB(0,R) \to \R$ and $\psi\colon\R^d \to \R$ such that $\varphi$ and $\psi$ are the Legendre conjugate of each other and, up to an non-relabeled subsequence, $\varphi_{h_n} \to \varphi$ locally uniformly in $\rmB(0,R)$, $\psi_{h_n} \to \psi$ locally uniformly in $\R^d$ and $\nabla \psi_{h_n} \to \nabla \psi$ $\Leb{d}$-a.e.~in $\R^d$.
}

\paragraph{Proof of claim $(3)$} We consider the shifted pairs
\[
\tilde{\varphi}_{h_n}\eqdef  \varphi_{h_n} + \psi_{h_n}(0)\comma \qquad \tilde{\psi}_{h_n}\eqdef  \psi_{h_n} - \psi_{h_n}(0)\fstop
\]
Notice that $\tilde{\psi}_{h_n}(0)=0$ by construction, $\varphi_{h_n}(0)\le C$ by Theorem \ref{teo:final}(1) for a constant $C>0$ that does not depend on $n$, and, again by Theorem \ref{teo:final}(3)$, \tilde{\varphi}_{h_n}$ and~$\tilde{\psi}_{h_n}$ are the Legendre transform of each other. Now we show that $\int_{\rmB(0,R)} \tilde{\varphi}_{h_n} \de \Leb{d} \le I$ for a constant $I>0$ that does not depend on $n$: since 
\[\int_{\R^d} q_{h_n} \d \mathsf{T}_\eps(\mu) + \int_{\rmB(0,R)} q_{h_n}^* \d \nu  \to \frac{1}{2} \GHK^2(\nu, \mathsf{T}_\eps(\mu))^2 \ge 0 \text{ as } n \to + \infty,\]
we can assume without loss of generality that 
\[
\int_{\rmB(0,R)} q_{h_n}^* \d \nu \ge -1 -\int_{\R^d} q_{h_n} \d \mathsf{T}_\eps(\mu) \ge -1- \tfrac{1}{2}\mu \R^d\fstop
\]
Now we estimate
\begin{align*}
\int_{\rmB(0,R)} \braket{\tfrac{1}{2}|x|^2-\varphi_{h_n}(x)} \de \Leb{d}(x) &\ge \int_{\rmB(0,R)} q_{h_n}^* \de \Leb{d} \\
&= \int_{\rmB(0,R)} (q_{h_n}^*-1/2) \de \Leb{d} + \frac{1}{2}\Leb{d}(\rmB(0,R)) \\
&\ge \frac{1}{\delta} \int_{\rmB(0,R)} (q_{h_n}^*-1/2) \de \nu + \frac{1}{2}\Leb{d}(\rmB(0,R)) \\
&= \frac{1}{\delta} \int_{\rmB(0,R)} q_{h_n}^* \de \nu - \frac{1}{2\delta} \nu \R^d + \frac{1}{2}\Leb{d}(\rmB(0,R))\\
& \ge - \frac{1}{\delta} - \frac{1}{2\delta}\mu \R^d-\frac{1}{2\delta} \nu \R^d + \frac{1}{2}\Leb{d}(\rmB(0,R)) 
\end{align*}
Thus
\[ \int_{\rmB(0,R)} \varphi_{h_n} \de \Leb{d} \le \tfrac{1}{2} \int_{\rmB(0,R)} |x|^2 \de \Leb{d}(x)+ \tfrac{1}{2\delta} \left ( 2+\mu \R^d+\nu \R^d-\delta \Leb{d}(\rmB(0,R)) \right ) =:I'  \]
Finally 
\[ \int_{\rmB(0,R)} \tilde{\varphi}_{h_n} \de \Leb{d} = \int_{\rmB(0,R)} \varphi_{h_n} \de \Leb{d} + C_n \Leb{d}(\rmB(0,R)) \le I' + C\Leb{d}(\rmB(0,R)) \defeq I. \]
Notice moreover that, by convexity of $\varphi_{h_n}$, we deduce that
\[ \varphi_{h_n}(0) \le \frac{I'}{\Leb{d}(\rmB(0,R))R^d}\]
so that
\[
C_n\eqdef  \psi_{h_n}(0) \ge - \varphi_{h_n}(0) \ge - \frac{I'}{\Leb{d}(\rmB(0,R))R^d}\fstop
\]
This gives that the sequence $C_n$ is uniformly bounded. By \cite[Lemma 3.5]{FSS22} applied to $\tilde{\varphi}_{h_n}$ and $\psi_{h_n}$, we deduce the thesis of the present claim for the shifted pairs and, since we have showed that $C_n$ is bounded, also for the original pairs. This concludes the proof of the third claim.

\paragraphn{Claim $(4)$. Let $\mu \in \mcM(\R^d)$ and let $(h_n)_n \subset \N$ be such that
\[
\int_{\R^d} q_{h_n} \d \mathsf{T}_\eps(\mu) + \int_{\rmB(0,R)} q_{h_n}^* \d \nu = w_{h_n}(\mu) \to \tfrac{1}{2} \GHK^2(\nu, \mathsf{T}_\eps(\mu))^2 \quad \text{ as } n \to + \infty\fstop
\]
Then
\[
\limsup_{n \to + \infty} \int \braket{ |\nabla q_{h_n} |^2  + 4|q_{h_n} |^2} \de \mathsf{T}_\eps(\mu) \le \GHK(\nu, \mathsf{T}_\eps(\mu)^2\fstop
\]}

\paragraph{Proof of claim $(4)$} We extract a non-relabeled subsequence such that the $\limsup$ in the statement of the present claim is achieved as a limit. By the convergence in claim~$(3)$, as $n \to + \infty$, we have that
\begin{align*}
q_{h_{n}}(y)=\frac{1-\rme^{-|y|^2+2\psi_{h_{n}}(y)}}{2} &\to q(y)\eqdef \frac{1-\rme^{-|y|^2+2\psi(y)}}{2} \quad &&\text{ for every } y \in \R^d,\\
q_{h_{n}}^*(y)=\frac{1-\rme^{-|y|^2+2\varphi_{h_{}}(y)}}{2} &\to q^*(y)\eqdef \frac{1-\rme^{-|y|^2+2\varphi(y)}}{2} \quad &&\text{ for every } y \in \rmB(0,R).
\end{align*}
Since $q_{h_n}$ is uniformly bounded, by Dominated Convergence Theorem, we deduce that
\begin{equation}\label{eq:firstlimit}
\lim_{n \to + \infty} \int_{\R^d}q_{h_{n}}\de \mathsf{T}_\eps(\mu) =  \int_{\R^d}q\de \mathsf{T}_\eps(\mu) \fstop
\end{equation}
On the other hand, $q^*_{h_n}$ is uniformly bounded from above so that Fatou's lemma shows that
\begin{equation}\label{eq:secondlimit}
\limsup_{n \to + \infty} \int_{\R^d}q_{h_{n}}^*\de \nu \le \int_{\R^d}q^*\de \nu \fstop
\end{equation}
We can thus conclude that
\begin{align*}
\int_{\R^d} q \de \mathsf{T}_\eps(\mu) + \int_{\rmB(0,R)} q^* \d \nu \geq&\ \limsup_n \braket{ \int_{\R^d} q_{h_n} \d \mathsf{T}_\eps(\mu) + \int_{\rmB(0,R)} q_{h_n}^* \d \nu }
\\
=&\ \tfrac{1}{2} \GHK(\nu, \mathsf{T}_\eps(\mu))^2\fstop
\end{align*}
Note that $\int_{\R^d}q \de \mu \in \R$ since $q$ is uniformly bounded so that also $\int_{\R^d} q^* \de \nu \in \R$ since $q^*$ is bounded above. By Theorem \ref{teo:ghk} the one above must be an equality so that, by the
uniqueness part of Theorem \ref{teo:final}, we deduce that $(\varphi,\psi)$ is the unique pair of functions whose existence is stated in Theorem \ref{teo:final} for the pair $(\nu, \mu)$. In particular, by Theorem \ref{teo:final}(5), we deduce that
\begin{align*}
\GHK(\nu, \mathsf{T}_\eps(\mu))^2 =&\int_{\R^d} \Bigg[ \tbraket{1-\rme^{-|y|^2+2\psi(y)}}^2
\\
&+\rme^{-2|y|^2+4\psi(y)}\tbraket{\rme^{|y-\nabla \psi(y)|^2}-1} \Bigg] \de \mathsf{T}_\eps(\mu)(y)\fstop
\end{align*}
Since $q_{h_n}$ is uniformly Lipschitz and bounded, we can use the Dominated Convergence Theorem and conclude that
\begin{align*}
\lim_{n \to + \infty}& \int \braket{ |\nabla q_{h_n} |^2  + 4|q_{h_n} |^2  } \de \mathsf{T}_\eps(\mu) =
\\
=& \int \braket{ |\nabla q |^2  + 4|q |^2 } \de \mathsf{T}_\eps(\mu) 
\\
=& \int \braket{ |y-\nabla \psi(y)|^2\rme^{-2|y|^2+4\psi(y)}  + (1-\rme^{-|y|^2+2\psi(y)})^2  } \de \mathsf{T}_\eps(\mu)
\\
\leq& \int_{\R^d} \braket{ \tbraket{1-\rme^{-|y|^2+2\psi(y)}}^2+\rme^{-2|y|^2+4\psi(y)}\tbraket{\rme^{|y-\nabla \psi(y)|^2}-1} } \de \mathsf{T}_\eps(\mu)(y)
\\
=&\ \GHK(\nu, \mathsf{T}_\eps(\mu))^2.
\end{align*}
This concludes the proof of fourth claim.

\paragraphn{Claim $(5)$. Let $v_k$ be defined as in claim $(2)$. Then, for $\mcQ$-a.e.~$\mu \in \mcM(\R^d)$,
\[
\limsup_k |\rmD v_k|_{\star, 2, \AA}(\mu) \le \zeta' \left ( \frac{1}{2} \GHK(\nu, \mathsf{T}_\eps(\mu))^2 \right )  \GHK(\nu, \mathsf{T}_\eps(\mu)) \fstop
\]
}

\paragraph{Proof of claim $(5)$} Let $A \subset \mcM(\R^d)$ be defined as
\[
A\eqdef  \bigcap_k \bigcup_{h=1}^k A_h^k\comma
\]
where $A_h^k$ is the full $\mcQ$-measure subset of $E_h^k\eqdef \{\mu : z_k(\mu)=w_h(\mu)\}$ where \eqref{eq:limsup} holds. Notice that $A$ has full $\mcQ$-measure. Let $\mu \in A$ be fixed and let us pick a non-decreasing sequence $k \mapsto h_k$ such that 
\[
z_k(\mu) = w_{h_k}(\mu)\fstop
\] 
By claim $(2)$ we have that 
\begin{align*}
|\rmD v_k|_{\star, 2, \AA}(\mu) \leq&\ u_k(\mu) (\zeta'(w_{h_k}(\mu))) \braket{\int \braket{ |\nabla q_{h_k} |^2  + 4|q_{h_k} |^2 } \de \mathsf{T}_\eps(\mu)}^{1/2} 
\\
&\qquad + \frac{\sqrt{4+\pi^2}}{k}|\varsigma'(\mu \R^d/k)|\sqrt{\mu \R^d} \|\zeta\|_\infty
\end{align*}
and by claim (4) we know that 
\[ \limsup_{k \to + \infty} \int \tbraket{ |\nabla q_{h_k} |^2  + 4|q_{h_k} |^2 } \de \mathsf{T}_\eps(\mu) \le \GHK(\nu, \mathsf{T}_\eps(\mu)^2.\]
Noting that $u_k(\mu) \to 1$ and passing to the $\limsup_{k \to + \infty}$, we obtain the statement of the fifth claim.

\paragraph{Claim $(6)$. Conclusion}
We deduce the assertion of the present proposition from Lemma \ref{le:trunc2}. We set
\[
v\eqdef \zeta \paren{\tfrac{1}{2} \GHK\tparen{\nu, \mathsf{T}_\eps(\emparg)}^2}\comma \quad G\eqdef \zeta' \paren{\tfrac{1}{2} \GHK\tparen{\nu, \mathsf{T}_\eps(\emparg)}^2}  \GHK\tparen{\nu, \mathsf{T}_\eps(\emparg)} \fstop
\]
By claim (2), $v_k \in D^{1,2}(\X(\mcQ); \AA)$ and it is clear that $v_k \in L^\infty(\mcM(\R^d), \mcQ)$ since it is bounded by $\|\zeta\|_\infty$ (in particular the sequence $v_k$ is uniformly bounded in every bounded set of $\mcM(\R^d)$). The sequence of relaxed gradients $|\rmD v_k|_{\star,2, \AA}$ is also uniformly bounded in every bounded set of $\mcM(\R^d)$ since by claim (2) and \eqref{eq:claim2} we have
\begin{align*}
\begin{aligned}
|\rmD v_k|_{\star,2, \AA}(\mu) \le&\ u_k(\mu) \|\zeta'\|_\infty C \sqrt{5}\sqrt{\eps+\mu \R^d}
\\
&+ \frac{2\sqrt{4+\pi^2}}{k}|\varsigma'(\mu \R^d/k)|\sqrt{\mu \R^d} \|\zeta\|_\infty\comma
\end{aligned}
 \mu \in \mcM(\R^d) \fstop
 \end{align*}
Clearly $v$ and $G$ are Borel functions with $v \in L^2(\mcM(\R^d), \mcQ)$ (since $\|v\|_\infty \le \|\zeta\|_\infty$) and $G$ non-negative. By claim (1) and claim (5) we have that 
\[ \lim_k v_k(\mu) = v(\mu), \quad \limsup_{k\to\infty}|\rmD v_k|_{\star,2, \AA}(\mu)\le G(\mu) \quad\text{$\mcQ$-a.e.~in $\mcM(\R^d)$}.\]
We conclude by Lemma \ref{le:trunc2} that $v \in H^{1,2}(\X(\mcQ); \AA)$ and that $|\rmD v|_{\star, 2,\AA} \le G$ $\mcQ$-a.e.~in $\mcM(\R^d)$. This concludes the proof of the proposition.
\end{proof}

Arguing precisely as in \cite[Corollary 4.18]{FSS22}, we deduce the \eqref{eq:aimsec}.

\begin{corollary}\label{cor:ineq} Let $\nu \in \mcM^{ac}(\R^d)$ and let $R, \delta>0$ be such that $\supp(\nu)=\overline{\rmB(0,R)}$, and $\nu \ge \delta\Leb{d} \mres \rmB(0,R)$. Then 
\[ \left |\rmD \GHK(\nu, \cdot) \right |_{2,\star, \AA} \le 1 \quad \text{$\mcQ$-a.e.~in $\mcM(\R^d)$}.\]
\end{corollary}

As a further corollary we obtain our main result concerning the space $\X(\mcQ)$.

\begin{theorem}\label{teo:density} For every finite, non-negative Borel measure $\mcQ$ on $(\mcM(\R^d), \HK)$ the unital and point-separating subalgebra $\AA= \FC{1,1}{u,b}{1}{b}(\mcM(\R^d)) \subset \Lip_b(\mcM(\R^d), \HK)$ is dense in $2$-energy in $D^{1,2}(\mcM(\R^d), \HK, \mcQ)$. In particular $H^{1,2}(\mcM(\R^d), \HK, \mcQ)$ is a Hilbert space and $(\mcM(\R^d), \HK)$ is universally infinitesimal Hilbertian.
\end{theorem}
\begin{proof} The first statement follows by Corollary \ref{cor:ineq} and Theorem \ref{theo:startingpoint}, also recalling that $\HK$ is the length distance induced by $\GHK$.
The Hilbertianity of the metric measure space $H^{1,2}(\mcM(\R^d), \HK, \mcQ)$ is then an immediate consequence of Proposition \ref{prop:equalityriem} and \cite[Theorem 2.17]{FSS22}. The universal infinitesimal Hilbertianity of $(\mcM(\R^d), \HK)$ then follows by Proposition~\ref{prop:extt}.
\end{proof}

\subsection{Extensions of the Hilbertianity and density results}\label{s:Extensions}

\subsubsection{Cylinder functions in the Riemannian setting} 

In this section we extend the density result from measures on $\R^d$ to measures on a Riemannian manifold. We proceed as in \cite[Section 6]{FSS22} and we recall the following result \cite[Theorem 2.24]{FSS22} after introducing some notation. 

If $X$ is a set and $Y \subset X$, $\jmath\colon Y \to X$ is the inclusion map and $g\colon X\to \R$, then we define $\jmath^* g\colon Y \to \R$ as $\jmath^* g\eqdef  g \circ \jmath$.

\begin{theorem}\label{thm:224} Let $(X, \sfd)$ be a complete and separable metric space and let $Y \subset X$ be a closed subset endowed with a metric $\delta$ such that  $(Y,\dY)$ is complete and separable and 
    \begin{equation}
    \label{eq:39}
    \sfd_1(y_1,y_2)\leq \dY(y_1,y_2)\leq
   \hat \sfd_{Y}(y_1,y_2)\comma \qquad y_1,y_2\in Y \fstop
  \end{equation}
Let $\mssm$ be a non-negative, finite, Borel measure on $(Y, \delta)$ and let $\jmath\colon Y \to X$ be the inclusion map. Consider the Polish metric-measure spaces $\X\eqdef  (X, \sfd, \jmath_\sharp \mssm)$ and $\Y\eqdef (Y, \delta, \mssm)$.
If $\AA \subset \Lip_b(X, \sfd)$ is a unital and point separating subalgebra dense in $q$-energy in $D^{1,q}(\X)$, $q \in (1,+\infty)$, then $\jmath^* (\AA) \subset \Lip_b(Y, \delta)$ and 
\[
\CE_{\Y, q} \circ \jmath^* = \CE_{\Y, q, \jmath^*(\AA)} \circ \jmath^* = \CE_{\X, q}  = \CE_{\X, q, \AA} \quad \text{ on } L^0(X, \jmath_\sharp \mssm)\fstop
\]
\end{theorem}

\begin{theorem}\label{teo:csubs} Let $C \subset \R^d$ be a closed set and let $\sigma\colon C \times C \to [0,+\infty)$ be a metric on $C$ such that $(C,\sigma)$ is complete and separable and satisfying
\[ 
\sfd_e(y_1,y_2)\le \sigma(y_1,y_2)\le
\hat {(\sfd_e)}_{Y}(y_1,y_2)\quad\text{for every }y_1,y_2\in C\comma
\]
where $\sfd_e$ is the distance induced by the Euclidean norm on $\R^d$. Let $\mcQ$ be a non-negative, finite Borel measure on $(\mcM(C), \HK_\sigma)$ and let $\jjmath\eqdef \jmath_\sharp$ where $\jmath\colon C \to \R^d$ is the inclusion. Set $\Y\eqdef (\mcM(C), \HK_\sigma, \mcQ)$ and $\X\eqdef (\mcM(\R^d), \HK_{\sfd_e}, \jjmath_\sharp \mcQ)$. Then $H^{1,2}(\X)$ is linearly isomorphic to $H^{1,2}(\Y)$ and
\begin{equation}\label{eq:140}
 |\rmD (\jjmath^* u)|_{\star, \Y, 2} = \jjmath^* (|\rmD u|_{\star, \X, 2}) \quad \text{ for every } u \in H^{1,2}(\X).
\end{equation}
In particular $H^{1,2}(\Y)$ is a Hilbert space and the unital and point separating subalgebra $\jjmath^*(\FC{1,1}{u,b}{1}{b}(\mcM(\R^d)))$ is dense in $2$-energy in $D^{1,2}(\Y)$.
\end{theorem}
\begin{proof} Since the topologies induced by $\sigma$ and $\sfd_e$ on $C$ are both Polish and comparable, $\mathcal{B}(C, \sfd_e)$ and $\mathcal{B}(C, \sigma)$ coincide so that we can simply write $\mcM(C)$ without specifying the distance inducing the topology on $C$. Clearly every measure in $\mcM(\R^d)$ with support contained in $C$ can be identified with an element of $\mcM(C)$ and every measure in $\mcM(C)$ can be viewed as a measure in $\mcM(\R^d)$ with support contained in $C$. Therefore we can identify $\mcM(C)$ as a subset of $\mcM(\R^d)$ and the inclusion map is provided by $\jjmath$. 

\medskip

We want to apply Theorems \ref{thm:224} and \ref{teo:density} with~$X\eqdef \mcM(\R^d)$, $\sfd\eqdef \HK_{\sfd_e}$, \ $Y\eqdef \mcM(C)$, and $\delta\eqdef \HK_\sigma$. Since $C$ is closed, also $\mcM(C)$ is closed in $(\mcM(\R^d), \HK_{\sfd_e})$.
We only have to check the inequalities \eqref{eq:39}: the first one
\[
\HK_{\sfd_e} \wedge 1 \le  \HK_\sigma \quad \text{ on } \mcM(C) \times \mcM(C)\comma
\]
is immediately satisfied since $(\rmd_e)_{\pi, \f{C}} \le \sigma_{ \pi, \f{C}}$ on $\f{C}(C)$, as a consequence of the definition of the distance on the cone (cf.~\eqref{ss22:eq:distcone}) and since the same inequality being satisfied by $\sfd_e$ and $\sigma$.
This gives the stronger inequality $\HK_{\sfd_e} \le  \HK_\sigma$.
To prove the second inequality in \eqref{eq:39}, we consider two measures $\mu_0, \mu_1 \in \mcM(C)$ and a $\HK_{\sfd_e}$-Lipschitz curve $\mu:[0,\ell]\to \mcM(C)$ such that $\mu(0)=\mu_0$ and $\mu(\ell)=\mu_1$ parametrized by $\HK_{\sfd_e}$-arc-length.
By \cite[Theorem 8.4]{LMS18}) there exists a measure  $\eeta\in \mcP(\rmC([0,\ell];\f{C}[\R^d]))$ concentrated on $(\sfd_e)_{\pi, \f{C}}$-absolutely continuous curves such that $(\f{h} \circ \mathsf e_t)_\sharp (\eeta)= \mu_t  $ for every $t\in [0,\ell]$ and
\[
1= |\dot{\mu}|^2_{\HK_{\sfd_e}} (t) = \int |\dot{\f{w}}|^2_{(\sfd_e)_{\pi, \f{C}}} (t) \de \eeta(\f{w}) \quad \text{ for a.e.~$t \in [0,\ell]$}\fstop
\]
Let us now consider the continuous function $\zeta:\f{C}[\R^d] \to [0,+\infty)$ defined as
\[
\zeta([x,r])= r^2 \mathrm{dist}(x,C)\comma \qquad [x,r] \in \f{C}[\R^d]\comma
\]
and note that $\zeta$ vanishes precisely on $\f{C}[C]$ considered as a subset of $\f{C}[\R^d]$. Fubini's Theorem yields
    \begin{align*}
      \int\int_0^\ell \zeta(\f{w}(t))\,\d
      t\,\d\eeta(\f{w})
    &=
     \int_0^\ell \int \zeta(\sfe_t(\f{w}))\,
      \d\eeta(\f{w})\,\d t
      =\int_0^\ell \int_{ \R^d}\mathrm{dist}(x,C) \,\d\mu_t(x)\,\d t=0
  \end{align*}
  since $\int \mathrm{dist}(x,C)\,\d\mu_t(x)=0$ for every $t\in
  [0,\ell]$. It follows that $\int_0^\ell \zeta(\f{w}(t))\,\d t=0$ for $\eeta$-a.e.~$\f{w}$, so
  that the set of $t\in [0,\ell]$ for which
  $\f{w}(t)\in \f{C}[C]$ is dense in $[0,\ell]$. Being $C$ closed,
  we conclude that $\f{w}$ takes values in $\f{C}[C]$ for
  $\eeta$-a.e.~$\f{w}$. 
Observe also that every $\f{w} \in \mathrm{AC}([0,\ell]; (\f{C}[C], (\sfd_e)_{\pi, \f{C}}))$ belongs to $\mathrm{AC}([0,\ell]; (\f{C}[C], \sigma_{\pi, \f{C}}))$ and $|\dot{\f{w}}|_{(\sfd_e)_{\pi, \f{C}}} (t) = |\dot{\f{w}}|_{\sigma_{\pi, \f{C}}} (t)$ for a.e.~$t \in [0,\ell]$: this is a consequence of the representation in \cite[Lemma 8.1]{LMS18}.  We can now estimate the $\HK_{\sigma}$ distance between  the two measures $\mu_{0}$ and $\mu_{1}$, viz.
  \begin{align*}
    \HK_\sigma(\mu_{0},\mu_{1})^2  &\le \int_{\f{C}[C] \times \f{C}[C]} \sigma_{\pi, \f{C}}^2 \de ((\mathrm e_0, \mathrm e_1)_\sharp \eeta) =  \int \sigma_{\pi, \f{C}}(\f{w}(0), \f{w}(1))^2 \de \eeta(\f{w}) \\
    & \le \ell \int \int_0^\ell |\dot{\f{w}}|^2_{\sigma_{\pi, \f{C}}} (t) \de t \de \eeta(\f{w}) =  \ell \int \int_0^\ell |\dot{\f{w}}|^2_{(\sfd_e)_{\pi, \f{C}}} (t) \de t \de \eeta(\f{w}) \\
    & = \ell \int_0^\ell \int |\dot{\f{w}}|^2_{(\sfd_e)_{\pi, \f{C}}} (t) \de \eta(\f{w}) \de t = \ell \int_0^\ell |\dot{\mu}|^2_{\HK_{\sfd_e}} (t) \de t \\
    &= \ell^2
    \end{align*}
  where we have used  that $(\mathsf e_{0},\mathsf e_{1})_\sharp \eeta \in \f{H}(\mu_{0}, \mu_{1})$.
  Taking the infimum w.r.t.~$\ell$, we deduce that
  \[
  \HK_\sigma(\mu_0, \mu_1 ) \le \hat {(\HK_{\sfd_e})}_{\mcM(C)}(\mu_0, \mu_1)
  \]
  which is the second inequality in \eqref{eq:39}.
  \end{proof}
  
In the rest of this section we assume that $(M,g)$ is a smooth, connected, complete Riemannian manifold with Riemannian distance~$\mssd_g$.  
  
\begin{theorem}\label{thm:hilbm}
Let $\mcQ$ be a finite, non negative Borel measure on $\mcM(M)$.
Then, the unital and point-separating subalgebra $\FC{1,1}{u,b}{1}{b}$ is dense in $2$-energy in \linebreak $D^{1,2}(\mcM(M), \HK_{\mssd_g}, \mcQ)$ and $H^{1,2}(\mcM(M), \HK_{\mssd_g}, \mcQ)$ is a Hilbert space.
\end{theorem}
\begin{proof}
  By Nash Isometric Embedding Theorem \cite{Nash54} we can find a dimension $d$, and a smooth isometric embedding
  $\iota:M\to
  \iota(M)\subset \R^d$.
  On $C\eqdef \iota(M)$ we can define
  the (Riemannian) metric $\sfd_C$ inherited by $\mssd_g$:
  $\sfd_C(\iota(x),\iota(y))=\mssd_g(x,y)$ so that $\iota$ is an
  isometry and
  $(C, \sfd_C)$ is a complete and separable metric space.
  We denote by $\iiota\eqdef \iota_\sharp$ the corresponding isometry between
  $(\mcM(M), \HK_{\mssd_g})$ and $(\mcM(C), \HK_{\sfd_C})$  and 
  we also set $\tilde\mcQ\eqdef \iiota_\sharp \mcQ $  which is a non-negative and finite Borel measure on $\mcM(C)$. It is clear that the
  map $\iiota^*:u\mapsto u\circ\iiota$ induces a linear isometric isomorphism
  between 
  $H^{1,2}(\mcM(C), \sfd_C,\tilde\mcQ)$ and
  $H^{1,2}(\mcM(M), \HK_{\mssd_g} ,\mcQ)$. 
  
  Since $M$ is complete and $\iota$ is an embedding, $C$ is a closed
  subset of $\R^d$ and $\sfd_C$ induces on $C$ the relative topology
  of $\R^d$. Since $\iota$ is isometric, we also have
  \begin{equation}
    \label{eq:89}
    \sfd_e(y_1,y_2) \le \sfd_C(y_1,y_2)=(\hat\sfd_e)_{C}(y_1,y_2) \quad \text{ for every } y_1,y_2 \in C \fstop
  \end{equation}

We can introduce the inclusion map $\jmath\colon C\to \R^d$ and the
 corresponding
 $\jjmath=\jmath_\sharp\colon \mcM(C) \to \mcM(\R^d)$.
 By Theorem \ref{thm:224}
 we have that
 the map $\jjmath^*\colon u\mapsto u\circ\jjmath$ provides a linear isometric
 isomorphism
 between $H^{1,2}(\mcM(\R^d), \HK_{\sfd_e}, \jjmath_\sharp \tilde\mcQ)$ and
 $H^{1,2}(\mcM(C), \HK_{\sfd_C},\tilde\mcQ)$ satisfying \eqref{eq:140};
 we conclude that the map
 $\kkappa^*\eqdef \iiota^*\circ \jjmath^*=(\jjmath\circ \iiota)^*$
 is a linear isometric
 isomorphism
 between $H^{1,2}(\mcM(\R^d), \HK_{\sfd_e},\kkappa_\sharp \mcQ)$
 (note that $\kkappa_\sharp=\jjmath_\sharp\circ \iiota_\sharp$) and
 $H^{1,2}(\mcM(M), \HK_{\mssd_g},\mcQ)$
 satisfying
 \begin{equation}
  | \rmD (\kkappa^* u)|_{\star, \Y, 2}  = \kkappa^* \left ( |\rmD u|_{\star, \X, 2} \right ) \text{
     for every } u \in H^{1,2}(\mcM(\R^d), \HK_{\sfd_e},\kkappa_\sharp \mcQ),\label{eq:140bis}
\end{equation}
where $\Y\eqdef  (\mcM(M), \HK_{\mssd_g},\mcQ)$ and $\X\eqdef  (\mcM(\R^d), \HK_{\sfd_e},\kkappa_\sharp \mcQ)$.
This property in particular yields the  Hilbertianity of $H^{1,2}(\mcM(M), \HK_{\mssd_g},\mcQ)$. To prove the density of the subalgebra  $\FC{1,1}{u,b}{1}{b}$ it is enough to prove that $\AA' \eqdef  \kkappa^* (\FC{1,1}{u,b}{1}{b}(\mcM(\R^d))) \subset \FC{1,1}{u,b}{1}{b}$ is dense in $2$-energy in $D^{1,2}(\mcM(M), \HK_{\mssd_g},\mcQ)$. First of all observe that for every $\tilde{u} \in \FC{1,1}{u,b}{1}{b}(\mcM(\R^d))$ and every $\mu \in \mcM(M)$ it holds
\[ \lip_{\HK_{\mssd_g}} (\kkappa^* \tilde{u}) (\mu) \le \kkappa^* ( \lip_{\HK_{\sfd_e}} \tilde{u} )(\mu).\]
This follows by the fact that $\kkappa$ is a $\HK_{\mssd_g}$-$\HK_{\sfd_e}$ contraction, since $\sfd_e(\kappa(y_0), \kappa(y_1)) \le \mssd_g(y_0, y_1)$ for every $y_0,y_1 \in M$. Let now $u = \kkappa^* \tilde{u} $ for $u \in H^{1,2}(\mcM(\R^d), \HK_{\sfd_e},\kkappa_\sharp \mcQ)$. By Theorem \ref{teo:csubs}, we can find a sequence $(\tilde{u}_n)_n \subset \FC{1,1}{u,b}{1}{b}(\mcM(\R^d))$ such that 
\[ \tilde{u}_n \to \tilde{u}, \quad \lip_{\HK_{\sfd_e}} \tilde{u}_n \to |\rmD \tilde u |_{\star, \X, 2} \quad \text{ in } L^2(\mcM(\R^d), \kkappa_\sharp \mcQ).\]
We deduce that, setting $u_n\eqdef  \kkappa^* \tilde u_n \in \AA'$, we have 
\begin{equation*}
\begin{gathered}
u_n \to u\comma
\\
\lip_{\HK_{\mssd_g}} u_n\le \kkappa^* (\lip_{\HK_{\sfd_e}} \tilde{u}_n ) \to \kkappa^* ( |\rmD \tilde u |_{\star, \X, 2}) = |\rmD u|_{\star, \Y, 2}
\end{gathered}
\qquad \text{ in } L^2(\mcM(M), \mcQ) \fstop
\end{equation*}

Up to extracting a suitable (not relabelled) subsequence we can suppose that $\lip_{\HK_{\mssd_g}} u_n$
converges weakly in $L^2(\mcM(M), \mcQ)$ to some $G \in
L^2(\mcM(M), \mcQ)$ which is a $\Y$-relaxed gradient of $u$. We also see that 
\begin{align*}
  \int G^2\,\d\mcQ\leq&
  \limsup_{n\to\infty} \int (\lip_{\HK_{\mssd_g}} u_n)^2 \, \d \mcQ
  \le \limsup_{n\to\infty} \int \Big(\kkappa^* (\lip_{\HK_{\sfd_e}} \tilde{u}_n ))\Big)^2 \,
  \d \mcQ
  \\
  =&
  \int |\rmD u|_{\star, \Y, 2}^2 \, \d \mcQ,
\end{align*}
showing that $G=|\rmD u|_{\star, \Y, 2}$ and
$\lip_{\HK_{\mssd_g}} u_n\to |\rmD u|_{\star, \Y, 2}$ strongly in $L^2(\mcM(M), \mcQ)$.
\end{proof}

We prove now the density result for smaller but dense algebras of functions, following \cite[Proposition 4.19]{FSS22}.

We say that a function $F: \R^\infty \to \R$ is \emph{finitary} if $F(t_1, t_2, \dotsc) = f(t_1, \dots, t_k)$ for some $f \colon \R^k \to \R$ and $k \in \N_0$. For an algebra $\AA$ of $\R$-valued functions on $\R^\infty$ we denote by $\AA_{\fin}$ the subset of finitary elements in $\AA$.

\begin{proposition}\label{prop:dolore} Let $\mcQ$ be a finite, non negative Borel measure on $\mcM(M)$. Let
\begin{enumerate}[$(a)$, leftmargin=2em]
\item $\AA_1 \subset \rmC^1(\R_0^+)$ be a subalgebra locally uniformly $\rmC^1(\R_0^+)$-dense in $\rmC_c^1(\R^+_0)$;
\item $\AA_2 \subset \rmC^1(\R^\infty)_{\fin}$ be a subalgebra locally uniformly $\rmC^1(\R^\infty)$-dense in $\rmC_b^1(\R^\infty)_{\fin}$;
\item $\AA_3 \subset \rmC^1(M)$ be a subalgebra such that for every $f \in \rmC^1_b(\R^d)$ there exists a sequence $\seq{f_n}_n \subset \AA_3$ satisfying
\begin{equation}\label{eq:densc1mu}
\begin{aligned}
\sup_{n} \sup_{x \in M}\braket{ |f_{n}(x)| + |\nabla f_{n}(x)| }&< + \infty \comma\\
 \qquad \lim_{n \to + \infty} \int_{M} (|f-f_{n}|^2 + |\nabla f - \nabla f_{n}|^2) \de \mu &=0 
\end{aligned}
\quad \text{ for every } \mu \in \mcM(M) \fstop
\end{equation}
\end{enumerate}
Then the algebra 
\begin{equation}
\AA \eqdef \set{u\colon \mcM(M)\to\R : \begin{gathered} u=(\nchi\circ \car^\trid) \cdot (F\circ \mbff^\trid)\comma
\\
\nchi\in \AA_1\comma
\\
F\in \AA_2 \comma \mbff \in \AA_3^\infty
\end{gathered}}\comma
\end{equation}
is dense in $2$-energy in $D^{1,2}(\mcM(M), \HK_{\sfd_g}, \mcQ)$.
\end{proposition}

\begin{proof}
Thanks to Theorem \ref{teo:density}, it is enough to show that for every $u \in \FC{1,1}{u,b}{1}{b}$ there exists a sequence $(u_n)_n \subset \AA$ such that
\begin{equation}\label{eq:theaimcon}
u_n \to u  \quad \text{ and } \quad \| (\boldnabla u_n)_{\cdot}-(\boldnabla u)_{\cdot}\|_{ T_{\cdot} } \to 0 \quad \text{ in } L^2(\mcM(M), \mcQ) \text{ as } n \to + \infty.
\end{equation} 
If $u \equiv 1$, since $\AA_1$ is locally uniformly $\rmC^1(\R_0^+)$-dense in $\rmC_c^1(\R^+_0)$, we can find a sequence $(\nchi_k)_k \subset \AA_1$ approximating the functions $t \mapsto \varsigma(t/k)$ with $\varsigma$ as in Lemma~\ref{le:trunc1}. The sequence $u_k \eqdef \nchi_k \circ \car^\trid$ converges in $D^{1,2}(\mcM(M), \HK_{\sfd_g}, \mcQ)$ to $u$. Let us now consider a general $u=(\nchi \circ \car^\trid)\cdot (F\circ \mbff^\trid) \in \FC{1,1}{c,b}{1}{b}$, where $\nchi \in \rmC_c^1(\R)$, $\boldsymbol f=(f_1,\dotsc, f_N)$ is a vector of functions in $\rmC^1_b(M)$ and $F\in \rmC^1_b(\R^N)$, $N \in \N$. Since $\nchi$ is compactly supported, there exists some $r>1$ such that $\supp(\nchi) \subset [-r, r]$. Since in \eqref{eq:theaimcon} the convergence is in $L^2(\mcM(M), \mcQ)$, we can proceed by steps approximating $\nchi$, $F$ and~$\mbff$ and then use a diagonal argument to conclude.

\paragraph{Approximation of $\nchi$}
Let $(\nchi_n)_n \subset \AA_1$ be such that $\sup_{|t| \le r} |\nchi_n(t)-\nchi(t)|+ |\nchi'_n(t)-\nchi'(t)| \le 1/n$ for every $n \in \N$ and let us set $u_n\eqdef  (\nchi_n \circ \car^\trid)\cdot (F \circ \mbff^\trid)$. We clearly have
\begin{align*}
\|u_n-u\|^2_{L^2(\mcM(M), \mcQ)} =&  \int_{\mcB_r} |(\nchi_n(\mu M)-\nchi_n(\mu M)|^2 | (F \circ \mbff^\trid)(\mu)|^2 \de \mcQ(\mu)
\\
\le&\ \frac{\mcQ(\mcB_r )}{n^2} \|F\|_\infty^2
\end{align*}
so that clearly $u_n \to u$ in $L^2(\mcM(M), \mcQ)$ as $n \to + \infty$. Moreover by Proposition \ref{prop:equalityriem} we have
\begin{align*}
(\boldnabla u_n)_\mu(x) &=(F \circ \mbff^\trid)(\mu) \nchi'_n(\mu M) (0,1)+ \nchi_n(\mu M) \boldnabla (F \circ \mbff^\trid)_\mu (x)\comma
\\
(\boldnabla u)_\mu(x) &=(F \circ \mbff^\trid)(\mu) \nchi'(\mu M) (0,1)+ \nchi(\mu M) \boldnabla (F \circ \mbff^\trid)_\mu (x)\comma
\end{align*}
so that 
\begin{align*}
\int& \| (\boldnabla u_n)_\mu-(\boldnabla u)_\mu\|^2_{T_\mu^{1,4}} \de \mcQ(\mu)
\\
\leq &\ 8\int_{\mcB_r} \int_{M}|(F \circ \mbff^\trid)(\mu)|^2 |\nchi_n'(\mu M) - \nchi'(\mu M)|^2 \de \mu(x)\de  \mcQ(\mu)
\\
& \quad + 2\int_{\mcB_r} \int_{M} \abs{\boldnabla (F \circ \mbff^\trid)_\mu(x)}_\toplus^2 \abs{\nchi_n(\mu M)- \nchi(\mu M)}^2  \de \mu (x)\de \mcQ(\mu)
\\
\leq &\ \frac{8r\mcQ(\mcB_r)}{n^2} \|F\|_\infty^2 + \frac{2r\mcQ(\mcB_r)}{n^2} \sup_{(\mu, x)} \abs{\boldnabla (F \circ \mbff^\trid)_\mu (x)}_\toplus^2,
\end{align*}
giving that \eqref{eq:theaimcon} holds for $u_n$. We can thus assume that $\nchi \in \AA_1$.

\paragraph{Approximation of $F$}
Setting $R\eqdef r \sup_{M, \, 1 \le k \le N}\braket{|f_{k}|+|\nabla f_{k}|}$, we can find a sequence $(F_n)_n \subset \AA_2$ such that
  \begin{equation}
    \label{eq:86}
    \sup_{|z|\le R}|F_n(z)-F(z)|+|\nabla
    F_n(z)-\nabla F(z)| \le \frac{1}{n} \text{ for every } n\in \N.
  \end{equation}
  Let us set $u_n\eqdef  (\nchi \circ \car^\trid)\cdot (F_n\circ \mbff^\trid) $. We have
\begin{align*}
\|u_n-u\|^2_{L^2(\mcM(M), \mcQ)}  =& \int_{\mcB_r} |\nchi(\mu M)|^2 |F_n(\mbff^\trid \mu) - F(\mbff^\trid \mu)|^2 \de \mcQ(\mu)
\\
\leq&\ \frac{\mcQ(\mcB_r)}{n^2}\|\nchi\|_\infty^2
\end{align*}
so that clearly $u_n \to u$ in $L^2(\mcM(M), \mcQ)$ as $n \to + \infty$. By Proposition \ref{prop:equalityriem} we also have
\begin{align*}
(\boldnabla u_n)_\mu(x) &=(F \circ \mbff^\trid)(\mu) \nchi'(\mu M) (0,1)+ \nchi(\mu M) \boldnabla (F_n \circ \mbff^\trid)_\mu(x),\\
(\boldnabla u)_\mu(x) &=(F \circ \mbff^\trid)(\mu) \nchi'(\mu M) (0,1)+ \nchi(\mu M) \boldnabla (F \circ \mbff^\trid)_\mu(x),
\end{align*}
so that 
\begin{align*}
\int &\| (\boldnabla u_n)_\mu-(\boldnabla u)_\mu\|^2_{T_\mu^{1,4}} \de \mcQ(\mu)
\\
&\le \int_{\mcB_r} \int_{M} |\nchi(\mu M)|^2 |\boldnabla (F_n \circ \mbff^\trid)_\mu(x) - \boldnabla (F \circ \mbff^\trid)_\mu(x)|_\toplus^2 \de \mu(x) \de \mcQ(\mu)
\\
& \le  \|\nchi\|_\infty^2 \int_{\mcB_r} \int_{M} \Bigg | \sum_{k=1}^N \tbraket{\partial_k F(\mbff^\trid \mu)- \partial_k F_n (\mbff^\trid(\mu))}
\\
& \hspace{35mm} \cdot (\nabla f_k(x), f_k(x)) \Bigg |_\toplus^2 \de \mu(x) \de \mcQ(\mu)
\\
& \le  \|\nchi\|_\infty^2 \int_{\mcB_r} \int_{M}  \sum_{k=1}^N \abs{ \partial_k F(\mbff^\trid \mu) - \partial_k F_n (\mbff^\trid(\mu))}^2
\\
& \hspace{35mm} \cdot \abs{(\nabla f_k(x), f_k(x))}_\toplus^2 \de \mu(x) \de \mcQ(\mu)
\\
& \le \|\nchi\|_\infty^2 \frac{5R^2 r \mcQ(\mcB_r)}{n^2} 
\end{align*}
giving that \eqref{eq:theaimcon} holds for $u_n$. We can thus assume both that $\nchi \in \AA_1$ and $F \in \AA_2$.

\paragraph{Approximation of $\mbff$}
Let us consider bounded sequences $(f_{k,n})_{n} \subset \AA_3$,
  $k=1,\dotsc, N$, approximating $f_k$ in the sense that
  \[ R'\eqdef  \sup_{k,n} \sup_{x \in M}( |f_{n,k}(x)| + |\nabla f_{n,k}(x)|)< + \infty\]
  and
   \[ \lim_{n \to + \infty} \sup_k \int_{M} (|f_k-f_{k,n}|^2 + |\nabla f_k - \nabla f_{k,n}|^2) \de \mu =0 \quad \text{ for every } \mu \in \mcM(M).\]
  Let us set $\mbff_n \eqdef (f_{1,n}, f_{2,n}, \dots, f_{N,n})$ and $u_n\eqdef  (\nchi \circ \car^\trid)\cdot (F \circ\mbff_n^\trid) \in \AA$. Let us denote by $R''$ the quantity 
  \[ R''\eqdef  R' \vee \sup_{k} \sup_{M}( |f_{k}(x)| + |\nabla f_{k}(x)|)< + \infty\]
and by $L$ the maximum of the Lipschitz constants of
  $F$ and $\partial_kF$ in the cube $[-rR'',rR'']^N$ with respect to the $\infty$-norm in $\R^N$. We see that
  \begin{align*}
\|u_n-u\|^2_{L^2(\mcM(M), \mcQ)} &=  \int_{\mcB_r} |\nchi(\mu M)|^2 | F(\mbff_n^\trid(\mu)) -  F(\mbff^\trid(\mu))|^2 \de \mcQ(\mu)
\\
& \le L^2\|\nchi\|_\infty^2  \int_{\mcB_r} \sup_k |f_{n,k}^\trid(\mu) - f_k^\trid(\mu)|^2 \de \mcQ(\mu) 
\\
& \le r L^2\|\nchi\|_\infty^2  \int_{\mcB_r} \sup_k \int_{M} |f_{n,k} - f_k|^2 \de \mu \de \mcQ(\mu) \fstop
\end{align*}
By Dominated Convergence we deduce that $u_n \to u$ in $L^2(\mcM(M), \mcQ)$ as $n \to + \infty$.
By Proposition \ref{prop:equalityriem},
\begin{align*}
(\boldnabla u_n)_\mu(x) =&\ (F \circ\mbff_n^\trid)(\mu) \nchi'(\mu M) (0,1)
\\
&+ \nchi(\mu M) \sum_{k=1}^N \partial_k F (\mbff_n^\trid \mu) (\nabla f_{k,n}(x), f_{k,n}(x))\comma
\\
(\boldnabla u)_\mu(x) =&\ (F \circ \mbff^\trid)(\mu) \nchi'(\mu M) (0,1)
\\
&+ \nchi(\mu M) \sum_{k=1}^N \partial_k F (\mbff^\trid \mu) (\nabla f_{k}(x), f_{k}(x))
\end{align*}
so that
\begin{align*}
\int &\| (\boldnabla u_n)_\mu-(\boldnabla u)_\mu\|^2_{T_\mu^{1,4}} \de \mcQ(\mu)
\\
\leq&\ 8 \int_{\mcB_r} \int_{M} |(F \circ\mbff_n^\trid)(\mu)-(F \circ \mbff^\trid)(\mu)|^2|\nchi'(\mu M)|^2 \de \mu(x)\de  \mcQ(\mu)
\\
& + 2 \int_{\mcB_r} \int_{M} \Bigg | \nchi(\mu) \sum_{k=1}^N \bigg(\partial_k F (\mbff_n^\trid \mu) (\nabla f_{k,n}(x), f_{k,n}(x)) 
\\
& \hspace{40mm} - \partial_k F (\mbff^\trid \mu) (\nabla f_{k}(x), f_{k}(x)) \bigg ) \Bigg |^2_\toplus \de \mu(x) \de \mcQ(\mu)
\\
& \leq  8r^2 L^2\|\nchi'\|_\infty^2  \int_{\mcB_r} \sup_k \int_{M} |f_{n,k} - f_k|^2 \de \mu \de \mcQ(\mu)
\\
& + 2 \|\nchi\|_\infty^2 \int_{\mcB_r} \int_{M}  \sum_{k=1}^N \Bigg |\partial_k F (\mbff_n^\trid \mu) (\nabla f_{k,n}(x), f_{k,n}(x))
\\
&\hspace{40mm} - \partial_k F (\mbff^\trid \mu) (\nabla f_{k}(x), f_{k}(x)) \Bigg|_\toplus^2  \de \mu(x) \de \mcQ(\mu) \fstop
\end{align*}
The first summand in the last inequality converges to $0$ as $n \to + \infty$ by Dominated Convergence. We then concentrate on the integral in the second summand:
\begin{align*}
\int_{\mcB_r}& \int_{M}  \sum_{k=1}^N \bigg|\partial_k F (\mbff_n^\trid \mu) (\nabla f_{k,n}(x), f_{k,n}(x)) 
\\
&\hspace{35mm} - \partial_k F (\mbff^\trid \mu) (\nabla f_{k}(x), f_{k}(x)) \bigg |_\toplus^2  \de \mu(x) \de \mcQ(\mu)
\\
\leq&\ 2 \int_{\mcB_r} \int_{M}  \sum_{k=1}^N \bigg|\partial_k F (\mbff_n^\trid \mu) (\nabla f_{k,n}(x), f_{k,n}(x)) 
\\
&\hspace{35mm} - \partial_k F (\mbff_n^\trid \mu) (\nabla f_{k}(x), f_{k}(x)) \bigg|_\toplus^2  \de \mu(x) \de \mcQ(\mu)
\\
&+ 2 \int_{\mcB_r} \int_{M}  \sum_{k=1}^N \bigg | \partial_k F (\mbff_n^\trid \mu) (\nabla f_{k}(x), f_{k}(x))
\\
&\hspace{35mm} - \partial_k F (\mbff^\trid \mu) (\nabla f_{k}(x), f_{k}(x)) \bigg |_\toplus^2  \de \mu(x) \de \mcQ(\mu)
\\
& \le 2\|\nabla F\|_\infty^2  \sum_{k=1}^N\int_{\mcB_r}   \int_{M} (|\nabla f_k - \nabla f_{k,n}|^2 + 4|f_k-f_{k,n}|^2) \de \mu \de \mcQ(\mu)
\\
& + 10r(R'')^2 \sum_{k=1}^N \int_{\mcB_r} |\partial_k F (\mbff_n^\trid \mu) - \partial_k F (\mbff^\trid \mu)|^2 \de \mcQ(\mu)
\\
& \le 2\|\nabla F\|_\infty^2  \sum_{k=1}^N\int_{\mcB_r}   \int_{M} (|\nabla f_k - \nabla f_{k,n}|^2 + 4|f_k-f_{k,n}|^2) \de \mu \de \mcQ(\mu)
\\
& + 10r^2L^2 (R'')^2 \sum_{k=1}^N \int_{\mcB_r} \sup_k \int_{M} |f_{n,k} - f_k|^2 \de \mu \de \mcQ(\mu).
\end{align*}
Both summands converge to $0$ thanks to the Dominated Convergence Theorem. This shows that~\eqref{eq:theaimcon} holds for $u_n$ and concludes the proof.
\end{proof}

\begin{corollary}\label{c:DensityCylC}
The space $(\mcM(M), \HK_{\mssd_g})$ is universally infinitesimally Hilbertian. Let $\mcQ$ be a non-negative Borel measure on $\mcM(M)$.
If
\begin{equation*}
\int \rme^{-t \mu M} \, \diff \mcQ(\mu) < + \infty \quad \text{ for every } t>0 \comma
\end{equation*}
then, the subalgebra $\FC{\infty,\infty}{c,c}{\infty}{c}$ is strongly dense in $H^{1,2}(\mcM(M), \HK_{\mssd_g}, \mcQ)$, and for every $u \in H^{1,2}(\mcM(M), \HK_{\mssd_g}, \mcQ)$ there exists a sequence $(u_n)_n \subset \FC{\infty,\infty}{c,c}{\infty}{c}$ such that 
\[
u_n \to u, \quad \lip_{\HK_{\mssd_g}} u_n \to |\rmD u |_{\star, 2} \quad \text{ in } L^2(\mcM(M), \mcQ)\fstop
\]
\end{corollary}
\begin{proof} 
Combining Theorem \ref{thm:hilbm} with Proposition \ref{prop:extt} we immediately obtain the Hilbertianity result. The density of smooth cylinder functions for a finite measure~$\mcQ$ follows by Proposition~\ref{prop:dolore} with the choices $\AA_1= \rmC_c^\infty(\R_0^+)$, $\AA_3= \rmC_c^\infty(M)$, and $\AA_2 \subset \rmC^\infty(\R^\infty)_\fin$ is the set of functions $F\colon \R^\infty \to \R$ such that $F(t_1, t_2, \dotsc) = f(t_1, \dotsc, t_k)$ for some $f \in \rmC_c^\infty(\R^k)$ and some $k \in \N_1$.

This, together with Remark~\ref{rem:strong} and Lemma~\ref{le:densinf}, gives the desired density results for possibly infinite measures as above, also noting that both $\FC{\infty,\infty}{c,c}{\infty}{c}$ and $\lip (\FC{\infty,\infty}{c,c}{\infty}{c})$ are subsets of $L^2(\mcM(\R^d), \mcQ)$ as a consequence of the truncation induced by $\nchi \circ \car^\trid$, and that the function $\vartheta_k \eqdef \rme^{-\frac{1}{k} \car^\trid}$ satisfies the assumptions in Lemma \ref{le:densinf}. 
\end{proof}

\subsubsection{Smooth cylinder functions in the Euclidean setting}
We now show how to extend the density result for cylinder functions to the extended metric-topological measure spaces 
\[ \X_{\He}(\mcQ)\eqdef (\mcM(\R^d), \tau_w, \He, \mcQ) \quad \text{ and } \quad \X_{\W}(\mcQ)\eqdef (\mcM(\R^d), \tau_w, \W, \mcQ), \]
where $\He=\He_2$ and $\W=\W_{2,\sfd_e}$ are the Hellinger and extended Wasserstein distances introduced in Definitions \ref{def:he} and \ref{def:wass}, respectively (recall that $\sfd_e$ is the distance induced by the Euclidean norm on $\R^d$), $\tau_w = \sigma(\mcM(\R^d), \rmC_b(\R^d))$ is the usual weak topology on $\mcM(\R^d)$ and $\mcQ$ is finite, non-negative Borel measure on $(\mcM(\R^d), \tau_w)$. 
\medskip

We know by \cite[Thm.s 7.22, 7.23]{LMS18} that
\begin{equation}\label{eq:convlambda}
\HK_{\lambda \sfd_e} \uparrow \mathsf{He}, \quad \lambda \HK_{\sfd_e/\lambda} \uparrow \mathsf{W} \quad \text{ as } \lambda \to + \infty.
\end{equation}
We deduce by e.g.~\cite[Lemma 2.4]{Savare22} that both $\X_{\He}(\mcQ)$ and $\X_{\W}(\mcQ)$ are indeed e.m.t.m.\ spaces. 
\medskip

If we show that the unital algebra of cylinder functions is dense in $2$-energy in $D^{1,2}(\X_\lambda(\mcQ))$, being $\X_\lambda(\mcQ) \eqdef  (\mcM(\R^d), \HK_{\lambda \sfd_e}, \mcQ)$, for every $\lambda>0$, the density result for $D^{1,2}(\X_{\He}(\mcQ))$ and $D^{1,2}(\X_{\W}(\mcQ))$ will follow as a consequence of Lemma~\ref{le:easy} and Proposition \ref{prop:erbar}.

We use the following notation: if $\lambda>0$ we denote by $\mathsf{T}^\lambda\colon \mcM(\R^d) \to \mcM(\R^d)$ the map
\[
\mathsf{T}^\lambda(\mu) \eqdef  t^\lambda_\sharp \mu\comma \qquad \mu \in \mcM(\R^d)\comma
\]
where we denote by~$t^\lambda: \R^d \to \R^d$ the dilation $x \mapsto \lambda x$.

\begin{lemma}\label{le:scaling} Let $\lambda >0$ be fixed. Then
\[ \HK_{\lambda \sfd_e}(\mu, \nu) = \HK_{ \sfd_e}(\mathsf{T}^\lambda (\mu), (\mathsf{T}^\lambda (\nu)) \quad \text{ for every } \mu, \nu \in \mcM(\R^d).\]
In particular $\mathsf{T}^\lambda$ is an isometric isomorphism from $(\mcM(\R^d), \HK_{\lambda \sfd_e})$ to $(\mcM(\R^d), \HK_{ \sfd_e})$.
\end{lemma}
\begin{proof} This immediately follows since the map $u^\lambda\colon ([x,r],[y,s]) \mapsto  ([\lambda x,r],[\lambda y,s])$ induces a bijection $u^\lambda_\sharp$ from $\f{H}(\mu, \nu)$ to $\f{H}(\mathsf{T}^\lambda (\mu), (\mathsf{T}^\lambda (\nu))$. (Recall the notation in~\eqref{eq:defhk}.)
\end{proof}

\begin{corollary}\label{cor:densl} Let $\lambda >0$; then the unital and point-separating subalgebra $\AA= \FC{\infty,\infty}{u,c}{\infty}{c}(\mcM(\R^d))$ is dense in $2$-energy in $D^{1,2}(\X_\lambda(\mcQ))$. 
\end{corollary}
\begin{proof} 
Let $\X'_\lambda(\mcQ)\eqdef (\mcM(\R^d), \HK_{\sfd_e}, \mathsf{T}^\lambda_\sharp \mcQ)$ and note that the map $(\mathsf{T}^\lambda)^*$, defined as $(\mathsf{T}^\lambda)^* (u)\eqdef  u \circ \mathsf{T}^\lambda$, is a bijection between $\Lip_b(\mcM(\R^d), \HK_{\sfd_e})$ and $\Lip_b(\mcM(\R^d), \HK_{\lambda \sfd_e})$ and the algebra $\AA$ is invariant under its action i.e.~$\AA = (\mathsf{T}^\lambda)^*(\AA)$. We deduce by Lemma \ref{le:scaling} and  \cite[Proposition~5.15]{Savare22} that $(\mathsf{T}^\lambda)^*$ is an isomorphism of $H^{1,2}(\X'_\lambda; \AA)$ onto $H^{1,2}(\X_\lambda; \AA)$ and an isomorphism of $H^{1,2}(\X'_\lambda)$ onto $H^{1,2}(\X_\lambda)$. 
By the density in $2$-energy of $\FC{\infty,\infty}{c,c}{\infty}{c}(\mcM(\R^d)) \subset \AA$ in $D^{1,2}(\X'_\lambda)$ provided by Corollary~\ref{c:DensityCylC} we deduce that
\[
H^{1,2}(\X'_\lambda(\mcQ); \AA) = H^{1,2}(\X'_\lambda(\mcQ))\comma
\] so that
\[
H^{1,2}(\X_\lambda(\mcQ); \AA) = H^{1,2}(\X_\lambda(\mcQ))\fstop \qedhere
\]
\end{proof}

\begin{corollary} The point-separating subalgebra $\FC{\infty,\infty}{c,c}{\infty}{c}(\mcM(\R^d))$ is dense in $2$-energy in $D^{1,2}(\X_{\He}(\mcQ))$ and $D^{1,2}(\X_{\W}(\mcQ))$, and $H^{1,2}(\X_{\He}(\mcQ))$ and $H^{1,2}(\X_{\W}(\mcQ))$ are Hilbert spaces.
\end{corollary}
\begin{proof} The density of $\AA = \FC{\infty,\infty}{u,c}{\infty}{c}(\mcM(\R^d))$ in $D^{1,2}(\X_{\He}(\mcQ))$ and $D^{1,2}(\X_{\W}(\mcQ))$ as well as the Hilbertianity result follow by Corollary \ref{cor:densl}, Lemma \ref{le:easy}, Proposition \ref{prop:erbar} and \eqref{eq:convlambda}. 
The density for the smaller (and non-unital) subalgebra $\FC{\infty,\infty}{c,c}{\infty}{c}(\mcM(\R^d))$ follows by \cite[Proposition 2.15]{FSS22}, Lemma \ref{le:trunc1} and Proposition~\ref{prop:equalityhew} which shows in particular the inequalities $\lip^{\tau_w}_{\He} \le \lip_{\HK}$, $\lip^{\tau_w}_{\W} \le \lip_{\HK}$.
\end{proof}

\section{The canonical form and the Cheeger energy}\label{s:CanonicalForm}
Let $X$ be a Polish space, and~$\ev\colon \mcM(X)\times X\to\R$ be the \emph{evaluation map}
\[
\ev\colon (\mu,x)\mapsto \ev_x\mu\eqdef \mu_x\comma
\]
where, for notational simplicity and throughout this section, we set
\[
\mu_x\eqdef \mu\!\set{x}\comma \qquad x\in X\fstop
\]

\subsection{The multiplicative infinite-dimensional Lebesgue measure}\label{sec:lebm}
Let~$\theta>0$, and recall that~$\LP{\theta,\nu}$ is the multiplicative infinite-dimensional Lebesgue measure with intensity~$\nu$ in~\eqref{eq:InfLebesgueIsomor}.
We collect here all the relevant properties of~$\LP{\theta,\nu}$.

\begin{lemma}\label{l:FinitenessBalls}
$\LP{\theta,\nu}$ is finite on every~$\HK_{\mssd_g}$-ball for every~$\theta>0$.

\begin{proof}
By triangle inequality it suffices to show the statement for sets of the form~$\mcB_r$ for some~$r\geq 0$.
By~\eqref{eq:InfLebesgueIsomor},
\[
\LP{\theta,\nu}(\mcB_r)= \lambda_\theta([0,r)) < \infty \fstop \qedhere
\]
\end{proof}
\end{lemma}

\subsubsection{Invariance and uniqueness}
We collect here some results about invariance and uniqueness of multiplicative infinite-dimensional Lebesgue measures.

\begin{proposition}\label{p:UniquenessLP}
Let~$\mcQ$ be a non-negative (non-zero) projectively~$\mfM(X)$-invariant Borel measure on~$\mcM(X)$, and such that~$\mcQ\mcB_0=0$ and~$\mcQ\mcB_r\in \R^+$ for some~$r\in\R^+$.
Then,
\begin{enumerate}[$(i)$]
\item\label{i:p:UniquenessLP:0} $k_t \eqdef \int \rme^{-t \mu X} \diff\mcQ(\mu)$ is (positive and) finite for every $t>0$;

\item\label{i:p:UniquenessLP:1} $\alpha \eqdef \int \mu(\emparg)\, \rme^{- \mu X} \diff\mcQ(\mu)$ is a finite non-negative (non-zero) Borel measure on~$X$;
\item\label{i:p:UniquenessLP:2} letting~$\theta\eqdef \alpha X$ and~$\nu\eqdef \normaliz(\alpha)$, it holds that~$\mcQ= k_1\,\LP{\theta,\nu}$.
\end{enumerate}

\begin{proof}
For any~$a \in \rmB_b(X)$ let~$\mcQ_a\eqdef ( \rme^a\cdot )_\pfwd \mcQ$ be the shift of~$\mcQ$ by~$\rme^a$, and~$d(\rme^a)\eqdef \frac{\diff \mcQ_a}{\diff\mcQ}$ be the Radon--Nikod\'ym derivative of~$\mcQ_a$ w.r.t.~$\mcQ$. 

\paragraph{Claim 1: for every $a \in \rmB_b(X)$ the function $d(\rme^a) \in (0,+\infty)$ is a constant and satisfies $d(\rme^a) \in (0,1]$ if $a \ge 0$. If, additionally, $a$ is constant, then $d(\rme^a) \in (0,1)$ if $a>0$ and $d(\rme^a)=d(1)=1$ if $a=0$}
If $a \in \rmB_b(X)^+$ and $r>0$, we have
\[ d(\rme^a) \mcQ(\mcB_r) = \mcQ_a(\mcB_r)= \int \car_{\mcB_r}(\rme^a \cdot \mu) \, \diff \mcQ(\mu) \le \int \car_{\mcB_r}(\mu) \, \diff \mcQ(\mu)  = \mcQ(\mcB_r) \fstop\]
Choosing $r>0$ such that $\mcQ(\mcB_r) \in (0, + \infty)$ shows that $d(\rme^a) \in [0,1]$. If, by contradiction, $d(\rme^a)=0$, then $\mcQ_a(\mcB_r)=0$ for every $r>0$ so that $\mcQ_a$ is the zero measure, a contradiction, since $\mcQ$ is not the zero measure. This proves that $d(\rme^a) \in (0,1]$ for every $a \in \rmB_b(X)^+$.
For~$a \in \rmB_b(X)$ we have
\[ d(\rme^a) = \frac{d(\rme^{a^+})}{d(\rme^{a^-})}  \in (0, + \infty).\]
If additionally $a$ is constant and positive, setting $c\eqdef\rme^a \in (1,+\infty)$, we assume by contradiction that $d(c)=1$. Then, for any $r>0$ such that $\mcQ(\mcB_r) \in (0, + \infty)$, we have $\mcQ(\mcB_r)= \mcQ(\mcB_{r/c})$, so that also $\mcQ(\mcB_{r/c})>0$; we deduce that $\mcQ(\mcB_{r/e^n} \setminus \mcB_{r/e^{n+1}}) =0$ for every $n \in \N$ so that
\[ \mcQ(\mcB_r) = \mcQ(\set{0}) + \sum_{n=0}^{\infty} \mcQ(\mcB_{r/c^n} \setminus \mcB_{r/c^{n+1}}) = 0 \comma\]
a contradiction. Finally, if $a=0$, $\mcQ_a=\mcQ_0=\mcQ$ so that $d(\rme^a)=d(1)=1$.

\paragraph{Claim 2: $a\mapsto d(\rme^a)$ is continuous on bounded sets w.r.t.\ the pointwise convergence in~$\rmB_b(X)$}
Let~$\seq{a_n}_n\in \rmB_b(X)$ be pointwise convergent to~$a\in \rmB_b(X)$ and satisfying~$\sup_n \|a_n\|_\infty<\infty$.
Then,~$\rme^{a_n} \mu$ converges weakly to~$\rme^{a} \mu$ for every~$\mu\in\mcM(X)$ by Dominated Convergence in~$L^1(\mu)$ with dominating function~$\abs{f}\rme^{\sup_n \|a_n\|_\infty}$ for any~$f\in \Cb(X)$.
Since~$\mcQ$ is finite on mass-bounded sets, for every bounded weakly continuous~$u\colon \mcM(X)\to\R$ with mass-bounded support,
\[
\lim_n \int u(\mu) \diff\mcQ_{a_n}(\mu) = \lim_n \int u(\rme^{a_n} \cdot  \mu) \diff\mcQ(\mu) = \int u(\rme^a \cdot \mu) \diff\mcQ(\mu) = \int u(\mu) \diff\mcQ_a(\mu)
\]
by Dominated Convergence in~$L^1(\mcQ)$ with dominating function~$\abs{u}\car_{\supp u}$.
Then, by projective invariance of~$\mcQ$,
\begin{equation}\label{eq:p:UniqunessLP:1}
\lim_n d(\rme^{a_n}) \int u \diff\mcQ = \lim_n \int u \diff\mcQ_{a_n} = \int u \diff\mcQ_a = d(\rme^a) \int u \diff\mcQ\fstop
\end{equation}
Since~$\mcQ$ is finite on mass-bounded sets, we may choose~$u$ so that~$\int u\diff\mcQ\in (0,\infty)$, and cancelling it from~\eqref{eq:p:UniqunessLP:1} concludes the assertion.

\paragraph{Claim 3: $d(\rme^{ar})=d(\rme^a)^r$ for every~$r\in\R$ and every $a \in \rmB_b(X)$}
For~$r\in\Z$, the assertion holds since~$a\mapsto d(\rme^a)$ is a group homomorphism.
The case~$r\in \Q$ follows from the case~$r\in\Z$ by a standard iteration argument.
The case~$r\in\R$ follows from the case~$r\in\Q$ by the continuity of~$a\mapsto d(\rme^a)$ in Claim~1.

\paragraph{Claim 4: the assignment
\begin{equation}\label{eq:p:UniquenessLP:1}
\alpha\colon A \longmapsto - \log d(\rme^{\car_A})
\end{equation}
defines a finite non-negative Borel measure on $X$ with total mass $\theta>0$}
Since $\car_A \ge 0$, we have by Claim~1 that $d(\rme^{\car_A}) \in (0,1]$ so that $\alpha A \ge 0$. Again by Claim~1 we get that $\theta=\alpha X = - \log (d(\rme))>0$ and $\alpha(\emp)=-\log(d(1))=0$. Since~$\rme^a\mapsto d(\rme^a)$ is a group homomorphism,~$\alpha$ is finitely additive on pairwise disjoint sets. We are only left to show that $\tilde{\nu}$ is countably additive. Let~$\seq{A_i}_i$ be any countable pairwise disjoint collection of Borel subsets of~$X$, and set~$A\eqdef \cup_i A_i$. 
For each~$n\in \N_1$, define a finite partition~$\seq{B^n_i}_{i\leq n}$ of~$A$ by setting~$B^n_i\eqdef A_i$ if~$i<n$ and~$B^n_n \eqdef \cup_{i \ge n} A_i $.
It suffices to show that~$d(\rme^{\car_A})= \prod_i^\infty d(\rme^{\car_{A_i}})$.
We have
\begin{equation}\label{eq:p:UniqunessLP:2}
d(\rme^{\car_A}) = \prod_i^n d(e^{\car_{B^n_i}}) = d(\rme^{\car_{B^n_n}})\prod_i^{n-1} d(\rme^{\car_{A_i}}) \fstop 
\end{equation}
Since~$\car_{B^n_n}$ is pointwise non-increasing to~$\car_\emp=0$ and uniformly bounded by~$1$, by Claim~2, we have~$\lim_n d(\rme^{\car_{B^n_n}})= d(\rme^0)=d(1)=1$.
The claim is proved by letting~$n$ to infinity in~\eqref{eq:p:UniqunessLP:2}.

\paragraph{Conclusion}
We now detail some arguments in the proof of~\cite[Thm.~4.2]{TsiVerYor01}.
For~$k\in\Z$ let~$\mcA_k\eqdef \set{\mu\in \mcM(X) : 2^k \leq \mu X < 2^{k+1}}$.
Since~$2\cdot \mcA_k= \mcA_{k+1}$ we have~$\mcQ(\mcA_{k+1})= d(2)^{-1} \mcQ(\mcA_k)$, hence~$\mcQ(\mcA_k)= d(2)^{-k} \mcA_0$.
Then, since~$\mcQ(\set{0})=0$ and $d(2) < 1$ by Claim~1,
\begin{equation}\label{eq:p:UniqunessLP:2.5}
k_t= \int \rme^{-t\mu X} \diff\mcQ(\mu) \leq \sum_{k\in\Z} \mcQ(\mcA_0)\, d(2)^{ -k} \exp(-2^k t) <\infty, \quad t>0 \fstop
\end{equation}
Choosing~$t=1$ proves~\ref{i:p:UniquenessLP:0}.
Furthermore,~\eqref{eq:p:UniqunessLP:2.5} states that the Laplace transform of~$\mcQ$ is finite on constant functions, and it is readily verified that it is finite on all positive simple functions: if $N \in \N_1$ and $A_1, \dots, A_N$ is a pairwise disjoint partition of $X$ we set
\[
u_{\boldsymbol{t}} \eqdef\sum_{i=1}^N t_i \car_{A_i} \ge \min_{1 \le i \le N} t_i \car_{X} \defeq \underline{t} \car_X, \quad \boldsymbol{t} = (t_1, \dots, t_N) \in (\R^+)^N \fstop
\]
We have
\[ L(\boldsymbol{t}) \eqdef \int \rme^{-\int u_{\boldsymbol{t}}\, \diff \mu} \diff\mcQ(\mu) \le \int \rme^{-\underline{t} \mu X} \diff\mcQ(\mu) <+\infty \fstop\]
For every $t >0$, $1 \le j \le N$ and $\boldsymbol{s}, \boldsymbol{t} \in (\R^+)^N$ let us set
\[
b_{j,t} \eqdef \car_{ X \setminus A_j} + t \car A_j, \quad \boldsymbol{s} \diamond \boldsymbol{t} \eqdef (s_1 t_1, s_2 t_2, \dots, s_N t_N)\comma
\]
and observe that
\[
u_{\boldsymbol{s} \diamond \boldsymbol{t}} = u_{\boldsymbol{t}} \prod_{i=1}^N b_{i,s_i}\fstop
\]
We have
\begin{align*}
L(\boldsymbol{s} \diamond \boldsymbol{t} )=& \int \rme^{-t \int u_{\boldsymbol{s} \diamond \boldsymbol{t}}\, \diff \mu} \diff\mcQ(\mu) =  \ L(\boldsymbol{t}) \prod_{i=1}^N d\tparen{b_{i,s_i}} 
\\
= & \ L(\boldsymbol{t}) \prod_{i=1}^N d\tparen{\rme^{(\log s_i) \car_{A_i}}}= \ L(\boldsymbol{t}) \prod_{i=1}^N d(\rme^{\car_{A_i}})^{\log s_i}
\end{align*}
by Claim~2, so that
\begin{align}\label{eq:LaplaceTransform}
L(\boldsymbol{s} \diamond \boldsymbol{t})= L(\boldsymbol{t}) \prod_{i=1}^N  s_i^{-\theta\, \nu(A_i)} \comma
\end{align}
with $\nu \eqdef \normaliz(\alpha)$. We deduce that  
\begin{align*}
\int \rme^{- \int u_{\boldsymbol{t}}\, \diff \mu} \diff\mcQ(\mu) &= L(\boldsymbol{t}) = L(1,1, \dots, 1) \cdot \prod_{i=1}^N t_i^{ -  \theta\, \nu(A_i)}
\\
& = k_1  \prod_{i=1}^N t_i^{ -  \theta\, \nu(A_i)} 
\\
& = k_1 \exp \paren{-\theta \int \log u_{\boldsymbol{t}} \, \diff \nu} \fstop
\end{align*}
Thus 
\begin{equation}\label{eq:formulatrans}
\int \rme^{- \int u\, \diff \mu} \diff\mcQ(\mu) =  k_1 \exp \paren{-\theta \int \log u \, \diff \nu} 
\end{equation}
for every positive simple function $u$.
By a standard approximation argument, we deduce that \eqref{eq:formulatrans} holds for every Borel function $u\colon X \to [0,+\infty]$.

Now, define a measure~$\mcQ_*$ on~$\mcM(X)$ by~$\diff\mcQ_*(\mu)\eqdef k_1^{-1}\rme^{-\mu X}\diff\mcQ(\mu)$.
Note that~$\mcQ_*$ is a probability measure by~\ref{i:p:UniquenessLP:0}.

Choosing~$u=\car+v$ in~\eqref{eq:formulatrans}, we deduce from~\eqref{eq:formulatrans} that
\[
\int \rme^{-\int v\diff\mu} \diff\mcQ_*(\mu) = \exp\paren{-\theta\int \log(1+v)\diff\nu}
\]
for every Borel function $v\colon X \to [0,+\infty]$.
By the uniqueness part in the generalized Bernstein theorem for Laplace transforms of probability measures (on abelian semigroups~\cite[\S47, p.~261]{Cho54} or on continuous bounded functions~\cite[Thm.~2.3]{HofRes77}),
together with the Laplace transform characterization of the Gamma measure, e.g.~\cite[Eqn.~(7)]{TsiVerYor01}, we conclude that~$\mcQ_*$ is the Gamma measure~$\GP{\theta, \nu}$ with intensity measure~$\alpha=\theta\nu$.
It then follows from~\cite[Eqn.~9]{TsiVerYor01} that~$\mcQ= k_1 \LP{\theta, \nu}$, which concludes the proof of the representation in~\ref{i:p:UniquenessLP:2} for the measure~$\alpha=\theta\nu$ in~\eqref{eq:p:UniquenessLP:1}.

It remains to show that the measure~$\alpha$ in~\eqref{eq:p:UniquenessLP:1} satisfies the representation in~\ref{i:p:UniquenessLP:1}.
Since~$\mcQ_*=\GP{\theta, \nu}$, the representation follows by definition of intensity measure, see also \eqref{eq:MeckeG}.
\end{proof}
\end{proposition}

\begin{corollary}
Let~$\mcQ$ be a non-negative projectively~$\mfM(X)$-invariant Borel measure on~$\mcM(X)$ such that~$\mcQ\mcB_r\in \R^+$ for some~$r\in\R^+$.
Then, either~$\mcQ=a_0\delta_{\mathbf{0}}$ or~$\mcQ=a_1\LP{\theta,\nu}$ for some constants~$a_0,a_1\geq 0$, some constant~$\theta>0$ and some Borel probability measure~$\nu$ on~$X$.

\begin{proof}
Assume by contradiction that $\mcQ$ charges both $\set{\mathbf{0}}$ and its complement. In this case, the restriction of $\mcQ$ to $\set{\mathbf{0}}$ is invariant, so that the measure $\mcQ$ too must be  invariant.
Thus the restriction of $\mcQ$ to $\mcM (X)\setminus \mcB_0$ is invariant. Since by Proposition \ref{p:UniquenessLP} the only non-negative (non-zero) projectively invariant Borel measure on $\mcM (X)\setminus \mcB_0$ giving finite mass to some ball is the multiplicative infinite-dimensional Lebesgue measure, and the latter is not invariant, we conclude that the restriction of $\mcQ$ to $\mcM (X)\setminus \mcB_0$ is the zero measure.
\end{proof}
\end{corollary}

\begin{proposition}\label{p:FullIdentification}
Let~$G<\Shift(X)$ be any subgroup. Then, $\LP{\theta,\nu}$ is invariant for the \eqref{eq:ActionShifts}-action of $G$ on~$\mcM(X)$ if and only if~$\nu$ is invariant for the natural action of~$G$ on~$X$.

In particular, if~$X$ has a structure of smooth, connected, orientable  Riemannian manifold~$(M,g)$ with \emph{finite} volume~$\vol_g$ and $\LP{\theta,\nu}$ is invariant for the \eqref{eq:ActionShifts}-action on~$\mcM(M)$ of the group $\Diff^+_0(g)$ of all compactly non-identical,  orientation-preserving,  $\vol_g$-preserving diffeomorphisms, then~$\theta\nu=\vol_g$.

\begin{proof}
Assume first that~$\LP{\theta,\nu}$ is invariant for the \eqref{eq:ActionShifts}-action of $G$. Then, for every~$\iota\in G$ we have
\begin{align*}
\theta \iota_\pfwd \nu &= \iota_\pfwd \int \mu(\emparg)\, \rme^{-\mu X} \diff\LP{\theta,\nu}(\mu) = \int (\iota_\pfwd \mu)(\emparg)\, \rme^{-(\iota_\pfwd \mu)X}\diff\LP{\theta,\nu}(\mu) 
\\
&= \int \mu(\emparg)\, \rme^{-\mu X} \diff {\iota_\pfwd}_\pfwd \LP{\theta,\nu}(\mu) = \int \mu(\emparg)\, \rme^{-\mu X} \diff \LP{\theta,\nu} = \theta\nu \fstop
\end{align*}
Since~$\theta>0$, this shows that~$\iota_\pfwd\nu=\nu$, that is,~$\nu$ is invariant for the natural action of~$G$ on~$X$.

Vice versa, assume that~$\nu$ is invariant for the natural action of~$\Shift(X)$ on~$X$.
By Proposition~\ref{p:Mapping} below, for every~$\iota\in G$ we have
\[
{\iota_\pfwd}_\pfwd \LP{\theta,\nu} = \LP{\theta,\iota_\pfwd \nu} = \LP{\theta,\nu} \fstop
\]
This concludes the proof of the first assertion.

\medskip

Now, assume that~$X=(M,g)$ and that~$\LP{\theta,\nu}$ is invariant for the \eqref{eq:ActionShifts}-action of $\Diff^+_0(g)$ on~$\mcM(M)$.
It follows from the first assertion that~$\nu$ is invariant for the natural action of~$\Diff^+_0(g)$ on~$M$.
Thus, it suffices to show that~$\nu\propto \vol_g$, which is shown in Proposition~\ref{p:InvariantVolG} below.
\end{proof}
\end{proposition}

\subsubsection{Functoriality}
We collect here some facts about the functoriality of the assignment~$\alpha\mapsto \LP{\alpha X, \normaliz(\alpha)}$, greatly extending the convolution property of~$\theta\mapsto \LP{\theta,\nu}$ in Proposition~\ref{p:TVY}\ref{i:p:TVY:1}.

Let~$\alpha\in\mcM(X)$ and denote by $\mcM(X)^+ \eqdef  \mcM(X) \setminus \set{\mathbf{0}}$.
In order to state the next results, it is convenient to write~$\LP{\alpha}\eqdef \LP{\alpha X,\normaliz(\alpha)}$.
For a Polish space~$Y$, we denote by~$\mcM_\sigma(Y)$ the space of all non-negative $\sigma$-finite Borel measures on~$Y$.
If~$Y=(Y,+)$ is a \emph{cone} in a topological linear space, we define the \emph{convolution} of~$\mcQ_1,\mcQ_2\in\mcM_\sigma(Y)$ by
\begin{equation}\label{eq:ConvolutionDfn}
(\mcQ_1*\mcQ_2) A \eqdef \int_Y \car_{A}(y_1+y_2) \diff\mcQ_1(y_1) \diff\mcQ_2(y_2)
\end{equation}
with~$A\subset Y$ Borel.
Since Polish spaces are strongly Radon, the existence of the convolution of $\sigma$-finite (Radon) measures follows similarly to the proof of the same assertion for (Radon) probability measures in~\cite[Prop.~I.4.4, p.~64]{VakTarCho87}.

\begin{corollary}[Convolution property]
The assignment~$\tparen{\mcM(X)^+,+}\ni \alpha \mapsto \mcL_{\alpha} \in \tparen{\mcM_\sigma(\mcM(X)^+)^+,*}$ is a homomorphism of magmas, that is
\[
\LP{\alpha_1+\alpha_2} = \LP{\alpha_1}* \LP{\alpha_2} \comma \qquad \alpha_1,\alpha_2\in\mcM(X)^+\fstop
\]
\begin{proof}
It is readily seen that the convolution~$\LP{\alpha_1}*\LP{\alpha_2}$ satisfies all the assumptions in Proposition~\ref{p:UniquenessLP}.
For example, let us briefly verify the projective $\Multi(X)$-invariance.
From~\eqref{eq:ConvolutionDfn}, for every Borel~$A\subset\mcM(X)$, for every~$a\in \rmB_b(X)$,
\begin{align*}
(\rme^a.)_\pfwd (\LP{\alpha_1}*\LP{\alpha_2}) A =& \int \car_{(\rme^a.)^{-1}(A)}(\mu_1+\mu_2) \diff\LP{\alpha_1}(\mu_1)\diff\LP{\alpha_2}(\mu_2)
\\
=& \int \car_{A}(\rme^a.\mu_1+\rme^a.\mu_2) \diff\LP{\alpha_1}(\mu_1)\, \diff\LP{\alpha_2}(\mu_2)
\\
=& \int \car_{A}(\mu_1+\mu_2) \diff(\rme^a.)_\pfwd\LP{\alpha_1}(\mu_1)\, \diff(\rme^a.)_\pfwd\LP{\alpha_2}(\mu_2)
\\
=&\ d_1(\rme^a) d_2(\rme^a) \int \car_{A}(\mu_1+\mu_2) \diff\LP{\alpha_1}(\mu_1)\, \diff\LP{\alpha_2}(\mu_2)
\\
=&\ d_1(\rme^a) d_2(\rme^a) (\LP{\alpha_1}*\LP{\alpha_2})(A)\comma
\end{align*}
where, for~$i=1,2$, we set~$d_i(\rme^a)\eqdef \frac{\diff(\rme^a.)_\pfwd\LP{\alpha_i}}{\diff \LP{\alpha_i}}$, so that there exists~$\frac{\diff (\rme^a.)_\pfwd (\LP{\alpha_1}*\LP{\alpha_2})}{\diff (\LP{\alpha_1}*\LP{\alpha_2})}= d_1(\rme^a) d_2(\rme^a)$.

Thus, by Proposition~\ref{p:UniquenessLP}, we have~$\LP{\alpha_1}*\LP{\alpha_2}= k_1 \LP{\alpha}$ for some~$\alpha\in \mcM(X)^+$ and some constant~$k_1>0$, and it suffices to show that~$\alpha=\alpha_1+\alpha_2$ and~$k_1=1$; we show only the first assertion, the proof of the second is similar. By Fubini's Theorem,
\begin{align*}
\int \mu(\emparg)\, \rme^{-\mu X} \diff (\LP{\alpha_1}*\LP{\alpha_2})(\mu) =& \iint \tparen{\mu_1(\emparg)+\mu_2(\emparg)}\, \rme^{-\mu_1 X-\mu_2 X} \diff\LP{\alpha_1}(\mu_1)\, \LP{\alpha_2}(\mu_2)
\\
=& \int \mu_1(\emparg)\, \rme^{-\mu_1 X} \diff\LP{\alpha_1}(\mu_1)\, \int \rme^{-\mu_2 X} \LP{\alpha_2}(\mu_2)
\\
&+ \int \mu_2(\emparg)\, \rme^{-\mu_2 X} \diff\LP{\alpha_2}(\mu_2)\, \int \rme^{-\mu_1 X} \LP{\alpha_1}(\mu_1)
\\
=&\ (\alpha_1 +\alpha_2)(\emparg)
\end{align*}
again by Proposition~\ref{p:UniquenessLP}.
\end{proof}
\end{corollary}

\begin{proposition}[Mapping theorem]\label{p:Mapping}
Let~$X$ and~$Y$ be Polish spaces, and~$f\colon X\to Y$ be a Borel function. Then, for every~$\alpha\in\mcM(X)$,
\[
{f_\pfwd}_\pfwd \LP{\alpha} = \LP{f_\pfwd\alpha} \fstop
\]
\end{proposition}
\begin{proof}
Let~$\alpha=\theta\nu$ with~$\theta\eqdef \alpha X>0$ and~$\nu\eqdef \normaliz(\alpha)\in\mcP(X)$. 
It suffices to combine the representation of~$\LP{\alpha}=\LP{\theta,\nu}$ in~\eqref{eq:InfLebesgueIsomor} with the Mapping Theorem~\cite[Thm.~3.9]{LzDS19a} for the Dirichlet--Ferguson measure~$\DF{\nu}$.
\end{proof}

\subsubsection{Mecke identity}
As it is clear from the study of the simplicial part~$\DF{\nu}$ of~$\LP{\theta,\nu}$, and of the Gamma measure~$\GP{\theta, \nu}$, the characterization of such measures via Mecke-type integral identities plays a key role.
A similar, strikingly simple identity holds as well for the multiplicative infinite-dimensional Lebesgue measure~$\LP{\theta,\nu}$.

\begin{proposition}[Mecke identity for~$\LP{\theta,\nu}$]\label{p:MeckeLP}
For every Borel measurable~$F\colon \mcM(X)\times  \R^+_0\times X  \to [0,\infty]$,
\begin{equation}\label{eq:MeckeLP}
\begin{aligned}
\int \braket{\int_X F(\mu,\mu_x,x)\, \diff\mu(x) }& \diff\LP{\theta,\nu}(\mu) = 
\\
&\theta\int \braket{\int_X \int_0^\infty F(\mu+s\delta_x,s,x) \, \diff s\,\diff\nu(x)} \diff\LP{\theta,\nu}(\mu) \fstop
\end{aligned}
\end{equation}
\begin{proof}
Recall the Mecke identity for the Gamma measure~$\GP{\theta}$, e.g.~\cite[Eqn.~(2.4)]{LzDSLyt17} viz.
\begin{equation}\label{eq:MeckeG}
\begin{aligned}
\int \braket{\int_X G(\mu,\mu_x,x)\diff\mu(x)}& \diff\GP{\theta, \nu}(\mu)=
\\
&\ \theta\int \braket{\int_X \int_0^\infty G(\mu+s\delta_x,s,x)\, e^{-s} \diff s\, \diff\nu(x)} \diff\GP{\theta, \nu} (\mu)\comma
\end{aligned}
\end{equation}
and recall that~$\diff  \LP{\theta,\nu} (\mu)=e^{\mu X}\diff\GP{\theta, \nu}(\mu)$, see e.g.~ \cite[p.~165]{TsiVer99} . Applying the above identity to~$G(\mu,\mu_x,x)= e^{\mu X} F(\mu,\mu_x,x)$, we then have
\begin{align*}
\int & \braket{\int_X F(\mu,\mu_x,x)\diff\mu X}\diff\LP{\theta,\nu}(\mu)
\\
=&\ \theta\int \braket{\int_X \int_0^\infty e^{(\mu+s\delta_x)X} F(\mu+s\delta_x, s ,x)\, e^{-s} \diff s\, \diff\nu(x)} \diff\GP{\theta, \nu}(\mu)
\\
=&\ \theta\int \braket{\int_X \int_0^\infty F(\mu+s\delta_x, s ,x)\, \diff s\, \diff\nu(x)} e^{\mu X}\diff\GP{\theta, \nu} (\mu)
\\
=&\ \theta\int \braket{\int_X \int_0^\infty F(\mu+s\delta_x, s ,x)\, \diff s\, \diff\nu(x)} \diff\LP{\theta,\nu}(\mu)\comma
\end{align*}
which is the assertion.
\end{proof}
\end{proposition}

\begin{lemma}\label{l:Negligible}
The set
\begin{equation}\label{eq:Negligible}
\mcN\eqdef \ev^{-1}((0,+\infty))=\set{(\mu,x) : \mu_x>0} \subset \mcM(X)\times X
\end{equation}
is Borel measurable and~$\LP{\theta,\nu}\otimes \nu$-negligible.

\begin{proof}
Since the map~$\ev\colon (\mu,x)\mapsto \mu_x$ is Borel measurable on~$\mcM(X)\times X$ the set~$\mcN$ is Borel measurable.
Furthermore,
\begin{equation*}
\begin{aligned}
\int_{\mcM(X)} \mu_x \, \diff\LP{\theta,\nu}(\mu) =& \int_0^\infty \int_{\mcP(X)} \eta_x\, \diff \DF{\nu}(\eta)\,  t\, \diff \lambda_\theta(t)
\\
=&\ \nu_x \int_0^\infty \,  t\, \diff\lambda_\theta(t)=0  \comma  
\end{aligned}
\qquad x\in X \comma
\end{equation*}
 where we used the identity $\int \eta_x\, \diff \DF{\nu}(\eta) = \nu_x$ in \cite[Prop.~1, p.~214]{Fer73}, and the fact that $\nu$ is non-atomic. 
Thus, the section~$\mcN_x\eqdef \set{\mu\in\mcM(X): (\mu,x)\in\mcN}$ is~$\LP{\theta,\nu}$-negligible for each~$x\in X$.
Since a Borel measurable set with negligible sections is negligible for the product measure, the conclusion follows.
\end{proof}
\end{lemma}

Throughout the rest of this section, let~$(M,g)$ be a  smooth, connected, orientable, complete  Riemannian manifold with Riemannian distance~$\sfd_g$ and Riemannian volume measure $\vol_g$ and let~$\rho\in\Cb^\infty(M)$ be so that~$\rho>0$ everywhere on~$M$ and~$\nu\eqdef \rho\vol_g$ is an element of~$\mcP_2(M)$, i.e.
\[
 \nu M=1 \quad \text{ and } \quad \int_M \mssd_{g,x_0}^2 \de \nu <\infty \qquad \text{for some~} x_0\in M\fstop
\]
 We refer to the triplet $(M, g, \nu)$ as above as to a \emph{weighted Riemannian manifold}.

Finally, for every function~$\hat f\in\rmC^1(\R^+_0\times M)$ and every~$(s,x)\in\R^+_0\times M$, set
\[
\hat f'(s,x)\eqdef (\partial_s \hat f)(s,x), \quad \nabla \hat f (s,x)\eqdef (\nabla_x \hat f)(s,x), \quad \Delta \hat f(s,x) \eqdef (\Delta_x \hat f)(s,x) \fstop
\]

In the following, we will make use of standard definitions and results in theory of Dirichlet forms. We refer the reader to the monographs~\cite{FukOshTak11, MaRoe92} for a standard treatment.

\subsection{Extended cylinder functions, extended form}\label{sec:extcyl}
The following sets of cylinder functions will be instrumental to the proof of Theorem~\ref{t:IntroForm}.

\begin{definition}[Cylinder functions of reduced potential energies]\label{d:ExtCylinderF}
For every $\hat f \in  \rmC_b  (\R^+_0\times M)$, we define
\begin{align}\label{eq:d:ExtCylinderF:0}
\hat f^\trid(\mu)\eqdef \int_M \hat f(\mu_x,x) \,\diff\mu(x)\comma \qquad \mu \in \mcM(M) \fstop
\end{align}
We also define the following sets of cylinder functions for $\eps \in [0,+\infty)$:
\begin{align*}
\hFC{m_1}{\sharp_1}{m_2}{\sharp_2}\eqdef& \set{\hat u\colon \mcM(M)\to\R : \begin{gathered} \hat u=F\circ \hat\mbff^\trid\comma F \in\rmC^{m_1}_{\sharp_1}(\R^{k+1};\R) \comma
\\
k\in \N\comma \hat\mbff\eqdef\tseq{\hat f_i}_{0\leq i\leq k} \comma \hat f_0\equiv 1 \comma 
\\
\hat f_i\in \rmC^{m_2}_{\sharp_2}\tparen{\R^+_0\times M} \text{ for } 1\leq i\leq k\end{gathered}}\comma
\intertext{and}
\hFC{m_1}{\sharp_1}{m_2}{\sharp_2,\eps}\eqdef& \set{\hat u\colon \mcM(M)\to\R : \begin{gathered} \hat u=F\circ \hat\mbff^\trid\comma F \in\rmC^{m_1}_{ \sharp_1 }(\R^{k+1};\R) \comma
\\
k\in  \N  \comma \hat\mbff\eqdef\tseq{\hat f_i}_{0\leq i\leq k} \comma \hat f_0\equiv 1 \comma 
\\
\hat f_i\in \rmC^{m_2}_{\sharp_2}\tparen{(\eps,\infty)\times M} \text{ for } 1\leq i\leq k\end{gathered}} \comma
\end{align*}
where~$m_1,m_2\in \N\cup\set{\infty}$ and~$\sharp$ stands for either~$b$ for \emph{bounded} or~$c$ for \emph{compact support}.
Clearly
\[
 \hFC{\infty}{c}{\infty}{c,\eps}  \subsetneq \hFC{\infty}{c}{\infty}{c,0}  \subsetneq  \hFC{\infty}{c}{\infty}{c}   \subsetneq \hFC{0}{b}{0}{b}   \comma \qquad \eps >0\fstop
\]
 and 
\[ 
 \FC{\infty,\infty}{c,c}{\infty}{c} \subsetneq \hFC{\infty}{c}{\infty}{c} \fstop
\]
\end{definition}

Let us start by showing that cylinder functions of reduced potential energies are dense in~$L^2(\LP{\theta,\nu})$, so  that all  forms in the following will be densely defined.

\begin{lemma}\label{l:L2DensityCylinder}
The following assertions hold:
\begin{enumerate}[$(i)$]
\item\label{i:l:L2DensityCylinder:0} all functions in~$\hFC{0}{ b  }{0}{b}$ are Borel measurable;
\item\label{i:l:L2DensityCylinder:1.5} $\FC{\infty,\infty}{c,c}{\infty}{c}$ is dense in~$L^p(\LP{\theta,\nu})$ for every~$\theta>0$ and every~$p\geq 1$;
\item\label{i:l:L2DensityCylinder:2} $\hFC{\infty}{c}{\infty}{c,0}$ is dense in~$L^p(\LP{\theta,\nu})$ for every~$\theta>0$ and every~$p\geq 1$.
\end{enumerate}
\begin{proof}
\ref{i:l:L2DensityCylinder:0} It suffices to show that functions in~$\hFC{0}{b}{0}{c}$ are Borel measurable.
The assertion for functions in~$\hFC{0}{b}{0}{b}$ follows by approximation.
By~\cite[Rmk.~2.6]{EthKur94}, all functions of the form~\eqref{eq:d:ExtCylinderF:0} with~$\hat f\in\rmC_c(\R^+\times M)$ are continuous w.r.t.\ the weak atomic topology introduced in~\cite{EthKur94}.
In particular, since as noted in~\cite[p.~5, below Eqn.~(2.3)]{EthKur94}, the Borel $\sigma$-algebra of the weak atomic topology on~$\mcM(M)$ coincides with the Borel $\sigma$-algebra of the narrow topology on~$\mcM(M)$, these functions are Borel measurable.
It follows that all functions in~$\hFC{0}{c}{0}{c}$ are Borel measurable, being the composition of a continuous function~$F\in\rmC^0(\R^{k+1};\R)$ with the Borel measurable functions~$\hat f_i^\trid$,~$1\leq i\leq k$ and with the Borel measurable function~$\hat f_0^\trid= \car^\trid$.

\ref{i:l:L2DensityCylinder:1.5}
 Let~$u\in L^p(\LP{\theta,\nu})$ and set, for every $k \in \N$, $\mcQ_k \eqdef \nchi_k\LP{\theta,\nu}$, where $\nchi_k \eqdef \vartheta_k \circ \car^\trid$ and $\vartheta_k \in \rmC^\infty_c(\R_0^+)$ with $\supp \vartheta_k \subset [0,k]$ and $\vartheta_k \uparrow 1$ as $k \to + \infty$; note that $\mcQ_k \le \LP{\theta,\nu}$ is a finite non-negative Borel measure. Arguing as in~\cite[Lemma 2.1.27]{Savare22}, we see that we can find sequences $\seq{\tilde{u}_n^k}_n \subset \FC{\infty,\infty}{c,c}{\infty}{c}$ such that $\tilde{u}_n^k \to u$ in $L^p(\mcQ_k)$ as $n \to + \infty$. Setting $u^k\eqdef u \nchi_{k}$, and $u_n^k\eqdef \tilde{u}_n^k \nchi_k$, $k,n \in \N$, we see that $u_n^k \in \FC{\infty,\infty}{c,c}{\infty}{c}$, $u_n^k \to u^k$ in $L^p(\LP{\theta,\nu})$ as $n \to + \infty$ and $u^k \to u$ in $L^p(\LP{\theta,\nu})$ as $k \to + \infty$.
 We conclude by the diagonal argument in $L^p(\LP{\theta,\nu})$.

\ref{i:l:L2DensityCylinder:2}  By \ref{i:l:L2DensityCylinder:1.5}, it suffices to show that any $u \in \FC{\infty,\infty}{c,c}{\infty}{c}$ can be approximated in $L^p(\LP{\theta,\nu})$ by functions in $\hFC{\infty}{c}{\infty}{c,0}$. So let $u = F\circ\mbff^\trid$ for $F \in \rmC^\infty_c(\R^{k+1})$ and $\mbff = (\car, f_1, \dots, f_k)$ with $f_i \in \rmC^\infty_c(M)$ for $1\le i \le k \in \N$. We set, for every $n \in \N$, 
\begin{equation}\label{eq:approxrho}
\hat u_n \eqdef F\circ \tparen{\car^\trid, (\varrho_n \otimes f_1)^\trid,\cdots (\varrho_n \otimes f_k)^\trid} \comma
\end{equation}
with $\varrho_n \in \rmC_c^\infty(\R^+)$, $0\leq \varrho_n\nearrow 1$, and $\supp\varrho_n\subset [1/n, n]$. Clearly $\hat u_n \in \hFC{\infty}{c}{\infty}{c,0}$ and it is not difficult to see that
\[ \|u-\hat u_n\|_{L^p(\LP{\theta,\nu})}^p \le k (\Li{F})^p \left ( \max_{1 \le i \le k} \|f_i\|_\infty \right )^p \int_{\mcB_r} \int_M \left |\varrho_n(\mu_x)-1 \right |^p \diff \mu(x)\, \diff \LP{\theta,\nu}(\mu)\comma\]
where $r>0$ is such that $\supp F \subset [-r,r]^{k+1}$. By Dominated Convergence Theorem we conclude the assertion.
\end{proof}
\end{lemma}

Let us now show that~$\LP{\theta,\nu}$ fits the abstract framework of~\S\ref{ss:IntroGeometric} in that it is partially \ref{eq:ActionSemidirect}-quasi-invariant.
Whereas we will not make explicit use of this fact in the following, we believe it to be of interest in its own right, from the representation-theoretical point of view detailed in~\S\ref{sss:PQI}.

Firstly, we note that it is not restrictive to relax the definition of partial quasi-invariance in the following way.
Let~$\seq{\msF_t}_{t\in T}$ be the filtration in Definition~\ref{d:PQI}\ref{i:d:PQI:4}, and denote by~$\msF_\vee$ its terminal $\sigma$-algebra.
Rather than requiring~$\msF_\vee=\msF$, it suffices to let~$A\subset \Omega$ be $\mcQ$-conegligible and to require that~$\msF_\vee=\msF_A$, the trace $\sigma$-algebra of~$\msF$ on~$A$.
In the present setting, note that~$\LP{\theta,\nu}$ is concentrated on the subset~$\mcM^{\pa}(M)$ of purely atomic measures, since~$\DF{\nu}$ is concentrated on purely atomic probability measures (e.g.~\cite[Thm.~2, p.~219]{Fer73}).
We now choose~$A=\mcM^{\pa}(M)$ in the above reasoning.
\begin{proposition}\label{p:PQI}
For every~$\theta>0$, the measure~$\LP{\theta,\nu}$~is partially quasi-invariant under the \ref{eq:ActionSemidirect}-action of~$\mfG(M)\eqdef \Diff^+_0(M)\rtimes_\pb \exp[\Cc^\infty(M)]$.
\begin{proof}
Since~$\LP{\theta,\nu}$ is projectively invariant w.r.t.\ the \ref{eq:ActionMultipliers} of~$\exp[\Cc^\infty(M)]$, and since the push-forward of measures is homogeneous (in fact, linear), it suffices to show that~$\LP{\theta,\nu}$ is partially quasi-invariant under the \ref{eq:ActionShifts}-action of~$ \Diff^+_0(M)$.
Since the action~\eqref{eq:ActionSemidirect} splits over the decomposition~\eqref{eq:IsomorphismMbp}, it suffices to verify that there exists a filtration~$\msF_\bullet\eqdef \seq{\msF_t}_{t\in T}$ of~$\mcM^{\pa}(M)$ with the following properties:
\begin{enumerate*}[$(a)$]
\item the terminal $\sigma$-algebra~$\msF_\vee$ of~$\msF_\bullet$ coincides with the Borel $\sigma$-algebra of~$\mcM^{\pa}(M)$;
\item $\msF_\bullet$ is $J$-saturated, i.e.~$\msF_t= (J^{-1}\circ J)(\msF_t)$ for every~$t\in T$;
\item the simplicial part~$\DF{\nu}$ of~$\LP{\theta,\nu}$ is partially quasi-invariant w.r.t.\ the split action~\eqref{eq:ActionShifts} on~$\mcP^{\pa}(M)$ for the image filtration~$J(\msF_\bullet)\eqdef \seq{J(\msF_t)}_{t\in T}$.
\end{enumerate*}

To this end, let~$T\eqdef [0,1]$ be unit interval with the reverse of its natural order, and let~$\msF_t\eqdef \sigma\tparen{\hFC{\infty}{c}{\infty}{c,t}}$ be the $\sigma$-algebra generated by the family of cylinder functions in Definition~\ref{d:ExtCylinderF}.
Note that~$\msF_\vee=\sigma\tparen{\hFC{\infty}{c}{\infty}{c,0}}$, and recall that it generates the Borel $\sigma$-algebra on~$\mcM^{\pa}(M)$ as noted in the proof of Lemma~\ref{l:L2DensityCylinder}.

Finally, it was shown in~\cite[Prop.~5.20]{LzDS17+} that~$\DF{\beta}$ is partially quasi-invariant under the \ref{eq:ActionShifts}-action of~$ \Diff^+_0(M)$ for the filtration $\seq{\msF_t}_{t\in T}$ with~$\msF_t=\sigma\tparen{\hFC{\infty}{c}{\infty}{c,t}}$.
\end{proof}
\end{proposition}

For each~$\hat u=F\circ\hat\mbff^\trid\in \hFC{0}{ b }{0}{ b ,0}$ define a function
\begin{equation}\label{eq:U}
U\colon \mcM(M)\times M \times \R^+\to\R \comma \qquad U\colon (\mu,x,s)\longmapsto \hat u(\mu+s\delta_x) \fstop
\end{equation}
It is not difficult to show that the map~$(\mu,x,s)\mapsto \mu+s\delta_x$ is continuous. 
This fact, together with Lemma~\ref{l:L2DensityCylinder}\ref{i:l:L2DensityCylinder:0}, implies that~$U$ is Borel measurable.

\subsubsection{Vertical differentiability of extended cylinder functions}\label{sss:ClosabilityVer}
Define an operator $\tparen{\widehat \mcL^\ver,\hFC{\infty}{c}{\infty}{c,0}}$ by
\begin{align*}
\begin{aligned}
(\widehat \mcL^\ver \hat u)_\mu\eqdef&\  \int_M  \partial^2_s\restr{s=\mu_x} \hat u(\mu+s\delta_x-\mu_x\delta_x) \diff\mu(x)
\\
&\qquad\ensuremath{+ \theta  \int_M \partial_s\restr{s=0} \hat u(\mu+s\delta_x) \diff\nu(x)}
\comma 
\end{aligned}
\qquad \hat u\in\hFC{\infty}{c}{\infty}{c,0}\fstop
\end{align*}
\begin{proposition}\label{p:ClosVer}
The form~$\tparen{\mcE^\ver,\hFC{\infty}{c}{\infty}{c,0}}$  defined as
\begin{align*}
\mcE^\ver(\hat u,\hat v) &\eqdef \int \tscalar{(\gradH \hat u)_\mu}{(\gradH \hat v)_\mu}_{T^\ver_\mu}\, \diff \LP{\theta,\nu}(\mu) 
\\
&= \int \int_M (\gradH \hat u)_\mu(x) (\gradH \hat v)_\mu(x) \, \diff \mu(x) \, \diff \LP{\theta,\nu}(\mu)
\end{align*}  
is well-defined and closable on~$L^2(\mcM(M), \LP{\theta,\nu})$.
Its closure~$\tparen{\widehat\mcE^\ver,\dom{\widehat\mcE^\ver}}$ is a closed bilinear form with generator the Friedrichs extension on~$L^2(\mcM(M), \LP{\theta,\nu})$ of the operator~$\tparen{\widehat \mcL^\ver,\hFC{\infty}{c}{\infty}{c,0}}$. 
\end{proposition}

In order to prove Proposition~\ref{p:ClosVer} we shall need several auxiliary results on the differentiability of functions on~$\mcM(M)$.
Some of them may be inferred from the general computations for derivatives on spaces of measures in~\cite{RenWan21}, or from the analogous results on the differentiability of functions on~$\mcP(M)$ in~\cite{LzDS17+}, but we report them for ease of reference.

\begin{lemma}\label{l:LaplacianVerAux}
Fix~$\hat u=F\circ\hat\mbff^\trid\in \hFC{0}{c}{0}{c,0}$  with $F:\R^{k+1} \to \R$ . Then,
\begin{enumerate}[$(i)$]
\item\label{i:l:LaplacianVerAux:2} if~$\hat u$ is in~$\hFC{1}{c}{1}{c,0}$, then for~$\LP{\theta,\nu}\otimes \nu$-a.e.~$(\mu,x)$ the function~$s\mapsto U(\mu,x,s)$  as in \eqref{eq:U}  is differentiable on~$\R^+$, and, for every~$s\in\R^+$,
\begin{equation}\label{eq:l:LaplacianVerAux:1}
\partial_s U(\mu,x,s) = \sum_{i=0}^k (\partial_i F)\tparen{\hat\mbff^\trid(\mu+s\delta_x)} \braket{\hat f_i(s,x)+s \hat f_i'(s,x)} 
 \as{\LP{\theta,\nu}\otimes\nu}\semicolon
\end{equation}

\item\label{i:l:LaplacianVerAux:3} if~$\hat u$ is in~$\hFC{2}{c}{2}{c,0}$, then for~$\LP{\theta,\nu}\otimes \nu$-a.e.~$(\mu,x)$ the function~$s\mapsto U(\mu,x,s)$  as in \eqref{eq:U}  is twice differentiable on~$\R^+$, and, for every~$s\in\R^+$,
\begin{equation}\label{eq:l:LaplacianVerAux:2}
\begin{aligned}
\partial^2_s  U(\mu,x,s)  =& \sum_{i,j=0}^{k,k} (\partial^2_{ij} F)\tparen{\hat\mbff^\trid(\mu+s\delta_x)} \braket{\hat f_i(s,x)+s \hat f_i'(s,x)} 
\braket{\hat f_j(s,x)+s \hat f_j'(s,x)} 
\\
& + \sum_{i=0}^k (\partial_i F)\tparen{\hat\mbff^\trid(\mu+s\delta_x)} \braket{2\hat f_i'(s,x)+s \hat f_i''(s,x)} 
  \as{\LP{\theta,\nu}\otimes\nu} \fstop
\end{aligned}
\end{equation}
\end{enumerate}
In particular, the left -hand sides of~\eqref{eq:l:LaplacianVerAux:1} and~\eqref{eq:l:LaplacianVerAux:2} are Borel measurable.
\begin{proof}
\ref{i:l:LaplacianVerAux:2}
Fix~$\hat f\in\rmC^1_c(\R^+\times M)$.
Let~$\mcN$ be as in~\eqref{eq:Negligible} and recall that it is~$\LP{\theta,\nu}\otimes\nu$-negligible by Lemma~\ref{l:Negligible}. For every~$(\mu,x)\not\in\mcN$, for every~$s\in\R^+$,
\begin{align*}
\partial_s& \tbraket{\hat f^\trid(\mu+s\delta_x)}=
\\
=&\ \partial_s \braket{\int \hat f\tparen{\mu_y+s \car_{x}(y),y}\, \diff\mu(y) + s\hat f\tparen{\mu_x+s,x}} 
\intertext{hence, by definition of~$\mcN$, continuing the above chain of equalities,}
=&\ \partial_s \braket{\int \hat f(\mu_y,y)\,\diff\mu(y)+s\hat f(s,x)}
\\
=&\ \hat f(s,x)+ s \hat f'(s,x) \fstop
\end{align*}
We used that, since~$\hat f\in \rmC^1_0(\emparg,x)$ has compact support in~$\R^+$ away from~$0$ uniformly in~$x\in M$, and since the total mass of~$\mu$ is finite, every integral above is in fact a finite sum, hence we may freely differentiate under integral sign.

On the complement of the $\LP{\theta,\nu}\otimes\nu$-negligible set~$\mcN$, the equality in~\eqref{eq:l:LaplacianVerAux:1} for~$U$ readily follows from the above equality and the standard chain rule.

\ref{i:l:LaplacianVerAux:3} A proof is similar to the one of~\ref{i:l:LaplacianVerAux:2} and therefore it is omitted.
\end{proof}
\end{lemma}

\begin{lemma}\label{l:TridVerGrad}
 Fix~$\hat u=F\circ\hat\mbff^\trid\in \hFC{1}{c}{1}{c,0}$ with $F\colon\R^{k+1} \to \R$ . 
Then, for every~$(\mu,x)\in \mcM(M)\times M$ there exists
\begin{equation}\label{eq:l:TridVerGrad:1}
\begin{aligned}
(\gradH \hat u)_\mu(x) =&\  \sum_{i=0}^k  (\partial_i F)\tparen{\hat\mbff^\trid(\mu)} \braket{\mu_x \hat f_i'(\mu_x,x)+\hat f_i(\mu_x,x)}
\end{aligned}
\end{equation}
and
\begin{align}\label{eq:l:TridVerGrad:2}
 (\gradH \hat u)_\mu(x) =\partial_s\restr{s=\mu_x}  U  (\mu,x,s) \as{\LP{\theta,\nu}\otimes\nu}
\end{align}
 where $U$ is as in \eqref{eq:U}.  

\begin{proof} 
Fix~$\hat f\in\rmC^1_c(\R^+\times M)$. For every~$\phi\in\rmC^\infty_c(M)$ we have
\begin{align*}
(\partial_\phi \hat f^\trid)_\mu \eqdef&\ \diff_t\restr{t=0} \hat f^\trid\tparen{e^{t\phi}\mu} = \int \diff_t\restr{t=0} \braket{ \hat f\tparen{e^{t\phi(x)}\mu_x,x} e^{t\phi(x)} }\diff\mu(x)
\\
=& \int \Bigg[e^{t\phi(x)} \mu_x \phi(x)\hat f'\tparen{e^{t\phi(x)}\mu_x,x} e^{t\phi(x)}
\\
&+\hat f\tparen{e^{t\phi(x)}\mu_x,x}\phi(x) e^{t\phi(x)} \Bigg]\restr{t=0} \diff\mu(x)
\\
=& \int \braket{\mu_x \hat f'(\mu_x,x)+ \hat f(\mu_x,x)}\phi(x)\, \diff\mu(x) \fstop
\end{align*}
We used that, since~$\hat f\in \rmC^1_0(\emparg,x)$ has compact support in~$\R^+$ away from~$0$ uniformly in~$x\in M$, and since the total mass of~$\mu$ is finite, every integral above is in fact a finite sum, hence we may freely differentiate under integral sign.

By arbitrariness of~$\phi\in \rmC^\infty_c(M)$ and density of~$\rmC^\infty_c(M)$ in~$T^\ver_\mu$ for every~$\mu\in\mcM(M)$, we have
\begin{align*}
(\gradH \hat f^\trid)_\mu(x) =& \braket{ \mu_x \hat f'(\mu_x,x)+ \hat f(\mu_x,x)} \comma  
\end{align*}
 which clearly belongs to $L^2(M, \mu)$.  

The equality in~\eqref{eq:l:TridVerGrad:1} for~$\hat u$ readily follows from the above equality and the standard chain rule.
The equality in~\eqref{eq:l:TridVerGrad:2} holds by comparison of~\eqref{eq:l:TridVerGrad:1} and~\eqref{eq:l:LaplacianVerAux:1}  also using the fact that $\mu_x =0$ for $\LP{\theta,\nu} \otimes \nu$-a.e.~$(\mu, x)$.  
\end{proof}
\end{lemma}

\begin{proof}[Proof of Proposition~\ref{p:ClosVer}]
Fix~$\hat u=F\circ\hat\mbff^\trid, \hat v=G\circ \hat\mbfg^\trid\in \hFC{\infty}{c}{\infty}{c,0}$ with $F\colon \R^{k+1}\to \R$ and $G\colon \R^{h+1} \to \R$.
Then, by definition of~$\mcE^\ver$ and by~\eqref{eq:l:TridVerGrad:1},
\begin{align*}
\mcE^\ver&(\hat u,\hat v)  \eqdef  \int  \tscalar{(\gradH \hat u)_\mu}{(\gradH \hat v)_\mu}_{T^\ver_\mu}\diff\LP{\theta,\nu}(\mu)
\\
=& \int \int_M \Bigg[ \sum_{i=0}^k (\partial_i F)(\hat\mbff^\trid(\mu)) \tbraket{\mu_x \hat f_i'(\mu_x,x) +\hat f_i(\mu_x,x)}
\\
&\qquad\qquad\qquad\qquad \cdot \sum_{j=0}^h (\partial_j G)(\hat\mbfg^\trid(\mu)) \tbraket{\mu_x\hat g_j'(\mu_x,x) +\hat g_j(\mu_x,x)} \Bigg] \diff\mu(x)\, \diff\LP{\theta,\nu}(\mu)
\intertext{hence, by Proposition~\ref{p:MeckeLP} and by~\eqref{eq:l:LaplacianVerAux:1},}
=&\ \theta \int \int_M  \int_0^\infty \Bigg[ \sum_{i=0}^k (\partial_i F)(\hat\mbff^\trid(\mu+r\delta_x)) \tbraket{r\hat f_i'(r,x) +\hat f_i(r,x)}
\\
&\qquad\qquad\qquad \cdot \sum_{j=0}^h (\partial_j G)(\hat\mbfg^\trid(\mu+r\delta_x)) \tbraket{r\hat g_j'(r,x) +\hat g_j(r,x)} \Bigg] \diff r\,\diff\nu(x)\, \diff\LP{\theta,\nu}(\mu)
\\
=&\ \theta \int \int_M \int_0^\infty \partial_r \hat u(\mu+r\delta_x)\, \partial_r \hat v(\mu+r\delta_x)\, \diff r\, \diff\nu(x)\, \diff\LP{\theta,\nu}(\mu) \fstop
\end{align*}
Integrating by parts on~$\R^+$ and applying both~\eqref{eq:l:LaplacianVerAux:1} and~\eqref{eq:l:LaplacianVerAux:2},
\begin{align*}
=&\ \theta \int \int_M \tbraket{\hat u(\mu+r\delta_x)\, \partial_r \hat v(\mu+r\delta_x)}\Bigg\lvert^{r=\infty}_{r=0} \, \diff\nu(x)\,\diff \LP{\theta,\nu}(\mu)
\\
&\qquad  -\theta \int \int_M \int_0^\infty \hat u(\mu+r\delta_x)\, \partial_s^2\restr{s=r} \hat v(\mu+s\delta_x) \, \diff r\,\diff\nu(x)\,\diff \LP{\theta,\nu}(\mu)
\\
 =  &\ -\theta \int \hat u(\mu) \int_M \partial_s\restr{s=0} \hat v(\mu+s\delta_x) \diff\nu(x)\,\diff \LP{\theta,\nu}(\mu)
\\
&\qquad -\theta \int \int_M \int_0^\infty \hat u(\mu+r\delta_x)\, \partial^2_s\restr{s=r} \hat v(\mu+s\delta_x)\, \diff r\, \diff\nu(x)\,\diff \LP{\theta,\nu}(\mu)
\\
=&\ -\theta \int \hat u(\mu) \int_M \partial_s\restr{s=0} \hat v(\mu+s\delta_x) \diff\nu(x)\,\diff \LP{\theta,\nu}(\mu)
\\
&\qquad -\theta \int \int_M \int_0^\infty \hat u(\mu+r\delta_x)\, \partial^2_s\restr{s=r} \hat v\tparen{(\mu+r\delta_x)+s\delta_x-r\delta_x)} \,\diff r\,\diff\nu(x)\,\diff \LP{\theta,\nu}(\mu)
\\
=&\ -\theta \int \hat u(\mu) \int_M \partial_s\restr{s=0} \hat v(\mu+s\delta_x) \diff\nu(x)\,\diff \LP{\theta,\nu}(\mu)
\\
&\qquad -\int \hat u(\mu)  \int_M \partial^2_s\restr{s=\mu_x} \hat v(\mu+s\delta_x-\mu_x\delta_x)\, \diff\mu(x)\,\diff \LP{\theta,\nu}(\mu) \fstop
\end{align*}
where we used that~$\hat u$ vanishes outside a ball of the origin in~$\mcM(M)$ to cancel the boundary term at~$r=\infty$.  
Furthermore, for $\LP{\theta,\nu}$-a.e.~$\mu$,
\begin{align*}
\int_M &\partial_s\restr{s=0} \hat v (\mu+s\delta_x) \diff\nu(x) 
\\
&= \int_M \sum_i^h (\partial_i G)\tparen{\hat\mbfg^\trid(\mu
)}
\partial_s\restr{s=0} \braket{\int \hat g_i(\mu_y,y) \diff\mu_{\setminus x}(y) + (\mu_x+s) \hat g_i(\mu_x+s,x) } \diff\nu(x)
\\
&= \int_M \sum_i^h (\partial_i G)\tparen{\hat\mbfg^\trid(\mu)} \braket{\mu_x \hat g_i'(\mu_x,x) + \hat g_i(\mu_x,x)} \diff\nu(x)
\\
&= \int_M \sum_i^h (\partial_i G)\tparen{\hat\mbfg^\trid(\mu)} \braket{\hat g_i(0,x)} \diff\nu(x) \comma
\end{align*}
since~$\mu_x=0$ for~$\LP{\theta,\nu}\otimes \nu$-a.e.~$(\mu,x)$ by Lemma~\ref{l:Negligible}.
Since~$\hat g_i(0,x)\equiv 0$ for every~$\hat g_i$ and every~$x\in M$ by definition of~$\hat v=G\circ\hat\mbfg^\trid\in \hFC{\infty}{c}{\infty}{c,0}$, we conclude that the boundary term at~$r=0$ vanishes too, $\LP{\theta,\nu}$-a.s.
Thus, cancelling this boundary term in the above chain of inequality we obtain the desired assertion 
\begin{align*}
\mcE^\ver(\hat u,\hat v)= \tscalar{\hat u}{-\widehat \mcL^\ver \hat v}_{L^2(\LP{\theta,\nu})}  \fstop
\end{align*}

Since~$\hFC{\infty}{c}{\infty}{c,0}$ is dense in~$L^2(\LP{\theta,\nu})$ by Lemma~\ref{l:L2DensityCylinder}\ref{i:l:L2DensityCylinder:2}, the above equality shows at once that $\widehat \mcL^\ver \hFC{\infty}{c}{\infty}{c,0}\subset L^2(\LP{\theta,\nu})$ and that $\tparen{\widehat \mcL^\ver, \hFC{\infty}{c}{\infty}{c,0}}$ generates~$\tparen{\mcE^\ver,\hFC{\infty}{c}{\infty}{c,0}}$.
Thus,~$\tparen{\mcE^\ver,\hFC{\infty}{c}{\infty}{c,0}}$ is closable and its closure~$\tparen{\widehat\mcE^\ver,\dom{\widehat\mcE^\ver}}$ is generated by the Friedrichs extension of $\tparen{\widehat \mcL^\ver, \hFC{\infty}{c}{\infty}{c,0}}$ by~\cite[Thm.~X.23]{ReeSim75}.
\end{proof}

\subsubsection{Horizontal differentiability of extended cylinder functions}\label{sss:ClosabilityHor}
Define an operator~$\tparen{\widehat \mcL^\hor,\hFC{\infty}{c}{\infty}{c,0}}$ by setting  for every~$\mu \in \mcM(M)$  
\begin{align*}
(\widehat \mcL^\hor \hat u)_\mu\eqdef& \int_M \frac{\Delta_z\restr{z=x} \hat u(\mu+\mu_x\delta_z-\mu_x\delta_x)}{{\mu_x}^2} \diff\mu(x)
\\
&\qquad + \scalar{\paren{\frac{\nabla\log\rho}{\ev}}_{\mu}}{(\gradW \hat u)_{\mu}}_{T^\hor_{\mu}\mcM(M)} \comma \quad \hat u\in\hFC{\infty}{c}{\infty}{c,0}\fstop
\end{align*}

\begin{proposition}\label{p:ClosHor}
The form~$\tparen{\mcE^\hor,\hFC{\infty}{c}{\infty}{c,0}}$  defined as
\begin{align*}
\mcE^\hor( \hat u,\hat v) &\eqdef \int \tscalar{(\gradW \hat u)_\mu}{(\gradW \hat v)_\mu}_{T^\hor\mu}\, \diff \LP{\theta,\nu}(\mu) 
\\
&= \int \int_M \tscalar{(\gradW \hat u)_\mu(x)}{(\gradW \hat v)_\mu(x)}_{g_x} \, \diff \mu(x) \, \diff \LP{\theta,\nu}(\mu)
\end{align*}  
is  well-defined and  closable on~$L^2(\mcM(M), \LP{\theta,\nu})$.
Its closure~$\tparen{\widehat\mcE^\hor,\dom{\widehat\mcE^\hor}}$ is a closed bilinear form with generator the Friedrichs extension on~$L^2(\mcM(M), \LP{\theta,\nu})$ of the operator~$\tparen{\widehat \mcL^\hor,\hFC{\infty}{c}{\infty}{c,0}}$.
\end{proposition}

In order to prove Proposition~\ref{p:ClosHor} we shall need several auxiliary results.

\begin{lemma}\label{l:LaplacianHorAux}
Fix~$\hat u=F\circ\hat\mbff^\trid\in \hFC{0}{c}{0}{c,0}$  with $F:\R^{k+1} \to \R$ .
Then,
\begin{enumerate}[$(i)$]
\item\label{i:l:LaplacianHorAux:2} if~$\hat u$ is in~$\hFC{1}{c}{1}{c,0}$, then for~$\LP{\theta,\nu}\otimes \nu$-a.e.~$(\mu,z)$ the function~$z\mapsto U(\mu,z,s)$  as in \eqref{eq:U}  is differentiable on~$M$ at~$z$, and, for every~$s\in\R^+$,
\begin{equation}\label{eq:l:LaplacianHorAux:1}
\nabla_z U(\mu,z,s) = s\sum_{i=0}^k (\partial_i F)\tparen{\hat\mbff^\trid(\mu+s\delta_z)} \nabla \hat f_i(s,z) 
 \as{\LP{\theta,\nu}\otimes\nu}\semicolon
\end{equation}

\item\label{i:l:LaplacianHorAux:3} if~$\hat u$ is in~$\hFC{2}{c}{2}{c,0}$, then for~$\LP{\theta,\nu}\otimes \nu$-a.e.~$(\mu,z)$ the function~$z\mapsto U(\mu,z,s)$  as in \eqref{eq:U}  is twice differentiable on~$M$ at~$z$, and, for every~$s\in\R^+$,
\begin{equation}\label{eq:l:LaplacianHorAux:2}
\begin{aligned}
\Delta_z U(\mu,z,s)=&\ s^2\sum_{i,j=0}^k (\partial^2_{ij} F)\tparen{\hat\mbff^\trid(\mu+s\delta_z)}\, g_z\tparen{\nabla\hat f_i(s,z),\nabla\hat f_j(s,z)}
\\
&\qquad + s\sum_{i=0}^k (\partial_i F)\tparen{\hat\mbff^\trid(\mu+s\delta_z)} \Delta \hat f_i(s,z) 
  \as{\LP{\theta,\nu}\otimes\nu} \fstop
\end{aligned}
\end{equation}
\end{enumerate}
 In particular, the  left-hand sides of~\eqref{eq:l:LaplacianHorAux:1} and~\eqref{eq:l:LaplacianHorAux:2} are Borel measurable.
\begin{proof}
\ref{i:l:LaplacianHorAux:2}
Fix~$\hat f\in\rmC^1_c(\R^+\times M)$.
Let~$\mcN$ be as in~\eqref{eq:Negligible} and recall that it is~$\LP{\theta,\nu}\otimes\nu$-negligible by Lemma~\ref{l:Negligible}. For every~$(\mu,z)\not\in\mcN$, for every~$s\in\R^+$,
\begin{align*}
\nabla_z \tbraket{\hat f^\trid(\mu+s\delta_z)}=&\ \nabla_z \braket{\int \hat f\tparen{\mu_y+s\car_{z}(y),y}\, \diff\mu(y) + s\hat f\tparen{\mu_z+s,z}}
\intertext{hence, by definition of~$\mcN$, continuing the above chain of equalities,}
=&\ \nabla_z \braket{\int \hat f\tparen{\mu_y,y}\, \diff\mu(y) + s\hat f\tparen{s,z}}
\\
=&\ s\nabla\hat f(s,z)\fstop
\end{align*}

On the complement of the $\LP{\theta,\nu}\otimes\nu$-negligible set~$\mcN$, the equality in~\eqref{eq:l:LaplacianHorAux:1} for~$U$ readily follows from the above equality and the standard chain rule.
\ref{i:l:LaplacianHorAux:3} A proof is similar to the one of~\ref{i:l:LaplacianHorAux:2} and therefore it is omitted.
\end{proof}
\end{lemma}

\begin{lemma}\label{l:TridHorGrad}
Fix~$\hat u\in\hFC{1}{c}{1}{c,0}$.
Then, for every~$(\mu,x)\in \mcM(M)\times M$ there exists
\begin{equation}\label{eq:l:TridHorGrad:1}
\begin{aligned}
(\gradW \hat u)_\mu(x) =&\ \sum_{i=0}^k (\partial_i F)\tparen{\hat\mbff^\trid(\mu)} \nabla \hat f_i(\mu_x,x)
\end{aligned}
\end{equation}
and
\begin{align}\label{eq:l:TridHorGrad:2}
(\gradW \hat u)_\mu(x)= \mu_x^{-1} \nabla_z\restr{z=x} U(\mu,z,\mu_x) \as{\LP{\theta,\nu}\otimes\nu}
\end{align}
 where $U$ is as in \eqref{eq:U}. 

\begin{proof} 
Fix~$\hat f\in\rmC^1_c(\R^+\times M)$. For every~$w\in\mfX^\infty_c(M)$ we have
\begin{align*}
(\partial_w \hat f^\trid)_\mu \eqdef&\ \diff_t\restr{t=0} \hat f^\trid\tparen{{\psi^w_t}_\pfwd\mu} = \diff_t\restr{t=0} \int \hat f\tparen{({\psi^w_t}_\pfwd\mu)_y,y} \diff {\psi^w_t}_\pfwd\mu(y)
\\
=&\ \diff_t\restr{t=0} \int \hat f\tparen{({\psi^w_t}_\pfwd\mu)_{\psi^w_t(y)},\psi^w_t(y)} \diff \mu(y)
\\
=&\ \int \diff_t\restr{t=0} \hat f\tparen{\mu_y,\psi^w_t(y)} \diff \mu(y)
\\
 =  & \int g_y( \nabla \hat f(\mu_y, y), w_y) \de \mu(y)  \fstop 
%
\end{align*}
We used that, since~$\hat f\in \rmC^1_0(\emparg,x)$ has compact support in~$\R^+$ away from~$0$ uniformly in~$x\in M$, and since the total mass of~$\mu$ is finite, every integral above is in fact a finite sum, hence we may freely differentiate under integral sign.

By arbitrariness of~$w\in \mfX^\infty_c(M)$ and density of~$\mfX^\infty_c(M)$ in~$T^\hor_\mu\mcM(M)$ for every~$\mu\in\mcM(M)$, we have
\begin{align*}
(\gradW \hat f^\trid)_\mu(x) = \nabla \hat f(\mu_x,x) \fstop
\end{align*}

The equality in~\eqref{eq:l:TridHorGrad:1} for~$\hat u$ readily follows from the above equality and the standard chain rule.
The equality in~\eqref{eq:l:TridHorGrad:2} holds by comparison of~\eqref{eq:l:TridHorGrad:1} and~\eqref{eq:l:LaplacianHorAux:1}  also using the fact that $\mu_x =0$ for $\LP{\theta,\nu} \otimes \nu$-a.e.~$(\mu, x)$. 
\end{proof}
\end{lemma}

\begin{proof}[Proof of Proposition~\ref{p:ClosHor}]
Fix~$\hat u=F\circ\hat\mbff^\trid, \hat v=G\circ \hat\mbfg^\trid\in \hFC{\infty}{c}{\infty}{c,0}$  with $F\colon \R^{k+1}\to \R$ and $G\colon \R^{h+1} \to \R$.
Then, by definition of~$\mcE^\hor$ and by~\eqref{eq:l:TridHorGrad:1},
\begin{align*}
\mcE^\hor&(\hat u,\hat v) \eqdef \int \tscalar{(\gradW \hat u)_\mu}{(\gradW \hat v)_\mu}_{T^\hor_\mu}\diff\LP{\theta,\nu}(\mu)
\\
=& \int \int_M g_{x}\Bigg(\sum_{i=0}^k (\partial_i F)(\hat\mbff^\trid(\mu)) \nabla \hat f_i(\mu_x,x),
\\
&\qquad\qquad\qquad\qquad \sum_{j=0}^h (\partial_j G)(\hat\mbfg^\trid(\mu)) \nabla\hat g_j(\mu_x,x) \Bigg) \diff\mu(x)\, \diff\LP{\theta,\nu}(\mu)
\intertext{hence, by Proposition~\ref{p:MeckeLP} and by~\eqref{eq:l:LaplacianHorAux:1},}
=&\ \theta \int \int_M \int_0^\infty r^{-2}\, g\tparen{\nabla_x \hat u(\mu+r\delta_x), \nabla_x \hat v(\mu+r\delta_x)} \diff r\, \diff\nu(x)\, \diff\LP{\theta,\nu}(\mu) \fstop
\end{align*}
Now, we integrate by parts on~$M$.
Since~$M$ is boundaryless, no boundary term appears.
Since the~$\hat f_i$'s and the~$\hat g_j$'s are compactly supported in the $M$-variable, no boundary term at the infinity of~$M$ appears either.
Thus, continuing the above chain of equalities
\begin{align*}
=& -\theta \int \int_M \int_0^\infty r^{-2}\, \hat u(\mu+r\delta_x)\,  \div_z  \restr{z=x} \braket{ \rho(z)\, \nabla_z\hat v(\mu+ r  \delta_z) }\, \diff r\,\diff\vol_g(x)\,\diff \LP{\theta,\nu}(\mu)
\\
=& -\theta \int \int_M \int_0^\infty r^{-2}\, \hat u(\mu+r\delta_x)\Big[ g\paren{\nabla_z\restr{z=x} \rho(z), \nabla_z\restr{z=x}\hat v(\mu+r\delta_z)} 
\\
&\qquad\qquad\qquad\qquad\qquad\qquad\qquad+ \rho(x)\,\Delta_z\restr{z=x} \hat v(\mu+r\delta_z)\Big]\, \diff r\,\diff\vol_g(x)\,\diff \LP{\theta,\nu}(\mu)
\\
=& -\theta \int \int_M \int_0^\infty r^{-2}\, \hat u(\mu+r\delta_x)\Big[ g\paren{\nabla_z\restr{z=x} \log\tparen{\rho(z)}, \nabla_z\restr{z=x}\hat v(\mu+r\delta_z)} 
\\
&\qquad\qquad\qquad\qquad\qquad\qquad\qquad+ \Delta_z\restr{z=x} \hat v(\mu+r\delta_z)\Big]\, \diff r\,\diff\nu(x)\,\diff \LP{\theta,\nu}(\mu)
\\
=& -\theta \int \int_M \int_0^\infty r^{-2}\, \hat u(\mu+r\delta_x)
\\
&\qquad\qquad\cdot\Big[ g\paren{\nabla_z\restr{z=x} \log\tparen{\rho(z)}, \nabla_z\restr{z=x}\hat v\tparen{(\mu+r\delta_x)+r\delta_z-r\delta_x}} 
\\
&\qquad\qquad\qquad+ \Delta_z\restr{z=x} \hat v\tparen{(\mu+r\delta_x)+r\delta_z-r\delta_x}\Big]\, \diff r\,\diff\nu(x)\,\diff \LP{\theta,\nu}(\mu) \comma
\end{align*}
and, applying again Proposition~\ref{p:MeckeLP},
\begin{align*}
=&\ - \int \hat u(\mu)\int_M {\mu_x}^{-2} \Big[ g\tparen{(\nabla\log\rho)_x, \nabla_z\restr{z=x} \hat v(\mu+\mu_x\delta_z-\mu_x\delta_x)}
\\
&\qquad\qquad\qquad\qquad\qquad\qquad+ \Delta_z\restr{z=x} \hat v(\mu+\mu_x\delta_z-\mu_x\delta_x) \Big] \diff\mu(x)\,\diff \LP{\theta,\nu}(\mu)
\\
=&\ \tscalar{\hat u}{-\widehat \mcL^\hor \hat v}_{L^2(\LP{\theta,\nu})} \fstop
\end{align*}

Since~$\hFC{\infty}{c}{\infty}{c,0}$ is dense in~$L^2(\LP{\theta,\nu})$ by Lemma~\ref{l:L2DensityCylinder}\ref{i:l:L2DensityCylinder:2}, the above chain of equalities shows at once that~$\widehat \mcL^\hor \hFC{\infty}{c}{\infty}{c,0}\subset L^2(\LP{\theta,\nu})$ and that $\tparen{\widehat \mcL^\hor, \hFC{\infty}{c}{\infty}{c,0}}$ generates~$\tparen{\mcE^\hor,\hFC{\infty}{c}{\infty}{c,0}}$.
Thus,~$\tparen{\mcE^\hor,\hFC{\infty}{c}{\infty}{c,0}}$ is closable and its closure~$\tparen{\widehat\mcE^\hor,\dom{\widehat\mcE^\hor}}$ is generated by the Friedrichs extension of $\tparen{\widehat \mcL^\hor, \hFC{\infty}{c}{\infty}{c,0}}$ by~\cite[Thm.~X.23]{ReeSim75}.
\end{proof}

\subsubsection{Extended form}
In light of Lemmas~\ref{l:TridVerGrad} and~\ref{l:TridHorGrad},
\[
\widehat\boldGamma(\hat u,\hat v)_\mu\eqdef \int \braket{g\tparen{(\boldnabla^\hor \hat u)_\mu , (\boldnabla^\hor \hat v)_\mu} +4 (\boldnabla^\ver \hat u)_\mu (\boldnabla^\ver \hat v)_\mu }\diff\mu
\]
is a well-defined bilinear form on~$\hFC{\infty}{c}{\infty}{c,0}$.
Analogously,
\[
(\widehat \mcL \hat u)_\mu \eqdef (\widehat \mcL^\hor \hat u)_\mu +  4   (\widehat \mcL^\ver \hat u)_\mu
\]
is a well-defined operator~$\tparen{\widehat \mcL,\hFC{\infty}{c}{\infty}{c,0}}$.

\begin{proposition}\label{p:ExtendedForm}
Fix~$\theta>0$.
The bilinear form~$\tparen{\mcE,\hFC{\infty}{c}{\infty}{c,0}}$ defined by
\[
\mcE(\hat u,\hat v)\eqdef \int \widehat\boldGamma(\hat u,\hat v)\, \diff\LP{\theta,\nu} \comma \qquad \hat u,\hat v\in \hFC{\infty}{c}{\infty}{c,0}\comma
\]
is closable on~$L^2(\LP{\theta,\nu})$. Its closure~$\tparen{\widehat\mcE,\dom{\widehat\mcE}}$ is a Dirichlet form on~$L^2(\LP{\theta,\nu})$ generated by the Friedrichs extension~$\tparen{\widehat \mcL,\dom{\widehat \mcL}}$ of the operator~$\tparen{\widehat \mcL,\hFC{\infty}{c}{\infty}{c,0}}$.

\begin{proof}
Since~$\hFC{\infty}{c}{\infty}{c,0}$ is dense in~$L^2(\LP{\theta,\nu})$ by Lemma~\ref{l:L2DensityCylinder}\ref{i:l:L2DensityCylinder:2}, by Propositions~\ref{p:ClosVer} and~\ref{p:ClosHor} the operator~$\tparen{ \widehat  \mcL  , \hFC{\infty}{c}{\infty}{c,0}}$ satisfies~$\widehat \mcL \hFC{\infty}{c}{\infty}{c,0}\subset L^2(\LP{\theta,\nu})$ and generates $\tparen{\mcE,\hFC{\infty}{c}{\infty}{c,0}}$.
Thus,~$\tparen{ \mcE  ,\hFC{\infty}{c}{\infty}{c,0}}$ is closable and its closure~$\tparen{  \widehat\mcE  ,\dom{ \widehat\mcE }}$ is generated by the Friedrichs extension of $\tparen{ \widehat \mcL , \hFC{\infty}{c}{\infty}{c,0}}$ by~\cite[Thm.~X.23]{ReeSim75}.

By the chain rule for~$\boldnabla$ consequence of Lemmas~\ref{l:TridVerGrad} and~\ref{l:TridHorGrad}, the Markov property holds on the Markovian core~$\hFC{\infty}{c}{\infty}{c,0}$.
This suffices to conclude the Markov property for~$\tparen{ \widehat\mcE ,\dom{ \widehat\mcE }}$ by, e.g.,~\cite[Prop.~I.4.10, p.~35]{MaRoe92}.
\end{proof}
\end{proposition}

Finally, let us show  the  relation between the domain~$\dom{\widehat\mcE}$ and standard cylinder functions.

\begin{lemma}\label{l:StandardCylinderSubset}
Fix~$\theta>0$. Then,~$\FC{1,1}{c,c}{1}{c}\subset \dom{\widehat\mcE}$.

\begin{proof}
We show that~$\FC{\infty,\infty}{c,c}{\infty}{c}\subset \dom{\widehat\mcE}$. The conclusion for less regular cylinder functions follows by approximation in a straightforward way.

For~$u=F\circ \mbff^\trid\in\FC{\infty,\infty}{c,c}{\infty}{c}$  with $F:\R^{k+1} \to \R$, let~$\hat u_n\in  \hFC{\infty}{c}{\infty}{c,0}  $ be the sequence of extended cylinder functions $L^2(\LP{\theta,\nu})$-converging to~$u$  defined as in \eqref{eq:approxrho} with $\varrho_n$ additionally satisfying   
\begin{equation}\label{eq:l:StandardCylinderSubset:0.5}
\car_{[2/n,n-1/n]} \leq \varrho_n \leq \car_{[1/n,n]} \qquad \text{and} \qquad \abs{\varrho_n'}\leq 2n \tparen{\car_{[1/n,2/n]}+\car_{[n-1/n,n]}}\fstop
\end{equation}

By a standard lower-semicontinuity argument, see e.g.~\cite[Lem.~I.2.12, p.~22]{MaRoe92}, it suffices to show that~$\sup_n \widehat\mcE( \hat u_n  )<\infty$.
By definition of~$\tparen{\widehat\mcE,\dom{\widehat\mcE}}$ we have $\sup_n\widehat\mcE(  \hat u_n )\leq \sup_n \widehat\mcE^\hor( \hat u_n )+  4  \sup_n \widehat\mcE^\ver( \hat u_n )$.  Let~$R> 0$ be so that~$\supp F\subset [0,R]^{k+1}$. Using \eqref{eq:l:TridHorGrad:1} and the fact that $|\varrho_n| \le 1$, it is not difficult to see that
\[ \sup_n \widehat\mcE^\hor( \hat u_n) \le k^2 \Li{F}^2 \max_{i\leq k} \norm{ \nabla f_i}_{g, \rmC^0}^2 R \, \LP{\theta,\nu}(\mcB_R).\]
We are left to  show that~$\sup_n \widehat\mcE^\ver( \hat u_n )<\infty$.

By~\eqref{eq:l:TridVerGrad:1},
\begin{align}
\nonumber
\widehat\mcE^\ver(\hat u_n) =&\  \int  \int_M \abs{\sum_{i=0}^k (\partial_i F)\tparen{\hat\mbff^\trid(\mu)} \tbraket{\mu_x \varrho_n'(\mu_x)+\varrho_n(\mu_x)} f_i(x)}^2 \diff\mu(x)\, \diff\LP{\theta,\nu}(\mu)
\\
\nonumber
\leq&\ \underbrace{  k^2\Li{F}^2 \max_{i\leq k} \norm{f_i}_{\rmC^0}^2}_{\defeq C_u} \int_{\mcB_R} \int_M \abs{\mu_x \varrho_n'(\mu_x)+\varrho_n(\mu_x)}^2 \diff\mu(x)\, \diff\LP{\theta,\nu}(\mu)
\\
\label{eq:l:StandardCyilinderSubset:1}
\leq&\ 2C_u \int \car_{\mcB_R}(\mu) \int_M \abs{\mu_x \varrho_n'(x)}^2\diff\mu(x)\, \diff\LP{\theta,\nu}(\mu) + 2C_u \int_{\mcB_R} \int_M\diff\mu \, \diff\LP{\theta,\nu}(\mu) 
\end{align}
where we used that~$\abs{\varrho_n}\leq 1$ in the last inequality.
The second summand above is bounded by~$2C_u R\, \LP{\theta,\nu}(\mcB_R)$, which is finite by Lemma~\ref{l:FinitenessBalls} and independent of~$n$.
As for the first summand, by the Mecke identity~\eqref{eq:MeckeLP}, and by~\eqref{eq:InfLebesgueIsomor},
\begin{align}
\nonumber
\int \car_{\mcB_R}(\mu) &\int_M \abs{\mu_x \varrho_n'(x)}^2\diff\mu(x)\, \diff\LP{\theta,\nu}(\mu)
\\
\nonumber
=&\ \theta \int \int_M \int_0^\infty \car_{\mcB_R}(\mu+s\delta_x) s^2\abs{\varrho_n'(s)}^2\diff s\, \diff\nu(x)\, \diff\LP{\theta,\nu}(\mu)
\\
\nonumber
= &\ \theta\,  \nu M \int_{\mcP(M)}  \int_0^\infty \int_0^\infty  \car_{[0,R]}(t+s) \abs{\varrho_n'(s)}^2 s^2\diff s\, \diff\lambda_\theta(t)\, \diff\DF{\nu}(\eta)
\\
\label{eq:l:StandardCyilinderSubset:2}
=&\ \theta \int_0^R \int_0^{R-t} s^2 \abs{\varrho_n'(s)}^2 \diff s\, \diff\lambda_\theta(t) \fstop
\end{align}
Then, by the estimate for~$\abs{\varrho_n'}$ in~\eqref{eq:l:StandardCylinderSubset:0.5} we have, for all~$n\geq R+1  $,
\begin{align*}
\int_0^R \int_0^{R-t} s^2 \abs{\varrho_n'(s)}^2 \diff s\, \diff\lambda_\theta(t) \leq&\ \lambda_\theta([0,R])\, (2n)^2\int_{1/n}^{2/n} s^2 \diff s
\\
\leq&\  10 n^{-1} \lambda_\theta([0,R]) =  10  n^{-1} \LP{\theta,\nu}(\mcB_R) \fstop
\end{align*}
Finally, combining this latter estimate with~\eqref{eq:l:StandardCyilinderSubset:1} and~\eqref{eq:l:StandardCyilinderSubset:2},
\begin{equation}
\sup_{n\geq R}  \widehat\mcE^\ver  (\hat u_n) \leq 2 C_u\, \LP{\theta,\nu}(\mcB_R) \braket{R+ 10  \,\theta} <\infty\comma
\end{equation}
which suffices to conclude the assertion.
\end{proof}
\end{lemma}

\subsection{Identification of the canonical form and the Cheeger energy}\label{sec:ident}
We now turn to the identification of the geometric and metric measure structures on~$\mcM(M)$.
In this section we prove our main results Theorems~\ref{t:IntroForm} and~\ref{t:CoincidenceIntro}, recalled below.

\begin{theorem}\label{t:Main}
Let~$(M,g,\nu)$ be a weighted Riemannian manifold as in the beginning of~\S\ref{s:CanonicalForm}.
Then, for every~$\theta>0$, 
\begin{enumerate}[$(i)$]
\item\label{i:t:Main:1} the form~$\tparen{\mcE, \FC{\infty,\infty}{c,c}{\infty}{c}}$ is densely defined and closable on~$L^2(\mcM(M),\LP{\theta,\nu})$;
\item\label{i:t:Main:2} its closure~$(\mcE,\dom{\mcE})$ is a (symmetric) conservative strongly local Dirichlet form on $L^2(\LP{\theta,\nu})$; 
\item\label{i:t:Main:3} $(\mcE,\dom{\mcE})$ is recurrent if~$\theta \in (0,1]$ and transient if~$\theta\in (1,\infty)$;
\item\label{i:t:Main:4} $(\mcE,\dom{\mcE})$ is quasi-regular and coincides with the Cheeger energy of the metric measure space $(\mcM(M), \HK_{\mssd_g},\LP{\theta,\nu})$;
\item\label{i:t:Main:5} $(\mcE,\dom{\mcE})$ is properly associated with an $\LP{\theta,\nu}$-invariant Hunt diffusion process~$\mu_\bullet$ with state space~$\mcM(M)$, recurrent if~$\theta\in (0,1]$ and transient if~$\theta\in (1,\infty)$.
\end{enumerate}
\end{theorem}

\begin{remark}
Assume~$\dim M\geq 2$.
Arguing similarly to the proof of~\cite[Cor.~5.19, Lem.~A.20]{LzDS17+}, it would be possible to show that~$\FC{\infty,\infty}{c,c}{\infty}{c}$ is in fact a \emph{form core} for~$\tparen{\widehat\mcE,\dom{\widehat\mcE}}$ so that, in fact, the forms~$\tparen{\mcE,\dom{\mcE}}$ and~$\tparen{\widehat\mcE,\dom{\widehat\mcE}}$ \emph{coincide}.
The details of this identification, instrumental to a complete description of the Vershik diffusion, will be addressed elsewhere.
\end{remark}

\begin{proof}[Proof of Theorem~\ref{t:Main}]
The proof will require several auxiliary results which are collected below and in Appendix~\ref{app:Bessel}.

\paragraph{Proof of~\ref{i:t:Main:1}}
The closability of~$\tparen{\mcE, \FC{\infty,\infty}{c,c}{\infty}{c}}$ follows from Proposition~\ref{p:ExtendedForm}.
Indeed, we have~$\FC{\infty,\infty}{c,c}{\infty}{c}\subset \dom{\widehat\mcE}$ by Lemma~\ref{l:StandardCylinderSubset}, and closability is inherited by restrictions.
It is shown in Lemma~\ref{l:L2DensityCylinder}\ref{i:l:L2DensityCylinder:1.5} that~$\FC{\infty,\infty}{c,c}{\infty}{c}$ is dense in $L^2(\LP{\theta,\nu})$, that is~$\tparen{\mcE, \FC{\infty,\infty}{c,c}{\infty}{c}}$ is also densely defined.

\paragraph{Proof of~\ref{i:t:Main:2} and~\ref{i:t:Main:3}}
Since~$\FC{\infty,\infty}{c,c}{\infty}{c}$ is closed under post-compos\-ition with smooth functions,~$\dom{\mcE}$ is a Markovian subspace of~$L^2(\LP{\theta,\nu})$, and therefore, for every~$u\in \FC{\infty,\infty}{c,c}{\infty}{c}$,
\[
u^+\wedge 1\in \dom{\mcE}\quad \text{and} \quad \mcE(u^+\wedge 1)=\widehat\mcE(u^+\wedge 1)\leq \widehat\mcE(u)=\mcE(u)\fstop 
\]
By, e.g.~\cite[Prop.~I.4.10]{MaRoe92}, this suffices to establish that~$\tparen{\mcE,\dom{\mcE}}$ has the Dirichlet property.

Conservativeness is shown in Lemma~\ref{l:Conservativeness}.
Locality follows from the diffusion property for~$\tparen{\mcE, \FC{\infty,\infty}{c,c}{\infty}{c}}$ which is readily verified, or else from the identification with the Cheeger energy in~\ref{i:t:Main:4}.
Strong locality follows from locality and conservativeness, see e.g.~\cite[Thm.~4.5.4, p.~187]{FukOshTak11}.

Recurrence if~$\theta\in (0,1]$ holds combining Lemma~\ref{l:Recurrence} and Lemma~\ref{l:CBIRecurrence}.
Transience if~$\theta \in (1,\infty)$ holds combining Lemma~\ref{l:Transient} and Lemma~\ref{l:CBIRecurrence}.

\paragraph{Proof of~\ref{i:t:Main:4}}
Let~$\X\eqdef \tparen{\mcM(M),\HK_{\mssd_g},\LP{\theta,\nu}}$.
On the one hand, it follows from Proposition~\ref{p:UniquenessLP}, Corollary~\ref{c:DensityCylC} and Proposition~\ref{prop:equalityriem} that~$\CE_{2,\X}$ extends~$(\mcE,\dom{\mcE})$.
On the other hand,~$H^{1,2}(\X)$ coincides with the closure of~$\FC{\infty,\infty}{c,c}{\infty}{c}$ by Corollary~\ref{c:DensityCylC}, thus it must be~$\tparen{\mcE,\dom{\mcE}}=\tparen{\CE_{2,\X},H^{1,2}(\X)}$.
Since the Cheeger energy of a ($\sigma$-finite Radon) metric measure space is always quasi-regular by~\cite[Prop.~3.21]{LzDSSuz20} (also cf.\ the proof of~\cite[Thm.~4.1]{Sav14}) and local, we conclude that~$(\mcE,\dom{\mcE})$ is quasi-regular and local.

\paragraph{Proof of~\ref{i:t:Main:5}}
All assertions follow from~\ref{i:t:Main:1}-\ref{i:t:Main:4} and standard arguments in the theory of Dirichlet forms.
\end{proof}

Let us now turn to the proof of recurrence, transience and of conservativeness, which will complete the proof of points~\ref{i:t:Main:2}, \ref{i:t:Main:3} and~\ref{i:t:Main:5} in Theorem~\ref{t:Main}. Note that we will however make use of point \ref{i:t:Main:1} in Theorem \ref{t:Main} and that $(\mcE,\dom{\mcE})$ is a (symmetric) local Dirichlet form on~$L^2(\LP{\theta,\nu})$: both facts have been proven above without relying on recurrence, transience, conservativeness and point~\ref{i:t:Main:5} in Theorem~\ref{t:Main}.

In the rest of this section we will make use of the following classical characterizations of recurrence and conservativeness, see e.g.~\cite[Thm.~1.6.3(ii), p.~58]{FukOshTak11} and~\cite[Thm.~1.6.6, p.~63]{FukOshTak11}, which in fact hold in full generality without the need for topological assumptions.

A Dirichlet form~$\tparen{E,\dom{E}}$ on~$L^2(\mssm)$ is recurrent if and only if there exist functions~$u_n\in\dom{E}$ such that 
\begin{subequations}\label{eq:l:Recurrence:1}
\begin{gather}
\label{eq:l:Recurrence:1a}
0\leq u_n\leq 1 \quad \text{and} \quad \lim_n u_n =1 \as{\mssm}\comma
\\
\label{eq:l:Recurrence:1b}
\lim_n E(u_n)=0 \fstop
\end{gather}
\end{subequations}

A Dirichlet form~$\tparen{E,\dom{E}}$ on~$L^2(\mssm)$ is conservative if and only if there exist functions~$u_n\in\dom{E}$ with
\begin{gather}\label{eq:Conservative:1}
0\leq u_n\leq 1 \quad \text{and} \quad \lim_n u_n =1 \as{\mssm}\comma
\end{gather}
such that one of the following equivalent conditions holds.
Either
\begin{subequations}
\begin{gather}\label{eq:Conservative:1b}
\lim_n E(u_n,v)=0 \comma \qquad v\in \dom{E}\cap L^1(\mssm)\semicolon
\end{gather}
or there exists~$f\in L^1(\mssm)\cap L^2(\mssm)$ with~$f>0$ $\mssm$-a.e., and~$\alpha>0$, such that
\begin{gather}
\label{eq:Conservative:1c}
\lim_n E(u_n,v)=0 \comma \qquad v=G_\alpha f  \comma 
\end{gather}
\end{subequations}
 where $\seq{G_\alpha}_\alpha$ is the resolvent of~$\tparen{E,\dom{E}}$.

\subsubsection{Radial projection and intertwining}
Let~$\emparg^\rad\colon L^2(\LP{\theta,\nu})\to L^2(\LP{\theta,\nu})$ be the radial projection
\[
\emparg^\rad\colon u\longmapsto \paren{u^\rad\colon \mu \mapsto \int u\tparen{\car^\trid(\mu)\cdot \eta}\, \diff \DF{\nu}(\eta)} \comma
\]
and~$(\mcE^\rad,\dom{\mcE^\rad})$ be the \emph{radial part} of~$(\mcE,\dom{\mcE})$, i.e.\ the quadratic form defined by
\[
\mcE^\rad(u) \eqdef   \mcE(u) \comma \qquad u\in \dom{\mcE^\rad}\eqdef \set{u\in \dom{\mcE} : u=u^\rad} \fstop
\]
Since~$\DF{\nu}$ is a probability measure,~$\emparg^\rad$ is a projection operator (in particular: idempotent), acting as the identity on~$\dom{\mcE^\rad}$.  Note also that
$\norm{u^\rad}_{L^2(\LP{\theta,\nu})}\leq \norm{u}_{L^2(\LP{\theta,\nu})}$ for every~$u\in L^2(\LP{\theta,\nu})$ so that $\emparg^\rad$ is continuous.  

Further denote by~$L^2(\LP{\theta,\nu})^\rad$ the image of~$L^2(\LP{\theta,\nu})$ via~$\emparg^\rad$.
Then, it is readily verified that
\[
L^2(\LP{\theta,\nu})^\rad\cong L^2(\lambda_\theta) \comma \qquad \theta>0  \comma  
\]
 where the isomorphism is simply given by 
\begin{equation}\label{eq:OperatorTilde}
\tilde\emparg\colon L^2(\LP{\theta,\nu})^\rad \to L^2(\lambda_\theta)  \comma \qquad u\longmapsto \tilde u : u =  \tilde{u} \circ \car^\trid \fstop
\end{equation}

\begin{lemma}\label{l:RadGrad}
If~$u\in\dom{\mcE}$, then~$u^\rad\in\dom{\mcE^\rad}$ and
\begin{gather}
\nonumber
\gradW u^\rad \equiv 0\comma 
\\
\label{eq:l:RadGrad:0}
\begin{aligned}
(\grad u^\rad)_\mu &= (\gradH u^\rad)_\mu
\\
&= \int \int (\gradH u)_{\mu M \eta} \diff\eta\, \diff\DF{\nu}(\eta)  \quad \text{in}\quad T^\ver_\mu\mcM(M) \forallae{\LP{\theta,\nu}} \mu \fstop  
\end{aligned}
\end{gather}
Moreover, $\mcE(u^\rad) = 4\mcE^\ver(u^\rad) \le 4\mcE^\ver(u) \le\mcE(u)$ for every $u \in \dom{\mcE}$. In particular $\mcE^\rad$ is densely defined.

\begin{proof}
 The assertion regarding the horizontal gradient is obvious and holds for every $u \in \dom{\mcE}$. Let $u \in \dom{\mcE}$ and observe that the Borel function $(y, \mu, \eta) \mapsto \sqrt{\mu M}(\gradH u)_{\mu M \eta} (y)$ satisfies
\begin{align*}
\int \int_{\mcP}& \int_M \mu M | (\gradH u)_{\mu M \eta} (y)|^2\, \diff \eta(y)\, \diff \DF{\nu}(\eta)\, \diff \LP{\theta,\nu}(\mu) =
\\
&=\int \int_M \int_{\mcP} \int_M | (\gradH u)_{\mu M \eta} (y)|^2\, \diff \eta(y)\, \diff \DF{\nu}(\eta)\, \diff \mu(x)\, \diff \LP{\theta,\nu}(\mu)
\\
&= \int_0^\infty \int_{\mcP} \int_M \int_{\mcP} \int_M| (\gradH u)_{t \eta} (y)|^2\, \diff \eta(y)\, \diff \DF{\nu}(\eta)\, \diff (t \eta')(x)\, \diff \DF{\nu}(\eta')\, \diff \lambda_\theta(t)
\\
&= \int_0^\infty \int_{\mcP} \int_M | (\gradH u)_{t \eta} (y)|^2\, \diff (t\eta)(y)\, \diff \DF{\nu}(\eta)\, \diff \lambda_\theta(t)
\\
&= \int \int_M | (\gradH u)_{\mu} (y)|^2\, \diff \mu (y)\, \diff \LP{\theta,\nu}(\mu)
\\
&= \mcE^\ver(u).
\end{align*}
Fubini's theorem gives that the  function
\[ \mu \mapsto H(\mu;u) \eqdef \sqrt{\mu M} \int \int (\gradH u)_{\mu M \eta} \diff\eta\, \diff\DF{\nu}(\eta)\]
is Borel measurable and finite for $\LP{\theta,\nu}$-a.e.~$\mu$; the same holds in particular for the right hand side of \eqref{eq:l:RadGrad:0}. Jensen's inequality also gives that $\|H(\emparg; u)\|^2_{L^2(\LP{\theta,\nu})}$ is bounded by $ \mcE^\ver(u)$.  
Let us now show that~\eqref{eq:l:RadGrad:0} holds for cylinder functions;  let~$u\eqdef F\circ \mbff^\trid\in\FC{\infty,\infty}{c,c}{\infty}{c}$  with $F: \R^{k+1} \to \R$. 
Then, for every~$f\in\rmC^\infty_c(M)$,
\begin{align*}
\diff_t\restr{t=0}  u^\rad \tparen{\rme^{tf}\mu} &= \diff_t\restr{t=0} \int u\tparen{(\rme^{tf}\mu)(M) \eta}\diff\DF{\nu}(\eta)
\\
&= \int \diff_t\restr{t=0} u\tparen{(\rme^{tf}\mu)(M)\eta} \diff\DF{\nu}(\eta)
\\
&= \int \sum_{i=0}^{k} (\partial_i F)(\mu M\mbff^\trid\eta)\, \diff_t\restr{t=0} f_i^\trid \tparen{(\rme^{tf} \mu)(M)\eta} \diff\DF{\nu}(\eta)
\\
&= \int \sum_{i=0}^{k} (\partial_i F)(\mu M\mbff^\trid\eta)\, \diff_t\restr{t=0} f_i^\trid \tparen{(\rme^{tf} \mu)(M)\eta} \diff\DF{\nu}(\eta)
\\
&= \int \sum_{i=0}^{k} (\partial_i F)(\mu M\mbff^\trid\eta)\, f_i^\trid\eta\, \diff_t\restr{t=0} \tparen{(\rme^{tf} \mu)(M)} \diff\DF{\nu}(\eta)
\\
&= \int \sum_{i=0}^{k} (\partial_i F)(\mu M\mbff^\trid\eta)\, f_i^\trid\eta\, f^\trid \mu\, \diff\DF{\nu}(\eta)
\\
&= \int f(x) \int \sum_{i=0}^{k} (\partial_i F)(\mu M\mbff^\trid\eta)\, f_i^\trid\tparen{\eta}\, \diff\DF{\nu}(\eta)\, \diff \mu(x)
\\
&= \int \braket{\int \int (\gradH u)_{\mu M \eta}\diff\eta\, \diff\DF{\nu}(\eta)} f\diff\mu 
\\
&= \scalar{\int \int (\gradH u)_{\mu M \eta} \diff\eta\, \diff\DF{\nu}(\eta)}{f}_{T^\ver_\mu} \fstop
\end{align*}
By arbitrariness of~$f\in\rmC^\infty_c(M)$ we conclude that  $u^\rad \in \dom{\mcE}$ for every cylinder function $u$ and that~\eqref{eq:l:RadGrad:0} holds for cylinder functions. In particular
\begin{equation}\label{eq:cylcontr}
\begin{aligned}
\mcE(u^\rad) =&\ 4\mcE^\ver(u^\rad) = 4\| H(\emparg; u) \|^2_{L^2(\LP{\theta,\nu})} \le 4\mcE^\ver(u) 
\\
\le&\ \mcE(u)\comma 
\end{aligned}
\qquad u \in \FC{\infty}{c}{\infty}{c}\fstop
\end{equation}
Let now $u \in \dom{\mcE}$ be arbitrary and let $\seq{u_n}_n \subset \FC{\infty}{c}{\infty}{c}$ be such that $u_n$ $\mcE_1^{1/2}$-converges to $u$. By the inequality in \eqref{eq:cylcontr} we deduce that $\seq{u_n^\rad}_n$ is $\mcE_1^{1/2}$-Cauchy so that $\mcE_1^{1/2}$-converges to some $v \in \dom{\mcE}$ which has to coincide with $u^\rad$, since the radial projection is continuous in $L^2(\LP{\theta,\nu})$. We deduce in particular that $u^\rad \in \dom{\mcE}$. We are left to show that \eqref{eq:l:RadGrad:0} holds for $u \in \dom{\mcE}$. We have
\begin{align*}
\int \biggl \| (\gradH u^\rad)_\mu- &\int \int (\gradH u)_{\mu M \eta} \diff\eta\, \diff\DF{\nu}(\eta) \biggr \|^2_{T^\ver_\mu} \, \diff \LP{\theta,\nu}(\mu) =
\\
\le &\ 2\, \mcE^\ver(u^\rad-u_n^\rad) 
\\
&+2\, \int \int \left | \int \int (\gradH (u_n-u)_{\mu M \eta} \diff\eta\, \diff\DF{\nu}(\eta) \right |^2 \, \diff \mu\, \diff \LP{\theta,\nu}(\mu)
\\
= &\ 2\, \mcE^\ver(u^\rad-u_n^\rad) + 2 \left \| H(\emparg; u_n-u) \right \|^2_{L^2(\LP{\theta,\nu})} 
\\
\le &\ 2\, \mcE^\ver(u^\rad-u_n^\rad) + 2 \mcE^\ver(u_n-u) \to 0.
\end{align*}
This proves \eqref{eq:l:RadGrad:0} for a general $u \in \dom{\mcE}$, which in particular implies the inequality $\mcE^\ver(u^\rad) \le \mcE^\ver(u)$ because of the bound on the $L^2(\LP{\theta,\nu})$-norm of $H(\emparg; u)$. 
\end{proof}
\end{lemma}

More importantly, we have the following.

\begin{proposition}\label{p:OrthoProj}
The operator~$\emparg^\rad\colon L^2(\LP{\theta,\nu})\to L^2(\LP{\theta,\nu})$ is an $L^2(\LP{\theta,\nu})$-ortho\-gonal projection and
\begin{equation}\label{eq:p:OrthoProj:1}
\mcE(u^\rad, v)= \mcE(u,v^\rad) \comma \qquad u,v\in\dom{\mcE}\fstop
\end{equation}
In particular, $\emparg^\rad\colon \dom{\mcE}_1\to\dom{\mcE}_1$  is an $\mcE^{1/2}_1$-orthogonal projection.

\begin{proof}
We have already commented that~$\emparg^\rad$ is a projection operator, i.e.\ linear and idempotent, which holds irrespectively of the chosen domain.  To show that it is an $L^2(\LP{\theta,\nu})$-orthogonal projection it suffices to show that
\begin{equation}\label{eq:p:OrthoProj:2}
\scalar{u^\rad}{v}_{L^2(\LP{\theta,\nu})}=\scalar{u}{v^\rad}_{L^2(\LP{\theta,\nu})} \comma \qquad u,v\in L^2(\LP{\theta,\nu}) \comma
\end{equation}
which is easily verifiable. Regarding the second assertion, note that $\emparg^\rad\colon \dom{\mcE} \to \dom{\mcE}$ by Proposition \ref{l:RadGrad}, so that combining~\eqref{eq:p:OrthoProj:1} and~\eqref{eq:p:OrthoProj:2} proves that~$\emparg^\rad$ is an $\mcE^{1/2}_1$-orthogonal projection.  
We prove~\eqref{eq:p:OrthoProj:1} below.

By definition of~$\tparen{\mcE,\dom{\mcE}}$ and by the chain of equalities in~\eqref{eq:l:RadGrad:0},
\begin{align*}
\mcE(u^\rad,v)&= \int \int_M \braket{g\paren{(\gradW u^\rad)_\mu }{(\gradW v)_\mu} +  4  (\gradH u^\rad)_\mu (\gradH v)_\mu} \diff \mu\, \diff\LP{\theta,\nu}(\mu)
\\
&= 4 \int \int_M (\gradH u^\rad)_\mu (\gradH v)_\mu \diff\mu\, \diff\LP{\theta,\nu}(\mu) 
\\
&=4\int \int_\mcP \int_M (\gradH v)_{\mu} \int_M (\gradH u)_{\mu M \eta} \diff\eta\, \diff\DF{\nu}(\eta)\, \diff\mu\, \diff\LP{\theta,\nu}(\mu) \quad (\mu= t\eta')
\\
&= 4 \int_0^\infty \iint_\mcP \int_M (\gradH v)_{t\eta'} (\gradH u)_{t\eta} \, \diff\DF{\nu}(\eta)\, \diff\eta\, \diff\DF{\nu}(\eta')\, t\diff\lambda_\theta(t)\comma
\end{align*}
where we applied~\eqref{eq:InfLebesgueIsomor}. 
By an application of Fubini's Theorem we conclude that
\begin{align*}
\mcE(u^\rad,v)&=4\int_0^\infty t \iint_\mcP \iint_M (\gradH v)_{t\eta'} (\gradH u)_{t\eta}\, \diff\eta\, \diff\eta'\, \diff\DF{\nu}(\eta)\, \diff\DF{\nu}(\eta')\, \diff\lambda_\theta(t)\fstop
\end{align*}
As this latter expression is symmetric in~$\eta,\eta'$, applying the previous arguments in reverse order to~$v$ in place of~$u$ proves~\eqref{eq:p:OrthoProj:1}.
\end{proof}
\end{proposition}

Proposition~\ref{p:OrthoProj} implies that~$\emparg^\rad\colon L^2(\LP{\theta,\nu})\to L^2(\LP{\theta,\nu})$ and~$\emparg^\rad\colon \dom{\mcE}_1\to \dom{\mcE}_1$ are self-adjoint operators.
This fact, together with the intertwining property
\[
\dom{\mcE^\rad}=\dom{\mcE}^\rad \comma \qquad \mcE(u^\rad)=\mcE^\rad(u^\rad) \comma \quad u \in\dom{\mcE}\comma
\]
has the standard yet important consequence that~$\emparg^\rad$ \emph{intertwines} the semigroups, resolvents and generators corresponding to~$\tparen{\mcE,\dom{\mcE}}$ and~$\tparen{\mcE^\rad,\dom{\mcE^\rad}}$.
In the next result, we denote by~$\mcT_\bullet\eqdef\seq{\mcT_t}_{t\geq 0}$, resp.~$\mcG_\bullet\eqdef \seq{\mcG_\alpha}_{\alpha\geq 0}$, and~$\tparen{\mcL,\dom{\mcL}}$, the semigroup, resp.\ resolvent, and generator, of~$\tparen{\mcE,\dom{\mcE}}$ on~$L^2(\LP{\theta,\nu})$.
We adopt analogous notations for $\tparen{\mcE^\rad,\dom{\mcE^\rad}}$ \emph{on the Hilbert space $L^2(\LP{\theta,\nu})^\rad$}.

\begin{corollary}[Intertwining]\label{c:PreservationRadiality}
The operator~$\emparg^\rad$ intertwines~$\tparen{\mcE,\dom{\mcE}}$ on $L^2(\LP{\theta,\nu})$ with~$\tparen{\mcE^\rad,\dom{\mcE^\rad}}$ on $L^2(\LP{\theta,\nu})^\rad$.
That is, for every~$u\in L^2(\LP{\theta,\nu})$,
\begin{enumerate}[$(i)$]
\item\label{i:c:PreservationRadiality:3} $\mcT^\rad_t u^\rad=(\mcT_t u)^\rad$ for every~$t>0$;
\item\label{i:c:PreservationRadiality:4} $\mcG^\rad_\alpha u^\rad=(\mcG_\alpha u)^\rad$ for every~$\alpha>0$;
\item\label{i:c:PreservationRadiality:2} if additionally~$u\in\dom{\mcL}$ then~$u^\rad\in\dom{\mcL^\rad}$ and~$\mcL^\rad u^\rad=(\mcL u)^\rad$.
\end{enumerate}
\end{corollary}

\begin{lemma}\label{l:RadCore}
 For every $u \in \dom{\mcE^\rad}$ there exists a sequence of cylinder functions $\seq{u_n}_n$ such that $u_n \to u$ w.r.t.~$(\mcE^\rad)_1^{1/2}$.  In particular  $\FC{\infty,\infty}{c,c}{\infty}{c} \cap \dom{\mcE^\rad}$  is a form core for~$\tparen{\mcE^\rad,\dom{\mcE^\rad}}$.
\end{lemma}

\begin{proof}
Fix~$u\in\dom{\mcE^\rad}\subset \dom{\mcE}$. By Proposition \ref{prop:dolore} with the choices $\AA_1= \rmC_c^\infty(\R_0^+)$, $\AA_3= \rmC_c^\infty(M)$, and $\AA_2 = \R[t_1, t_2, \dotsc]_\fin$, we can find functions~$u_k\in \FC{\infty,\infty}{c,c}{\infty}{c}$ of the form
\[
u_k=(\nchi_k\circ \car^\trid) \cdot (F_k\circ \mbff_k^\trid) \comma \qquad \mbff_k=\seq{f_{n,1},\dotsc,f_{n,N_k}}\comma
\]
as in~\eqref{eq:CylinderF}, converging to~$u$ w.r.t.~$\mcE^{1/2}_1$, and additionally so that~$F_k$ coincides with a multivariate polynomial~$p_k$ on the hypercube
\[
\prod_{i=1}^{N_k} \tbraket{0,\norm{\phi_{n,i}}_{\rmC^0} \max_k\supp(\nchi_k)}\fstop
\]

\paragraph{Claim: ${u_k}^\rad\in \FC{\infty,\infty}{c,c}{\infty}{c}$}
For simplicity of notation, we drop the subscript~$k$ and write~$q\eqdef N_k$. Then
\begin{align*}
u^\rad = (\nchi\circ\car^\trid ) \int F\tparen{(\emparg M)\, \mbff^\trid(\eta)}\, \diff\DF{\nu}(\eta) = (\nchi\circ\car^\trid) \int p\tparen{\car^\trid\cdot \mbff^\trid(\eta) } \diff\DF{\nu}(\eta) \fstop
\end{align*}
Since~$\FC{\infty,\infty}{c,c}{\infty}{c}$ is a vector space, it suffices to show the assertion in the case when~$p\colon \seq{t_1,\dotsc, t_q}\mapsto \prod_j^q t_j^{n_j}$ is a monic monomial.
In this case,
\[
u^\rad = (\nchi\circ \car^\trid) \cdot (\car^\trid)^{n_1+\cdots + n_q} \int \tparen{p\circ \mbff^\trid}(\eta)\, \diff\DF{\nu}(\eta)\comma
\]
where the integral is a constant~$c_p$ independent of~$\mu$. (See~\cite[Cor.~3.5]{LzDSQua23} for an explicit computation of~$c_p$.)
Since~$\tilde\nchi\colon t \mapsto \nchi(t)\, t^{n_1+\cdots + n_q}$ is itself in~$\rmC^\infty_c(\R^+_0)$, this shows that~$u^\rad=c_p\tilde\nchi\circ\car^\trid$ is itself a cylindrical function, proving the claim.

\medskip

In order to conclude the assertion, it suffices to show that~${u_k}^\rad$ is $(\mcE^\rad)^{1/2}_1$-convergent to~$u$.
This follows from  Proposition \ref{p:OrthoProj}  and the assumption on~$\seq{u_k}_k$.
\end{proof}

\subsubsection{The radial-part process}
In the next propositions, we relate (properties of) $\tparen{\mcE,\dom{\mcE}}$ with (properties of) the Dirichlet form~$\tparen{E^\theta,\dom{E^\theta}}$ of the $\theta$-dimensional squared Bessel process discussed in Appendix~\ref{app:Bessel}.

\begin{proposition}\label{p:IsomorphicForms}
The operator
\begin{equation}
\tilde\emparg\colon L^2(\LP{\theta,\nu})^\rad \to L^2(\lambda_\theta) \comma \qquad u\longmapsto \tilde u \ : \ u=\tilde u\circ \car^\trid
\end{equation}
is a unitary operator.
Its non-relabeled restriction to~$\dom{\mcE^\rad}$ is a unitary operator~$\tilde\emparg\colon ( \frac{1}{4} \mcE^\rad,\dom{\mcE^\rad})\to\tparen{E^\theta,\dom{E^\theta}}$.
\end{proposition}

\begin{proof}
The fact that~$\tilde\emparg\colon L^2(\LP{\theta,\nu})^\rad \to L^2(\lambda_\theta)$ is unitary follows from~\eqref{eq:InfLebesgueIsomor}.  To show the second assertion, by virtue of  Lemma~\ref{l:RadCore} and Lemma~\ref{l:CBIcore}, it suffices to show that $\tilde\emparg\colon (\frac{1}{4}\mcE^\rad,\FC{\infty}{0}{\infty}{0}\mcM(M) \cap \dom{\mcE^\rad})\to\tparen{E^\theta,\rmC^\infty_c(\R^+_0)}$ is unitary w.r.t.~$(\frac{1}{4}\mcE^\rad)_1^{1/2}$ and $(E^\theta)_1^{1/2}$. Clearly, for every $\nchi \in  \rmC^\infty_c(\R_0^+)$ the function $u\eqdef\nchi\circ \car^\trid$ satisfies $\tilde{u}=\nchi$ so that surjectivity is proven.  Let now $u \in  \FC{\infty,\infty}{c,c}{\infty}{c} \cap \dom{\mcE^\rad}$  and observe that it has to be  of the form~$u= \nchi\circ \car^\trid$ for some~$\nchi\in \rmC^\infty_c( \R_0^+)$, which then also satisfies $\tilde{u}=\nchi$. Then, by~\eqref{eq:Fdiff},
\begin{align*}
 \tfrac{1}{4}\,\mcE^\rad(u)=&\  \tfrac{1}{4}\,\mcE(u)=\int \mu M\, \nchi'(\mu M)^2\, \diff\LP{\theta,\nu}(\mu) = \int_0^\infty t \nchi'(t)^2\, \diff \lambda_\theta(t)=  E^\theta(\tilde{u})\fstop \qedhere
\end{align*}
\end{proof}

\begin{lemma}\label{l:Conservativeness}
The form~$\tparen{\mcE,\dom{\mcE}}$ is conservative.
\begin{proof}
By the characterization of conservativeness in~\eqref{eq:Conservative:1},~\eqref{eq:Conservative:1c}, it suffices to show that there exist functions~$u_n\in \dom{\mcE}$ satisfying
\begin{equation}\label{eq:l:Conservativeness:1}
0\leq u_n \leq 1\comma \qquad \lim_{n\to\infty} u_n= 1  \as{\LP{\theta,\nu}}\comma \qquad \lim_{n\to\infty} \mcE(u_n,v)=0\comma
\end{equation}
for some~$v= \mcG_\alpha f$, with~$\alpha>0$ and~$f\in L^1(\LP{\theta,\nu})\cap L^2(\LP{\theta,\nu})$ and~$f>0$ $\LP{\theta,\nu}$-a.e..

Since~$\tparen{E^\theta,\dom{E^\theta}}$ is conservative by Proposition~\ref{p:CBI}, there exist functions~$\tilde u_n\in \dom{E^\theta}$ satisfying~\eqref{eq:Conservative:1},~\eqref{eq:Conservative:1b} (with~$\lambda_\theta$ in place of~$\mssm$ and~$E^\theta$ in place of~$E$.)
Let~$u_n$  be  defined by~$u_n\eqdef  \tilde u_n\circ \car^\trid$ and note that~$u_n=u_n^\rad$.
By~\eqref{eq:InfLebesgueIsomor}, it is readily verified~$u_n$ satisfies~\eqref{eq:Conservative:1} (with~$\LP{\theta,\nu}$ in place of~$\mssm$).

Now, fix~$\tilde f\in L^1(\lambda_\theta)\cap L^2(\lambda_\theta)$ with~$\tilde f>0$ $\lambda_\theta$-a.e., and let~$f$ be its radial extension to~$\mcM(M)$ defined as above.
Note that~$f=f^\rad$ satisfies~$f\in L^1(\LP{\theta,\nu})\cap L^2(\LP{\theta,\nu})$ and~$f>0$ $\LP{\theta,\nu}$-a.e..
Fix any~$\alpha>0$ and set~$v\eqdef \mcG_\alpha f$.
By standard properties of resolvents,~$\mcG_\alpha L^1(\LP{\theta,\nu})\subset L^1(\LP{\theta,\nu})$ and~$\mcG_\alpha L^2(\LP{\theta,\nu})\subset \dom{\mcE}$. 
Thus~$v\in\dom{\mcE}\cap L^1(\LP{\theta,\nu})$.

Since~$f=f^\rad$, we have~$v=v^\rad$ by Corollary~\ref{c:PreservationRadiality}\ref{i:c:PreservationRadiality:4}, and then~$v\in\dom{\mcE^\rad}$.
Thus, by Proposition~\ref{p:IsomorphicForms}, we have~$\tilde v\in\dom{E^\theta}$.
Furthermore,~$\tilde v\in L^1(\lambda_\theta)$ by~\eqref{eq:InfLebesgueIsomor}.
Finally, again by  Proposition~\ref{p:IsomorphicForms},
\begin{align*}
\lim_{n\to\infty} \mcE(u_n, v) =  \lim_{n \to  \infty} \mcE^{\rad}(u_n, v) =  \lim_{n \to  \infty} 4  E^\theta(\tilde u_n, \tilde v) = 0
\end{align*}
by~\eqref{eq:Conservative:1},~\eqref{eq:Conservative:1b} for~$\tilde u_n$ as shown above.
This shows~\eqref{eq:l:Conservativeness:1} and concludes the proof.
\end{proof}
\end{lemma}

\begin{lemma}\label{l:Recurrence}
The form~$\tparen{\mcE,\dom{\mcE}}$ is recurrent if and only if so is~$\tparen{E^\theta,\dom{E^\theta}}$.
\begin{proof}
 By Proposition~\ref{p:IsomorphicForms} it is enough to show that $\tparen{\mcE,\dom{\mcE}}$ is recurrent if and only if $\tparen{\frac{1}{4}\mcE^\rad,\dom{\mcE^\rad}}$ is recurrent. To prove this equivalence we will make use of the characterization of recurrence in \eqref{eq:l:Recurrence:1}.

Assume that there exist functions $u_n \in \dom{\mcE^\rad}$ satisfying \eqref{eq:l:Recurrence:1} with $\mssm=\LP{\theta,\nu}$ and $E=\mcE^\rad$. Since~$\dom{\mcE^\rad}\subset \dom{\mcE}$ and~$\mcE^\rad=\mcE$ on~$\dom{\mcE^\rad}$, \eqref{eq:l:Recurrence:1} holds also with $\mssm=\LP{\theta,\nu}$ and $E=\mcE$.

Assume now that there exist functions $v_n \in \dom{\mcE}$ satisfying~\eqref{eq:l:Recurrence:1} with $\mssm=\LP{\theta,\nu}$ and $E=\mcE$. It is readily seen that~${v_n}^\rad$ satisfies~\eqref{eq:l:Recurrence:1a} (with~$\LP{\theta,\nu}$ in place of~$\mssm$). 
Furthermore,~${v_n}^\rad$ satisfies
\[
\lim_n \mcE^\rad({v_n}^\rad)\leq \lim_n \mcE(v_n) =0
\]
 by Lemma \ref{l:RadGrad}. Thus,~\eqref{eq:l:Recurrence:1} (with $\mssm=\LP{\theta,\nu}$ and $E=\mcE^\rad$) holds with~$u_n\eqdef {v_n}^\rad\in\dom{\mcE^\rad}$.
\end{proof}
\end{lemma}

Recall that a Dirichlet form~$(E,\dom{E})$ on~$L^2(\mssm)$ is called \emph{transient} if there exists a function~$g\in L^1(\mssm)^+\cap L^\infty(\mssm)$ such that
\begin{equation}\label{eq:Transient}
\int \abs{u} g\, \diff\mssm \leq E(u)^{1/2}\comma \qquad u\in\dom{E}\fstop
\end{equation}
Also note that, if~$(E,\dom{E})$ is additionally local, then it is enough to check \eqref{eq:Transient} on non-negative functions $u \in \dom{E}$.

\begin{lemma}\label{l:Transient}
If the form~$\tparen{E^\theta,\dom{E^\theta}}$ is transient, then so is the form~$\tparen{\mcE,\dom{\mcE}}$.

\begin{proof}
Assume~$\tparen{E^\theta,\dom{E^\theta}}$ is transient, and let~$ \tilde g_0  \colon \R_+\to\R$ be as in~\eqref{eq:Transient} for~$\tparen{E^\theta,\dom{E^\theta}}$.
Further let~$g\eqdef  2 \, (\tilde g_0\circ\car^\trid)\colon \mcM(M)\to\R$ be its lift to~$\mcM(M)$.  Note that $g\in L^1(\LP{\theta,\nu})^+\cap L^\infty(\LP{\theta,\nu})$.  
Then, for every non-negative~$u\in\dom{\mcE}$,
\begin{align*}
\int {u} g\, \diff \LP{\theta,\nu} =& \int {u} g^\rad\, \diff\LP{\theta,\nu} = \int {u}^\rad g\, \diff\LP{\theta,\nu}= \int_0^\infty {\widetilde{u^\rad}} \tilde{g} \,\diff\lambda_\theta
\\ 
 =& \ 2 \, \int_0^\infty {\widetilde{u^\rad}} \tilde{g_0}  \leq  2  \, \sqrt{ E^\theta\tparen{\widetilde{u^\rad}} }=  \mcE^\rad(u^\rad)^{1/2} \leq   \mcE(u)^{1/2}\comma
\end{align*}
where we used, respectively, that:~$\emparg^\rad$ is self-adjoint on~$L^2(\LP{\theta,\nu})$ by Proposition~\ref{p:OrthoProj}; that~$\tilde\emparg\colon L^2(\LP{\theta,\nu})^\rad\to L^2(\lambda_\theta)$ is unitary by Proposition~\ref{p:IsomorphicForms}; that $\tparen{E^\theta,\dom{E^\theta}}$  is transient  by assumption, so that we can apply~\eqref{eq:Transient}  with $u = \widetilde{u^\rad} \in \dom{E^\theta}$; that~$\tilde\emparg\colon ( \frac{1}{4}  \mcE^\rad,\dom{\mcE^\rad})\to\tparen{E^\theta,\dom{E^\theta}}$ is unitary  by Proposition~\ref{p:IsomorphicForms}; and finally  Lemma~\ref{l:RadGrad} for the last inequality. 
\end{proof} \end{lemma}

%

\appendix
\section{The Dirichlet form of the squared Bessel process}\label{app:Bessel}
In this Appendix we construct the Dirichlet forms associated with \emph{squared Bessel processes}.
It is helpful to regard these processes as parametrized by their dimension, that is as a class of \emph{Continuous-state Branching Processes with Immigration} (shortly: \emph{CBI processes}).
A general treatment of CBI processes by means of Feller theory is classical and can be found in~\cite{KawWat71,ShiWat73}.
CBI processes form one-parameter families (in fact: convolution semigroups)~$\tseq{\mbbP^\theta_{a,b,c}}_{\theta\geq 0}$ of diffusion processes on the real line, indexed by real parameters~$a,b,c$ with~$a,c\geq 0$, cf.~\cite[Thm.~1.2 and Eqn.~(1.23)]{ShiWat73}.

Let~$\seq{W_t}_t$ be a standard Brownian motion on~$\R$ and denote by~$x^+\eqdef x \vee 0$ the positive part.
Here, we choose parameters~$a=1$,~$b=0$, and~$c=1$, resulting in the \emph{squared Bessel process of dimension~$\theta$}, the unique strong solution to the~\textsc{sde}
\begin{equation}\label{eq:SDE}
\diff x_t= \sqrt{2 x_t^+\, } \diff W_t + \theta\diff t\comma \qquad t> 0 \fstop
\end{equation}

Now, consider the quadratic form on~$L^2(\lambda_\theta)$ given by
\begin{equation}\label{eq:RadialForm}
\begin{gathered}
\dom{E^\theta}\eqdef \set{f\in \mathrm{AC}_\loc( \R^+)\cap L^2(\lambda_\theta): \int_0^\infty t\, f'(t)^2\, \diff\lambda_\theta<\infty }\comma 
\\
E^\theta(f,g)\eqdef \int_0^\infty t\, f'(t)\, g'(t)\, \diff\lambda_\theta(t) \fstop
\end{gathered}
\end{equation}

\begin{lemma}\label{l:CBI}
The form in~\eqref{eq:RadialForm} is a Dirichlet form on~$L^2(\lambda_\theta)$, with generator
\begin{equation}\label{eq:l:CBI:0}
\begin{gathered}
\dom{L^\theta}\eqdef \set{f\in \dom{E^\theta} : f'\in \mathrm{AC}_\loc(\R^+)\text{ and } \exists \lim_{t\downarrow 0} t^\theta\, f'(t)=0}\comma
\\
(L^\theta f)(t) \eqdef t f''(t) +\theta f'(t) \fstop
\end{gathered}
\end{equation}

\begin{proof}
In order to show that~\eqref{eq:RadialForm} is a Dirichlet form, it suffices to show that it coincides, in the notation of~\cite{FukOshTak11}, with the form~$\tparen{\mcE^+,\dom{\mcE^+}}$ in~\cite[Eqn.s~(3.3.20)--(3.3.21), p.~134]{FukOshTak11} for the choice~$D=\R^+$,~$\rho=\rho_\theta=\Gamma(\theta)^{-1}t^{\theta-1}$, and~$a(t)=\Gamma(\theta)^{-1}t^\theta$.
Indeed, by a standard integration by parts on~$\rmC^\infty_c(\R^+)$ (i.e.\ with vanishing boundary terms), we see that
\begin{equation}\label{eq:l:CBI:1}
E^\theta(\phi,\psi)= -\int_0^\infty \phi L^\theta \psi\, \diff \lambda_\theta(t)\comma \qquad \phi,\psi\in\rmC^\infty_c(\R^+) \fstop
\end{equation}
It then follows from~\cite[Thm.~3.3.1, p.~135]{FukOshTak11} that the form~$\tparen{E^\theta,\dom{E^\theta}}$ is a Dirichlet form corresponding to the maximal Markovian non-positive self-adjoint extension of~$\tparen{L^\theta,\rmC^\infty_c(\R^+)}$ with core~$\dom{E^\theta}\cap \rmC^\infty(\R^+)$ by~\cite[Lem.~3.3.3, p.~134]{FukOshTak11}.
Note that~$\rmC^\infty_c(\R^+_0)$ embeds injectively and continuously into~$\dom{E^\theta}$ by restriction to~$\R^+$.
We always regard functions in~$\rmC^\infty_c(\R^+_0)$ as elements of~$\dom{E^\theta}$ up to this identification.

Now, consider the space
\begin{align}\label{eq:l:CBI:2}
\msV_\theta&\eqdef\set{\psi\in \mathrm{AC}_\loc(\R^+) : \psi'\in\mathrm{AC}_\loc(\R^+)\comma \exists a_\psi\eqdef \lim_{t\downarrow 0} t^\theta\psi'(t)\in\R}\fstop
\end{align}

By integration by parts with~$\phi\in \rmC^\infty_c(\R^+_0)\subset\dom{E^\theta}$ and~$\psi\in\msV_\theta$,
\begin{equation}\label{eq:l:CBI:3}
E^\theta(\phi,\psi)= -\int_0^\infty \phi L^\theta \psi\, \diff \lambda_\theta(t) - \Gamma(\theta)^{-1}\braket{\lim_{t\downarrow 0} t^\theta \psi'(t)}\phi(0)\comma
\end{equation}
so that the distributional generator of~$\tparen{E^\theta,\rmC^\infty_c(\R^+_0)}$ is formally given by
\[
L^\theta\psi+\Gamma(\theta)^{-1}\braket{\lim_{t\downarrow 0} t^\theta \psi'(t)}\delta_0\comma \qquad \psi\in\msV_\theta\comma
\]
where~$\delta_0$ is the linear functional on~$\rmC^\infty_c(\R^+_0)$ defined by~$\delta_0\phi\eqdef \lim_{t\downarrow 0}\phi(t)$.
Since $\rmC^\infty_c(\R^+_0)\subset\dom{E^\theta}$, this shows that~$\dom{L^\theta}\subset \msV_\theta\cap\dom{E^\theta}$.

Finally, let~$\psi\in\msV_\theta$.
On the one hand, since~$\lambda_\theta$ is absolutely continuous w.r.t.~$\Leb{1}$, imposing~$L^\theta\psi\in L^2(\lambda_\theta)$ implies that~$a_\psi=0$.
As a consequence
\[
\dom{L^\theta}\subset \set{\psi\in \msV_\theta : a_\psi=0}\cap\dom{E^\theta}\fstop
\]
On the other hand, the integration by parts in~\eqref{eq:l:CBI:3} extends to~$\phi\in\dom{E^\theta}$ and~$\psi\in \set{\psi\in \msV_\theta : a_\psi=0}\cap\dom{E^\theta}$, so that~$\dom{L^\theta}= \set{\psi\in \msV_\theta : a_\psi=0}\cap\dom{E^\theta}$.
This concludes the proof of~\eqref{eq:l:CBI:0}. 
\end{proof}
\end{lemma}

\begin{proposition}\label{p:CBI}
The form~$\tparen{E^\theta,\dom{E^\theta}}$ is a regular conservative strongly local Dirichlet form on~$L^2(\lambda_\theta)$, properly associated with the squared Bessel process solving the \textsc{sde}~\eqref{eq:SDE}.
\end{proposition}

\begin{lemma}\label{l:CBIcore}
$\rmC^\infty_c(\R^+_0)$ is a core for~$\tparen{E^\theta,\dom{E^\theta}}$.
\begin{proof}
As in the proof of Lemma~\ref{l:CBI}, we regard~$\rmC^\infty_c(\R^+_0)$ as a subset of~$\dom{E^\theta}$. 
Let~$F_\theta$ denote the $E^\theta_1$-closure of~$\rmC^\infty_c(\R^+_0)$.
It suffices to show that~$F_\theta^\perp=\set{0}$, where~$F_\theta^\perp$ denotes the $E^\theta_1$-orthogonal complement to~$F_\theta$ in~$\dom{E^\theta}$.
Note that~$f\in F_\theta^\perp$ if and only if
\[
\int_0^\infty t f'(t) \phi'(t) \diff\lambda_\theta(t)+\int_0^\infty f(t) \phi(t)\, \diff\lambda_\theta(t) =0\comma \qquad \phi\in \rmC^\infty_c(\R^+_0) \comma 
\]
so that~$f$ is a distributional solution in~$\msV_\theta$ to~$L^\theta f= f$, where~$\msV_\theta$ is defined in~\eqref{eq:l:CBI:2}.

It follows from e.g.~\cite[14.1.2.62, p.~526 or 14.1.2.108, p.~531]{PolZai18} that solutions to~$L^\theta f=f$ form a two-dimensional linear space.
A basis for the space of solutions is expressible in terms of Bessel functions, see~\cite[\emph{ibid}.]{PolZai18}.
We rather express them in terms of the confluent hypergeometric function~${}_0F_1$ and of its regularized form~${}_0\widetilde F_1$, respectively defined by
\begin{align}
\label{eq:Hypergeom}
{}_0F_1(;a;z) &\eqdef \sum_{k=0}^\infty \frac{z^k}{\langle a\rangle_k k!}\comma \qquad z\in \C\comma \quad a\in\C\setminus \Z^-_0\comma
\\
\nonumber
{}_0\widetilde F_1(;a;z)&\eqdef \frac{{}_{0}F_1(;a;z)}{\Gamma(a)} \comma \qquad a,z\in \C \comma
\end{align}
where~$\Gamma$ is Euler's Gamma function and~$\langle a\rangle_k\eqdef\Gamma(a+k)/\Gamma(a)$ is the Pochhammer symbol.
Indeed, setting
\begin{equation}\label{l:CBIcore:1}
f_0(t)\eqdef {}_0F_1(;\theta;t) \qquad \text{and} \qquad f_1(t)\eqdef t^{1-\theta}{}_0\widetilde F_1(;2-\theta;t)
\end{equation}
and differentiating the series expression in~\eqref{eq:Hypergeom}, it is readily verified that~$f_0$ and~$f_1$ are linearly independent solutions to~$L^\theta f=f$.
Thus, each solution to~$L^\theta f=f$ has the form~$a_0 f_0+a_1 f_1$ for some real constants~$a_0,a_1$.

Since~$\lim_{t\to\infty} t^{\theta-1} f_i(t)^2=\infty$, we have~$f_i\notin L^2(\lambda_\theta)$ for~$i=0,1$ and every~$\theta>0$.
This shows that solutions in~$\msV_\theta$ to~$L^\theta f=f$ are never in~$L^2(\lambda_\theta)$.
As a consequence,~$F_\theta^\perp=\set{0}$ as desired, and the conclusion follows.
\end{proof}
\end{lemma}

\begin{proof}[Proof of Proposition~\ref{p:CBI}]
Consider the quadratic form in~\eqref{eq:RadialForm}.
Closedness and the Markov property were proved in Lemma~\ref{l:CBI}.
Regularity follows from Lemma~\ref{l:CBIcore}.
Strong locality follows by inspection.
The part~$\tparen{L^\theta,\rmC^\infty_c(\R^+_0)}$ of~$\tparen{L^\theta,\dom{L^\theta}}$ on~$\rmC^\infty_c(\R^+_0)$ is thus the generator of the CBI process solving~\eqref{eq:SDE} by, e.g.,~\cite[Eqn.~(1.25)]{ShiWat73}.

For the converse implications, we argue as follows.
 
\paragraph{Distributional adjoint of the generator}
Consider the generator~$\tparen{L^\theta,\rmC^\infty_c(\R^+_0)}$ and note that, again by a standard integration by parts,
\begin{align*}
\int_0^\infty \psi\, L^\theta \phi\, \diff t &= \phi' t \psi\Big\lvert_0^\infty - \int_0^\infty (\phi'\psi + \phi' t\psi')\diff t +\theta\phi\psi\Big\lvert_0^\infty -\int_0^\infty\theta \phi\psi'\diff t
\\
&= 0 - \int_0^\infty \phi'(\psi+t\psi')\diff t - \theta \phi(0)\psi(0) - \theta\int_0^\infty\phi\psi'\diff t
\\
&= \int_0^\infty \phi (2\psi'-\theta\psi'+t\psi'')\diff t - \phi(\psi+t\psi')\Big\lvert_0^\infty -\theta\phi(0)\psi(0)
\\
&= \int_0^\infty \phi \tparen{(2-\theta)\psi'+t\psi''}\diff t -(\theta+1)\phi(0)\psi(0) \fstop
\end{align*}
Thus, the distributional adjoint~$(L^{\theta})^*$ of~$\tparen{L^\theta,\rmC^\infty_c(\R^+_0)}$ is given by
\[
(L^{\theta})^*\psi = (2-\theta)\psi' +t\psi''-(\theta+1)\delta_0 \fstop
\]

\paragraph{Invariant measures}
The unique positive stationary solutions~$g_\theta\eqdef a_\theta t^{\theta-1}$ on~$(0,\infty)$, $a_\theta>0$, to the corresponding Kolmogorov forward equation~$(L^{\theta})^*g_\theta \equiv 0$ are thus the density of a candidate invariant measure~$\diff \mu_\theta=g_\theta\diff t$ for~$L^\theta$.
Furthermore, in light of~\cite[Eqn.~(1.23) and Thm.~1.1]{ShiWat73},~$\seq{\mu_\theta}_{\theta>0}$ must be a convolution semigroup.
This property, together with~$(L^{\theta})^*\mu_\theta =0$, implies that the function~$\theta\mapsto a_\theta$ must satisfy the functional equation
\begin{equation}\label{eq:p:CBI:2}
\frac{a_\theta a_\tau}{a_{\theta+\tau}}=\frac{\Gamma(\theta+\tau)}{\Gamma(\theta)\Gamma(\tau)} \fstop
\end{equation}

For the solution~$a_\theta\eqdef \Gamma(\theta)^{-1}$ to~\eqref{eq:p:CBI:2} we then have that~$\mu_\theta=\lambda_\theta$ is an invariant measure and such that~$\seq{\mu_\theta}_{\theta>0}$ is a convolution semigroup on~$\R^+$.

\paragraph{Dirichlet form}
By the integration by parts~\eqref{eq:l:CBI:1} we may apply Lemma~\ref{l:CBI}.
Thus,~$\tparen{E^\theta,\dom{E^\theta}}$ is the strongly local Dirichlet form generated by~$\tparen{L^\theta,\dom{L^\theta}}$ and the latter is the maximal Markovian self-adjoint extension of~$\tparen{L^\theta,\rmC^\infty_c(\R^+)}$.
The form is conservative since so is the associated CBI process solving~\eqref{eq:SDE}, which is in turn a consequence of~\cite[Thm.~1.2, p.~41]{KawWat71}, as noted in~\cite[Ex.~1.1, p.~42, after~(1.24)]{KawWat71}.
\end{proof}

The next result follows from known properties of the squared Bessel process.
For \emph{integer}~$\theta>0$, it is a trivial consequence of the representation of the squared Bessel process as the squared Euclidean norm of a standard Brownian motion.
For completeness, we provide a proof by Dirichlet-form methods.

\begin{lemma}\label{l:CBIRecurrence}
The form~$\tparen{E^\theta,\dom{E^\theta}}$ is irreducible.
It is recurrent if~$\theta\in (0,1]$ and transient if~$\theta\in (1,\infty)$.
\end{lemma}
\begin{proof}
By Proposition~\ref{p:CBI}, the form~$\tparen{E^\theta,\dom{E^\theta}}$ is properly associated with the squared Bessel process.
As it is well-known that the latter is irreducible, so is $\tparen{E^\theta,\dom{E^\theta}}$.
In light of the transience/recurrence dichotomy for irreducible Dirichlet forms, e.g.~\cite[Lem.~1.6.4(iii), p.~55]{FukOshTak11}, it suffices to show that~$\tparen{E^\theta,\dom{E^\theta}}$ is recurrent if and only if~$\theta\in (0,1]$.

Assume first that~$\theta\leq 1$.
Let~$u\in\rmC^\infty_c(\R^+_0)$ be so that~$u(0)=1$.
Further set~$u_n(t)\eqdef u\tparen{(1+t)^{1/n}-1}$ for~$t\geq 0$.
Then,~$u_n\in\rmC^\infty_c(\R^+_0)\subset \dom{E^\theta}$ and $\lim_n u_n(t)=u(0)=1$ for every~$t\geq 0$, which verifies~\eqref{eq:l:Recurrence:1a}.
Furthermore, since~$\theta \leq 1$,
\begin{align*}
\int_0^\infty \abs{u_n'(t)}^2 t^\theta \diff t =&\ \frac{1}{n}\int_0^\infty \abs{u'(s)}^2 (s+1)^{1-n} \tparen{(s+1)^n-1}^\theta\diff s
\\
\leq&\ \frac{1}{n} \int_0^\infty \abs{u'(s)}^2(s+1)\diff s \fstop
\end{align*}
Since~$u$ has compact support, the latter integral is a constant~$c_u>0$.
It follows that
\[
\limsup_n E^\theta(u_n) \leq \limsup_n \frac{c_u}{\Gamma(\theta)\, n}=0\comma
\]
which verifies~\eqref{eq:l:Recurrence:1b}.
This concludes the assertion by the above characterization of recurrence.

Assume now that~$\theta>1$ and argue by contradiction that there exist functions~$u_n$ satisfying~\eqref{eq:l:Recurrence:1} (with~$\lambda_\theta$ in place of~$\mssm$).
Since~$u_n\in\dom{E^\theta}\subset \mathrm{AC}_\loc(\R^+)$, we will always consider~$u_n$ as its unique continuous representative.
Without loss of generality, up to truncation, we may and will assume that~$u_n\geq 0$.

In particular, there exists~$t_0>0$ such that~$\lim_n u_n(t_0)=1$.
Furthermore, since~$u_n\in L^2(\lambda_\theta)$ and~$\theta\geq 1$, there exists an increasing sequence~$\seq{t_n}_n$ with~$t_0\leq t_n\nearrow \infty$ and such that~$u_n(t_n)\leq 1/2$.
On the one hand, by the fundamental theorem of calculus,
\begin{equation}\label{eq:l:RecurrenceCBI:3}
\limsup_n \abs{\int_{t_0}^{t_n} u_n'(t) \diff t}^2 = \limsup_n \abs{u_n(t_n)-u_n(t_0)}^2 \geq 1/4\fstop
\end{equation}

On the other hand, since~$\theta>1$,
\begin{align*}
\abs{\int_{t_0}^{t_n} u_n'(t) \diff t}^2\leq& \int_{t_0}^{t_n} \abs{u_n'(t)\, t^{\theta/2}}^2\diff t \int_{t_0}^{t_n} t^{-\theta} \diff t \leq \int_0^\infty \abs{u_n'(t)}^2 t^\theta \diff t \int_{t_0}^\infty t^{-\theta} \diff t
\\
=&\ \Gamma(\theta)\, \frac{t_0^{1-\theta}}{\theta-1}\, E^\theta(u_n)\comma
\end{align*}
so that, by the assumption in~\eqref{eq:l:Recurrence:1b},
\begin{align*}
\limsup_n \abs{\int_{t_0}^{t_n} u_n'(t)\diff t}^2 \leq \limsup_n \Gamma(\theta-1)\, t_0^{1-\theta}\, E^\theta(u_n) = 0\comma
\end{align*}
which contradicts~\eqref{eq:l:RecurrenceCBI:3} and concludes the assertion.
\end{proof}

\section{Measure-preserving diffeomorphisms}\label{app:diffeo}
Let~$(M,g)$ be a smooth  connected, orientable  Riemannian manifold with Riemannian volume measure~$\vol_g$. 
We collect here some auxiliary results about the group~$\Diff^+_0(g)$ of  all compactly non-identical, orientation-preserving, $\vol_g$-preserving diffeomorphisms.

Firstly, let us recall that~$\Diff^+_0(g)$ is the (infinite-dimensional) Lie group corresponding to the Lie algebra of $\div_g$-free vector fields on~$M$ with the Lie derivative as its Lie bracket.
Let us further recall some virtually well-known results about the natural action of~$\Diff^+_0(g)$ on~$M$.

The following may be easily inferred from the arguments in~\cite[\S3]{Boo69}.

\begin{lemma}[Extension lemma]\label{l:ExtensionLemma}
Let~$(M,g)$  be in addition open or boundaryless, and~$K\subset M$ be any contractible compact subset.
Then, every smooth vector field on~$K$ has a  compactly supported,  $\div_g$-free extension to the whole of~$M$.
\end{lemma}

As an immediate consequence, we see that the Lie algebra of $\div_g$-free vector fields is infinite-dimensional, thus so is~$\Diff^+_0(g)$.

\begin{proposition}[Transitivity,~{\cite[Thm.~A, \S3, p.~98]{Boo69}}]\label{p:Transitive}
$\Diff^+_0(g)$ acts $k$-transitively on~$M$ for every~$k\in\N_1$. In particular, it acts transitively on~$M$.
\end{proposition}

\subsection{Actions on measures}
In the following, let~$B_1\subset \R^d$ be the open unit ball equipped with the standard Euclidean metric~$g_e$.

\begin{lemma}\label{l:Ball}
Let~$\nu$ be any finite measure on~$B_1$ invariant for the \eqref{eq:ActionShifts}-action of~$\Diff^+_0( g_e  )$.
Then~$\nu\propto \Leb{d}\restr{B_1}$. In particular,~$\nu$ has no singular continuous part.
\end{lemma}

\begin{proof}
Let~$Q_1,Q_2\subset B_1$ be arbitrary closed cubes  with equal volume, define~$K$ as the closed convex hull of~$Q_1\cup Q_2$, and let~$w$ be the vector field on~$K$ defining the  translation  of~$Q_1$ to~$Q_2$.
By Lemma~\ref{l:ExtensionLemma} applied to~$K$ and~$w$, there exists a  compactly supported  $\div_{g_e}$-free vector field~$w'$ on~$M$ extending~$w$ on~$K$.
Then, the flow of~$w'$ at time~$1$ is a  compactly non-identical  orientation-preserving, $\Leb{d}$-preserving diffeomorphism on~$M$ mapping~$Q_1$ to~$Q_2$.
In other words,~$\Diff^+_0( g_e  )$ acts transitively on all closed cubes in~$B_1$, hence, by continuity, on  the family $\mathcal{C}$ of all semi-closed cubes~$Q$ of the form~$Q\eqdef x+q\, [0,1)^d$ with $x\in \R^d$, $q \in \Q^+$ and such that $\overline{Q} \subset B_1$.

Since~$\nu$ is invariant, by a standard decomposition argument, for every rational~$q\in\Q^+$ and every  $Q \in \mathcal{C}$ with~$qQ\in \mathcal{C}$  we have~$\nu (q Q)= q^d  \nu Q  $.
Now,  let $Q_0 \in \mathcal{C}$ be fixed, and set~$c\eqdef \nu Q_0/\Leb{d}Q_0$.
For every  $Q_1 \in \mathcal{C}$  there exists~$q\in \Q^+$ and~$x\in\R^d$ so that~$Q_1=x+q Q_0$.
For such~$q$, by invariance of~$\nu$ and of~$\Leb{d}$ (also under rescaling),
\[
\nu Q_1 = \nu (qQ_0) = q^d \nu Q_0 = q^d c \Leb{d} Q_0 = c \Leb{d}(q Q_0) = c \Leb{d} Q_1 \fstop
\]
Since  $\mathcal{C}$ is  a $\pi$-system generating the Borel $\sigma$-algebra of~$B_1$, it follows by a standard monotone class argument that~$\nu=c\Leb{d}$ as Borel measures on~$B_1$, which concludes the proof.
\end{proof}

\begin{proposition}\label{p:InvariantVolG}
Let~~$\nu$ be any finite Borel measure on~$M$ invariant for the \eqref{eq:ActionShifts}-action of~$\Diff^+_0(g)$.
Then,~$\nu\propto\vol_g$.
\end{proposition}

\begin{proof}
Let~$\nu$ be invariant, with Lebesgue decomposition~$\nu=\nu_{\pa} + \nu_{ac}+\nu_{sc}$.
Since~$\nu$ is a finite invariant measure, and since~$\Diff^+_0(g)$ acts transitively on~$M$ by Proposition~\ref{p:Transitive}, a simple argument shows that~$\nu_{\pa}=0$.

\paragraph{Claim~1: $\nu_{ac}$ and~$\nu_{sc}$ are both invariant}
By Lebesgue decomposition, there exist Borel subsets~$M_{ac},M_{sc}\subset M$ with
\[
M_{ac}\cup M_{sc}=M\comma \qquad M_{ac}\cap M_{sc}=\emp\comma \qquad \vol_gM_{sc}=0 \comma \qquad \nu_{sc} M_{ac}=0 \fstop
\]
Let~$\nu_{ac}=\rho\vol_g$ and fix an arbitrary~$\psi\in\Diff^+_0(g)$.
Since~$\vol_g M_{sc}=0$, we have
\begin{equation}\label{eq:p:FullIdentification:1}
\nu_{ac} M_{sc}=0 \qquad \text{and} \qquad \psi_\pfwd \nu_{ac} M_{sc} = \int_{M_{sc}} \rho\circ\psi^{-1} \diff\psi_\pfwd\vol_g = \int_{M_{sc}} \rho\circ \psi^{-1}\diff\vol_g=0 \fstop
\end{equation}
Furthermore, since~$\nu$ is invariant, for every Borel~$A\subset M$,
\[
\psi_\pfwd\nu_{ac}A+\psi_\pfwd \nu_{sc} A = \psi_\pfwd \nu A = \nu A = \nu_{ac} A + \nu_{sc} A\comma
\]
whence
\begin{equation}\label{eq:p:FullIdentification:2}
\psi_\pfwd\nu_{ac}A - \nu_{ac} A =  \nu_{sc} A - \psi_\pfwd \nu_{sc} A\fstop
\end{equation}

Combining~\eqref{eq:p:FullIdentification:1} and~\eqref{eq:p:FullIdentification:2}, we conclude that
\begin{equation}\label{eq:p:FullIdentification:3}
\psi_\pfwd \nu_{sc} A= \nu_{sc} A\comma \qquad A \quad \text{Borel}\comma A\subset M_{sc}\fstop
\end{equation}
Choosing~$A=M_{sc}$ in~\eqref{eq:p:FullIdentification:3}, we further have
\begin{align}
\nonumber
\psi_\pfwd \nu_{sc} M_{ac}=&\ \psi_\pfwd \nu_{sc} (M\setminus M_{sc}) = \psi_\pfwd \nu_{sc} M -\psi_\pfwd \nu_{sc} M_{sc} = \nu_{sc}M - \psi_\pfwd \nu_{sc} M_{sc}
\\
\label{eq:p:FullIdentification:4}
=&\ \nu_{sc}M_{sc} - \psi_\pfwd \nu_{sc} M_{sc} =0
\fstop
\end{align}
Since~$M_{ac}$ and~$M_{sc}$ form a (disjoint) partition of~$M$, combining~\eqref{eq:p:FullIdentification:3} with~\eqref{eq:p:FullIdentification:4} shows that~$\nu_{sc}$ is invariant.

Then, for every~$\psi\in \Diff^+_0(g)$,
\[
\nu_{ac} + \nu_{sc} = \nu = \psi_\pfwd\nu = \psi_\pfwd \nu_{ac} +\psi_\pfwd \nu_{sc} = \psi_\pfwd\nu_{ac} + \nu_{sc} \comma
\]
and cancelling~$\nu_{sc}$ shows that~$\nu_{ac}$ too is invariant.
This concludes the proof of the claim.

\medskip

\paragraph{Claim 2:~$\nu_{sc}=0$}
Again since~$\Diff^+_0(g)$ acts transitively on~$M$ by Proposition~\ref{p:Transitive}, and since each element of~$\Diff^+_0(g)$ is an open map, it is not difficult to show that either~$\nu_{sc}=0$, or~$\supp \nu_{sc}=M$.
Thus, it suffices to show that the restriction of~$\nu_{sc}$ to some small $\mssd_g$-ball in~$M$ is identically vanishing.
Note that the Lebesgue decomposition of~$\nu$ is preserved by push-forward via local charts.
Thus, without loss of generality up to push-forward by a chart and rescaling, the assertion is equivalent to the same assertion for the standard Euclidean unit ball, which is shown in Lemma~\ref{l:Ball}.

\paragraph{Claim 3:~$\nu_{ac}\propto \vol_g$}
For~$\nu_{ac}=\rho \vol_g$ we have 
\[
\rho \vol_g= \nu_{ac} =\psi_\pfwd\nu_{ac} = (\rho\circ \psi^{-1}) \psi_\pfwd \vol_g = (\rho\circ\psi^{-1}) \vol_g 
\comma \qquad \psi\in \Diff^+_0(g) \fstop
\]
Again since~$\Diff^+_0(g)$ acts transitively on~$M$, we conclude that~$\rho$ is $\vol_g$-a.e.\ constant, thus~$\nu_{ac}\propto \vol_g$.
This concludes the proof since~$\nu_{sc}=\nu_{\pa}=0$ by the previous claims.
\end{proof}

\medskip
\paragraph{\em\bfseries Data availability statement}
Data sharing not applicable to this article as no datasets were generated or analysed in the current study.

\bibliographystyle{abbrvurl}
\bibliography{2025-02-16-SobolevHK}
\end{document}